\documentclass[reqno]{amsart}
\usepackage{hyperref}
\usepackage{amssymb}
\usepackage{mathrsfs}
\usepackage{esint}
\usepackage{graphicx}
\usepackage{wrapfig}
\usepackage{float}
\usepackage{inputenc} 
\usepackage{amsmath}
\usepackage{bm}
\numberwithin{equation}{section}
\newtheorem{theorem}{Theorem}[section]
\newtheorem{lemma}[theorem]{Lemma}

\newtheorem{remark}[theorem]{Remark}

\numberwithin{equation}{section}
\newtheorem{thm}{Theorem}[section]
\newtheorem{lema}[thm]{Lemma}

\title[]
{Classical solutions to the Boltzmann equations for  gas mixture with unequal molecular masses}

\author[
]{Gaofeng Wang, Weike Wang and Tianfang Wu}

\address{G.F. Wang. School of Mathematics and Statistics, Nantong University,
	Nantong 226019, P. R. China.}
\email{gfwang@ntu.edu.cn}

\address{W.K. Wang. School of Mathematical Sciences, CMA-Shanghai and Institute of Natural
	Science, Shanghai Jiao Tong University, Shanghai, 200240, P.R.China.}
\email{wkwang@sjtu.edu.cn}

\address{T.F. Wu. (corresponding author). Wenzhou Business College, Wenzhou 325000, P. R. China.}
\email{20249237@wzbc.edu.cn}

\thanks{}

\begin{document}	
	
	\maketitle
	\begin{abstract}
	‌The Boltzmann equation is essential for gas thermodynamics, as it models how the molecular density distribution $F(t,x,v)$ changes over time. However, existing research primarily focuses on the single-species Boltzmann equation, while investigations into gas mixtures with unequal molecular masses remain relatively limited. Notably, mixed-gas studies have broader applications---exemplified by Earth's atmosphere, composed of 78\% nitrogen, 21\% oxygen, and 1\% trace gases, where the $N_2$ to $O_2$ molecular mass ratio is 28:32 (simplified as 7:8). This work addresses the Boltzmann equations for such mixtures with unequal molecular masses $(m^A\neq m^B)$, establishing the global-in-time existence of classical solutions near Maxwellians for soft potentials ($-3<\gamma<0$) in a periodic spatial domain. Our analysis encompasses arbitrary molecular mass ratios. ‌The main contribution of this paper lies in ‌the‌ detailed characterization of the linear collision operator's structure and establishing estimates for the nonlinear terms under unequal mass conditions.  ‌Consequently, these results‌ may help advance‌ spectral analysis for soft potentials ‌as well as $L^2-L^{\infty}$ frameworks in future studies of multi-component Boltzmann equations. \\
		
		 \vspace{5pt}
	\end{abstract}
	
	\allowdisplaybreaks

	\section{Introduction and main result}
	
	\subsection{Introduction}
	We consider the Boltzmann equations for two species of particles
	\begin{equation}\label{MAMAMAMAzhuyaotuimox}
		\left\{
		\begin{gathered}
			\partial_tF^A+v\cdot \nabla_xF^A=Q^{AA}(F^A,F^A)+Q^{AB}(F^A,F^B),
			\\
			\partial_tF^B+v\cdot \nabla_xF^B=Q^{BA}(F^B,F^A)+Q^{BB}(F^B,F^B),
		\end{gathered}
		\right.
	\end{equation}
	where the particle species is denoted by the Greek letters $\alpha,\beta\in\{A,B\}$. $F^\alpha(t,x,v)$ is the spatially periodic density distribution function for the particles species $\alpha$ at time $t\geq0$, position $x=(x_1,x_2,x_3)\in [-\pi,\pi]^3=\mathbb{T}^3$, and velocity $v=(v_1,v_2,v_3)\in \mathbb{R}^3$. Particle mass of species $\alpha$ is denoted by $m^\alpha$.  
	The collision operator  for $\alpha$-$\beta$ particle pairs takes the form
	$$ Q^{\alpha \beta}(F^{\alpha},F^{\beta})=\int_{\mathbb{R}^3}\int_{\mathbb{S}^2}b^{\alpha \beta}(\theta) |v-v_*|^{\gamma} \Big[F^{\alpha}(v') F^{\beta}(v'_*)-F^{\alpha}(v)F^{\beta}(v_*) \Big] d\omega dv_*. $$
	Here $-3<\gamma<0, \omega\in\mathbb{S}^2$, and $b^{\alpha \beta}(\theta)$ satisfies the Grad angular cutoff assumption,
	\begin{equation*}
		0< b^{\alpha \beta}(\cos \theta)\leq C_{b}  |\cos \theta |,
	\end{equation*}
	where, $C_{b}$ is a constant, and $\cos\theta=\frac{\omega\cdot(v-v_*)}{|v-v_*|}$. Moreover, the velocities $v, v_*$
	and $v', v'_*$ correspond to the species $\alpha$ and $\beta$ before and after the collision occurs, which obey the laws of momentum conservation and energy conservation:
	\begin{equation*}
		m^{\alpha} v_*' +m^{\beta}v' =m^{\alpha} v_*+m^{\beta} v, \quad
		m^{\alpha}|v_*'|^{2}+m^{\beta}|v'|^2 =m^{\alpha}|v_*|^2+m^{\beta}|v|^2.
	\end{equation*}
	So, $ v', v'_* $ can be expressed as	
	\begin{equation}\label{F1vuvusacl}
			v'=v-\frac{2m^{\beta}}{m^\alpha+m^\beta}[(v-v_*)\cdot \omega]\omega,\qquad 
			v'_{*}=v_{*}+\frac{2m^{\alpha}}{m^\alpha+m^\beta}[(v-v_*)\cdot \omega]\omega.
	\end{equation}

Over the past few decades, many mathematical theories have emerged regarding the Boltzmann equation, most of which focus on the single-species Boltzmann equation. Relatively few discussions address the Boltzmann equation for multi-species gases. For gas mixtures with equal molecular collision masses ($m^A = m^B$), H.L.Li, T.Yang and M.Y.Zhong \cite{[meq]LYZ} studied the decay rates of the bipolar Vlasov -Poisson-Boltzmann equations. Z.D.Fang and K.L.Qi \cite{[FQ]ARC} investigated the hydrodynamic transition process from Boltzmann equations to two-fluid systems. C.Bardos and X.F.Yang \cite{[7]Bardos2012CPDE} gave the classification of well-posed kinetic boundary
layer for hard sphere gas mixtures. Meanwhile, Y.J.Wang \cite{YJWangSIMA} established the diffusive limit behavior of the VPB system in periodic domains under the same mass condition ($m^A = m^B$), and subsequently \cite{[49]WangJDE2013} provided rigorous analysis of its temporal decay properties. N.Jiang, Y.J.Lei, and H.J.Zhao \cite{[2025X]Jiang} investigated the Euler-Maxwell limit for potential  $-3<\gamma\leq1$, ($m_{\pm}=e_{\pm}=1$).
However, when dealing with a mixed gas containing two species with different particle masses, more difficulties arise, primarily in three aspects:
\begin{enumerate}
	\item The collision operator incorporates the mass parameters of different particles, complicating the analysis of its properties.
	\item Collisions between particles of unequal mass lack symmetry.
	\item Collision operator on the right-hand side cannot decouple the system of equations through $``(f^{+}+f^{-})''$ and $``(f^{+}-f^{-})''$, which is different from the case of $m^A=m^B$.
\end{enumerate}

Consequently, studies on collisions involving particles of different masses ($m^A \neq m^B$) are far fewer. K.Aoki, C.Bardos and S.Takata \cite{[1]Aoki2003JSP} investigated the Knudsen layer for gas mixtures, establishing that the equilibrium state for particles of unequal masses is a bi-Maxwellian distribution and proving that is a self-adjoint and non-negative operator.  
In \cite{[YU2010]}, A.Sotirov and S.H.Yu investigate the Boltzmann equation for a binary gas mixture in one spatial dimension, with initial conditions where one gas is near vacuum and the other in proximity to a Maxwellian equilibrium state. They decompose the operator 
$\bm{L}(v)$ into four scalar operators $L^{AA}, L^{AB}, L^{BA}, L^{BB}$, where $L^{AA}$ and $L^{BB}$ share the same properties with the single-species operator $L$. Specifically, they establish the coercivity of the operator $L^{AB}$. Their work establishes a spectral analysis framework for asymmetric initial conditions, paving the way for subsequent research by other scholars—as demonstrated in \cite{[2014]LSQ}, \cite{you3F}, \cite{you4F} and \cite{[51]Wu2024AA}.

Later, Briant and Daus \cite{[60]Briant2016ARM} studied the Boltzmann equation system for hard potentials with particles of different masses, employing spectral methods to construct operator semigroups and establish the well-posedness of solutions for small perturbations near equilibrium initial conditions. Their paper makes a contribution ‌by‌ providing  mathematically rigorous treatment of gases composed of particles of unequal masses without structural constraints, ‌and by‌ establishing key properties of essential linear operators — ‌specifically‌ the coercivity of $\bm{L}(v)$ and‌ decay estimates for  $\bm{K}(v)$, thereby‌ providing a direct foundation for subsequent research $($one can refer to \cite{Aoki20020JDE}, \cite{Alonso2024}, \cite{BondesanCPAA}, \cite{[DL]VPB} and \cite{[50]Wu2023JDE}$)$. 

However, the conclusions and proofs ‌in \cite{[60]Briant2016ARM} no longer apply for soft potentials. Specifically, Theorem 3.3's method for proving coercivity of the linear collision operator $\bm{L}(v)$ does not apply in case of $-3<\gamma <0$, as the spectrum of 
$\bm{L}(v)$ no longer has a gap ‌near zero‌. $($boldface denotes vector-valued operators$)$.  Additionally, in \cite{[60]Briant2016ARM} Lemma 5.1 they prove that  $\bm{L}=-\bm{\nu} + \bm{K}$ and the integral kernel $\bm{k}(v,v_*)$ of the $\bm{K}(v)$ satisfies 
\begin{equation}\label{estMKB1}
	\frac{\left \langle v \right\rangle^l}{\left \langle v_* \right\rangle^l}\big|\bm{k}(v,v_*)\big|\leq C \big[|v-v_*|^{\gamma}+|v-v_*|^{\gamma-2}\big] e^{-c(|v-v_*|^2+\frac{||v|^2-|v_*|^2|^2}{|v-v_*|^2})},
\end{equation}
as well as in \cite{[60]Briant2016ARM} Lemma 5.2:
\begin{align}
	&\int_{\mathbb{R}^3 }\frac{\left \langle v \right\rangle^l}{\left \langle v_* \right\rangle^l} \big|\bm{k}(v,v_*)\big| d v_*\notag\\
	&\leq \int_{\mathbb{R}^3 }  \big[|v-v_*|^{\gamma}+|v-v_*|^{\gamma-2}\big] e^{-c(|v-v_*|^2+\frac{||v|^2-|v_*|^2|^2}{|v-v_*|^2})} d v_* \leq \frac{C}{1+|v|},\label{estMKB2}
\end{align}
where $0\leq\gamma\leq 1$ and $\left \langle v \right\rangle=(1+|v|),\,l \geq 0$. 

Note that  estimate \eqref{estMKB1} fails to provide any  decay‌ about velocity $v$ in $L^{\infty}_v$, because one can take $|v|$ and $|v_*|$ arbitrarily‌ large while satisfying $$|v-v_*|=\big||v|^2-|v_*|^2\big|=1.$$ As shown in Figure 1:
\begin{figure}[H]
	\centering
	\includegraphics[width=14cm,height=7cm]{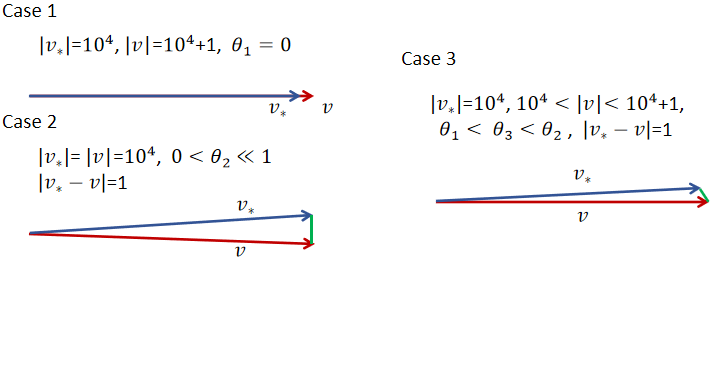}
	\vspace{-1.8cm}
	\caption{Why no decay in \eqref{estMKB1}}
\end{figure}
\begin{itemize}
	\item In Case 1, $\big||v|^2-|v_*|^2\big|=2\times10^4+1$;
	\item In Case 2, $\big||v|^2-|v_*|^2\big|=0$;
	\item Thus‌, Case 3 exists where $\big||v|^2-|v_*|^2\big|=1$.
\end{itemize}
Moreover, when $-3<\gamma\leq-1$, estimate \eqref{estMKB2} ‌is not integrable‌.  This estimate ‌is also likely crude for the hard potentials case‌, ‌as‌ its decay rate $\frac{C}{1+|v|}$, does not reflect‌ the dependence on the potential parameter $\gamma$.

For the Boltzmann equation of a single species of particles, in fact, the corresponding integral kernel $k(v,v_*)$ satisfies the following improved decay estimate: 		
\begin{align}
		\frac{\left \langle v \right\rangle^l}{\left \langle v_* \right\rangle^l}\big|k(v,v_*)\big|\leq C \frac{(1+|v|+|v_*|)^{\gamma-1}}{|v-v_*|} \, e^{-c\{|v-v_*|^2+\frac{||v|^2-|v_*|^2|^2}{|v-v_*|^2}\}}, \hspace{0.5cm}-1<\gamma \leq 1,\label{FSC15}\\
		\frac{\left \langle v \right\rangle^l}{\left \langle v_* \right\rangle^l}\big|k^{\chi}(v,v_*)\big|\leq C_{\varepsilon} \frac{(1+|v|+|v_*|)^{\gamma-1}}{|v-v_*|} \, e^{-c\{|v-v_*|^2+\frac{||v|^2-|v_*|^2|^2}{|v-v_*|^2}\}}, \hspace{0.5cm}-3<\gamma \leq -1,\label{FSC16}
\end{align}
and 
\begin{align}
	\int_{\mathbb{R}^3 } \frac{\left \langle v \right\rangle^l}{\left \langle v_* \right\rangle^l}\big|k(v,v_*)\big| d v_*\leq C \left \langle v \right\rangle^{\gamma-2},\hspace{0.5cm}-1<\gamma \leq 1,\label{FSC17}\\
	\int_{\mathbb{R}^3 } \frac{\left \langle v \right\rangle^l}{\left \langle v_* \right\rangle^l}\big|k^{\chi}(v,v_*)\big| d v_*\leq C_{\varepsilon} \left \langle v \right\rangle^{\gamma-2},\hspace{0.5cm}-3<\gamma \leq -1,\label{FSC18}
\end{align}
 where $\chi=\chi_{\varepsilon}$ is a cutoff function defined in \eqref{cutoffuctiondef} and $k^{\chi}(v,v_*)$ is the kernel of $K^{\chi}(v)$, defined as the regular part of $K(v)$. Estimates \eqref{FSC15} and \eqref{FSC17} (see \cite{[DiJMP]Guo}, \cite{[19]Guo2010ARMA}), ‌exhibit‌ a decay of $(1+|v|)^{\gamma-1}$.  \eqref{FSC16} and \eqref{FSC18}, proved in \cite{[impvR]Guo2003} and \cite{[2017]ARMLSQ}, provide finer estimates than \eqref{estMKB2}.‌ These finer decay analyses are useful, ‌especially‌ when dealing with problems of spectral analysis for soft potentials, one can refer to \cite{[DYZSF]VPB},\cite{[meq]LYZ},\cite{[PG]Liu2004}. They are also crucial  for  $L^{\infty}$  estimates along the characteristics within the $L^2-L^{\infty}$ framework (see \cite{[20]Guo2021ARMA},\cite{[22]Guo2010CPAM},\cite{[inp]LIWANG2023SIAM},\cite{[OY]2025J}), where any slight decay shortfall would prevent the energy estimate from closing.
 
‌‌Since these refined estimates are critically important, it naturally raises the question: Does the corresponding integral kernel 
$\bm{k}(v,v_*)$ for multi-particle systems similarly satisfy such estimates? Therefore, a more detailed structural dissection of the linear operator $\bm{K}$ is necessary. ‌Here we focus on the case of soft potentials.  $\bm{K}$ is decomposed into $\bm{K}_1 + \bm{K}_2$. The structure of operator $\bm{K}_1$ is relatively simple. Using momentum and energy conservation in collisions establishes its exponential decay, thus satisfying estimates \eqref{FSC16},\eqref{FSC18}. For the part of $\bm{K}_2$, it is decomposed via a cutoff function $\chi=\chi_{\varepsilon}$, as $\bm{K}_2 = \bm{K}^{\chi}_2 + \bm{K}^{1-\chi}_2$. If $\varepsilon$ is chosen sufficiently small, $\bm{K}^{1-\chi}_2$ becomes small that can always be absorbed by the right-hand side. ‌Let  $\bm{f}=(f^A,f^B)^T$ be a vector valued function. The part of $\bm{K}_{2}^{\chi}\mathbf{f}=(K_{2}^{\alpha,\chi}\mathbf{f},K_{2}^{\alpha,\chi}\mathbf{f})^T$ requires focused analysis. It admits the following decomposition:
 \begin{equation*}
 	K_{2}^{\alpha,\chi}\mathbf{f} = \underbrace{K_{2,(1)}^{\alpha,\chi}\mathbf{f}}_{\text{Typical part}} + \underbrace{K_{2,(2)}^{\alpha,\chi}\mathbf{f}}_{\text{Hybrid part}},
 \end{equation*}
 \begin{figure}[H]
 	\vspace{-0.4cm}
 	\centering
 	\includegraphics[width=14.5cm,height=5.7cm]{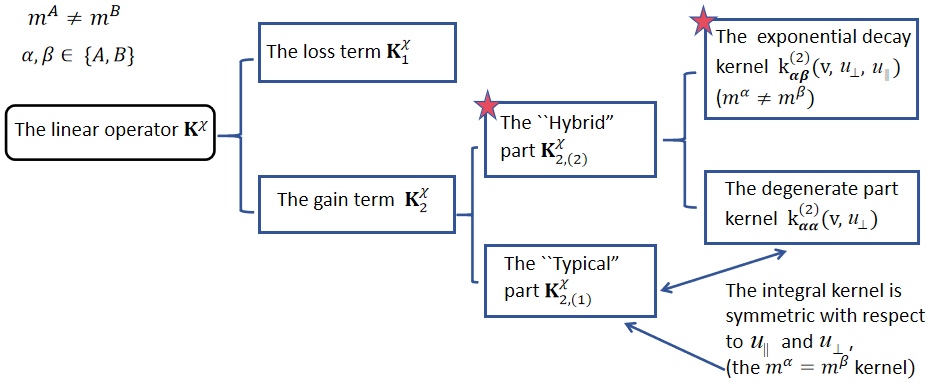}
 	\vspace{-0.2cm}
 	\caption{Decomposition of the Linear Operator $\mathbf{K}^{\chi}$}
 \end{figure}
 \noindent where 
 \begin{equation*}
 	\begin{split}
 		&	 K_{2}^{\alpha,\chi} \mathbf{f}=  \frac{1}{\sqrt{\mu^\alpha}} \sum_{\beta=A,B} \iint_{\mathbb{R}^3 \times \mathbb{S}^2}
 		|v-v_*|^{\gamma}\chi_{\varepsilon}(v-v_*) \Big[\mu^\alpha(v')\sqrt{\mu^\beta(v_*')}f^\beta(v_*')\\
 		& \hspace{7.7cm} +\mu^{\beta}(v_*') \sqrt{\mu^{\alpha}(v')}\, f^\alpha(v')\Big]b^{ \alpha \beta}(\theta) d\omega dv_*.
 	\end{split}	
 \end{equation*}	
 
 Both parts $\alpha \neq \beta$ and $\alpha = \beta$ in ``Typical part" exhibit properties analogous to ‌those in the Boltzmann equation for a single species, which is equivalent to the case of identical molecules $(m_A=m_B)$. So, its integral kernel $k_{\alpha \beta}^{(1)}(v,u_{\parallel})$ thus fulfills estimates \eqref{FSC16},\eqref{FSC18}.  ``Hybrid part" comprises a ‌degenerate  part $(m^{\alpha}=m^{\beta})$ and an ‌exponentially decaying part $(m^{\alpha} \neq m^{\beta})$. ‌It is straightforward to verify that the integral kernel $k_{\alpha \beta}^{(2)}(v,u_{\perp},u_{\parallel})$ of the exponentially decaying part satisfies estimates \eqref{FSC16},\eqref{FSC18}. While the degenerate part exhibits symmetry with the $\alpha = \beta$ part in ``Typical part", $($thus maintaining the same structure to that in equal-mass collisions$)$, its integral kernel $k_{\alpha \alpha}^{(2)}(v,u_{\perp})$ also fulfills these improved estimates. Details can be found in Lemma \ref{lemma2.2} and Remark \ref{RK2211}.  How to formally transform these integral kernels to be consistent with \eqref{FSC16},\eqref{FSC18}, see Remark \ref{RK2222}.
 
 For hard potentials, we can drop the cutoff function $\chi_{\varepsilon}$ while preserving the above conclusions. These notations here carry a subtle distinction from that in the proof of Lemma \ref{lemma2.2}, primarily due to the inclusion of a truncation function $\mathbf{1}_{|v|+|v_*|>m}$ introduced to simultaneously establish operator compactness in Lemma \ref{lemma2.2}. ‌Yet, this adjustment sustains the principal properties of the decay-related decomposition.  Therefore, the answer is ‌Yes!‌

Based on the structural analysis of the linear operator $\bm{K}$, we can further derive the coercivity for the linear operator $\bm{L}$. ‌The proof we employ ‌not only‌ establishes the result for soft potentials ‌but also‌ applies to hard potentials, because ‌when $0 \leq \gamma \leq 1$, $\bm{K}$ is a compact operator. Thus differs from the proof in \cite{[60]Briant2016ARM}. In a subsequent section, we provide some estimates for the nonlinear term $\Gamma^{\alpha \beta}$. Although these nonlinear collision terms involve parameters $m^\alpha$, $m^\beta$ and couple functions $f^\alpha_1$ and $f^\beta_2$ together, we achieve the desired control by considering their vector-valued weighted $L^2_{x,v}$ norms.  ‌Furthermore‌, ‌to extend the existence time of local solutions to long time, we ‌generalize‌ Guo's method $($see Lemma 9 in \cite{[impvR]Guo2003}$)$ to the case of collisions with unequal masses and prove that the coercivity of the linear operator $\bm{L}$  is ‌preserved under integration with respect to any large time t‌.  Utilizing the properties of these operators and established lemmas, we finally prove the main result—Theorem \ref{maintheorem11}, establishing the global-in-time existence of classical solutions near Maxwellians for soft potentials ($-3<\gamma<0$).

\subsection{Notations}
We use notation $\Vert \cdot\Vert_{L^{p}_{x,v} } $ to denote the norm of $L^p$ ( $1\leq p \leq \infty$) space with respect to the variables $(x,v) \in \mathbb{R}^3  \times \mathbb{R}^3$,  and $| \cdot|_{L^p_v} $, $| \cdot|_{L^p_x} $ are the $L^p$ norms   w.r.t. the variable $v$ or $x$. For vector functions $\mathbf{f}(t,x,v)=(f^A,f^B)^T$, its $L^p$ norms are denoted by
\begin{equation*}
	\Vert \mathbf{f} \Vert_{L^{p}_{x,v}}=\Vert f^A\Vert_{L^{p}_{x,v}}+\Vert f^B \Vert_{L^{p}_{x,v}},\qquad | \mathbf{f} |_{L^{p}_{v}}=| f^A|_{L^{p}_{v}}+| f^B|_{L^{p}_{v}}.
\end{equation*}
The notations  $\left \langle \cdot, \cdot	\right \rangle_{L^{2}_{x,v}}$ and $\left\langle \cdot, \cdot\right \rangle_{L^{2}_{v}}$ are the inner product of the function space $L^2_{x,v}$ and $L^2_{v}$ respectively. For vector functions $\mathbf{f}$ and $\mathbf{g}(t,x,v)=(g^A,g^B)^T$,  their inner product in $L^2_v$ and $L^2_{x,v}$ are denoted as 
\begin{gather*}
	\left \langle \mathbf{f}, \mathbf{g}\right \rangle_{L^{2}_{v}} = \left \langle f^A, g^A	\right \rangle_{L^{2}_{v}} + \left \langle f^B, g^B	\right \rangle_{L^{2}_{v}}, \\
	\left \langle \mathbf{f}, \mathbf{g}\right \rangle_{L^{2}_{x,v}} = \left \langle f^A, g^A	\right \rangle_{L^{2}_{x,v}} + \left \langle f^B, g^B	\right \rangle_{{L^{2}_{x,v}}}.
\end{gather*}
We introduce the matrix product of column vectors $\mathbf{f}$ and $\mathbf{g}$:
\begin{equation*}
	(\mathbf{f}:\mathbf{g})=(f^Ag^A,f^Bg^B)^T.
\end{equation*} 
The letter $C$ is a constant, which may vary from line to line. $J \lesssim K$ means $J \leq C K$. If both hold $J \lesssim K$ and $K \lesssim J$, then we use the notation $J\thicksim K$. 
We define the polynomial weight function $\left \langle v \right\rangle= 1+|v|$, and $w=w(v)=(1+|v|)^{\gamma}$. Then we introduce the following weighted $L^2$ norm
\begin{gather*}
	\Vert f \Vert_{\nu}^2 = \iint_{\mathbb{R}^3 \times \mathbb{R}^3} |f|^2 w \,dxdv, \qquad  | f |_{\nu}^2 = \int_{\mathbb{R}^3 } |f|^2 w \,dv.\\
	\Vert \mathbf{f} \Vert_{\nu}^2 = \Vert f^A \Vert_{\nu}^2 +\Vert f^B \Vert_{\nu}^2 , \quad \quad \quad  | \mathbf{f} |_{\nu}^2 =| f^A |_{\nu}^2 +| f^B |_{\nu}^2.
\end{gather*}
Let $\tilde{\alpha}, \tilde{\beta}$ denote multi-indices with length $|\tilde{\alpha}|$ and $|\tilde{\beta}|$ respectively, and 
\begin{equation*}
	\partial^{\tilde{\alpha}}_{\tilde{\beta}} \equiv \partial_x^{\tilde{\alpha}} \partial_v^{\tilde{\beta}}.
\end{equation*}
Furthermore, we denote $\tilde{\beta} \leq \tilde{\alpha}$ if no component of $\tilde{\beta}$ is greater than the corresponding component of $\tilde{\alpha}$, and $\tilde{\beta} \leq \tilde{\alpha}$ if $\tilde{\beta} \leq \alpha$ and $|\tilde{\beta}| < |\alpha|$, and use 
$C^{\tilde{\beta}}_{\tilde{\alpha}}$ to denote the combination number.

Global solutions are constructed in the following energy norm:
\begin{equation}\label{DFenergynorm}
	\mathcal{E}[\mathbf{f}(t, x, v)] =  \sum_{|\tilde{\alpha}|+|\tilde{\beta}|\leq N}  \left( \frac{1}{2} \big\| w^{|\tilde{\beta}|} \partial_{\tilde{\beta}}^{\tilde{\alpha}} \mathbf{f}(t) \big\|_{L^{2}_{x,v}}^2 + \int_0^t \big\| w^{|\tilde{\beta}|} \partial_{\tilde{\beta}}^{\tilde{\alpha}} \mathbf{f}(s) \big\|^2_\nu \, ds \right)
\end{equation}
with the initial data satisfies
\begin{equation*}
	\mathcal{E}[\mathbf{f}^{\rm{in}}] \equiv  \mathcal{E}[\mathbf{f}(0)] \equiv \sum_{|\tilde{\alpha}|+|\tilde{\beta}|\leq N}  \frac{1}{2} \big\| w^{|\tilde{\beta}|} \partial_{\tilde{\beta}}^{\tilde{\alpha}} \mathbf{f}_0 \big\|^2_{L^{2}_{x,v}}. 
\end{equation*}

Throughout this article, $ N \geq 8 $, and the Einstein summation convention is used 
from time to time.

	\subsection{Basic properties of collision operators \cite{BoudinSC7}.}

We introduce the bilinear Boltzmann collision operator $\mathscr C $ in vector form: 
\begin{equation}\label{equation1.4}
	\mathscr{C} \mathbf{F} =
	\left(
	\begin{array}{cccc}
		Q^{AA}(F^A,F^A)+Q^{AB}(F^A,F^B)\vspace{3pt}\\
		Q^{BA}(F^B,F^A)+Q^{BB}(F^B,F^B)
	\end{array}
	\right).
\end{equation}
Firstly, we list some  properties of the collision operators. 

 \noindent (Collision invariants)  $\mathbf{\Psi}=(\Psi^{A},\Psi^{B})^{T}$ is called a collision invariant of operator $\mathscr{C}$ with the inner product of $(L^2_v(\mathbb{R}^3))^2$ if it satisfies
	\begin{equation}\label{Collision invariants}
		\langle\mathscr{C}\mathbf{F}, ~~\mathbf{\Psi} \rangle_{(L_v^2(\mathbb{R}^3))^2} =\sum_{\alpha,\beta\in \{A,B\}}\left\langle Q^{\alpha  \beta}(F^{\alpha},F^{\beta}) , ~~\Psi^{\alpha} \right\rangle_{L_v^2}
		=0\quad {\rm for ~any ~vector}\,~ \mathbf{F}.
	\end{equation}
	If and only if $\mathbf{\Psi}$ belongs to ${\rm Span} \big\{\mathbf{e}_1, \mathbf{e}_2, v_1 \mathbf{m}, v_2 \mathbf{m}, v_3 \mathbf{m}, |v|^2 \mathbf{m}\big\} $, where  $\mathbf{m} = (m^A,m^B )^T$, and $ \mathbf{e}_j$ is the $j^{th}$ unit vector in $\mathbb{R}^2$.
	
 \noindent (H-theorem) 
	The entropy function of the two-species Boltzmann equations  satisfies
	\begin{equation}\label{equation1.5}
		\left \langle
		\mathscr{C}\mathbf{F},  \log \mathbf{F}
		\right \rangle_{(L_v^2(\mathbb{R}^3))^2}
		= \sum_{\alpha,\beta\in \{A,B\}}\left\langle Q^{\alpha  \beta}(F^{\alpha},F^{\beta}) , ~~\log F^{\alpha} \right\rangle_{L_v^2(\mathbb{R}^3)}   \leq 0,
	\end{equation}
	and this equality holds if and only if 
	\begin{equation}\label{basisHG}
		\log \mathbf{F} \in {\rm Span} \big\{\mathbf{e}_1, \mathbf{e}_2, v_1 \mathbf{m}, v_2 \mathbf{m}, v_3 \mathbf{m}, |v|^2 \mathbf{m}\big\}.
	\end{equation}
In this article, we study the initial value problem
\begin{equation}\label{iniprobmain}
(\partial_t + v \cdot \nabla_{x}) \mathbf{F} = \mathscr{C} \mathbf{F} , \quad \mathbf{F}(0, x, v) = \mathbf{F}_0(x, v),
\end{equation}
in a periodic box From the property of the collision invariants , the conservation of mass, momentum and energy for \eqref{iniprobmain} take the form
\begin{align*}
	\frac{d}{dt} \int_{\mathbb{T}^3} \left\langle \mathbf{F}(t), \mathbf{e}_i \right\rangle_{L_v^2}  dx  &\equiv 0, \quad i=1,2,\\
	\frac{d}{dt} \int_{\mathbb{T}^3} \left\langle \mathbf{F}(t), v_j \mathbf{m} \right\rangle_{L_v^2} dx &\equiv 0, \quad j=1,2,3,\\
		\frac{d}{dt} \int_{\mathbb{T}^3} \left\langle \mathbf{F}(t), \frac{|v|^2}{2}\mathbf{m}   \right\rangle_{L_v^2}dx &\equiv 0.
\end{align*}
is to construct classical solutions for \eqref{iniprobmain} near a global bi-Maxwellian
	\begin{gather}\label{GBM}
	\bm{\mu}=
	\left(
	\begin{array}{cccc}
		\mu^A\\
		\mu^B
	\end{array}
	\right)
	=
	\left(
	\begin{array}{cccc}
		\frac{({m^A})^{3/2}}{(2\pi )^{3/2}}	e^{- \frac{m_A|v|^2}{2}}\\
		\frac{({m^B})^{3/2}}{(2\pi )^{3/2}}	e^{- \frac{m_B|v|^2}{2}}
	\end{array}
	\right).
\end{gather}
We denote 
\begin{equation}\label{defmsmsmsmsms}
\sqrt{\bm{\mu}}=(\sqrt{\mu^A},\sqrt{\mu^B})^T,\hspace{1cm}m^s=\min\big\{m^A,m^B\big\}, \hspace{1cm}  \mu^{s}=\frac{({m^s})^{3/2}}{(2\pi )^{3/2}}	e^{- \frac{m_s|v|^2}{2}}.
\end{equation}
The ‌orthonormal basis of $\mathbb{R}^3(v)$ is defined
as
\begin{gather}
	\tilde{\bm{\chi}}_1=\kappa_1(\mathbf{e}_1:\bm{\mu}), \quad \tilde{\bm{\chi}}_2=\kappa_2(\mathbf{e}_2:\bm{\mu}), \quad\tilde{\bm{\chi}}_j=\kappa_jv_{j-2}( \mathbf{m}:\bm{\mu}),\,\, j=3,4,5,  \quad  \tilde{\bm{\chi}}_6=\kappa_6|v|^2(\mathbf{m}:\bm{\mu}).\label{deforthonormalB}
\end{gather}
where $\kappa_i$ are some constants such that $\|\tilde{\bm{\chi}}_i\|_{L_v^2}=1$.

 $\mathbf{L}$ is a non-negative operator, and $\left\langle \mathbf{L}\mathbf{f}, \mathbf{f}  \right\rangle_{L_v^2}=0$ if and only if 
\begin{equation}
	\mathbf{f}= \mathbf{P}_0 \mathbf{f}
\end{equation}
where the projection $\mathbf{P}_0$ in $L_2(\mathbb{R}^3)$ is defined as
\begin{equation*}
	\mathbf{P}_0 \mathbf{f}(t,x,v) = \sum^6_{j=1}\big\langle \mathbf{f}(t,x,\cdot), \tilde{\bm{\chi}}_j \big\rangle_{L_v^2}\tilde{\bm{\chi}}_j.
\end{equation*}
The perturbation function $\mathbf{f}=(f^A,f^B)^T$ to $\bm{\mu}$ is defined by
\begin{equation}
	\mathbf{F}=\bm{\mu}+(\sqrt{\bm{\mu}}:\mathbf{f})
\end{equation}
 then the Boltzmann equation \eqref{iniprobmain} can be written as 
\begin{align}\label{socal555EQSO}
	\begin{split}
	(\partial_t + v \cdot \nabla_{x}) \mathbf{f} +\mathbf{L}\mathbf{f} = \bf{\Gamma}(f,f),\\
	\mathbf{f}(0,x,v) = \mathbf{f}^{\rm{in}}(x,v)
	\end{split}
\end{align}
Here the standard linear operator $\mathbf{L}=(L^A,L^B)^T$ is
	\begin{gather}\label{4.2}
	\begin{split}
		L^\alpha \mathbf{f}=- \frac{1}{\sqrt{\mu^{\alpha}}} \sum_{\beta=A,B} \left[Q^{ \alpha\beta}(\mu^\alpha,\sqrt{\mu^{\beta}} f^{\beta})+Q^{\alpha \beta}(\sqrt{\mu^\alpha}f^{\alpha},\mu^\beta)\right],
	\end{split}
\end{gather}
which can be decomposed by $\mathbf{L}\mathbf{f}=(\bm{\nu}:\mathbf{f})+\mathbf{K}\mathbf{f}$, i.e. 
\begin{gather}\label{Ldeopvec}
	\begin{split}
		\mathbf{L}\mathbf{f}=\left(\begin{array} {lll} \mathbf{L}^{A} \mathbf{f} \\  \mathbf{L}^{B} \mathbf{f} \end{array}\right) = \left(\begin{array} {lll} \upsilon^{A} f^A \\ \upsilon^{B} f^B \end{array}\right)+\left(\begin{array} {lll}  K^A \mathbf{f}\\   K^B \mathbf{f}\end{array}\right),
	\end{split}	
\end{gather}
where
\begin{equation}\label{nuKsuanzidef}
	\begin{split}
		&\upsilon^{\alpha}=   \sum_{\beta=A,B} \int_{\mathbb{R}^3\times \mathbb{S}^2}
		b^{\alpha\beta}(\theta)|v-v_*|^{\gamma} \mu^{\beta}(v_*) d\omega dv_*, \quad  K^{\alpha} \mathbf{f}= K_{1}^{\alpha} \mathbf{f}+K_{2}^{\alpha} \mathbf{f},
		\\
		&	 K_{1}^{\alpha} \mathbf{f}=  \frac{1}{\sqrt{\mu^\alpha}} \sum_{\beta=A,B} \int_{\mathbb{R}^3 \times \mathbb{S}^2}
		b^{\alpha\beta}(\theta)|v-v_*|^{\gamma}  \mu^\alpha(v) \sqrt{\mu^\beta(v_*)}f^\beta(v_*)   d\omega dv_*,\\
		&	 K_{2}^{\alpha} \mathbf{f}= - \frac{1}{\sqrt{\mu^\alpha}} \sum_{\beta=A,B} \int_{\mathbb{R}^3 \times \mathbb{S}^2}
		b^{\alpha\beta}(\theta)|v-v_*|^{\gamma}  \Big[\mu^\alpha(v')\sqrt{\mu^\beta(v_*')}f^\beta(v_*')\\
		& \hspace{6.8cm} +\mu^{\beta}(v_*') \sqrt{\mu^{\alpha}(v')}\, f^\alpha(v')\Big] d\omega dv_*.
	\end{split}	
\end{equation}	
Moreover, the non-symmetric bi-linear form  
\begin{gather}\label{Blackgamdef}
	\begin{split}
		\bf{\Gamma}(f,f) = \left(\begin{array} {lll} \displaystyle{\sum_{\beta=A,B}} \Gamma^{A \beta }(f^{\alpha},f^{\beta})\\ \displaystyle{\sum_{\beta=A,B}}\Gamma^{B \beta }(f^{\alpha},f^{\beta})\end{array}\right),
	\end{split}	
\end{gather}
where
\begin{eqnarray}\label{gain loss part}
	\begin{split}
		\Gamma^{\alpha \beta }(f^{\alpha},f^{\beta})= & \frac{1}{\sqrt{\mu^{\alpha }}}Q^{ \alpha \beta}(\sqrt{\mu^{\alpha }}f^{\alpha},  \sqrt{\mu^{\beta }}f^{\beta})\\
		= &  \frac{1}{\sqrt{\mu^{\alpha }}}\int_{\mathbb{R}^3\times \mathbb{S}^2}|v-v_*|^{\gamma}b^{\alpha \beta}(\theta)\\
		&  \times \left(\sqrt{\mu^{\alpha }(v') \mu^{\beta }(v_*')}f^{\alpha}(v')f^{\beta}(v_*')
		-\sqrt{\mu^{\alpha }(v)\mu^{\beta }(v_*)}f^{\alpha}(v)f^{\beta}(v_*)\right)d\omega dv_* \\
		=:& \Gamma^{\alpha \beta }_{\rm gain} - \Gamma^{\alpha \beta }_{\rm loss}.
	\end{split}
\end{eqnarray}

In this paper, we suppose the initial data $\mathbf{F}^{\rm{in}}(x,v)$ has the same mass, momentum and energy as the bi-Maxwellian $\bm{\mu}$, then for any $t\geq0$,
\begin{equation}\label{conseerlaw}
\begin{aligned}
	 \int_{\mathbb{T}^3} \big\langle \mathbf{f}, (\mathbf{e}_i:\sqrt{\bm{\mu}}) \big\rangle_{L_v^2}  dx  &= 0, \quad i=1,2,\\
	 \int_{\mathbb{T}^3} \big\langle \mathbf{f}, (v_j \mathbf{m}:\sqrt{\bm{\mu}}) \big\rangle_{L_v^2} dx &= 0, \quad j=1,2,3,\\
	 \int_{\mathbb{T}^3} \big\langle \mathbf{f}, (\frac{|v|^2}{2}\mathbf{m} :\sqrt{\bm{\mu}})  \big\rangle_{L_v^2}dx &= 0.
\end{aligned}
\end{equation}

\subsection{The main result}

\begin{thm}\label{maintheorem11}
	Let $-3 < \gamma < 0$, and $\mathbf{f}^{\rm{in}}(x,v)$ satisfies the conservation law \eqref{conseerlaw} at $t = 0$ with
	\begin{equation*}
		\mathbf{F}^{\rm{in}}=\bm{\mu}^{\rm{in}}+(\sqrt{\bm{\mu}^{\rm{in}}}:\mathbf{f}^{\rm{in}})
	\end{equation*}
	 There are constants $C_0 > 0$ and $\varepsilon_0 > 0$, such that if $\mathcal{E}(\mathbf{f}^{\rm{in}}) \leq \varepsilon_0$, there exists a unique global solution $\mathbf{f}(t,x,v)$ to the Boltzmann equation \eqref{socal555EQSO} with $\mathbf{F}=\bm{\mu}+(\sqrt{\bm{\mu}}:\mathbf{f})$, and
	$$
	\sup_{0 \leq s \leq \infty} \mathcal{E}\big[\mathbf{f}(s)\big] \leq C_0\mathcal{E}(\mathbf{f}^{\rm{in}}).
	$$
\end{thm}
\begin{remark}
	‌This work builds on a energy method developed  by Guo Yan in \cite{[ivR]Guo2002},\cite{[impvR]Guo2003} to establish global classical solutions.  
	Despite the complexity of collisions between particles of differing masses and the limited mathematical literature available, examining each component in isolation leads to confusion and inconclusive results. However, adopting a holistic perspective on these collision operators reveals a key insight: by replacing the 
	space with a ‌2D vector-valued 
	space‌, the system as a whole satisfies conclusions analogous to those of the single-species Boltzmann equation.
\end{remark}

\begin{figure}[H]
	\vspace{-0.5cm}
	\centering
	\includegraphics[width=15.3cm,height=9.5cm]{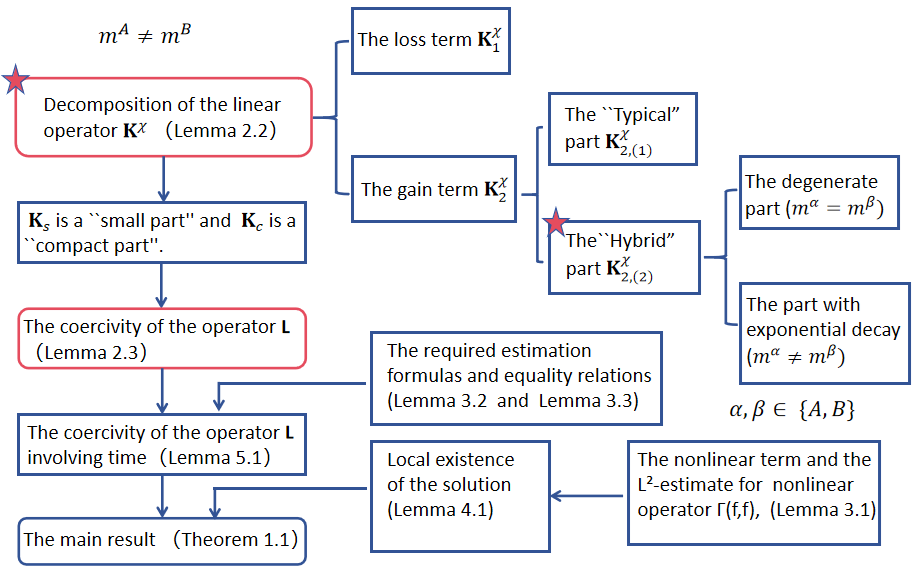}
	\vspace{-0.5cm}
	\caption{Structure presentation of the paper }
	\vspace{-0.1cm}
\end{figure}

\begin{thm}\label{maintheorem22}
	Let $-3 < \gamma < 0$ and assume  that
	$\mathbf{f}^{\rm{in}}(x, v)$ satisfies the conservation law \eqref{conseerlaw} at $t = 0$. Let $\mathbf{f}(t, x, v)$ be the classical solution to the Boltzmann
	equation \eqref{socal555EQSO} in $[0, T^*)$ such that for $0 \leq t < T^*$,
	$$
	\mathcal{E}\big[\mathbf{f} (t)\big] \leq M.
	$$
	Then there exists a small constant $0 < \delta_M < 1$, such that for any $t_1 \geq 0$, and any
	non-negative integer $n$ with $0 \leq t_1 + n < T^*$,
	\begin{equation*}
		\sum_{|\tilde{\alpha}| \leq N} \int_{t_1}^{t_1+n} \big\langle \mathbf{L}\big[\partial^{\tilde{\alpha}} \mathbf{f}(s)\big], \partial^{\tilde{\alpha}} \mathbf{f} (s) \big\rangle_{L^2_{x,v}} \, ds \leq \delta_M \sum_{|\tilde{\alpha}| \leq N} \int_{t_1}^{t_1+n} \| \partial^{\tilde{\alpha}} \mathbf{f} (s) \|^2_{\nu} \, ds. 
	\end{equation*}
\end{thm}

\begin{remark}
	‌The proof of this theorem resides primarily in Lemma \ref{socedlem9}, which is proved by contradiction. If \eqref{socedlem9ct} were false, we obtain a  normalized sequence $\mathbf{Z}_n=(Z^A_n,Z^B_n)^T$ satisfy \eqref{normalized8722} and \eqref{normalized888}. ‌The proof then proceeds in two steps:
	‌The first step is to prove that the limit function takes the form:
	\begin{equation*}
		\mathbf{Z}(t,x,v) = (\mathbf{a}(t,x):\sqrt{\bm{\mu}}) + (\mathbf{b}(t,x) \cdot v) (\mathbf{m}:\sqrt{\bm{\mu}}) + c(t,x)|v|^2(\mathbf{m}:\sqrt{\bm{\mu}}).
	\end{equation*}
	‌The second step demonstrates that 
	$\mathbf{a},\mathbf{b},c$ satisfy 
	\begin{equation*}
	\sum_{\tilde{\alpha} \neq 0, |\tilde{\alpha}| \leq N} \int_0^1 \left[ \big|\partial^{\tilde{\alpha}} \mathbf{a}\big|_{L_x^2}^2 + \big|\partial^{\tilde{\alpha}} \mathbf{b}\big|_{L_x^2}^2 + \big|\partial^{\tilde{\alpha}} c\big|_{L_x^2}^2 \right] (s) ds \le C M,
	\end{equation*}
	where $\mathbf{a}=(a_1,a_2),\, \mathbf{b}=(b_1,b_2,b_3)$. In the first step, ‌it is essential to treat the vector-valued function $\mathbf{Z}$ as an integrated entity‌, since these lemmas such as Lemma \ref{lemma2.1}---\ref{estimatesofnonlinears}, hold for the vector as a whole but not for its individual components. In the second step, ‌component-wise estimation of $Z^{\alpha}$ becomes feasible‌ since the limit function $\mathbf{Z}$ has been proven to satisfy equation \eqref{socad103103} in the sense of distributions. Specifically, $Z^A$
	satisfies equations related to $(a_1,\mathbf{b},c)$, while $Z^B$
	corresponds to $(a_2,\mathbf{b},c)$. The coefficients in both equations can be estimated separately.
\end{remark}

	\noindent \textbf{The paper is structured as follows}.
Section 2‌ provides a detailed discussion of the decomposition of the linear collision operator ‌$\bm{L}$\,$(m^A \neq m^B)$, which is further used to prove the coercivity. ‌Section 3‌ gives estimates for the nonlinear collision operator ‌$\bm{\Gamma}$. ‌Section 4‌ obtains a local-in-time solution to the Boltzmann equation by constructing an iterative sequence. ‌Section 5 proves the coercivity involving time (see \eqref{socedlem9ct}) under the amplitude assumption \eqref{socad8080}, utilizing the conservation law equations. Finally, in ‌Section 6‌, by combining the local existence of the solution established in ‌Section 4‌ and the key inequality \eqref{socedlem9ct} derived in ‌Section 5‌, we prove the global existence of the solution.

	\section{The linear operator $\mathbf{L}$}

	\begin{lema}\label{lemma2.1}
		Let $\mathbf{f}_i=(f^A_i,f^B_i)^T, (i=1,2)$ be two vector functions, and $\mathbf{K}$ be defined in \eqref{Ldeopvec}. For $\theta\geq0$,
		\begin{equation}\label{socaed1616}
			\big|\big\langle w^{2l}  \mathbf{K} \mathbf{f}_1, \mathbf{f}_2 \big\rangle_{L^{2}_{v}}\big| \leq C |w^{l}\mathbf{f}_1|_{\nu} |w^{l}\mathbf{f}_2|_{\nu}.
		\end{equation}
	\end{lema}
	\begin{proof}
		By the decomposition $K^{\alpha} \mathbf{f}= K_{1}^{\alpha} \mathbf{f}+K_{2}^{\alpha} \mathbf{f}$,  one has 
		\begin{equation*}
			\langle w^{2l}  \mathbf{K} \mathbf{f}_1, \mathbf{f}_2 \rangle_{L^{2}_{v}}=\langle w^{2l}  \mathbf{K}_1 \mathbf{f}_1, \mathbf{f}_2 \rangle_{L^{2}_{v}}+\langle w^{2l}  \mathbf{K}_2 \mathbf{f}_1, \mathbf{f}_2 \rangle_{L^{2}_{v}}.
		\end{equation*}
		For $ \mathbf{K}_1 $ part,  integrating $\omega$ over $   S^2 $ yields
		\begin{equation*}
		\begin{aligned}
			&\big\langle w^{2l} \mathbf{K}_1 \mathbf{f}_1, \mathbf{f}_2 \big\rangle_{L^{2}_{v}} \\
			&\hspace{1cm}= \sum_{\alpha=A,B}\sum_{\beta=A,B} C_{\alpha\beta} \iint_{(\mathbb{R}^3)^2} w^{2l}(v)|v - v_*|^\gamma  \sqrt{\mu^\alpha(v)} \sqrt{\mu^\beta(v_*)}f^\beta_1(v_*)f^\alpha_2(v)\,dv_*\,dv \\
			&\hspace{1cm}\lesssim \sum_{\alpha=A,B}\sum_{\beta=A,B}  \Bigg( \iint_{(\mathbb{R}^3)^2} w^{2l}(v)\sqrt{\mu^\alpha(v)} \sqrt{\mu^\beta(v_*)}|v - v_*|^\gamma [f^\beta_1(v_*)]^2 \,dv_*\,dv \Bigg)^{1/2} \\
			&\hspace{3.4cm} \times \Bigg( \iint_{(\mathbb{R}^3)^2} w^{2l}(v)\sqrt{\mu^\alpha(v)} \sqrt{\mu^\beta(v_*)}|v - v_*|^\gamma [f^\alpha_2(v)]^2 \,dv_*\,dv \Bigg)^{1/2}.
		\end{aligned}
		\end{equation*}
		Since $ w(v) = (1+|v|)^\gamma $ and $ \sqrt{\mu^\alpha} $ decays exponentially, we integrate the first factor of $f^\beta_1$ over $ v $, and integrate the second factor of $f^\alpha_2$ over $ v_* $, to obtain
		\begin{align}
			\big|\big\langle w^{2l} \mathbf{K}_1 \mathbf{f}_1, \mathbf{f}_2 \big\rangle_{L^{2}_{v}}\big| &
			\lesssim \sum_{\alpha=A,B}\sum_{\beta=A,B}  \left( \int_{\mathbb{R}^3} w^{2l+1}(v_*) [f^\beta_1(v_*)]^2 \,dv_* \right)^{1/2} \notag\\
		 &\hspace{3.5cm}\times\left( \int_{\mathbb{R}^3} w^{2l+1}(v) [f^\alpha_2(v)]^2 \,dv \right)^{1/2}\notag\\
		 &\lesssim \, |w^{l}\mathbf{f}_1|_{\nu} |w^{l}\mathbf{f}_2|_{\nu},\label{K1PTlem21}
		\end{align}
		which concludes the $\mathbf{K}_1$ part.
		
		 By the law of energy conservation, $K_{2}^{\alpha} \mathbf{f}$ defined in \eqref{nuKsuanzidef} can be written as 
		\begin{equation*}
			\begin{split}
				&	 K_{2}^{\alpha} \mathbf{f}= - \sum_{\beta=A,B} \iint_{\mathbb{R}^3 \times \mathbb{S}^2}
				b^{\alpha\beta}(\theta)|v-v_*|^{\gamma}  \Big[\sqrt{\mu^\alpha(v')}\sqrt{\mu^\beta(v_*)}f^\beta(v_*')\\
				& \hspace{6.5cm} +\sqrt{\mu^{\beta}(v_*')} \sqrt{\mu^{\alpha}(v_*)}\, f^\alpha(v')\Big] d\omega dv_*.
			\end{split}	
		\end{equation*}	
		Then, for the part of $\mathbf{K}_2$
		\begin{gather}\label{le21k2}
		\begin{aligned}
			&\big|\big\langle w^{2l} \mathbf{K}_2 \mathbf{f}_1, \mathbf{f}_2 \big\rangle_{L^{2}_{v}}\big|\\
			& \hspace{0.8cm}\lesssim \sum_{\alpha=A,B}\sum_{\beta=A,B}  \Bigg(\iint_{(\mathbb{R}^3)^2 }w^{2l}(v)|v-v_*|^{\gamma}\sqrt{\mu^{\beta}(v_*')} \sqrt{\mu^{\alpha}(v_*)}\, f^\alpha_1(v')f^\alpha_2(v)dv_*\,dv\\
			&\hspace{3.4cm}+ \iint_{(\mathbb{R}^3)^2 }w^{2l}(v)|v-v_*|^{\gamma}\sqrt{\mu^\alpha(v')}\sqrt{\mu^\beta(v_*)}f^\beta_1(v_*')f^\alpha_2(v)dv_*\,dv\Bigg).
		\end{aligned}
		\end{gather}
		The first term inside the summation symbol in \eqref{le21k2} is bounded by
		\begin{gather}\label{lem21f1fc}
		\begin{aligned}
		&\Bigg(\iint_{(\mathbb{R}^3)^2 }w^{2l}(v)|v-v_*|^{\gamma}\sqrt{\mu^{\beta}(v_*')} \sqrt{\mu^{\alpha}(v_*)}\, [f^\alpha_1(v')]^2dv_*\,dv \Bigg)^{1/2} \\
		&\hspace{0.5cm}\times \Bigg(\iint_{(\mathbb{R}^3)^2 }w^{2l}(v)|v-v_*|^{\gamma}\sqrt{\mu^{\beta}(v_*')} \sqrt{\mu^{\alpha}(v_*)}\, [f^\alpha_2(v)]^2dv_*\,dv \Bigg)^{1/2}.
		\end{aligned}
		\end{gather}
		From the law of momentum conservation, for $\theta\geq0$,
		\begin{equation}
			w^{2l}(v) \sqrt{\mu^{\alpha}(v_*)}\leq C w^{2l}(|v|+|v_*|)\leq C w^{2l}(|v'|+|v_*'|).
		\end{equation}
		Since 
		\begin{equation*}
			v'=v-\frac{2m^{\beta}}{m^\alpha+m^\beta}[(v-v_*)\cdot \omega]\omega,\qquad 
			v'_{*}=v_{*}+\frac{2m^{\alpha}}{m^\alpha+m^\beta}[(v-v_*)\cdot \omega]\omega.
		\end{equation*}
		The Jacobian matrix ‌$\mathbf{J}$ for the change of variables for $(v,v_*)\rightarrow(v',v_*')$  is the following $2\times2$ block matrix, where each block is a $3\times3$ matrix:
		\begin{equation*}
			\mathbf{J} =   \begin{pmatrix}
				 \frac{\partial v'}{\partial v} & \frac{\partial v'_*}{\partial v} \\
				 \frac{\partial v'}{\partial v_*}  & \frac{\partial v'_*}{\partial v_*} 
			\end{pmatrix}=\begin{pmatrix}
			 \mathbf{I}-\frac{2m^{\beta}}{m^\alpha+m^\beta}\omega \omega^{\rm{T}} & \frac{2m^{\alpha}}{m^\alpha+m^\beta}\omega \omega^{\rm{T}}   \\
			 \frac{2m^{\beta}}{m^\alpha+m^\beta}\omega \omega^{\rm{T}} & \mathbf{I}-\frac{2m^{\alpha}}{m^\alpha+m^\beta}\omega \omega^{\rm{T}}
			\end{pmatrix},
		\end{equation*}
	where $\mathbf{I}$ is a $3\times3$ identity matrix, and $\omega=(\omega_1,\omega_2,\omega_3)^{\rm{T}}$ is a column vector. According to the linear transformation of the determinant:
	\begin{align*}
		\det\mathbf{J} 
		=& \left|\begin{array}{ccc}
			 \mathbf{I} & \mathbf{I} \\
			 \frac{2m^{\beta}}{m^\alpha+m^\beta}\omega \omega^{\rm{T}}  & \mathbf{I}-\frac{2m^{\alpha}}{m^\alpha+m^\beta}\omega \omega^{\rm{T}}
		\end{array}\right|
		=\left|\begin{array}{ccc}
			\mathbf{I} & 0 \\
			\frac{2m^{\beta}}{m^\alpha+m^\beta}\omega \omega^{\rm{T}}  & \mathbf{I}-2\omega \omega^{\rm{T}}
		\end{array}\right|.
	\end{align*}
	Therefore
	\begin{align*}
		\det\mathbf{J} =& \det (\mathbf{I}-2\omega \omega^{\rm{T}}) \\
		=& \left|\begin{array}{ccc}
			 1-2\omega_1^2 & -2\omega_1\omega_2  & -2\omega_1\omega_3 \\
			 -2\omega_2\omega_1 & 1-2\omega_2\omega_2 & -2\omega_2\omega_3 \\
			 -2\omega_3\omega_1 & -2\omega_3\omega_2  & 1-2\omega_3\omega_3 
		\end{array}\right|\\
		=&(1-2\omega_1^2)\big[(1-2\omega_2^2)(1-2\omega_3^2)-4\omega_2^2\omega_3^2\big]\\
		&+2\omega_1\omega_2\big[-2\omega_1\omega_2(1-2\omega^2_3)-4\omega_1\omega_2\omega^2_3\big]
		\\&-2\omega_1\omega_3\big[4\omega_1\omega^2_2\omega_3+2\omega_1\omega_3(1-2\omega_2^2)\big]
		\\=& 1-2\omega^2_1-2\omega^2_2-2\omega^2_3,\\
		=& -1,
	\end{align*}
	which implies
	\begin{equation}\label{ddstardstar}
		dv_*'dv'=|\det\mathbf{J}|dv_*dv=dv_*dv.
	\end{equation}
	 Additionally, it is straightforward to verify
	 \begin{equation*}
	 	v'-v'_{*}=v-v_* - 2[(v-v_*)\cdot \omega]\omega.
	 \end{equation*}
	 As the picture below shows
	 \begin{figure}[H]
	 	\vspace{-0.1cm}
	 	\centering
	 	\includegraphics[width=9cm,height=3.35cm]{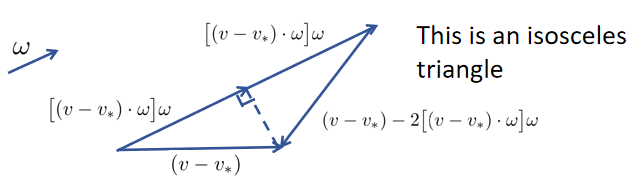 }
	 	\vspace{-0.3cm}
	 	\caption{$|v-v_*|=|(v-v_*) -2	\big[(v-v_*)\cdot \omega\big]\omega|$}
	 	\vspace{-0.3cm}
	 \end{figure}
	This is because any point on the perpendicular bisector is equidistant from the two endpoints of the line segment,  so we obtain
	 \begin{equation}\label{yipievvstar}
	 	|v-v_*|=|v'-v_*'|.
	 \end{equation}
	Then, the first line in $\eqref{lem21f1fc}$ is bounded by
	\begin{align*}
		&\Bigg(\iint_{(\mathbb{R}^3)^2 }w^{2l}(|v'|+|v_*'|)|v-v_*|^{\gamma}\sqrt{\mu^{\beta}(v_*')} \, \big[f^\alpha_1(v')\big]^2dv_*\,dv \Bigg)^{1/2} \\
		&\hspace{0.4cm}\lesssim \Bigg(\iint_{(\mathbb{R}^3)^2 }w^{2l}(|v'|)|v'-v_*'|^{\gamma}\sqrt{\mu^{\beta}(v_*')} \, \big[f^\alpha_1(v')\big]^2dv_*'\,dv' \Bigg)^{1/2}\\
		&\hspace{0.4cm}\lesssim \Bigg(\int_{\mathbb{R}^3}\int_{\mathbb{R}^3}|v'-v_*'|^{\gamma}\sqrt{\mu^{\beta}(v_*')}\, dv_*'\,w^{2l}(|v'|) \, \big[f^\alpha_1(v')\big]^2\,dv' \Bigg)^{1/2}\\
		&\hspace{0.4cm}\lesssim \, \big|w^{l}\mathbf{f}_1\big|_{\nu},
	\end{align*}
	 and the second line in $\eqref{lem21f1fc}$ is bounded by 
	\begin{align*}
		&\Bigg(\iint_{(\mathbb{R}^3)^2}w^{2l}(v)|v-v_*|^{\gamma} \sqrt{\mu^{\alpha}(v_*)}\, [f^\alpha_2(v)]^2dv_*\,dv \Bigg)^{1/2} \\
		&\hspace{0.4cm}\lesssim \Bigg( \int_{\mathbb{R}^3}\int_{\mathbb{R}^3}|v-v_*|^{\gamma}\sqrt{\mu^{\alpha}(v_*)}dv_* \, w^{2l}(v) \, [f^\alpha_2(v)]^2\,dv\Bigg)^{1/2}\\
		&\hspace{0.4cm}\lesssim \, \big|w^{l}\mathbf{f}_2\big|_{\nu}.
	\end{align*}
	Next, we consider the second term in \eqref{le21k2}, which is bounded by
	\begin{gather}\label{lem21f2fc22}
		\begin{aligned}
			C&\Bigg(\iint_{(\mathbb{R}^3)^2 }w^{2l}(v)|v-v_*|^{\gamma}\sqrt{\mu^\alpha(v')}\sqrt{\mu^\beta(v_*)}[f^\beta_1(v_*')]^2dv_*\,dv \Bigg)^{1/2} \\
			&\hspace{0.5cm}\times \Bigg(\iint_{(\mathbb{R}^3)^2 }w^{2l}(v)|v-v_*|^{\gamma}\sqrt{\mu^\alpha(v')}\sqrt{\mu^\beta(v_*)}[f^\alpha_2(v)]^2dv_*\,dv \Bigg)^{1/2}.
		\end{aligned}
	\end{gather}
	By a change of variables $(v,v_*)\rightarrow(v',v_*')$, utilizing \eqref{ddstardstar} and \eqref{yipievvstar}, the first line in \eqref{lem21f2fc22} is bounded by
	\begin{equation*}
	\begin{aligned}
		C&\Bigg(\iint_{(\mathbb{R}^3)^2 }w^{2l}(|v'|+|v_*'|)|v-v_*|^{\gamma}\sqrt{\mu^\alpha(v')}[f^\beta_1(v_*')]^2dv_*\,dv \Bigg)^{1/2} \\
		&\hspace{0.4cm}\lesssim \Bigg(\int_{\mathbb{R}^3}\int_{\mathbb{R}^3}|v'-v_*'|^{\gamma}\sqrt{\mu^{\alpha}(v')} dv_*' \, w^{2l}(|v_*'|) \, [f^\beta_1(v_*')]^2\,dv' \Bigg)^{1/2}\\
		&\hspace{0.4cm}\lesssim \, \big|w^{l}\mathbf{f}_1\big|_{\nu},
	\end{aligned}
	\end{equation*}
	and the second line in \eqref{lem21f2fc22} is bounded by
	\begin{align*}
		&\Bigg(\iint_{(\mathbb{R}^3)^2 }w^{2l}(v)|v-v_*|^{\gamma} \sqrt{\mu^{\beta}(v_*)}\, [f^\alpha_2(v)]^2dv_*\,dv \Bigg)^{1/2} \\
		&\hspace{0.4cm}\lesssim \Bigg( \int_{\mathbb{R}^3}\int_{\mathbb{R}^3}|v-v_*|^{\gamma}\sqrt{\mu^{\beta}(v_*)}dv_* \,w^{2l}(v) \, [f^\alpha_2(v)]^2\,dv\Bigg)^{1/2}\\
		&\hspace{0.4cm}\lesssim \, \big|w^{l}\mathbf{f}_2\big|_{\nu}.
	\end{align*}
	Therefore, \eqref{le21k2} is estimated and together with \eqref{K1PTlem21}, this lemma thus follows.
	\end{proof}

		For $m>0$, we define
	\begin{equation}\label{C2Kscdef}
		\begin{split}
			&	 K_{1,r}^{\alpha} \mathbf{f}=  \frac{1}{\sqrt{\mu^\alpha}} \sum_{\beta=A,B} \iint_{\mathbb{R}^3 \times \mathbb{S}^2}\mathbf{1}_{|v|+|v_*|>m}
			b^{ \alpha \beta}(\theta)|v-v_*|^{\gamma}  \mu^\alpha(v) \sqrt{\mu^\beta(v_*)}f^\beta(v_*)   d\omega dv_*,\\
			&	 K_{2,r}^{\alpha} \mathbf{f}= - \frac{1}{\sqrt{\mu^\alpha}} \sum_{\beta=A,B} \iint_{\mathbb{R}^3 \times \mathbb{S}^2}\mathbf{1}_{|v|+|v_*|>m}
			b^{ \alpha \beta}(\theta)|v-v_*|^{\gamma} \Big[\mu^\alpha(v')\sqrt{\mu^\beta(v_*')}f^\beta(v_*')\\
			& \hspace{8.5cm} +\mu^{\beta}(v_*') \sqrt{\mu^{\alpha}(v')}\, f^\alpha(v')\Big] d\omega dv_*.
		\end{split}	
	\end{equation}	
	To replace ‌$\mathbf{1}_{|v|+|v_*|>m}$ with ‌$\mathbf{1}_{|v|+|v_*|\leq m}$ in ‌\eqref{C2Kscdef}, the resulting expression defines ‌$K_{1,z}^{\alpha} \mathbf{f}$ and ‌$K_{2,z}^{\alpha} \mathbf{f}$, for $i=1,2$, satisfying
	\begin{equation*}
		\begin{split}
			K_{i}^{\alpha} \mathbf{f}=K_{i,r}^{\alpha} \mathbf{f}+K_{i,z}^{\alpha} \mathbf{f}, \qquad K_{z}^{\alpha} \mathbf{f}=K_{1,z}^{\alpha} \mathbf{f}+K_{2,z}^{\alpha} \mathbf{f}\\
			K_{r}^{\alpha} \mathbf{f}=K_{1,r}^{\alpha} \mathbf{f}+K_{2,r}^{\alpha} \mathbf{f}, \qquad K^{\alpha} \mathbf{f}=K_{r}^{\alpha} \mathbf{f}+K_{z}^{\alpha} \mathbf{f}
		\end{split}	
	\end{equation*}
	Furthermore‌, by choosing a smooth cutoff function $\chi_{r}$, 
	\begin{equation}\label{cutoffuctiondef}
		\chi(r) \equiv 1 \text{ for } r \geq 2\varepsilon;\quad \chi(r) \equiv 0 \text{ for } r \leq \varepsilon
	\end{equation}
	 we define
	\begin{equation}\label{C2KSchidef}
		\begin{split}
			&	 K_{1,r}^{\alpha,\chi} \mathbf{f}=  \frac{1}{\sqrt{\mu^\alpha}} \sum_{\beta=A,B} \iint_{\mathbb{R}^3 \times \mathbb{S}^2}\mathbf{1}_{|v|+|v_*|>m}
			b^{ \alpha \beta}(\theta)|v-v_*|^{\gamma} \chi(|v-v_*|) \mu^\alpha(v) \sqrt{\mu^\beta(v_*)}f^\beta(v_*)   d\omega dv_*,\\
			&	 K_{2,r}^{\alpha,\chi} \mathbf{f}= - \frac{1}{\sqrt{\mu^\alpha}} \sum_{\beta=A,B} \iint_{\mathbb{R}^3 \times \mathbb{S}^2}\mathbf{1}_{|v|+|v_*|>m}
			b^{ \alpha \beta}(\theta)|v-v_*|^{\gamma}\chi(|v-v_*|) \Big[\mu^\alpha(v')\sqrt{\mu^\beta(v_*')}f^\beta(v_*')\\
			& \hspace{9.1cm} +\mu^{\beta}(v_*') \sqrt{\mu^{\alpha}(v')}\, f^\alpha(v')\Big] d\omega dv_*.
		\end{split}	
	\end{equation}	
	and
	\begin{equation}\label{C2KSchidefse}
		\begin{split}
			&	 K_{1,z}^{\alpha,\chi} \mathbf{f}=  \frac{1}{\sqrt{\mu^\alpha}} \sum_{\beta=A,B} \iint_{\mathbb{R}^3 \times \mathbb{S}^2}\mathbf{1}_{|v|+|v_*|\leq m}
			b^{ \alpha \beta}(\theta)|v-v_*|^{\gamma} \chi(|v-v_*|) \mu^\alpha(v) \sqrt{\mu^\beta(v_*)}f^\beta(v_*)   d\omega dv_*,\\
			&	 K_{2,z}^{\alpha,\chi} \mathbf{f}= - \frac{1}{\sqrt{\mu^\alpha}} \sum_{\beta=A,B} \iint_{\mathbb{R}^3 \times \mathbb{S}^2}\mathbf{1}_{|v|+|v_*|\leq m}
			b^{ \alpha \beta}(\theta)|v-v_*|^{\gamma}\chi(|v-v_*|) \Big[\mu^\alpha(v')\sqrt{\mu^\beta(v_*')}f^\beta(v_*')\\
			& \hspace{9.1cm} +\mu^{\beta}(v_*') \sqrt{\mu^{\alpha}(v')}\, f^\alpha(v')\Big] d\omega dv_*.
		\end{split}	
	\end{equation}	
	Replace ‌$\chi(|v-v_*|)$ with ‌$\big[1-\chi(|v-v_*|)\big]$ in ‌\eqref{C2KSchidef} and \eqref{C2KSchidefse}, the resulting expression defines ‌$K_{1,r}^{\alpha,1-\chi} ,\,K_{2,r}^{\alpha,1-\chi} ,\,K_{1,z}^{\alpha,1-\chi} $ and ‌$K_{2,z}^{\alpha,1-\chi} $. It can be easily verified‌ that
	\begin{equation}\label{2121212Q1569}
		\begin{split}
		K^{\alpha} =\big\{K^{\alpha}_1+K^{\alpha}_2\big\} =&\big\{K_{1,r}^{\alpha,\chi}+K_{1,z}^{\alpha,\chi}+K_{1,r}^{\alpha,1-\chi}+K_{1,z}^{\alpha,1-\chi}\big\}\\
		&+\big\{K_{2,r}^{\alpha,\chi}+K_{2,z}^{\alpha,\chi}+K_{2,r}^{\alpha,1-\chi}+K_{2,z}^{\alpha,1-\chi}\big\}
		\end{split}	
	\end{equation}	
	
	Furthermore, we define 
	\begin{equation*}
		\mathbf{K}_s=(K^A_s,K^B_s)^{T},\hspace{0.6cm} \mathbf{K}_c=(K^A_c,K^B_c)^{T}
	\end{equation*}
	by
	\begin{equation}\label{2131213Q1569}
		\begin{split}
			&K^{\alpha}_s = K_{1,r}^{\alpha,1-\chi}+K_{1,z}^{\alpha,1-\chi}+K_{2,r}^{\alpha,1-\chi}+K_{2,z}^{\alpha,1-\chi}+K_{1,r}^{\alpha,\chi}+K_{2,r}^{\alpha,\chi}	\\
			&K^{\alpha}_c = K_{1,z}^{\alpha,\chi}+K_{2,z}^{\alpha,\chi}
		\end{split}	
	\end{equation}	
	and define
	\begin{equation}\label{2141214Q1569}
		\begin{split}
				K^{\alpha,1-\chi}_{1,s} &=K_{1,r}^{\alpha,1-\chi}+K_{1,z}^{\alpha,1-\chi}\\
			K^{\alpha,1-\chi}_{2,s} &=K_{2,r}^{\alpha,1-\chi}+K_{2,z}^{\alpha,1-\chi}\\
			K^{\alpha}_{1,s} &=K_{1,r}^{\alpha,1-\chi}+K_{1,z}^{\alpha,1-\chi}+K_{1,r}^{\alpha,\chi}\\
			K^{\alpha}_{2,s} &=K_{2,r}^{\alpha,1-\chi}+K_{2,z}^{\alpha,1-\chi}+K_{2,r}^{\alpha,\chi}
		\end{split}	
	\end{equation}	
	By combining \eqref{2121212Q1569} and \eqref{2131213Q1569}, we can verify that $\mathbf{K} = \mathbf{K}_c + \mathbf{K}_s$.
	
	Moreover, in the following lemma, we will prove that  $\mathbf{K}_s$  is a ``small part'' and   $\mathbf{K}_c$ is a ``compact part''.

	\begin{lema}\label{lemma2.2}
		Let $l\geq0$, $\mathbf{f}_i, (i=1,2)$ be two vector functions. The operator $\mathbf{K}$ defined in \eqref{Ldeopvec} and \eqref{nuKsuanzidef} is split as
		\begin{equation*}
				\mathbf{K} = \mathbf{K}_c + \mathbf{K}_s,
		\end{equation*}
		then for any small $1>\eta>0$,
		\begin{equation}\label{Lem222maineqv}
			\big|\big\langle w^{2l}  \mathbf{K}_s \mathbf{f}_1, \mathbf{f}_2 \big\rangle_{L^{2}_{v}}\big| \leq \eta |w^{l}\mathbf{f}_1|_{\nu} |w^{l}\mathbf{f}_2|_{\nu}.
		\end{equation}
		Moreover, $\mathbf{K}_c$ is a compact operator in $L^2_\nu(\mathbb{R}^3)$ with respect to $|\cdot|_\nu$.
	\end{lema}
	\begin{proof}
		By \eqref{2131213Q1569} and \eqref{2141214Q1569}, we decompose $K_{s}^{\alpha}$ as follows:
		\begin{equation}\label{KS1KS2DPPFJ}
			K_{s}^{\alpha}=K_{1,s}^{\alpha}+K_{2,s}^{\alpha}.
		\end{equation}
		We first consider the part of $K_{1,s}^{\alpha}$, which can be written as
		\begin{equation*}
			K_{1,s}^{\alpha}=K^{\alpha,1-\chi}_{1,s}+K_{1,r}^{\alpha,\chi}.
		\end{equation*}
		Since the angular part of the collision kernel  $|b^{\alpha \beta}|\leq C_{b}  |\cos \theta |$, $\left \langle  w^{2l}K_{1,s}^{\alpha,1-\chi}\mathbf{f}_1,f^{\alpha}_2\right \rangle_{L^{2}_{v}} $ is bounded by
		\begin{align}
			C&\sum_{\beta=A,B}\left( \iint_{(\mathbb{R}^3)^2} \sqrt{\mu^\alpha(v)} \sqrt{\mu^\beta(v_*)}|v - v_*|^\gamma w^{2l}(v) (1 - \chi) [f^\beta_1(v_*)]^2 \, dv_* dv \right)^{1/2} \notag\\
			&\hspace{1.2cm} \times \left( \iint_{(\mathbb{R}^3)^2} \sqrt{\mu^\alpha(v)} \sqrt{\mu^\beta(v_*)}|v - v_*|^\gamma w^{2l}(v) (1 - \chi) [f^{\alpha}_2(v)]^2 \, dv_* dv \right)^{1/2} \notag\\
			&\leq C \sum_{\beta=A,B}\Bigg( \int_{\mathbb{R}^3}  \int_{\mathbb{R}^3} |v - v_*|^\gamma (1 - \chi) dv\, w^{2l+\frac{\gamma}{2}}(v_*) [\mu^\beta(v_*)]^{\frac{1}{4}}[f^\beta_1(v_*)]^2dv_* \Bigg)^{1/2} \notag\\
			&\hspace{2cm} \times \Bigg( \int_{\mathbb{R}^3} \int_{\mathbb{R}^3} |v - v_*|^\gamma (1 - \chi) dv_* \, w^{2l+\frac{\gamma}{2}}(v)[\mu^\alpha(v)]^{\frac{1}{4}}[f^{\alpha}_2(v)]^2dv \Bigg)^{1/2} \notag\\
			&\leq C \sum_{\beta=A,B}\varepsilon^{3+\gamma} \Bigg( \int_{\mathbb{R}^3} w^{2l+\frac{\gamma}{2}}(v_*) [f^\beta_1(v_*)]^2 dv_* \Bigg)^{1/2} \Bigg( \int_{\mathbb{R}^3} w^{2l+\frac{\gamma}{2}}(v)[f^{\alpha}_2(v)]^2 dv \Bigg)^{1/2} \notag\\
			&\leq C \varepsilon^{3+\gamma} \big| w^l \mathbf{f}_1 \big|_\nu \big| w^l \mathbf{f}_2 \big|_\nu.\label{k1ssing}
		\end{align}
		Using the notation \eqref{defmsmsmsmsms}, $\left \langle  w^{2l}K_{1,r}^{\alpha,\chi}\mathbf{f}_1,f^{\alpha}_2\right \rangle_{L^{2}_{v}} $ is bounded by
		\begin{align}
			C& \sum_{\beta=A,B}\iint_{\{|v|+|u| \geq m\}} \sqrt{\mu^\alpha(v)} \sqrt{\mu^\beta(v_*)}|v - u|^\gamma w^{2l}\chi(|v - v_*|) f^{\alpha}_2(v)f^\beta_1(v_*) \, dv_* dv \notag\\
			&\leq C \sum_{\beta=A,B}[\mu^s(m)]^{\frac{1}{8}} \iint_{(\mathbb{R}^3)^2} [\mu^\beta(v_*)]^{\frac{1}{4}}[\mu^\alpha(v)]^{\frac{1}{4}}|v - v_*|^\gamma |f^{\alpha}_2(v)f^\beta_1(v_*)| \, dv_* dv \notag\\
			&\leq C \sum_{\beta=A,B}[\mu^s(m)]^{\frac{1}{8}} \left( \iint_{(\mathbb{R}^3)^2} [\mu^\beta(v_*)]^{\frac{1}{4}}[\mu^\alpha(v)]^{\frac{1}{4}}|v - v_*|^\gamma [f^\beta_1(v_*)]^2 \, dv_* dv \right)^{1/2} \notag\\
			&\hspace{3.4cm} \times \left( \iint_{(\mathbb{R}^3)^2} [\mu^\beta(v_*)]^{\frac{1}{4}}[\mu^\alpha(v)]^{\frac{1}{4}}|v - v_*|^\gamma [f^\alpha_2(v)]^2 \, dv_* dv \right)^{1/2} \notag\\
			&\leq C [\mu^s(m)]^{\frac{1}{8}}  \big| w^l \mathbf{f}_1 \big|_\nu \big| w^l \mathbf{f}_2 \big|_\nu. \label{k1sregu}
		\end{align}
		The proof for $K^{\alpha}_2$ part is more complicated than that of $K^{\alpha}_1$. 
		We write $K_{2,s}^{\alpha}$ as
		\begin{equation*}
			K_{2,s}^{\alpha}=K^{\alpha,1-\chi}_{2,s}+K_{2,r}^{\alpha,\chi}.
		\end{equation*}
		 The expression for $K^{\alpha}_2$ in \eqref{nuKsuanzidef} and \eqref{2141214Q1569} can be simplified, by using the conservation of energy:  $m^{\beta}|v_*'|^{2}+m^{\alpha}|v'|^2 =m^{\beta}|v_*|^2+m^{\alpha}|v|^2$,
		\begin{equation*}
			\begin{split}
				&	 K_{2,s}^{\alpha,1-\chi} \mathbf{f}_1= -  \sum_{\beta=A,B} \iint_{\mathbb{R}^3 \times \mathbb{S}^2}
				b^{ \alpha \beta}(\theta)|v-v_*|^{\gamma}(1 - \chi) \sqrt{\mu^\beta(v_*)}\\  &\hspace{6cm}\times\Big(\sqrt{\mu^\alpha(v')}f^\beta_1(v_*')
				 +\sqrt{\mu^{\beta}(v_*')} \, f^\alpha_1(v')\Big) d\omega dv_*.
			\end{split}	
		\end{equation*}	
		‌Therefore, $\left \langle  w^{2l}K_{2,s}^{\alpha,1-\chi}\mathbf{f}_1,f^{\alpha}_2\right \rangle_{L^{2}_{v}} $  is decomposed into two components based on $K_{2,s}^{\alpha}$. The part containing $f^\alpha_1(v')$ is bounded by
		\begin{equation}\label{lem22k2ssig}
			\begin{split}
				C\sum_{\beta=A,B} &\left(\iiint_{(\mathbb{R}^3)^2 \times \mathbb{S}^2}
				\sqrt{\mu^\beta(v_*)}\sqrt{\mu^{\beta}(v_*')}|v - v_*|^\gamma w^{2l} (1 - \chi)[f^\alpha_1(v')]^2
				d\omega dv_*dv\right)^{1/2}\\
				&\hspace{0.5cm}\times
				\left(\iiint_{(\mathbb{R}^3)^2 \times \mathbb{S}^2}
				\sqrt{\mu^\beta(v_*)}\sqrt{\mu^{\beta}(v_*')}|v - v_*|^\gamma w^{2l} (1 - \chi)[f^\alpha_2(v)]^2
				d\omega dv_*dv\right)^{1/2}
			\end{split}	
		\end{equation}	
		Although $m^A \neq m^B$, a direct computation shows that  $|v_* - v|=|v_*' - v'|$ still holds. Since $|v_* - v|\leq 2 \varepsilon$, it follows that
		\begin{align*}
			|v_*'| &= \big|v' - (v' - v_*')\big| \geq |v'| - |v- v_*| \geq |v'| - 2\varepsilon, \\
			|v_*| &= \big|v_*' + \frac{2m^{\alpha}}{m^{\alpha}+m^{\beta}}[(v_* - v) \cdot \omega]\omega\big| \geq |v_*'| - 2|v_* - v| \geq |v'| - 6\varepsilon.
		\end{align*}
		Therefore, for $0<\varepsilon<1$, 
		\begin{equation*}
			\sqrt{\mu^\beta(v_*)}\leq C\sqrt{\mu^{\beta}(v')}
		\end{equation*}
		By \eqref{ddstardstar}, the square of first line in \eqref{lem22k2ssig} is bounded by
		\begin{align*}
			C&\sum_{\beta=A,B} \iiint_{(\mathbb{R}^3)^2 \times \mathbb{S}^2} \sqrt{\mu^\beta(v_*)}\sqrt{\mu^{\beta}(v_*')}w^{2l}(v) |v - v_*|^\gamma \big[1 - \chi(|v_* - v|)\big] [f^\alpha_1(v')]^2 \,dv_*\,dv\,d\omega \\
			&\leq C \sum_{\beta=A,B} \iint_{(\mathbb{R}^3)^2} \sqrt{\mu^\beta(v')}\sqrt{\mu^{\beta}(v_*')}w^{2l}(v) |v' - v_*'|^\gamma \big[1 - \chi(|v_*' - v'|)\big] [f^\alpha_1(v')]^2 \,dv_*\,dv \\
			&\leq C \sum_{\beta=A,B} \iint_{(\mathbb{R}^3)^2} [\mu^\beta(v')]^{\frac{1}{4}}[\mu^{\beta}(v_*')]^{\frac{1}{4}}w^{2l}(v') |v' - v_*'|^\gamma \big[1 - \chi(|v_*' - v'|)\big] [f^\alpha_1(v')]^2 \,dv_*\,dv \\
			&\leq C \sum_{\beta=A,B} \int_{|v' - v_*'| \leq 2\varepsilon} [\mu^{\beta}(v_*')]^{\frac{1}{4}} |v' - v_*'|^\gamma \,dv_*' \int_{\mathbb{R}^3} [\mu^\beta(v')]^{\frac{1}{4}}w^{2l}(v') [f^\alpha_1(v')]^2 \,dv' \\
			&\leq C\varepsilon^{3+\gamma} |w^{l} \mathbf{f}_1|_\nu^2
		\end{align*}
		We note that
		\begin{align*}
			|v_*| &= |v - (v - v_*)| \geq |v| - |v - v_*| \geq |v| - 2\varepsilon, \\
			|v_*'| &= |v_* - \frac{2m^{\alpha}}{m^{\alpha}+m^{\beta}}[(v_* - v) \cdot \omega]\omega| \geq |v_*| - 2|v_* - v| \geq |v| - 6\varepsilon,
		\end{align*}
		so it holds that $\sqrt{\mu^\beta(v_*')}\leq C \sqrt{\mu^\beta(v)}$. 
		
		Hence, the square of second line in \eqref{lem22k2ssig} is
		bounded by
		\begin{align*}
			C &\sum_{\beta=A,B}\int_{\mathbb{R}^3} \int_{\mathbb{R}^3}\sqrt{\mu^\beta(v_*)}\sqrt{\mu^{\beta}(v)} |v - u|^\gamma \big[1 - \chi(|v_* - v|)\big] \,dv_*\,[f^\alpha_2(v)]^2\,dv \\
			&\leq C \sum_{\beta=A,B} \int_{\mathbb{R}^3} \int_{|v - v_*| \leq 2\varepsilon} \sqrt{\mu^\beta(v_*)} |v - v_*|^\gamma \,dv_*  \sqrt{\mu^{\beta}(v)} [f^\alpha_2(v)]^2 \,dv \\
			&\leq C\varepsilon^{3+\gamma} \big|w^{l} \mathbf{f}_2\big|_\nu^2
		\end{align*}
		So, \eqref{lem22k2ssig} is bounded by $C\varepsilon^{3+\gamma} |w^{l} \mathbf{f}_1|_\nu \big|w^{l} \mathbf{f}_2\big|_\nu$.
		
		Therefore, we have obtained the estimates of the first part in $\left \langle  w^{2l}K_{2,s}^{\alpha,1-\chi}\mathbf{f}_1,f^{\alpha}_2\right \rangle_{L^{2}_{v}} $.
		
		Similarly, the part of containing $f^\beta_1(v_*')$ in $\left \langle  w^{2l}K_{2,s}^{\alpha,1-\chi}\mathbf{f}_1,f^{\alpha}_2\right \rangle_{L^{2}_{v}} $  is bounded by
		\begin{equation}\label{lemmass213213so}
			\begin{split}
					C\sum_{\beta=A,B} &\left(\iiint_{(\mathbb{R}^3)^2 \times \mathbb{S}^2}
				\sqrt{\mu^\beta(v_*)}\sqrt{\mu^\alpha(v')}|v - v_*|^\gamma w^{2l} (1 - \chi)[f^\beta_1(v_*')]^2
				d\omega dv_*dv\right)^{1/2}\\
				&\hspace{0.5cm}\times
				\left(\iiint_{(\mathbb{R}^3)^2 \times \mathbb{S}^2}
				\sqrt{\mu^\beta(v_*)}\sqrt{\mu^\alpha(v')}|v - v_*|^\gamma w^{2l} (1 - \chi)[f^\alpha_2(v)]^2
				d\omega dv_*dv\right)^{1/2}
			\end{split}	
		\end{equation}	
		Since $|v_* - v|\leq 2 \varepsilon$, 
		\begin{align*}
			|v_*| &= |v_*' - \frac{2m^{\alpha}}{m^\alpha+m^\beta}[(v-v_*)\cdot \omega]\omega| \geq |v_*'| - 2|v_* - v| \geq |v_*'| - 4\varepsilon, \\
			|v'| &= |v-\frac{2m^{\beta}}{m^\alpha+m^\beta}[(v-v_*)\cdot \omega]\omega| \geq |v| - 2|v_* - v| \geq |v| - 4\varepsilon
		\end{align*}
		For $\varepsilon<1$, we have
		\begin{equation*}
			\sqrt{\mu^\beta(v_*)}\leq C\sqrt{\mu^{\beta}(v_*')},\quad \sqrt{\mu^{\alpha}(v')}\leq C\sqrt{\mu^{\alpha}(v)}.
		\end{equation*}
		So, the square of the first line in \eqref{lemmass213213so} is bounded by
		\begin{align*}
			C&\sum_{\beta=A,B} \iiint_{(\mathbb{R}^3)^2 \times \mathbb{S}^2} \sqrt{\mu^{\alpha}(v')}\sqrt{\mu^{\beta}(v_*')}w^{2l}(v) |v - v_*|^\gamma \big[1 - \chi(|v_* - v|)\big] [f^{\beta}_1(v_*')]^2 \,dv_*\,dv\,d\omega \\
			&\hspace{0.3cm}\leq C \sum_{\beta=A,B} \iint_{(\mathbb{R}^3)^2} \sqrt{\mu^{\alpha}(v')}\sqrt{\mu^{\beta}(v_*')}w^{2l}(v) |v' - v_*'|^\gamma \big[1 - \chi(|v_*' - v'|)\big] [f^{\beta}_1(v_*')]^2 \,dv_*'\,dv' \\
				&\hspace{0.3cm}\leq C \sum_{\beta=A,B} \iint_{(\mathbb{R}^3)^2} [\mu^{\alpha}(v')]^{\frac{1}{4}}[\mu^{\beta}(v_*')]^{\frac{1}{4}}w^{2l}(v_*') |v' - v_*'|^\gamma \big[1 - \chi(|v_*' - v'|)\big] [f^{\beta}_1(v_*')]^2 \,dv_*'\,dv' \\
				&\hspace{0.3cm}\leq C \sum_{\beta=A,B} \int_{|v' - v_*'| \leq 2\varepsilon} [\mu^{\alpha}(v')]^{\frac{1}{4}} |v' - v_*'|^\gamma \,dv' \int_{\mathbb{R}^3} [\mu^\beta(v'_*)]^{\frac{1}{4}}w^{2l}(v'_*) [f^{\beta}_1(v_*')]^2 \,dv_*' \\
				&\hspace{0.3cm}\leq C\varepsilon^{3+\gamma} \big|w^{l} \mathbf{f}_1\big|_\nu^2
		\end{align*}
		The square of the second line in \eqref{lemmass213213so} is bounded by
		\begin{align*}
			C &\sum_{\beta=A,B}\int_{\mathbb{R}^3} \int_{\mathbb{R}^3}\sqrt{\mu^\beta(v_*)}\sqrt{\mu^{\alpha}(v)} |v - u|^\gamma \big[1 - \chi(|v_* - v|)\big] \,dv_*\,[f^\alpha_2(v)]^2\,dv \\
			&\leq C \sum_{\beta=A,B} \int_{\mathbb{R}^3} \int_{|v - v_*| \leq 2\varepsilon} \sqrt{\mu^\beta(v_*)} |v - v_*|^\gamma \,dv_*  \sqrt{\mu^{\alpha}(v)} [f^\alpha_2(v)]^2 \,dv \\
			&\leq C \varepsilon^{3+\gamma} \big|w^{l} \mathbf{f}_2\big|_\nu^2
		\end{align*}
		So, \eqref{lemmass213213so} is also bounded by $C\varepsilon^{3+\gamma} \big|w^{l} \mathbf{f}_1\big|_\nu \big|w^{l} \mathbf{f}_2\big|_\nu$.
		
		‌In conclusion‌, we have
		\begin{equation}\label{lem22cclssig}
		\left \langle  w^{2l}\mathbf{K}_{2,s}^{1-\chi}\mathbf{f}_1,\mathbf{f}_2\right \rangle_{L^{2}_{v}}=	\sum_{\alpha=A,B}\left \langle  w^{2l}K_{2,s}^{\alpha,1-\chi}\mathbf{f}_1,f^{\alpha}_2\right \rangle_{L^{2}_{v}} \leq C \varepsilon^{3+\gamma} \big|w^{l} \mathbf{f}_1\big|_\nu \big|w^{l} \mathbf{f}_2\big|_\nu
		\end{equation}
		We now analyze the rest part  
		\begin{equation}\label{lem22K2schimainf}
			\begin{split}
				&	 -K_{2,r}^{\alpha,\chi} \mathbf{f}_1=   \sum_{\beta=A,B} \int_{\mathbb{R}^3 \times \mathbb{S}^2}\mathbf{1}_{|v|+|v_*|>m}
				b^{ \alpha \beta}(\theta)|v-v_*|^{\gamma} \sqrt{\mu^\beta(v_*)}\chi(|v-v_*'|)\\ & \hspace{5.6cm}\times\Big(\sqrt{\mu^\alpha(v')}f^\beta_1(v_*')
				 +\sqrt{\mu^{\beta}(v_*')} \, f^\alpha_1(v')\Big)d\omega dv_*.
			\end{split}	
		\end{equation}	
		Since its truncation function no longer supplies a small quantity, we perform a change of variable
		\begin{align}
			&u=v_*-v,\quad u_{\parallel} = (u \cdot \omega)\omega, \quad 
			u_{\perp} = u - (u \cdot \omega)\omega,  \notag\\
			&\zeta_{\parallel} + \zeta_{\perp} = v + \frac{m^{\beta}}{m^{\alpha}+m^{\beta}}u_{\parallel}, \quad 
			\zeta_{\parallel} \parallel u_{\parallel}, \quad \zeta_{\perp} \parallel u_{\perp},\label{CE2ESYCGV}
		\end{align}
		where $u \in \mathbb{R}^3$, $\omega \in \mathbb{S}^2$, $u_{\parallel} \in \mathbb{R}^3$, $u_{\perp} \in \mathbb{R}^2$, and
		\begin{equation*}
			|u_{\parallel}|^{2}(|u_{\parallel}|^{-1}d|u_{\parallel}|)d\omega = du_{\parallel} ,\qquad du ={du}^{+}+{du}^{-} =2(|u_{\parallel}|^{-1} d|u_{\parallel}|) du_{\perp}. 
		\end{equation*}
		Here, $(|u_{\parallel}|^{-1}d|u_{\parallel}|)$ can be understood as ``$dr$'' in spherical coordinate transformation. Since the direction of the coordinate variable $(|u_{\parallel}|^{-1}d|u_{\parallel}|)$ is only one direction of $\omega$ or $-\omega$. So, it yields
		\begin{equation*}
			{du}^{+}=  \mathbf{1}_{\{u\cdot\omega\geq0\}}du=(|u_{\parallel}|^{-1} d|u_{\parallel}|) du_{\perp},\hspace{0.5cm} {du}^{-}= \mathbf{1}_{\{u\cdot\omega<0\}} du=(|u_{\parallel}|^{-1} d|u_{\parallel}|) du_{\perp}.
		\end{equation*}
		Therefore, 
		\begin{equation}
			du d\omega = 2|u_{\parallel}|^{-2} du_{\perp} du_{\parallel}.
		\end{equation}
		Then $(-1)\times$ \eqref{lem22K2schimainf} can be expressed by
		\begin{align}
			\sum_{\beta=A,B}&\Bigg(\int_{|v|+|v_*|>m}\int_{\mathbb{S}^2} |u|^{\gamma} [\mu^{\beta}(u+v)]^{\frac{1}{2}}\chi(|u|)\big[\mu^{\beta}(v + u_{\perp}+\frac{m^{\beta}-m^{\alpha}}{m^{\alpha}+m^{\beta}}u_{\parallel})\big]^{\frac{1}{2}}\notag\\
			&\hspace{5.5cm}\times f^\alpha_1(v + \frac{2m^{\beta}}{m^{\alpha}+m^{\beta}}u_{\parallel})b^{ \alpha \beta}(\theta) \, du \, d\omega\Bigg) \label{CE2easyPT}\\
			& +\int_{|v|+|v_*|>m}\int_{\mathbb{S}^2} |u|^{\gamma} [\mu^{\beta}(u+v)]^{\frac{1}{2}}\chi(|u|)\big[\mu^{\alpha}(v + \frac{2m^{\beta}}{m^{\alpha}+m^{\beta}}u_{\parallel})\big]^{\frac{1}{2}}
			\hspace{1cm}(\alpha \neq \beta)\notag\\
			&\hspace{5cm}\times f^\beta_1(v + u_{\perp}+\frac{m^{\beta}-m^{\alpha}}{m^{\alpha}+m^{\beta}}u_{\parallel})b^{ \alpha \beta}(\theta) \, du \, d\omega \label{CE2hadPT}\\
			&+\int_{|v|+|v_*|>m}\int_{\mathbb{S}^2} |u|^{\gamma} [\mu^{\alpha}(u+v)]^{\frac{1}{2}}\chi(|u|)\big[\mu^{\alpha}(v + u_{\parallel})\big]^{\frac{1}{2}} f^\alpha_1(v + u_{\perp})b^{ \alpha \alpha}(\theta) \, du \, d\omega \label{CE2hadDGP}
		\end{align}
		Here, \eqref{CE2easyPT} is the integral for ``Typical part''. \eqref{CE2hadPT}$+$\eqref{CE2hadDGP} is the integral for ``Hybrid part'', where \eqref{CE2hadPT} exhibit exponential decay and \eqref{CE2hadDGP} is the degenerate part.
		
		Notice that 
		\begin{equation*}
			|u_{\parallel}| = |u \cos \theta| = \sqrt{|u_{\parallel}|^2 + |u_{\perp}|^2} \, |\cos \theta|
		\end{equation*}
		Furthermore, by \eqref{CE2ESYCGV}, it holds that
		\begin{align}
			u + v &= u_{\perp} + \zeta_{\perp} + \zeta_{\parallel} + \frac{m^\alpha}{m^\alpha+m^\beta}u_{\parallel}, \label{CEPT2ACA1vb}\\
			v + u_{\perp}&+\frac{m^{\beta}-m^{\alpha}}{m^{\alpha}+m^{\beta}}u_{\parallel} = u_{\perp} + \zeta_{\perp} + \zeta_{\parallel} - \frac{m^\alpha}{m^\alpha+m^\beta}u_{\parallel},\notag
		\end{align}
		where
		\begin{equation}\label{zetperzetpara}
			 \zeta_{\perp}=v_{\perp},\hspace{0.5cm}\zeta_{\parallel}=v_{\parallel}+\frac{m^\beta}{m^\alpha+m^\beta}u_{\parallel}.
		\end{equation}
		Therefore, the exponent in \eqref{CE2easyPT} is $-\frac{1}{2}\Big(|u_{\perp} + \zeta_{\perp}|^2+|\zeta_{\parallel}|^2+(\frac{m^\alpha}{m^\alpha+m^\beta})^2|u_{\parallel}|^2\Big)$.	\\
		Hence, \eqref{CE2easyPT} can be written as
		\begin{align}
			\sum_{\beta=A,B}& \int_{\mathbb{R}^3} \frac{1}{|u_{\parallel}|^2} e^{-\frac{m^\beta}{2}\big[|\zeta_{\parallel}|^2+(\frac{m^\alpha}{m^\alpha+m^\beta})^2|u_{\parallel}|^2\big]} f^\alpha_1(v + \frac{2m^{\beta}}{m^{\alpha}+m^{\beta}}u_{\parallel}) \notag\\
			&\quad \times \int_{\mathbb{R}^2} e^{-\frac{m^\beta}{2}|u_{\perp} + \zeta_{\perp}|^2} \mathbf{1}_{|v|+|v_*|>m}\left[|u_{\parallel}|^2 + |u_{\perp}|^2\right]^{\frac{\gamma}{2}} \chi\big(\sqrt{|u_{\parallel}|^2 + |u_{\perp}|^2}\big) b^{ \alpha \beta}(\theta) \, du_{\perp} du_{\parallel} \notag\\
			=&\sum_{\beta=A,B} \int_{\mathbb{R}^3} \frac{1}{|u_{\parallel}|} e^{-\frac{m^\beta}{2}\big[|\zeta_{\parallel}|^2+(\frac{m^\alpha}{m^\alpha+m^\beta})^2|u_{\parallel}|^2\big]} f^\alpha_1(v + \frac{2m^{\beta}}{m^{\alpha}+m^{\beta}}u_{\parallel})\notag \\
			&\quad \times \int_{\mathbb{R}^2} e^{-\frac{m^\beta}{2}|u_{\perp} + \zeta_{\perp}|^2} \mathbf{1}_{|v|+|v_*|>m}\left[|u_{\parallel}|^2 + |u_{\perp}|^2\right]^{\frac{\gamma - 1}{2}} \chi\big(\sqrt{|u_{\parallel}|^2 + |u_{\perp}|^2}\big) \frac{b^{ \alpha \beta}(\theta)}{|\cos \theta|} \, du_{\perp} du_{\parallel}\label{CE2BD221fSesD1}.
		\end{align}
		We introduce the integral kernel:
		\begin{align}
			k_{\alpha \beta}^{(1),r}(v,u_{\parallel}) &=:  \frac{1}{|u_{\parallel}|} e^{-\frac{m^\beta}{2}\big[|\zeta_{\parallel}|^2+(\frac{m^\alpha}{m^\alpha+m^\beta})^2|u_{\parallel}|^2\big]} \int_{\mathbb{R}^2} e^{-\frac{m^\beta}{2}|u_{\perp} + \zeta_{\perp}|^2} \mathbf{1}_{|v|+|v_*|>m}\notag\\
			&\hspace{2.5cm}\times  \left[|u_{\parallel}|^2 + |u_{\perp}|^2\right]^{\frac{\gamma - 1}{2}} \chi\big(\sqrt{|u_{\parallel}|^2 + |u_{\perp}|^2}\big) \frac{b^{ \alpha \beta}(\theta)}{|\cos \theta|} \, du_{\perp}.\label{ker1rrr}
		\end{align}
		Then \eqref{CE2easyPT} can be further written as
		\begin{equation*}
		\sum_{\beta=A,B}	\int_{\mathbb{R}^3}k_{\alpha \beta}^{(1),r}(v,u_{\parallel})\,f^\alpha(v + \frac{2m^{\beta}}{m^{\alpha}+m^{\beta}}u_{\parallel}) \,d u_{\parallel}
		\end{equation*}
		
		Next we turn to the other part of \eqref{lem22K2schimainf}. Note that in the part of \eqref{CE2hadPT}, $\mu^{\alpha}$ and $\mu^{\beta}$ may be different. To clearly characterize their exponential decay, we apply the following change of variable:
		\begin{eqnarray}\label{lem2and222cgv}
			\begin{split}
				\xi_{\parallel} + \xi_{\perp} &= \frac{\sqrt{m^\alpha}+\sqrt{m^\beta}}{2}v + \big(\frac{\sqrt{m^\beta}}{2}+\frac{\sqrt{m^{\alpha}}m^{\beta}}{m^{\alpha}+m^{\beta}}\big)u_{\parallel}+\frac{\sqrt{m^\beta}}{2}u_{\perp}, 
			\end{split}
		\end{eqnarray}
		where $\xi_{\parallel} \parallel u_{\parallel},\,\, \xi_{\perp} \parallel u_{\perp}$. It follows that
		\begin{align*}
			&\sqrt{m^\alpha}(v + \frac{2m^{\beta}}{m^{\alpha}+m^{\beta}}u_{\parallel})\\
			&\hspace{1.7cm}= \xi_{\parallel} + \xi_{\perp}-\Big[\frac{\sqrt{m^\beta}-\sqrt{m^\alpha}}{2}v+\big(\frac{\sqrt{m^\beta}}{2}-\frac{\sqrt{m^{\alpha}}m^{\beta}}{m^{\alpha}+m^{\beta}}\big)u_{\parallel}+\frac{\sqrt{m^\beta}}{2}u_{\perp}\Big], \\
			&\sqrt{m^\beta}(v +u_{\parallel}+ u_{\perp}) \\
			&\hspace{1.7cm}= \xi_{\parallel} + \xi_{\perp}+\Big[\frac{\sqrt{m^\beta}-\sqrt{m^\alpha}}{2}v+\big(\frac{\sqrt{m^\beta}}{2}-\frac{\sqrt{m^{\alpha}}m^{\beta}}{m^{\alpha}+m^{\beta}}\big)u_{\parallel}+\frac{\sqrt{m^\beta}}{2}u_{\perp}\Big]
		\end{align*}
		To keep the notation concise, we make the following change of variable
		\begin{eqnarray}\label{lem2223cgvb3}
			\begin{split}
				\eta_{\parallel} &= \frac{\sqrt{m^\beta}-\sqrt{m^\alpha}}{2}v_{\parallel} + \big(\frac{\sqrt{m^\beta}}{2}-\frac{\sqrt{m^{\alpha}}m^{\beta}}{m^{\alpha}+m^{\beta}}\big)u_{\parallel} \\
				\eta_{\perp} &= \frac{\sqrt{m^\beta}-\sqrt{m^\alpha}}{2}v_{\perp} +\frac{\sqrt{m^\beta}}{2}u_{\perp},
			\end{split}
		\end{eqnarray}
		where $v=v_{\parallel}+v_{\perp}, v_{\parallel} \parallel u_{\parallel}, v_{\perp}\parallel u_{\perp}$.
		Then, we deduce that the exponent in \eqref{CE2hadPT} is 
		\begin{equation*}
			-\frac{1}{4}\Big(|\xi_{\parallel}|^2+|\xi_{\perp}|^2+|\eta_{\parallel}|^2+|\eta_{\perp}|^2\Big).
		\end{equation*}
		Therefore, \eqref{CE2hadPT} can be written as
		\begin{align}
			& \sum_{\beta=A,B}\int_{\mathbb{R}^3} \frac{1}{|u_{\parallel}|^2} e^{-\frac{1}{4}(|\xi_{\parallel}|^2+|\eta_{\parallel}|^2)}\int_{\mathbb{R}^2}\mathbf{1}_{|v|+|v_*|>m}b^{ \alpha \beta}(\theta)\chi\big(\sqrt{|u_{\parallel}|^2 + |u_{\perp}|^2}\big)  \notag\\
			&\hspace{1cm}\times  e^{-\frac{1}{4}(|\xi_{\perp}|^2+|\eta_{\perp}|^2)} \left[|u_{\parallel}|^2 + |u_{\perp}|^2\right]^{\frac{\gamma}{2}}   f^\beta_1(v + u_{\perp}+\frac{m^{\beta}-m^{\alpha}}{m^{\alpha}+m^{\beta}}u_{\parallel})\, du_{\perp} du_{\parallel} \notag\\
			&\hspace{0.4cm}= \int_{\mathbb{R}^3} \frac{1}{|u_{\parallel}|} e^{-\frac{1}{4}(|\xi_{\parallel}|^2+|\eta_{\parallel}|^2)}\int_{\mathbb{R}^2}\mathbf{1}_{|v|+|v_*|>m} \frac{b^{ \alpha \beta}(\theta)}{|\cos \theta|}\chi\big(\sqrt{|u_{\parallel}|^2 + |u_{\perp}|^2}\big)  \notag\\
			&\hspace{1cm}\times  e^{-\frac{1}{4}(|\xi_{\perp}|^2+|\eta_{\perp}|^2) } \left[|u_{\parallel}|^2 + |u_{\perp}|^2\right]^{\frac{\gamma - 1}{2}}f^\beta_1(v + u_{\perp}+\frac{m^{\beta}-m^{\alpha}}{m^{\alpha}+m^{\beta}}u_{\parallel})  \, du_{\perp} du_{\parallel}.\label{CEs2BDSHH224D2}
		\end{align}
		We introduce the integral kernel $(\alpha \neq \beta)$:
		\begin{align}
			k_{\alpha \beta}^{(2),r}(v,u_{\perp},u_{\parallel})	&=:  \frac{1}{|u_{\parallel}|} e^{-\frac{1}{4}(|\xi_{\parallel}|^2+|\eta_{\parallel}|^2)} \chi\big(\sqrt{|u_{\parallel}|^2 + |u_{\perp}|^2}\big)\mathbf{1}_{|v|+|v_*|>m} \notag\\
			&\hspace{3cm}\times  e^{-\frac{1}{4}(|\xi_{\perp}|^2+|\eta_{\perp}|^2) } \left[|u_{\parallel}|^2 + |u_{\perp}|^2\right]^{\frac{\gamma - 1}{2}}\frac{b^{ \alpha \beta}(\theta)}{|\cos \theta|}.\label{ker2rrr}
		\end{align}
			Then \eqref{CE2hadPT} can be further written as
			\begin{equation*}
				\sum_{\beta=A,B}\iint_{\mathbb{R}^3\times\mathbb{R}^2}k_{\alpha \beta}^{(2),r}(v,u_{\perp},u_{\parallel}) \,f^\beta(v + u_{\perp}+\frac{m^{\beta}-m^{\alpha}}{m^{\alpha}+m^{\beta}}u_{\parallel})\,du_{\perp}\,d u_{\parallel}.
			\end{equation*}
		
		 If $m^\alpha>m^\beta$, we perform a change of variable:
		$-u\rightarrow u$, $-v\rightarrow v$.
		Without loss of generality, we assume $m^\beta\geq m^\alpha$, and we will divide the proof of \eqref{Lem222maineqv} for $\mathbf{K}_{2,r}^{\chi}$ part into two steps.
		
		$\mathbf{Step 1}$. We consider the decay estimate for the part \eqref{CE2BD221fSesD1} in $\left \langle  w^{2l}K_{2,r}^{\alpha,\chi}\mathbf{f}_1,f^{\alpha}_2\right \rangle_{L^{2}_{v}}$.
		
		 Noting that after performing the variable substitution $(v_*- v)\rightarrow u$ in  \eqref{CE2ESYCGV}, the region of integration becomes 
		\begin{equation*}
			\big(|v|+|v_*|\geq m\big) =  \big(|v|+|u+v|\geq m\big)\subseteq \big(|v|+|u|\geq \frac{m}{2}\big)
		\end{equation*}
		Therefore, we partition the integration domain as follows: 
		\begin{equation*}
			\begin{split}
				\big(|v|+|v_*|\geq m\big)\subseteq \big(|v|+|u|\geq \frac{m}{2}\big)\subseteq\Big[(|v|\geq \frac{m}{10})\cup(|u|\geq \frac{2m}{5})\Big]
			\end{split}
		\end{equation*}
		We further  split 
		\begin{equation*}
			\begin{split}
				\big(|v|\geq \frac{m}{10}\big)=&\Biggl\{\Big[(|v|\geq \frac{m}{10})\cap(|u_\parallel|\geq \frac{|v|}{8})\Big]\bigcup\Big[(|v|\geq \frac{m}{10})\cap(|u_\perp|\geq \frac{|v|}{8})\Big]\\
				&\hspace{0.4cm}\bigcup\Big[(|v|\geq \frac{m}{10})\cap(|v_\parallel|\geq \frac{|v|}{2})\cap(|u_\parallel|\leq \frac{|v|}{8})\cap(|u_\perp|\leq \frac{|v|}{8})\Big]\\
				&\hspace{0.4cm}\bigcup\Big[(|v|\geq \frac{m}{10})\cap(|v_\perp|\geq \frac{|v|}{2})\cap(|u_\parallel|\leq \frac{|v|}{8})\cap(|u_\parallel|\leq \frac{|v|}{8})\Big]\Biggr\}
			\end{split}
		\end{equation*}
		and
		\begin{equation*}
			\begin{split}
				\Big[(|u|\geq \frac{2m}{5})\cap	(|v|\leq \frac{m}{10})\Big]=&\Biggl\{\Big[(|u|\geq \frac{2m}{5})\cap(|v|\leq \frac{m}{10})\cap	(|u_\parallel|\geq \frac{|u|}{2})\Big]\\
				&\hspace{0.4cm}\bigcup\Big[(|u|\geq \frac{2m}{5})\cap(|v|\leq \frac{m}{10})\cap(|u_\perp|\geq \frac{|u|}{2})\Big]\Biggr\}
			\end{split}
		\end{equation*}
		Thus, the estimate of $\left \langle  w^{2l}K_{2,r}^{\alpha,\chi}\mathbf{f}_1,f^{\alpha}_2\right \rangle_{L^{2}_{v}}$ can be divided into the following four cases for discussion.

		\noindent {\bf Case A1. $D_{A1}=\Big[(|v|\geq \frac{m}{10})\cap(|u_\parallel|\geq \frac{|v|}{8})\Big]\bigcup\Big[(|u|\geq \frac{2m}{5})\cap(|v|\leq \frac{m}{10})\cap	(|u_\parallel|\geq \frac{|u|}{2})\Big]$}.

		Since $\chi = \chi_\varepsilon$ vanishes near the origin and
		$|B(\theta)| \leq C|\cos \theta|$,  for the chosen $\varepsilon > 0$ and any given $m > 0$,
		\begin{equation*}
			\begin{split}
				&\int_{\mathbb{R}^2} e^{-\frac{m^\beta}{2}|u_{\perp} + \zeta_{\perp}|^2} \left[|u_{\parallel}|^2 + |u_{\perp}|^2\right]^{\frac{\gamma - 1}{2}} \chi\big(\sqrt{|u_{\parallel}|^2 + |u_{\perp}|^2}\big) \frac{b^{ \alpha \beta}(\theta)}{|\cos \theta|} \, du_{\perp} \\
				&\hspace{0.5cm}\leq C_\varepsilon \int_{\mathbb{R}^2} e^{-\frac{m^\beta}{2}|u_\perp + \zeta_\perp|^2} du_\perp < \infty,
			\end{split}
		\end{equation*}
		Notice that in both case of A1, $|u_\parallel|\geq\frac{|v|}{8}\geq \frac{m}{80}>1$, we have
		\begin{equation}\label{caseA1wgams}
		C e^{-\frac{m^\beta}{16}(\frac{m^\alpha}{m^\alpha+m^\beta})^2|u_{\parallel}|^2}	\leq w^{\gamma-1}(v).
		\end{equation}
		Then, the part of \eqref{CE2BD221fSesD1} in $\left \langle  w^{2l}K_{2,r}^{\alpha,\chi}\mathbf{f}_1,f^{\alpha}_2\right \rangle_{L^{2}_{v}}$ is bounded by 
		\begin{align}
			& \sum_{\beta=A,B}\iint_{(\mathbb{R}^3)^2} \frac{1}{|u_{\parallel}|} e^{-\frac{1}{2}\big[|\zeta_{\parallel}|^2+(\frac{m^\alpha}{m^\alpha+m^\beta})^2|u_{\parallel}|^2\big]} f^\alpha_1(v + \frac{2m^{\beta}}{m^{\alpha}+m^{\beta}}u_{\parallel}) w^{2l}(v)f^\alpha_2(v)\notag\\
			&\hspace{1cm} \times \int_{\mathbb{R}^2} e^{-\frac{1}{2}|u_{\perp} + \zeta_{\perp}|^2} \left[|u_{\parallel}|^2 + |u_{\perp}|^2\right]^{\frac{\gamma - 1}{2}} \chi\big(\sqrt{|u_{\parallel}|^2 + |u_{\perp}|^2}\big) \frac{b^{ \alpha \beta}(\theta)}{|\cos \theta|} \mathbf{1}_{D_{A1}}\, du_{\perp} du_{\parallel}dv\notag\\
			&\hspace{0.4cm}\leq \frac{C_\varepsilon}{m} \iint_{(\mathbb{R}^3)^2}e^{-\frac{1}{2}(\frac{m^\alpha}{m^\alpha+m^\beta})^2|u_{\parallel}|^2}f^\alpha_1(v + \frac{2m^{\beta}}{m^{\alpha}+m^{\beta}}u_{\parallel})w^{\gamma}(v) w^{2l}(v)f^\alpha_2(v)du_{\parallel}dv\notag\\
			&\hspace{0.4cm}\leq \frac{C_\varepsilon}{m} \Bigg(\iint_{(\mathbb{R}^3)^2}e^{-\frac{1}{2}(\frac{m^\alpha}{m^\alpha+m^\beta})^2|u_{\parallel}|^2}[f^\alpha_1(v + \frac{2m^{\beta}}{m^{\alpha}+m^{\beta}}u_{\parallel})]^{2} w^{2l+\gamma}(v)du_{\parallel}dv\Bigg)^{1/2}\label{CE2262BDSesD1}\\
			&\hspace{4.4cm}\times  \Bigg(\int_{\mathbb{R}^3}e^{-\frac{1}{2}(\frac{m^\alpha}{m^\alpha+m^\beta})^2|u_{\parallel}|^2} du_{\parallel}\int_{\mathbb{R}^3}w^{2l+\gamma}(v)[f^\alpha_2(v)]^2dv\Bigg)^{1/2}\notag
		\end{align}
		Utilizing the inequality 
		\begin{equation}\label{convertvvvkgbc}
			w^{2l+\gamma}(v)\leq C w^{2l+\gamma}(v+ \frac{2m^{\beta}}{m^{\alpha}+m^{\beta}}u_{\parallel}) w^{2l+\gamma}( \frac{2m^{\beta}}{m^{\alpha}+m^{\beta}}u_{\parallel}),
		\end{equation}
		and \eqref{caseA1wgams},
		the last inequality in \eqref{CE2262BDSesD1} is bounded by
		\begin{equation*}
			\begin{split}
				& \frac{C_\varepsilon}{m} \Bigg(\int_{\mathbb{R}^3}e^{-\frac{1}{2}(\frac{m^\alpha}{m^\alpha+m^\beta})^2|u_{\parallel}|^2}w^{-2\gamma}(u_\parallel)w^{2l+\gamma}( \frac{2m^{\beta}}{m^{\alpha}+m^{\beta}}u_{\parallel})du_{\parallel}\\
				&\hspace{2.7cm}\times\int_{\mathbb{R}^3}f^\alpha_1(v + \frac{2m^{\beta}}{m^{\alpha}+m^{\beta}}u_{\parallel}) w^{2l+\gamma}(v + \frac{2m^{\beta}}{m^{\alpha}+m^{\beta}}u_{\parallel})dv\Bigg)^{1/2}\\
				&\hspace{0.5cm}\times  \Bigg(\int_{\mathbb{R}^3}e^{-\frac{1}{2}(\frac{m^\alpha}{m^\alpha+m^\beta})^2|u_{\parallel}|^2} du_{\parallel}\int_{\mathbb{R}^3}w^{2l+\gamma}(v)[f^\alpha_2(v)]^2dv\Bigg)^{1/2}\\
				&\leq \frac{C_\varepsilon}{m} \big|w^{l}\mathbf{f}_1\big|_\nu \big|w^{l} \mathbf{f}_2\big|_\nu.
			\end{split}
		\end{equation*}

		\noindent {\bf Case A2. $D_{A2}=\Big[(|v|\geq \frac{m}{10})\cap(|u_\perp|\geq \frac{|v|}{8})\Big]\bigcup\Big[(|u|\geq \frac{2m}{5})\cap	(|v|\leq \frac{m}{10})\cap	(|u_\perp|\geq \frac{|u|}{2})\Big]$}.
		
		In both case of A2, $|u_\perp|\geq \frac{m}{80}>1$, it holds that 
		\begin{equation}\label{LinfA2}
			\begin{split}
				&\int_{\mathbb{R}^2} e^{-\frac{m^\beta}{2}|u_{\perp} + \zeta_{\perp}|^2} \left[|u_{\parallel}|^2 + |u_{\perp}|^2\right]^{\frac{\gamma - 1}{2}} \chi\big(\sqrt{|u_{\parallel}|^2 + |u_{\perp}|^2}\big) \frac{b^{ \alpha \beta}(\theta)}{|\cos \theta|}\mathbf{1}_{D_{A2}} \, du_{\perp} \\
				&\hspace{0.4cm}\leq  \frac{C_\varepsilon}{1+|v|}\int_{\mathbb{R}^2} e^{-\frac{m^\beta}{2}|u_\perp + \zeta_\perp|^2} |u_\perp|^{\gamma}\mathbf{1}_{D_{A2}} du_\perp \\
				&\hspace{0.4cm}\leq  \frac{C_\varepsilon}{m}\int_{\mathbb{R}^2} e^{-\frac{m^\beta}{2}|u_\perp + \zeta_\perp|^2} |u_\perp|^{\gamma}\mathbf{1}_{D_{A2}} du_\perp \\
				&\hspace{0.4cm}\leq \frac{C_\varepsilon}{m}w^\gamma(v)\int_{\mathbb{R}^2} e^{-\frac{m^\beta}{2}|u_\perp + \zeta_\perp|^2} du_\perp 
			\end{split}
		\end{equation}
		Then, the part of \eqref{CE2BD221fSesD1} in $\left \langle  w^{2l}K_{2,r}^{\alpha,\chi}\mathbf{f}_1,f^{\alpha}_2\right \rangle_{L^{2}_{v}}$ is bounded by 
		\begin{align}
			& \sum_{\beta=A,B}\iint_{(\mathbb{R}^3)^2} \frac{1}{|u_{\parallel}|} e^{-\frac{1}{2}[|\zeta_{\parallel}|^2+(\frac{m^\alpha}{m^\alpha+m^\beta})^2|u_{\parallel}|^2]} f^\alpha_1(v + \frac{2m^{\beta}}{m^{\alpha}+m^{\beta}}u_{\parallel}) w^{2l}(v)f^\alpha_2(v)\notag\\
			&\hspace{0.9cm} \times \int_{\mathbb{R}^2} e^{-\frac{1}{2}|u_{\perp} + \zeta_{\perp}|^2} \left[|u_{\parallel}|^2 + |u_{\perp}|^2\right]^{\frac{\gamma - 1}{2}} \chi\big(\sqrt{|u_{\parallel}|^2 + |u_{\perp}|^2}\big) \frac{b^{ \alpha \beta}(\theta)}{|\cos \theta|} \mathbf{1}_{D_{A2}}\, du_{\perp} du_{\parallel}dv\notag\\
			&\hspace{0.4cm}\leq \frac{C_\varepsilon}{m} \iint_{(\mathbb{R}^3)^2}\frac{1}{|u_{\parallel}|}e^{-\frac{1}{2}(\frac{m^\alpha}{m^\alpha+m^\beta})^2|u_{\parallel}|^2}f^\alpha_1(v + \frac{2m^{\beta}}{m^{\alpha}+m^{\beta}}u_{\parallel})w^{\gamma}(v) w^{2l}(v)f^\alpha_2(v)du_{\parallel}dv\notag\\
			&\hspace{0.4cm}\leq \frac{C_\varepsilon}{m} \Bigg(\iint_{(\mathbb{R}^3)^2}\frac{1}{|u_{\parallel}|}e^{-\frac{1}{2}(\frac{m^\alpha}{m^\alpha+m^\beta})^2|u_{\parallel}|^2}[f^\alpha_1(v + \frac{2m^{\beta}}{m^{\alpha}+m^{\beta}}u_{\parallel})]^{2} w^{2l+\gamma}(v)du_{\parallel}dv\Bigg)^{1/2}\label{CwE2B45970DSesD1}\\
			&\hspace{3.3cm}\times  \Bigg(\int_{\mathbb{R}^3}\frac{1}{|u_{\parallel}|}e^{-\frac{1}{2}(\frac{m^\alpha}{m^\alpha+m^\beta})^2|u_{\parallel}|^2} du_{\parallel}\int_{\mathbb{R}^3}w^{2l+\gamma}(v)[f^\alpha_2(v)]^2dv\Bigg)^{1/2}\notag
		\end{align}
		By \eqref{convertvvvkgbc}, it can be concluded that the last inequality in \eqref{CwE2B45970DSesD1} is bounded by
		\begin{align*}
				& \frac{C_\varepsilon}{m} \Bigg(\int_{\mathbb{R}^3}\frac{1}{|u_{\parallel}|}e^{-\frac{1}{2}(\frac{m^\alpha}{m^\alpha+m^\beta})^2|u_{\parallel}|^2}w^{2l+\gamma}( \frac{2m^{\beta}}{m^{\alpha}+m^{\beta}}u_{\parallel})du_{\parallel}\\
				&\hspace{3.5cm}\times\int_{\mathbb{R}^3}f^\alpha_1(v + \frac{2m^{\beta}}{m^{\alpha}+m^{\beta}}u_{\parallel}) w^{2l+\gamma}(v + \frac{2m^{\beta}}{m^{\alpha}+m^{\beta}}u_{\parallel})dv\Bigg)^{1/2}\\
				&\hspace{0.4cm}\times  \Bigg(\int_{\mathbb{R}^3}e^{-\frac{1}{2}(\frac{m^\alpha}{m^\alpha+m^\beta})^2|u_{\parallel}|^2} du_{\parallel}\int_{\mathbb{R}^3}w^{2l+\gamma}(v)[f^\alpha_2(v)]^2dv\Bigg)^{1/2}\\
				&\leq \frac{C_\varepsilon}{m} \big|w^{l}\mathbf{f}_1\big|_\nu \big|w^{l} \mathbf{f}_2\big|_\nu.
		\end{align*}

		\noindent {\bf Case A3. $D_{A3}=\Big(|v|\geq \frac{m}{10}\Big)\cap\Big(|v_\parallel|\geq \frac{|v|}{2}\Big)\cap\Big(|u_\parallel|\leq \frac{|v|}{8}\Big)\cap\Big(|u_\perp|\leq \frac{|v|}{8}\Big)$}.
		
		From \eqref{CEPT2ACA1vb}, we have 
		\begin{equation}\label{CET2zetabds}
			\zeta_{\parallel}=\frac{m^\beta}{m^\alpha+m^\beta}u_{\parallel}+v_\parallel, \qquad \zeta_\perp=v_\perp,
		\end{equation}
		with the assumption $m^{\beta} >m^{\alpha}$, so that $|\zeta_{\parallel}|\geq\frac{|v|}{4}\geq\frac{m}{40}$. We deduce that
		\begin{equation}\label{LifA3}
			e^{-\frac{m^\beta}{2}|\zeta_{\parallel}|^2}\leq \frac{C}{m}w^{\gamma-1} (v).
		\end{equation}
		Then, the part of \eqref{CE2BD221fSesD1} in $\left \langle  w^{2l}K_{2,r}^{\alpha,\chi}\mathbf{f}_1,f^{\alpha}_2\right \rangle_{L^{2}_{v}}$ is bounded by 
		\begin{align*}
			& \sum_{\beta=A,B}\iint_{(\mathbb{R}^3)^2} \frac{1}{|u_{\parallel}|} e^{-\frac{1}{2}\big[|\zeta_{\parallel}|^2+(\frac{m^\alpha}{m^\alpha+m^\beta})^2|u_{\parallel}|^2\big]} f^\alpha_1(v + \frac{2m^{\beta}}{m^{\alpha}+m^{\beta}}u_{\parallel}) w^{2l}(v)f^\alpha_2(v)\\
			&\hspace{1cm} \times \int_{\mathbb{R}^2} e^{-\frac{1}{2}|u_{\perp} + \zeta_{\perp}|^2} \left[|u_{\parallel}|^2 + |u_{\perp}|^2\right]^{\frac{\gamma - 1}{2}} \chi\big(\sqrt{|u_{\parallel}|^2 + |u_{\perp}|^2}\big) \frac{b^{ \alpha \beta}(\theta)}{|\cos \theta|}\mathbf{1}_{D_{A3}} \, du_{\perp} du_{\parallel}dv\\
			&\hspace{0.4cm}\leq \frac{C_\varepsilon}{m} \iint_{(\mathbb{R}^3)^2}\frac{1}{|u_{\parallel}|}e^{-\frac{1}{2}(\frac{m^\alpha}{m^\alpha+m^\beta})^2|u_{\parallel}|^2}f^\alpha_1(v + \frac{2m^{\beta}}{m^{\alpha}+m^{\beta}}u_{\parallel}) w^{2l+\gamma}(v)f^\alpha_2(v)du_{\parallel}dv\\
			&\hspace{0.4cm}\leq \frac{C_\varepsilon}{m} \Bigg(\iint_{(\mathbb{R}^3)^2}\frac{1}{|u_{\parallel}|}e^{-\frac{1}{2}(\frac{m^\alpha}{m^\alpha+m^\beta})^2|u_{\parallel}|^2}[f^\alpha_1(v + \frac{2m^{\beta}}{m^{\alpha}+m^{\beta}}u_{\parallel})]^{2} w^{2l+\gamma}(v)du_{\parallel}dv\Bigg)^{1/2}\\
			&\hspace{3.3cm}\times  \Bigg(\int_{\mathbb{R}^3}\frac{1}{|u_{\parallel}|}e^{-\frac{1}{2}(\frac{m^\alpha}{m^\alpha+m^\beta})^2|u_{\parallel}|^2} du_{\parallel}\int_{\mathbb{R}^3}w^{2l+\gamma}(v)[f^\alpha_2(v)]^2dv\Bigg)^{1/2}\notag\\
			&\hspace{0.4cm}\leq \frac{C_\varepsilon}{m} \big|w^{l}\mathbf{f}_1\big|_\nu \big|w^{l} \mathbf{f}_2\big|_\nu.
		\end{align*}

		\noindent {\bf Case A4. $D_{A4}=\Big(|v|\geq \frac{m}{10}\Big)\cap\Big(|v_\perp|\geq \frac{|v|}{2}\Big)\cap\Big(|u_\parallel|\leq \frac{|v|}{8}\Big)\cap\Big(|u_\perp|\leq \frac{|v|}{8}\Big)$}.
		
		In both cases of A4,  we have
		\begin{equation*}
			|u_{\perp} + \zeta_{\perp}|\geq \frac{|v|}{4} \geq \frac{m}{40}>1.
		\end{equation*}
		Then, we get 
		\begin{equation}\label{LifA4}
			e^{-\frac{m^\beta}{8}|u_{\perp} + \zeta_{\perp}|^2}\leq \frac{C}{m}w^{\gamma-1} (v).
		\end{equation}
		The part of \eqref{CE2BD221fSesD1} in $\left \langle  w^{2l}K_{2,r}^{\alpha,\chi}\mathbf{f}_1,f^{\alpha}_2\right \rangle_{L^{2}_{v}}$ is bounded by 
		\begin{align*}
			& \sum_{\beta=A,B}\int_{\mathbb{R}^3}\int_{\mathbb{R}^3} \frac{1}{|u_{\parallel}|} e^{-\frac{1}{2}\big[|\zeta_{\parallel}|^2+(\frac{m^\alpha}{m^\alpha+m^\beta})^2|u_{\parallel}|^2\big]} f^\alpha_1(v + \frac{2m^{\beta}}{m^{\alpha}+m^{\beta}}u_{\parallel}) w^{2l}(v)f^\alpha_2(v)\\
			&\quad \times \int_{\mathbb{R}^2} e^{-\frac{1}{2}|u_{\perp} + \zeta_{\perp}|^2} \left[|u_{\parallel}|^2 + |u_{\perp}|^2\right]^{\frac{\gamma - 1}{2}} \chi\left(\sqrt{|u_{\parallel}|^2 + |u_{\perp}|^2}\right) \frac{b^{ \alpha \beta}(\theta)}{|\cos \theta|}\mathbf{1}_{D_{A4}} \, du_{\perp} du_{\parallel}dv\\
			&\leq \frac{C_\varepsilon}{m} \int_{(\mathbb{R}^3)^2}\frac{1}{|u_{\parallel}|}e^{-\frac{1}{2}(\frac{m^\alpha}{m^\alpha+m^\beta})^2|u_{\parallel}|^2}f^\alpha_1(v + \frac{2m^{\beta}}{m^{\alpha}+m^{\beta}}u_{\parallel}) w^{2l+\gamma}(v)f^\alpha_2(v)du_{\parallel}dv\\
			&\leq \frac{C_\varepsilon}{m} \Bigg(\int_{(\mathbb{R}^3)^2}\frac{1}{|u_{\parallel}|}e^{-\frac{1}{2}(\frac{m^\alpha}{m^\alpha+m^\beta})^2|u_{\parallel}|^2}[f^\alpha_1(v + \frac{2m^{\beta}}{m^{\alpha}+m^{\beta}}u_{\parallel})]^{2} w^{2l+\gamma}(v)du_{\parallel}dv\Bigg)^{1/2}\\
			&\hspace{2.9cm}\times  \Bigg(\int_{\mathbb{R}^3}\frac{1}{|u_{\parallel}|}e^{-\frac{1}{2}(\frac{m^\alpha}{m^\alpha+m^\beta})^2|u_{\parallel}|^2} du_{\parallel}\int_{\mathbb{R}^3}w^{2l+\gamma}(v)[f^\alpha_2(v)]^2dv\Bigg)^{1/2}\notag\\
			&\leq \frac{C_\varepsilon}{m} \big|w^{l}\mathbf{f}_1\big|_\nu \big|w^{l} \mathbf{f}_2\big|_\nu.
		\end{align*}
		Hence, we have obtained  the estimates for the part of \eqref{CE2BD221fSesD1} in $\left \langle  w^{2l}K_{2,r}^{\alpha,\chi}\mathbf{f}_1,f^{\alpha}_2\right \rangle_{L^{2}_{v}}$.

		$\mathbf{Step 2}$. We consider the decay estimate for the part \eqref{CEs2BDSHH224D2} in $\left \langle  w^{2l}K_{2,r}^{\alpha,\chi}\mathbf{f}_1,f^{\alpha}_2\right \rangle_{L^{2}_{v}}$.
		
		By \eqref{lem2and222cgv}, it holds that
		\begin{eqnarray}\label{lem2230cgvb}
			\begin{split}
				\xi_{\parallel} &= \frac{\sqrt{m^\beta}+\sqrt{m^\alpha}}{2}v_{\parallel} + (\frac{\sqrt{m^\beta}}{2}+\frac{\sqrt{m^{\alpha}}m^{\beta}}{m^{\alpha}+m^{\beta}})u_{\parallel} \\
				\xi_{\perp} &= \frac{\sqrt{m^\beta}+\sqrt{m^\alpha}}{2}v_{\perp} +\frac{\sqrt{m^\beta}}{2}u_{\perp},
			\end{split}
		\end{eqnarray}
		By direct calculation, it shows that  $(\frac{\sqrt{m^\beta}}{2}-\frac{\sqrt{m^{\alpha}}m^{\beta}}{m^{\alpha}+m^{\beta}})>0$, and 
		\begin{equation}\label{CET2PARALCOMPAR}
			\big(\frac{\sqrt{m^\beta}-\sqrt{m^\alpha}}{2}\big):\big(\frac{\sqrt{m^\beta}}{2}-\frac{\sqrt{m^{\alpha}}m^{\beta}}{m^{\alpha}+m^{\beta}}\big)>\big(\frac{\sqrt{m^\beta}+\sqrt{m^\alpha}}{2}\big):\big(\frac{\sqrt{m^\beta}}{2}+\frac{\sqrt{m^{\alpha}}m^{\beta}}{m^{\alpha}+m^{\beta}}\big).
		\end{equation}
		Then, \eqref{lem2223cgvb3} and \eqref{lem2230cgvb} indicate that if we regard $u_{\parallel}$ and $v_{\parallel}$ as a set of basis, the two vectors $\xi_{\parallel}$ and $\eta_{\parallel}$ are linearly independent. Similarly, we can infer that under the basis $u_{\perp},v_{\perp}$, vectors $\xi_{\perp}$ and $\eta_{\perp}$ are also linearly independent. 
		
		By comparing the expressions for $\xi_{\parallel}$ and $\eta_{\parallel}$ in \eqref{lem2223cgvb3} and \eqref{lem2230cgvb}, one can see that there exist constants $0<c_{v_\parallel},c_{u_\parallel}<1$ such that
		\begin{equation}
			\begin{split}
				c_{v_\parallel}\big(\frac{\sqrt{m^\beta}+\sqrt{m^\alpha}}{2}\big)=&\frac{\sqrt{m^\beta}-\sqrt{m^\alpha}}{2},\\
				c_{u_\parallel}\big(\frac{\sqrt{m^\beta}}{2}+\frac{\sqrt{m^{\alpha}}m^{\beta}}{m^{\alpha}+m^{\beta}}\big)=&\frac{\sqrt{m^\beta}}{2}-\frac{\sqrt{m^{\alpha}}m^{\beta}}{m^{\alpha}+m^{\beta}}.
			\end{split}
		\end{equation}
		From \eqref{CET2PARALCOMPAR}, it can be obtained that
		\begin{equation*}
			\begin{split}
				c_{u,-}:=&c_{v_\parallel}\big(\frac{\sqrt{m^\beta}}{2}+\frac{\sqrt{m^{\alpha}}m^{\beta}}{m^{\alpha}+m^{\beta}}\big)-\big(\frac{\sqrt{m^\beta}}{2}-\frac{\sqrt{m^{\alpha}}m^{\beta}}{m^{\alpha}+m^{\beta}}\big)>0,\\
				c_{v,-}:=&\big(\frac{\sqrt{m^\beta}-\sqrt{m^\alpha}}{2}\big)-c_{u_\parallel}\big(\frac{\sqrt{m^\beta}+\sqrt{m^\alpha}}{2}\big)>0.
			\end{split}
		\end{equation*}
		We denote
		\begin{equation*}
			\begin{split}
				\mathbf{c}_{u,+}=\frac{1}{\sqrt{2}}(c_{v_\parallel}\xi_{\parallel}+\eta_{\parallel}), \qquad \mathbf{c}_{u,-}=\frac{1}{\sqrt{2}}(c_{v_\parallel}\xi_{\parallel}-\eta_{\parallel})=\frac{c_{u,-}}{\sqrt{2}}u_{\parallel},\\
				\mathbf{c}_{v,+}=\frac{1}{\sqrt{2}}(c_{u_\parallel}\xi_{\parallel}+\eta_{\parallel}), \qquad \mathbf{c}_{v,-}=\frac{1}{\sqrt{2}}(\eta_{\parallel}-c_{u_\parallel}\xi_{\parallel})=\frac{c_{v,-}}{\sqrt{2}}v_{\parallel},
			\end{split}
		\end{equation*}
		which gives the equality that
		\begin{equation*}
			\begin{split}
				(c_{v_\parallel})^2|\xi_{\parallel}|^2+|\eta_{\parallel}|^2= |\mathbf{c}_{u,+}|^2+|\mathbf{c}_{u,-}|^2=|\mathbf{c}_{u,+}|^2+\frac{(c_{u,-})^2}{2}|u_{\parallel}|^2,\\
				(c_{u_\parallel})^2|\xi_{\parallel}|^2+|\eta_{\parallel}|^2= |\mathbf{c}_{v,+}|^2+|\mathbf{c}_{v,-}|^2=|\mathbf{c}_{v,+}|^2+\frac{(c_{v,-})^2}{2}|v_{\parallel}|^2.
			\end{split}
		\end{equation*}
		Similarly, it can be observed that
		\begin{equation*}
			\big(\frac{\sqrt{m^\beta}-\sqrt{m^\alpha}}{2}\big) :\frac{\sqrt{m^\beta}}{2}< \big(\frac{\sqrt{m^\beta}+\sqrt{m^\alpha}}{2}\big):\frac{\sqrt{m^\beta}}{2}.
		\end{equation*}
		Through a process analogous to the above that there exist constants 
		$0<c_{v_{\perp}}<1$ and $\tilde{c}_{u,-}, \tilde{c}_{v,-}>0$, as well as vectors $\tilde{\mathbf{c}}_{u,+},\tilde{\mathbf{c}}_{v,+}$, such that
		\begin{equation*}
			\begin{split}
				(c_{v_\perp})^2|\xi_{\perp}|^2+|\eta_{\perp}|^2&= |\mathbf{c}_{u,+}|^2+|\mathbf{c}_{u,-}|^2=|\mathbf{c}_{u,+}|^2+\frac{(c_{u,-})^2}{2}|u_{\perp}|^2,\\
				|\xi_{\perp}|^2+|\eta_{\perp}|^2&= |\mathbf{c}_{v,+}|^2+|\mathbf{c}_{v,-}|^2=|\mathbf{c}_{v,+}|^2+\frac{(c_{v,-})^2}{2}|v_{\perp}|^2.
			\end{split}
		\end{equation*}
		The exponential term about $\xi_{\parallel},\xi_{\perp},\eta_{\parallel},\eta_{\perp}$, in \eqref{CEs2BDSHH224D2} can be bounded as follows:
		\begin{equation}\label{k2alpneqbetdecayzb}
			e^{-\frac{1}{4}\big(|\xi_{\parallel}|^2+|\xi_{\perp}|^2+|\eta_{\parallel}|^2+|\eta_{\perp}|^2\big)} \leq e^{-\bar{c}\big(|v_{\parallel}|^2+|v_{\perp}|^2+|u_{\parallel}|^2+|u_{\perp}|^2\big)},
		\end{equation}
		where 
		\begin{equation*}
			\bar{c}=\frac{1}{64}\min\Big\{\frac{1}{4},(c_{u,-})^2,(c_{v,-})^2,(\tilde{c}_{u,-})^2,(\tilde{c}_{v,-})^2\Big\}>0.
		\end{equation*}
	
		Since $u=v_*- v$, we partition the integration domain $(|v|+|v_*|\geq m)$ as follows:
		\begin{equation*}
			\big(|v|+|u+v|\geq m\big)\subseteq \big(|v|+|u|\geq \frac{m}{2}\big)\subseteq\Big[(|v|\geq \frac{m}{4})\cup(|u|\geq \frac{m}{4})\Big],
		\end{equation*}
		which could further be split into
		\begin{equation*}
			\begin{split}
				\big(|v|\geq \frac{m}{4}\big)&=\Big[(|v|\geq \frac{m}{4})\cap(|v_{\parallel}|\geq \frac{|v|}{2})\Big]\cup\Big[(|v|\geq \frac{m}{4})\cap(|v_{\perp}|\geq \frac{|v|}{2})\Big],\\
				\big(|u|\geq \frac{m}{4}\big)&=\Big[(|u|\geq \frac{m}{4})\cap(|u_{\parallel}|\geq \frac{|u|}{2})\Big]\cup\Big[(|u|\geq \frac{m}{4})\cap(|u_{\perp}|\geq \frac{|u|}{2})\Big].
			\end{split}
		\end{equation*}
		Now, we consider the decay estimate for the part of \eqref{CEs2BDSHH224D2}.
		
		\noindent {\bf Case B1. $\big(|v|\geq \frac{m}{4}\big)\cap\big(|v_{\parallel}|\geq \frac{|v|}{2}\big)$}.
		
		In this case $|v_{\parallel}|\geq \frac{m}{8}$, so we get 
		\begin{equation}\label{sec2CaseB1}
			e^{-\frac{\bar{c}}{4}|v_{\parallel}|^2} \leq \frac{C}{m}, \hspace{1.9cm} e^{-\frac{\bar{c}}{4}[|v_{\parallel}|^2+|v_{\perp}|^2]} \leq C w^{\gamma}(v).
		\end{equation}
		Then, the  part about \eqref{CEs2BDSHH224D2} in $\left \langle  w^{2l}K_{2,r}^{\alpha,\chi}\mathbf{f}_1,f^{\alpha}_2\right \rangle_v$ is bounded by
		\begin{align}
			&\int_{\mathbb{R}^3} \int_{\mathbb{R}^3} \frac{1}{|u_{\parallel}|} e^{-\bar{c}(|v_{\parallel}|^2+|u_{\parallel}|^2)}\int_{\mathbb{R}^2} \chi\big(\sqrt{|u_{\parallel}|^2 + |u_{\perp}|^2}\big) w^{2l}(v)f^\alpha(v) \notag\\
			&\hspace{1.4cm} \times  e^{-\bar{c}(|v_{\perp}|^2+|u_{\perp}|^2) } \left[|u_{\parallel}|^2 + |u_{\perp}|^2\right]^{\frac{\gamma - 1}{2}}f^\beta(v + u_{\perp}+\frac{m^{\beta}-m^{\alpha}}{m^{\alpha}+m^{\beta}}u_{\parallel})  \, du_{\perp} du_{\parallel} dv \notag\\
			&\hspace{0.5cm}\leq \frac{C}{m}\int_{\mathbb{R}^3} \int_{\mathbb{R}^3} \frac{1}{|u_{\parallel}|} e^{-\frac{\bar{c}}{2}(|v_{\parallel}|^2+|u_{\parallel}|^2)}\int_{\mathbb{R}^2} \chi\big(\sqrt{|u_{\parallel}|^2 + |u_{\perp}|^2}\big) w^{2l+\gamma}(v)f^\alpha(v) \notag\\
			&\hspace{1.6cm}\times  e^{-\frac{\bar{c}}{2}(|v_{\perp}|^2+|u_{\perp}|^2) } \left[|u_{\parallel}|^2 + |u_{\perp}|^2\right]^{\frac{\gamma - 1}{2}}f^\beta(v + u_{\perp}+\frac{m^{\beta}-m^{\alpha}}{m^{\alpha}+m^{\beta}}u_{\parallel})  \, du_{\perp} du_{\parallel} dv \notag\\
			&\hspace{0.5cm}\leq \frac{C_{\varepsilon}}{m} \Bigg(\int_{\mathbb{R}^2}  
			e^{-\frac{\bar{c}}{2}|u_{\perp}|^2 } 
			\int_{\mathbb{R}^3}\frac{1}{|u_{\parallel}|}e^{-\frac{\bar{c}}{2}|u_{\parallel}|^2} \int_{\mathbb{R}^3}w^{2l+\gamma}(v)\notag\\
			&\hspace{5.7cm}\times \big[f^\beta(v + u_{\perp}+\frac{m^{\beta}-m^{\alpha}}{m^{\alpha}+m^{\beta}}u_{\parallel})\big]^2dvdu_{\parallel}du_{\perp}\Bigg)^{1/2}\label{sec2socal237}\\
			&\hspace{2.3cm}\times  \Bigg(\int_{\mathbb{R}^2}  
			e^{-\frac{\bar{c}}{2}|u_{\perp}|^2 } du_{\perp}
			\int_{\mathbb{R}^3}\frac{1}{|u_{\parallel}|}e^{-\frac{\bar{c}}{2}|u_{\parallel}|^2} du_{\parallel}\int_{\mathbb{R}^3}w^{2l+\gamma}(v)[f^\alpha(v)]^2dv\Bigg)^{1/2}\label{sec2socal238}
		\end{align}
		Because of
		\begin{equation*}
			\int_{\mathbb{R}^2}  
			e^{-\frac{\bar{c}}{2}|u_{\perp}|^2 } du_{\perp} \leq C, \qquad  \int_{\mathbb{R}^3}\frac{1}{|u_{\parallel}|}e^{-\frac{\bar{c}}{2}|u_{\parallel}|^2} du_{\parallel} \leq C,
		\end{equation*}
		clearly \eqref{sec2socal238} is  bounded by $|w^{l} \mathbf{f}_2|_\nu$. 
		
		By the inequality that
		\begin{equation*}
			w^{2l+\gamma}(v)\leq C w^{2l+\gamma}\big(v + u_{\perp}+\frac{m^{\beta}-m^{\alpha}}{m^{\alpha}+m^{\beta}}u_{\parallel}\big) w^{2l+\gamma}(u_{\perp}) w^{2l+\gamma}\big(\frac{m^{\beta}-m^{\alpha}}{m^{\alpha}+m^{\beta}}u_{\parallel}\big), 
		\end{equation*}
		thus \eqref{sec2socal237} is  bounded by
		\begin{align}
			&\frac{C_{\varepsilon}}{m} \Bigg(\int_{\mathbb{R}^2}  w^{2l+\gamma}(u_{\perp})
			e^{-\frac{\bar{c}}{2}|u_{\perp}|^2 } d u_{\perp}
			\int_{\mathbb{R}^3}\frac{1}{|u_{\parallel}|}w^{2l+\gamma}(\frac{m^{\beta}-m^{\alpha}}{m^{\alpha}+m^{\beta}}u_{\parallel})e^{-\frac{\bar{c}}{2}|u_{\parallel}|^2} du_{\parallel}\notag\\
			&\hspace{1cm}\times \int_{\mathbb{R}^3} w^{2l+\gamma}(v + u_{\perp}+\frac{m^{\beta}-m^{\alpha}}{m^{\alpha}+m^{\beta}}u_{\parallel})[f^\beta(v + u_{\perp}+\frac{m^{\beta}-m^{\alpha}}{m^{\alpha}+m^{\beta}}u_{\parallel})]^2dv\Bigg)^{1/2}\label{malneqmbetuprupeeq}
		\end{align}
		Since 
		\begin{equation*}
			\int_{\mathbb{R}^2}  w^{2l+\gamma}(u_{\perp})
			e^{-\frac{\bar{c}}{2}|u_{\perp}|^2 } du_{\perp} \leq C, \qquad  \int_{\mathbb{R}^3}\frac{1}{|u_{\parallel}|}w^{2l+\gamma}(\frac{m^{\beta}-m^{\alpha}}{m^{\alpha}+m^{\beta}}u_{\parallel})e^{-\frac{\bar{c}}{2}|u_{\parallel}|^2} du_{\parallel} \leq C,
		\end{equation*}
		\eqref{sec2socal237} can be further bounded by $\frac{C_{\varepsilon}}{m} |w^{l} \mathbf{f}_1|_\nu$.
		
		Consequently, the  part about \eqref{CEs2BDSHH224D2} in $\left \langle  w^{2l}K_{2,r}^{\alpha,\chi}\mathbf{f}_1,f^{\alpha}_2\right \rangle_v$ is bounded by $\frac{C_{\varepsilon}}{m} |w^{l} \mathbf{f}_1|_\nu |w^{l} \mathbf{f}_2|_\nu$.

		\noindent {\bf Case B2. $\big(|v|\geq \frac{m}{4}\big)\cap\big(|v_{\perp}|\geq \frac{|v|}{2}\big)$}.
		
		In this case $|v_{\perp}|\geq \frac{m}{8}>1$, then we have 
		\begin{equation}\label{sec2CaseB2}
			e^{-\frac{\bar{c}}{4}|v_{\perp}|^2} \leq \frac{C}{m}, \hspace{1.9cm} e^{-\frac{\bar{c}}{4}(|v_{\parallel}|^2+|v_{\perp}|^2)} \leq C w^{\gamma}(v).
		\end{equation}
		\noindent {\bf Case B3. $\big(|u|\geq \frac{m}{4}\big)\cap\big(|u_{\parallel}|\geq \frac{|u|}{2}\big)$}.
		
		In this case $|u_{\parallel}|\geq \frac{m}{8}>1$, then we have 
		\begin{equation}\label{sec2CaseB3}
			e^{-\frac{\bar{c}}{4}|u_{\parallel}|^2} \leq \frac{C}{m}, \hspace{1.9cm} e^{-\frac{\bar{c}}{4}(|v_{\parallel}|^2+|v_{\perp}|^2)} \leq C w^{\gamma}(v).
		\end{equation}
		\noindent {\bf Case B4. $\big(|u|\geq \frac{m}{4}\big)\cap\big(|u_{\perp}|\geq \frac{|u|}{2}\big)$}.
		
		In this case $|u_{\perp}|\geq \frac{m}{8}>1$, then we have 
		\begin{equation}\label{sec2CaseB4}
			e^{-\frac{\bar{c}}{4}|u_{\perp}|^2} \leq \frac{C}{m}, \hspace{1.9cm} e^{-\frac{\bar{c}}{4}(|v_{\parallel}|^2+|v_{\perp}|^2)} \leq C w^{\gamma}(v).
		\end{equation}
		By replacing \eqref{sec2CaseB1} with \eqref{sec2CaseB2} or \eqref{sec2CaseB3} or \eqref{sec2CaseB4}, we follow the analysis process in Case B1. For Case B2, Case B3, and Case B4, the same conclusion as in Case B1 can be concluded.
		
		So we conclude the estimates for the part of \eqref{CEs2BDSHH224D2} with $m^{\alpha}\neq m^{\beta}$ in $\left \langle  w^{2l}K_{2,r}^{\alpha,\chi}\mathbf{f}_1,f^{\alpha}_2\right \rangle_{L^{2}_{v}}$.
		
	\noindent\textbf{The degenerate part*}.\hspace{0.1cm}	Next, we turn to the case of $\alpha=\beta$.
		
		$(1)$ For the part of $k_{\alpha \beta}^{(1),r}$, the proof for case $\alpha\neq\beta$ can be fully applied to the proof for case $\alpha=\beta$, yielding the same conclusion.
		
		$(2)^{*}$  The integration kernel‌ $k_{\alpha \alpha}^{(2),r}$ no longer adheres to exponential decay \eqref{ker2rrr} with \eqref{k2alpneqbetdecayzb},
		but instead satisfies the decay  analogous to $k_{\alpha \beta}^{(1),r}$.
		\begin{proof}
			Notice by \eqref{lem2223cgvb3} and \eqref{lem2230cgvb}, when $\alpha=\beta$, we obtain
			\begin{eqnarray*}
				\begin{split}
					\xi_{\parallel} &= \sqrt{m^\alpha}(v_{\parallel}+u_{\parallel}), \hspace{0.93cm}  \eta_{\parallel}=0\\
					\xi_{\perp} &= \sqrt{m^\alpha}(v_{\perp}+\frac{1}{2}u_{\perp}),\hspace{0.5cm}\eta_{\perp}=\frac{1}{2}u_{\perp}
				\end{split}
			\end{eqnarray*}
			Therefore, $\xi_{\parallel}$ and $\eta_{\parallel}$ are not linearly independent. We cannot obtain exponential decay for $v_{\parallel}$ and $u_{\parallel}$ individually; only  exhibits exponential decay with respect to $(v_{\parallel}+u_{\parallel})$. This results in $k_{\alpha \alpha}^{(2),r}$ no longer possessing exponential decay with respect to $v$.
			
			The part of $\alpha=\beta$ in \eqref{lem22K2schimainf} is
			\begin{align}
				&\Bigg(\int_{|v|+|v_*|>m}\int_{\mathbb{S}^2} |u|^{\gamma} [\mu^{\alpha}(u+v)]^{\frac{1}{2}}\chi(|u|)\big[\mu^{\alpha}(v + u_{\perp})\big]^{\frac{1}{2}}\label{CE2ek1asyPT}\\
				&\hspace{7cm}\times f^\alpha_1(v + u_{\parallel})b^{ \alpha \alpha}(\theta) \, du \, d\omega \notag\\
				& \hspace{0.4cm}+\int_{|v|+|v_*|>m}\int_{\mathbb{S}^2} |u|^{\gamma} [\mu^{\alpha}(u+v)]^{\frac{1}{2}}\chi(|u|)\big[\mu^{\alpha}(v + u_{\parallel})\big]^{\frac{1}{2}}\label{CE2hadPk2T}\\
				&\hspace{7cm}\times f^\alpha_1(v + u_{\perp})b^{ \alpha \alpha}(\theta) \, du \, d\omega \Bigg)\notag
			\end{align}
			where \eqref{CE2ek1asyPT} is the part for $k_{\alpha \alpha}^{(1),r}$, and \eqref{CE2hadPk2T} for $k_{\alpha \alpha}^{(2),r}$. ‌We note that the independent variables in function $f^\alpha_1(\cdot)$ are $v+u_{\parallel}$ and $v+u_{\perp}$ in the two different cases, respectively. ``$du\,d\omega$'' has five degrees of freedom. Therefore, when the independent variable involves $u_{\parallel}$, we can assign $u_{\parallel}$ to have all three degrees of freedom while $u_{\perp}$ retains two degrees of freedom. When the independent variable involves $u_{\perp}$, we must assign $u_{\perp}$ to have all three degrees of freedom while $u_{\parallel}$ retains two degrees of freedom.
			
			For \eqref{CE2ek1asyPT}, we assign $u_{\parallel}$ to have all three degrees of freedom,  $u \in \mathbb{R}^3$, $\omega \in \mathbb{S}^2$, $u_{\parallel} \in \mathbb{R}^3$, $u_{\perp} \in \mathbb{R}^2$, and
			\begin{equation*}
				|u_{\parallel}|^{2}(|u_{\parallel}|^{-1}d|u_{\parallel}|)d\omega = du_{\parallel} ,\qquad du = 2(|u_{\parallel}|^{-1} d|u_{\parallel}|) du_{\perp}. 
			\end{equation*}
			But for \eqref{CE2hadPk2T}, we assign $u_{\perp}$ to have all three degrees of freedom,  $u \in \mathbb{R}^3$, $\omega' \in \mathbb{S}^2$, $u_{\perp} \in \mathbb{R}^3$, $u_{\parallel} \in \mathbb{R}^2$, and
			\begin{equation}\label{cgvbuprepdd}
				|u_{\perp}|^{2}(|u_{\perp}|^{-1}d|u_{\perp}|)d\omega' = du_{\perp} ,\qquad du =2 (|u_{\perp}|^{-1} d|u_\perp|) du_{\parallel}. 
			\end{equation}
			We denote the three coordinates of the spherical coordinate transformation as $(\tilde{r},\tilde{\theta},\tilde{\phi})$, where $\tilde{r}\geq0$ is radial distance, $0\leq\tilde{\theta}\leq \pi$ is polar angle and $0\leq\tilde{\phi}\leq 2 \pi$ is azimuthal angle.  
			
			In case of \eqref{CE2ek1asyPT}, we take $(|u_{\parallel}|,\theta,\tilde{\phi})\rightarrow(\tilde{r},\tilde{\theta},\tilde{\phi})$. 
			
			The differential area element on the unit sphere: $d\omega=\sin\theta d\theta d\tilde{\phi}$.
		
			In case of \eqref{CE2hadPk2T}, we take $(|u_{\perp}|,\varphi,\tilde{\phi})\rightarrow(\tilde{r},\tilde{\theta},\tilde{\phi})$. 
			
			The differential area element on the unit sphere: $d\omega'=\sin\varphi d\varphi d\tilde{\phi}$.

			\begin{figure}
				\centering
				\includegraphics[width=9cm,height=5.5cm]{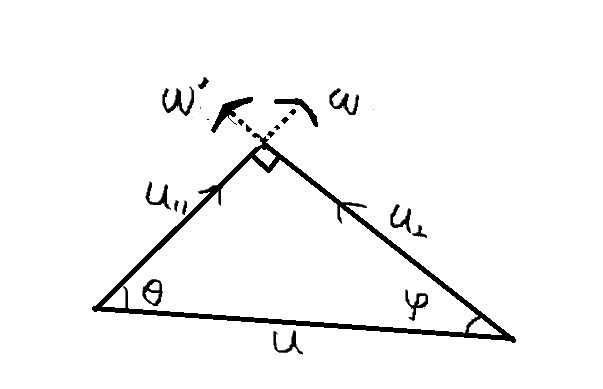}
				\vspace{-0.2cm}
				\caption{The illustration of variable substitution }
			\end{figure}
			From the picture (Figure 1), we have $|u_{\parallel}|=|u\cos\theta|$, $|u_{\perp}|=|u\cos\varphi|$, and $\varphi+\theta=\frac{\pi}{2}$. Therefore, it yields 
			\begin{equation*}
				du_{\perp}=|u_{\perp}|^{2}(|u_{\perp}|^{-1}d|u_{\perp}|)d\omega' = |u_{\perp}||u_{\parallel}|(|u_{\perp}|^{-1}d|u_{\perp}|)d\omega
			\end{equation*}
			Together with \eqref{cgvbuprepdd}, in the case of \eqref{CE2hadPk2T}, we obtain
			\begin{equation}\label{casaeqbedwdu}
				du d\omega = 2|u_{\parallel}|^{-1}|u_{\perp}|^{-1} du_{\perp} du_{\parallel}
			\end{equation}
			By \eqref{CE2easyPT} and \eqref{CEPT2ACA1vb}, \eqref{CE2ek1asyPT} can be written as
			\begin{align}
				& \int_{\mathbb{R}^3} \frac{1}{|u_{\parallel}|} e^{-\frac{m^\alpha}{2}\big[|\zeta_{\parallel}|^2+\frac{1}{4}|u_{\parallel}|^2\big]} f^\alpha_1(v + u_{\parallel}) \notag\\
				&\quad \times \int_{\mathbb{R}^2} e^{-\frac{m^\alpha}{2}|u_{\perp} + \zeta_{\perp}|^2} \mathbf{1}_{|v|+|v_*|>m}\left[|u_{\parallel}|^2 + |u_{\perp}|^2\right]^{\frac{\gamma - 1}{2}} \chi\big(\sqrt{|u_{\parallel}|^2 + |u_{\perp}|^2}\big) \frac{b^{ \alpha \beta}(\theta)}{|\cos \theta|} \, du_{\perp} du_{\parallel}\label{CE2BD2q21fSesD1}
			\end{align}
			where we take
			\begin{equation*}
				\zeta_{\perp}=v_{\perp},\hspace{0.5cm}\zeta_{\parallel}=v_{\parallel}+\frac{1}{2}u_{\parallel},\hspace{0.5cm}u_{\parallel}\in \mathbb{R}^3,\hspace{0.5cm}u_{\perp}\in \mathbb{R}^2.
			\end{equation*}
			Then the integral kernel in \eqref{ker1rrr} becomes
			\begin{align}
				k_{\alpha \alpha}^{(1),r}(v,u_{\parallel})	&=: \frac{1}{|u_{\parallel}|} e^{-\frac{m^\alpha}{2}\big[|\zeta_{\parallel}|^2+\frac{1}{4}|u_{\parallel}|^2\big]} \int_{\mathbb{R}^2} e^{-\frac{m^\alpha}{2}|u_{\perp} + \zeta_{\perp}|^2} \mathbf{1}_{|v|+|v_*|>m}\label{deftmambddk1f}\\
				&\hspace{2.5cm}\times  \left[|u_{\parallel}|^2 + |u_{\perp}|^2\right]^{\frac{\gamma - 1}{2}} \chi\big(\sqrt{|u_{\parallel}|^2 + |u_{\perp}|^2}\big) \frac{b^{ \alpha \beta}(\theta)}{|\cos \theta|} \, du_{\perp}. \notag 
			\end{align}
		 When $\alpha=\beta$ a process analogous to \eqref{CEPT2ACA1vb}--\eqref{CE2BD221fSesD1} shows that‌  ‌by combining \eqref{casaeqbedwdu}, \eqref{CE2hadPk2T} can be expressed as
			\begin{align}
				& \int_{\mathbb{R}^3} \frac{1}{|u_{\perp}|} e^{-\frac{m^\alpha}{2}\big[|\zeta^{\perp}|^2+\frac{1}{4}|u_{\perp}|^2\big]} f^\alpha_1(v + u_{\perp}) \notag\\
				&\quad \times \int_{\mathbb{R}^2} e^{-\frac{m^\alpha}{2}|u_{\parallel} + \zeta^{\parallel}|^2} \mathbf{1}_{|v|+|v_*|>m}\left[|u_{\parallel}|^2 + |u_{\perp}|^2\right]^{\frac{\gamma - 1}{2}} \chi\big(\sqrt{|u_{\parallel}|^2 + |u_{\perp}|^2}\big) \frac{b^{ \alpha \beta}(\theta)}{|\cos \theta|} \,  du_{\parallel}\,du_{\perp},\label{CE2BD2qg21fSesD1}
			\end{align}
			where 
			\begin{equation*}
				\zeta^{\parallel}=v_{\parallel},\hspace{0.5cm}\zeta^{\perp}=v_{\perp}+\frac{1}{2}u_{\perp},\hspace{0.5cm}u_{\perp}\in \mathbb{R}^3,\hspace{0.5cm}u_{\parallel}\in \mathbb{R}^2.
			\end{equation*}
			We define the integral kernel:
			\begin{align}
				k_{\alpha \alpha}^{(2),r}(v,u_{\perp})	&=: \frac{1}{|u_{\perp}|} e^{-\frac{m^\alpha}{2}\big[|\zeta^{\perp}|^2+\frac{1}{4}|u_{\perp}|^2\big]} \int_{\mathbb{R}^2} e^{-\frac{m^\alpha}{2}|u_{\parallel} + \zeta^{\parallel}|^2} \mathbf{1}_{|v|+|v_*|>m}\notag\\
				&\hspace{2.5cm}\times  \left[|u_{\parallel}|^2 + |u_{\perp}|^2\right]^{\frac{\gamma - 1}{2}} \chi\big(\sqrt{|u_{\parallel}|^2 + |u_{\perp}|^2}\big) \frac{b^{ \alpha \beta}(\theta)}{|\cos \theta|} \, du_{\parallel}.  \label{deftmambddk2}
			\end{align}
			Comparing the two expressions \eqref{CE2BD2q21fSesD1} and \eqref{CE2BD2qg21fSesD1} above, we observe that they become identical if we interchange the symbols ``$\parallel$'' and ``$\perp$'' in \eqref{CE2BD2qg21fSesD1}. Thus, the analysis of $Case A_1-Case A_4$ for $\eqref{CE2BD221fSesD1}$ is directly applicable to the case of \eqref{CE2hadPk2T}.
			
			Therefore, we can state that when $m^{\alpha}=m^{\beta}$, the integral kernel $k_{\alpha \beta}^{(2),r}(v,u_{\perp},u_{\parallel})$ degenerates into the integral kernel $k_{\alpha \alpha}^{(2),r}(v,u_{\perp})$, which is equal to $k_{\alpha \alpha}^{(1),r}(v,u_{\parallel})$ in the symmetric sense.
			
		\end{proof}
		
		According to \eqref{2141214Q1569} and \eqref{KS1KS2DPPFJ}, $K_{s}^{\alpha}$ is expressed as 
		\begin{equation*}
			K_{s}^{\alpha}=K_{1,s}^{\alpha,1-\chi}+K_{1,r}^{\alpha,\chi}+K_{2,s}^{\alpha,1-\chi}+K_{2,r}^{\alpha,\chi}
		\end{equation*}
		we now present the estimates for all these parts derived thus far as follows:
		
		By \eqref{k1ssing}, \eqref{k1sregu} and \eqref{lem22cclssig}, we obtain
		\begin{align*}
			\big \langle  w^{2l}K_{1,s}^{\alpha,1-\chi}\mathbf{f}_1,f^{\alpha}_2\big \rangle_{L^{2}_{v}} &\leq C \varepsilon^{3+\gamma} \big|w^{l} \mathbf{f}_1\big|_\nu \big|w^{l} \mathbf{f}_2\big|_\nu,\\
				\big \langle  w^{2l}K_{1,r}^{\alpha,\chi}\mathbf{f}_1,f^{\alpha}_2\big \rangle_{L^{2}_{v}} &\leq C \big[\mu^s(m)\big]^{\frac{1}{8}} \big|w^{l} \mathbf{f}_1\big|_\nu \big|w^{l} \mathbf{f}_2\big|_\nu,\\
				\big \langle  w^{2l}K_{2,s}^{\alpha,1-\chi}\mathbf{f}_1,f^{\alpha}_2\big \rangle_{L^{2}_{v}} &\leq C \varepsilon^{3+\gamma} \big|w^{l} \mathbf{f}_1\big|_\nu \big|w^{l} \mathbf{f}_2\big|_\nu.
		\end{align*}
		Based on all these analytical processes from \eqref{lem22K2schimainf} to \eqref{deftmambddk2}, it can be concluded that 
		\begin{equation*}
			\left \langle  w^{2l}K_{2,r}^{\alpha,\chi}\mathbf{f}_1,f^{\alpha}_2\right \rangle_{L^{2}_{v}} \leq \frac{C_\varepsilon}{m} \big|w^{l} \mathbf{f}_1\big|_\nu \big|w^{l} \mathbf{f}_2\big|_\nu.
		\end{equation*}
		\noindent For any $\eta \in (0,\frac{1}{2})$ we first choose sufficiently $\varepsilon$ small such that
		\begin{equation*}
			C\varepsilon^{3+\gamma} \leq \frac{\eta}{2}.
		\end{equation*}
		Then, for this fixed $1>\varepsilon>0$, there exists a sufficiently large $m>1$ such that 
		\begin{equation*}
			C \big[\mu^s(m)\big]^{\frac{1}{8}}+\frac{C_\varepsilon}{m} \leq \frac{\eta}{2}.
		\end{equation*}
	 Consequently, this completes the proof of \eqref{Lem222maineqv}.
	 
	\noindent	\textbf{Proof of Compactness}.\hspace{0.05cm}	In the last step, we prove that $K^{\alpha}_c = K_{1,z}^{\alpha,\chi}+K_{2,z}^{\alpha,\chi}$ is a compact operator. 
	Recall the expression that
		\begin{align*}
			 K_{1,z}^{\alpha,\chi} \mathbf{f}=  \frac{1}{\sqrt{\mu^\alpha}} \sum_{\beta=A,B} \iint_{\mathbb{R}^3 \times \mathbb{S}^2}\mathbf{1}_{|v|+|v_*|\leq m}
			b^{ \alpha \beta}(\theta)|v-v_*|^{\gamma} \chi(|v-v_*|) \mu^\alpha(v) \sqrt{\mu^\beta(v_*)}f^\beta(v_*)   d\omega dv_*
		\end{align*}
		Since $|b^{ \alpha \beta}(\theta)|\leq C$, for fixed small $\varepsilon>0$, 
		\begin{align*}
			|K_{1,z}^{\alpha,\chi} \mathbf{f}|\leq C_{\varepsilon} \sum_{\beta=A,B} \int_{\{|v_*|\leq m\}}
			 \sqrt{\mu^\beta(v_*)} \big|f^\beta(v_*)\big|  dv_*\sqrt{\mu^\alpha(v)}.
		\end{align*}
		Therefore, Hilbert-Schmidt theorem guarantees that $K_{1,z}^{\alpha,\chi}$ is a compact operator in $L^2_{\nu}(\mathbb{R}^3)$. 
		
		For the part of $K_{2,z}^{\alpha,\chi}$, we have
		\begin{align}
		K_{2,z}^{\alpha,\chi} \mathbf{f}&= - \frac{1}{\sqrt{\mu^\alpha}} \sum_{\beta=A,B} \iint_{\mathbb{R}^3 \times \mathbb{S}^2}\mathbf{1}_{|v|+|v_*|\leq m}
		b^{ \alpha \beta}(\theta)|v-v_*|^{\gamma}\chi(|v-v_*|)\notag\\ 
		&\hspace{3.3cm}\times \Big[\mu^\alpha(v')\sqrt{\mu^\beta(v_*')}f^\beta(v_*')
		  +\mu^{\beta}(v_*') \sqrt{\mu^{\alpha}(v')}\, f^\alpha(v')\Big] d\omega dv_*.\label{lem2KZZZC}
		\end{align}
		Replace ‌$\mathbf{1}_{|v|+|v_*|> m}$ with ‌$\mathbf{1}_{|v|+|v_*|\leq m}$ in ‌\eqref{ker1rrr} and \eqref{ker2rrr}, the resulting expression defines the integral kernels $k_{\alpha \beta}^{(1),z}(v,u_{\parallel})$ and $k_{\alpha \beta}^{(2),z}(v,u_{\perp},u_{\parallel})$.
		 Through analysis similar to that in the preceding discussion, we can prove that
		\begin{align}
			-K_{2,z}^{\alpha,\chi} \mathbf{f}&=   \sum_{\beta=A,B} \Bigg(\int_{\mathbb{R}^3}k_{\alpha \beta}^{(1),z}(v,u_{\parallel})\,f^\alpha(v + \frac{2m^{\beta}}{m^{\alpha}+m^{\beta}}u_{\parallel}) \,d u_{\parallel}\notag\\ 
			&\hspace{1.7cm}+\iint_{\mathbb{R}^3\times\mathbb{R}^2}k_{\alpha \beta}^{(2),z}(v,u_{\perp},u_{\parallel}) \,f^\beta(v + u_{\perp}+\frac{m^{\beta}-m^{\alpha}}{m^{\alpha}+m^{\beta}}u_{\parallel})\,du_{\perp}\,d u_{\parallel}\Bigg).\label{lem2KZZZC}
		\end{align}
		For the part of $k_{\alpha \beta}^{(1),z}(v,u_{\parallel})$, 
		\begin{align*}
			\big|k_{\alpha \beta}^{(1),z}(v,u_{\parallel})\big|\leq C_{\varepsilon}\frac{1}{|u_{\parallel}|} e^{-\frac{m^\beta}{2}\big[(\frac{m^\alpha}{m^\alpha+m^\beta})^2|u_{\parallel}|^2\big]}\mathbf{1}_{|v|+|v_*|\leq m}.
		\end{align*}
		
		Notice by \eqref{CEs2BDSHH224D2} and \eqref{k2alpneqbetdecayzb}, when $\alpha \neq \beta$, the part of $k_{\alpha \beta}^{(2),z}(v,u_{\perp},u_{\parallel})$ satisfies
		\begin{align*}
			\big|k_{\alpha \beta}^{(2),r}(v,u_{\perp},u_{\parallel})\big|	\leq C_{\varepsilon}  \frac{1}{|u_{\parallel}|} e^{-\bar{c}(|v|^2+|u|^2)} \mathbf{1}_{|v|+|v_*|\leq m} .
		\end{align*}
		Since $\frac{1}{|u_{\parallel}|}\in L^2_{loc}(\mathbb{R}^3)$, by Hilbert-Schmidt theorem, we know the part of $k_{\alpha \beta}^{(2),z}(v,u_{\perp},u_{\parallel})$ with $\alpha \neq \beta$ in $K_{2,z}^{\alpha,\chi}$ is compact.  
		
		When $\alpha = \beta$, for the part of $k_{\alpha \alpha}^{(2),z}(v,u_{\perp})$,
		\begin{align*}
			\big|k_{\alpha \alpha}^{(2),r}(v,u_{\perp})\big|	\leq C_{\varepsilon}  \frac{1}{|u_{\parallel}|} e^{-\frac{1}{4}(|\xi_{\parallel}|^2+|\eta_{\parallel}|^2+|\xi_{\perp}|^2+|\eta_{\perp}|^2)} \mathbf{1}_{|v|+|v_*|\leq m} .
		\end{align*}
		Here $u_{\perp}\in \mathbb{R}^3$ and $\frac{1}{|u_{\perp}|}\in L^2_{loc}(\mathbb{R}^3)$, by Hilbert-Schmidt theorem, this part of $K_{2,z}^{\alpha,\chi}$ is also compact.
		
		Consequently, we have completed the entire proof of Lemma \ref{lemma2.2}.
	\end{proof}
	
	\begin{remark}\label{RK2211}
		Since we need to discuss the compactness of operators within the $L^2$ framework, we thus introduce the truncation functions $\mathbf{1}_{|v|+|v_*|\geq m}$ . Disregard the truncation function related to compactness. Let $K_{2}^{\alpha,\chi}=K_{2,r}^{\alpha,\chi}+K_{2,z}^{\alpha,\chi}$, which can be decomposed as follows:
		\begin{equation}
			K_{2}^{\alpha,\chi}\mathbf{f} = \underbrace{K_{2,(1)}^{\alpha,\chi}\mathbf{f}}_{\text{Typical part}} + \underbrace{K_{2,(2)}^{\alpha,\chi}\mathbf{f}}_{\text{Hybrid part}},
		\end{equation}
		where 
		\begin{align}
			&K_{2,(1)}^{\alpha,\chi}\mathbf{f}=\sum_{\beta=A,B}	\int_{\mathbb{R}^3}k_{\alpha \beta}^{(1)}(v,u_{\parallel})\,f^\alpha(v + \frac{2m^{\beta}}{m^{\alpha}+m^{\beta}}u_{\parallel}) \,d u_{\parallel},\\
			&K_{2,(2)}^{\alpha,\chi}\mathbf{f}= \iint_{\mathbb{R}^3\times\mathbb{R}^2}k_{\alpha \beta}^{(2)}(v,u_{\perp},u_{\parallel}) \,f^\beta(v + u_{\perp}+\frac{m^{\beta}-m^{\alpha}}{m^{\alpha}+m^{\beta}}u_{\parallel})\,du_{\perp}\,d u_{\parallel} \quad \alpha \neq \beta \notag\\
			&\hspace{2cm}+\int_{\mathbb{R}^3}k_{\alpha \alpha}^{(2)}(v,u_{\perp}) \,f^\beta(v + u_{\perp})\,du_{\perp}. \label{HBD264}
		\end{align}
		These integral kernels satisfy
		\begin{align}
			k_{\alpha \beta}^{(1)}(v,u_{\parallel}) &=:  \frac{1}{|u_{\parallel}|} e^{-\frac{m^\beta}{2}\big[|\zeta_{\parallel}|^2+(\frac{m^\alpha}{m^\alpha+m^\beta})^2|u_{\parallel}|^2\big]} \int_{\mathbb{R}^2} e^{-\frac{m^\beta}{2}|u_{\perp} + \zeta_{\perp}|^2} \notag\\
			&\hspace{2.5cm}\times  \left[|u_{\parallel}|^2 + |u_{\perp}|^2\right]^{\frac{\gamma - 1}{2}} \chi\big(\sqrt{|u_{\parallel}|^2 + |u_{\perp}|^2}\big) \frac{b^{ \alpha \beta}(\theta)}{|\cos \theta|} \, du_{\perp}.\label{ker1rnomrr}
		\end{align}
		If $\alpha \neq \beta$, $($ the part with exponential decay in v $)$
		\begin{align}
			k_{\alpha \beta}^{(2)}(v,u_{\perp},u_{\parallel})	&=:  \frac{1}{|u_{\parallel}|} e^{-\frac{1}{4}(|\xi_{\parallel}|^2+|\eta_{\parallel}|^2)} \chi\big(\sqrt{|u_{\parallel}|^2 + |u_{\perp}|^2}\big) \notag\\
			&\hspace{3cm}\times  e^{-\frac{1}{4}(|\xi_{\perp}|^2+|\eta_{\perp}|^2) } \left[|u_{\parallel}|^2 + |u_{\perp}|^2\right]^{\frac{\gamma - 1}{2}}\frac{b^{ \alpha \beta}(\theta)}{|\cos \theta|}.\label{ker2rrnomr}
		\end{align}
		If $\alpha = \beta$, $($the degenerate part $)$
		\begin{align}
			k_{\alpha \alpha}^{(2)}(v,u_{\perp})	&=: \frac{1}{|u_{\perp}|} e^{-\frac{m^\alpha}{2}\big[|\zeta^{\perp}|^2+\frac{1}{4}|u_{\perp}|^2\big]} \int_{\mathbb{R}^2} e^{-\frac{m^\alpha}{2}|u_{\parallel} + \zeta^{\parallel}|^2}\notag \\
			&\hspace{2.5cm}\times  \left[|u_{\parallel}|^2 + |u_{\perp}|^2\right]^{\frac{\gamma - 1}{2}} \chi\big(\sqrt{|u_{\parallel}|^2 + |u_{\perp}|^2}\big) \frac{b^{ \alpha \beta}(\theta)}{|\cos \theta|} \, du_{\parallel}. \label{deftmanommbddk2}
		\end{align}
		In \eqref{ker1rnomrr} and \eqref{ker2rrnomr}, $u_{\parallel} \in \mathbb{R}^3, u_{\perp} \in \mathbb{R}^2$, while in \eqref{deftmanommbddk2}, $u_{\perp} \in \mathbb{R}^3, u_{\parallel} \in \mathbb{R}^2$. According to the analysis in Lemma \ref{lemma2.2}, we know \eqref{ker2rrnomr} has exponential decay in $v$, while \eqref{ker1rnomrr} and \eqref{deftmanommbddk2} do not.  Since \eqref{HBD264} contains integral kernels of two different types, we call it ``Hybrid part''. Furthermore, based on the previous  description,  $k_{\alpha \alpha}^{(1)}(v,u_{\parallel})$ and $k_{\alpha \alpha}^{(2)}(v,u_{\perp})$ are symmetric with respect to $u_{\parallel}$ and $u_{\perp}$. Moreover, $K_{2,(2)}^{\alpha,\chi}\mathbf{f}$ can be also expressed by
		\begin{align*}
			&K_{2,(2)}^{\alpha,\chi}\mathbf{f}= \iint_{\mathbb{R}^3\times\mathbb{R}^2}\hat{k}_{\alpha \beta}^{(2)}(v,u_{\parallel},u_{\perp}) \,f^\beta(v + u_{\parallel}+\frac{m^{\beta}-m^{\alpha}}{m^{\alpha}+m^{\beta}}u_{\perp})\,d u_{\parallel}\,du_{\perp} \quad \alpha \neq \beta \notag\\
			&\hspace{2cm}+\int_{\mathbb{R}^3}k_{\alpha \alpha}^{(2)}(v,u_{\perp}) \,f^\beta(v + u_{\perp})\,du_{\perp}. 
		\end{align*}
		where $u_{\perp} \in \mathbb{R}^3, u_{\parallel} \in \mathbb{R}^2$ and
		\begin{align*}
			\hat{k}_{\alpha \beta}^{(2)}(v,u_{\parallel},u_{\perp})	&=:  \frac{1}{|u_{\perp}|} e^{-\frac{1}{4}(|\xi_{\parallel}|^2+|\eta_{\parallel}|^2)} \chi\big(\sqrt{|u_{\parallel}|^2 + |u_{\perp}|^2}\big) \notag\\
			&\hspace{3cm}\times  e^{-\frac{1}{4}(|\xi_{\perp}|^2+|\eta_{\perp}|^2) } \left[|u_{\parallel}|^2 + |u_{\perp}|^2\right]^{\frac{\gamma - 1}{2}}\frac{b^{ \alpha \beta}(\theta)}{|\cos \theta|}.
		\end{align*}
		Due to $m^{\alpha} \neq m^{\beta}$, $u_{\parallel}$ and $u_{\perp}$ hold equal status, allowing us to interchange their order of integration in \eqref{sec2socal237}-\eqref{malneqmbetuprupeeq} to obtain the same result.
	\end{remark}
	
	\begin{remark}\label{RK2222} 
		``Typical part" can be formally transformed to coincide with the Boltzmann equation for a single species of particles, and it possesses the same decay estimates\,$($see \eqref{FSC16},\eqref{FSC18}$)$. 
		
		Let $v+\frac{2m^{\beta}}{m^{\alpha}+m^{\beta}}u_{\parallel}\rightarrow v^*$, $s<0$, it follows that
		\begin{equation*}
			\int_{\mathbb{R}^3}k_{\alpha \beta}^{(1)}(v,u_{\parallel})\frac{w^s(v)}{w^s(v+\frac{2m^{\beta}}{m^{\alpha}+m^{\beta}}u_{\parallel})}g^{\alpha}(v+\frac{2m^{\beta}}{m^{\alpha}+m^{\beta}}u_{\parallel}) d u_{\parallel}=	\int_{\mathbb{R}^3}\tilde{k}_{\alpha \beta}^{(1)}(v,v^*)\frac{w^s(v)}{w^s(v^*)}g^{\alpha}(v^*) d v^*,
		\end{equation*}
		which satisfies
		\begin{eqnarray}\label{222272}
			\begin{split}
				\frac{w^s(v)}{w^s(v_*)}|\tilde{k}_{\alpha \beta}^{(1)}(v,v_*)|\leq C_{\varepsilon} \frac{(1+|v|+|v_*|)^{\gamma-1}}{|v-v_*|} \, e^{-c\{|v-v_*|^2+\frac{||v|^2-|v_*|^2|^2}{|v-v_*|^2}\}},
			\end{split}
		\end{eqnarray}
		and
		\begin{eqnarray}\label{222273}
			\begin{split}
			\int_{\mathbb{R}^3}	\frac{w^s(v)}{w^s(v_*)}|\tilde{k}_{\alpha \beta}^{(1)}(v,v_*)|\leq C_{\varepsilon} \left \langle v \right\rangle^{\gamma-2}.
			\end{split}
		\end{eqnarray}
	\end{remark}
	\begin{proof}
	\noindent ‌Since the cases with $|u_{\parallel}|\leq 1$ or $|v|\leq 1$ are much simpler, we now consider the cases with  $|u_{\parallel}|\geq 1$ and $|v|\geq 1$. Notice by the expression \eqref{ker1rrr} for $k_{\alpha \beta}^{(1)}(v,u_{\parallel})$, and \eqref{zetperzetpara}, there exists sufficiently small $c>0$, such that
	\begin{align*}
		|k_{\alpha \beta}^{(1),r}(v,u_{\parallel})| &\leq:  \frac{1}{|u_{\parallel}|} e^{-2c\big[|v_{\parallel}|^2+|u_{\parallel}|^2\big]} \int_{\mathbb{R}^2} e^{-2c|u_{\perp} + \zeta_{\perp}|^2} \mathbf{1}_{|v|+|v_*|>m}\notag\\
		&\hspace{2.5cm}\times  \left[|u_{\parallel}|^2 + |u_{\perp}|^2\right]^{\frac{\gamma - 1}{2}} \chi\big(\sqrt{|u_{\parallel}|^2 + |u_{\perp}|^2}\big) \frac{b^{ \alpha \beta}(\theta)}{|\cos \theta|} \, du_{\perp},
	\end{align*} 
	and 
	\begin{align*}
		e^{-c|u_{\parallel}|^2} \leq C \frac{w^s(v)}{w^s(v+\frac{2m^{\beta}}{m^{\alpha}+m^{\beta}}u_{\parallel})} (1+|u_{\parallel}|)^{\gamma-1}.
	\end{align*} 
	Combining with \eqref{caseA1wgams} in Case A1; \eqref{LinfA2} in Case A2; \eqref{LifA3} in Case A3; and \eqref{LifA4} in Case A4, it yields that
	\begin{align}
		\frac{w^s(v)}{w^s(v+\frac{2m^{\beta}}{m^{\alpha}+m^{\beta}}u_{\parallel})}|k_{\alpha \beta}^{(1),r}(v,u_{\parallel})| &\leq  \frac{C_{\varepsilon}}{|u_{\parallel}|} (1+|u_{\parallel}|)^{\gamma-1}(1+|v|)^{\gamma-1}\notag\\
		&\leq  \frac{C_{\varepsilon}}{|u_{\parallel}|} (1+|u_{\parallel}|+|v|)^{\gamma-1}.\label{F247q}
	\end{align} 
	 Since $|v-v_*|=\frac{2m^{\beta}}{m^{\alpha}+m^{\beta}}|u_{\parallel}|$, we adjust the constant $c$ to obtain
	\begin{eqnarray}\label{rek4444}
		\begin{split}
			\frac{w(v)}{w(v_*)}|\tilde{k}_{\alpha \beta}^{(1)}(v,v_*)|&=\frac{w(v)}{w(v+\frac{2m^{\beta}}{m^{\alpha}+m^{\beta}}u_{\parallel})}|k_{\alpha \beta}^{(1)}(v,u_{\parallel})|\\
			&\leq C_{\varepsilon} \frac{(1+|v|+|v_*-v|)^{\gamma-1}}{|v-v_*|} \, e^{-c\{|v-v_*|^2+|v_{\parallel}|^2\}}.
		\end{split}
	\end{eqnarray}
	Since
	\begin{equation*}
		v+v_*=2v+\frac{2m^{\beta}}{m^{\alpha}+m^{\beta}}u_{\parallel}, \qquad  (|v_{\parallel}|^2+|u_{\parallel}|^2) \thicksim (|2v_{\parallel}+\frac{2m^{\beta}}{m^{\alpha}+m^{\beta}}u_{\parallel}|^2+|u_{\parallel}|^2),
	\end{equation*}
	 the exponential in \eqref{rek4444} can be replaced by
	\begin{equation*}
		e^{-c\{|v-v_*|^2+|2v_{\parallel}+\frac{2m^{\beta}}{m^{\alpha}+m^{\beta}}u_{\parallel}|^2\}}.
	\end{equation*}
	Then,  we further write $2v_{\parallel}+\frac{2m^{\beta}}{m^{\alpha}+m^{\beta}}u_{\parallel}$ as
	\begin{equation*}
		\big|2v_{\parallel}+\frac{2m^{\beta}}{m^{\alpha}+m^{\beta}}u_{\parallel}\big|=\frac{\big|\left \langle 2v+\frac{2m^{\beta}}{m^{\alpha}+m^{\beta}}u_{\parallel}, \frac{2m^{\beta}}{m^{\alpha}+m^{\beta}}u_{\parallel} \right \rangle_{v}\big|}{|\frac{2m^{\beta}}{m^{\alpha}+m^{\beta}}u_{\parallel}|}=\frac{|\left \langle v+v_*, v-v_* \right \rangle_{v}|}{|v-v_*|}=\frac{||v|^2-|v_*|^2|}{|v-v_*|}.
	\end{equation*}
	Therefore,
	\begin{equation*}
		\big(|v-v_*|^2+|v_{\parallel}|^2\big) \thicksim \big(|v-v_*|^2+\frac{||v|^2-|v_*|^2|^2}{|v-v_*|^2}\big).
	\end{equation*}
	Combine with
	\begin{equation*}
		(1+|v|+|v_*-v|) \thicksim (1+|v|+|v_*|),
	\end{equation*}
	then, \eqref{222272} follows. 
	
	Notice the inequality that  $|u_{\parallel}|^2+|v_{\parallel}|^2 \geq \frac{1}{2} (|u_{\parallel}|^2+|u_{\parallel}|\cdot|v_{\parallel}|)$, by \eqref{F247q}, it yields that
	\begin{align*}
		&\int_{\mathbb{R}^3}|k_{M,2}^{\alpha \beta(1)}(v,u_{\parallel})|\frac{w(v)}{w(v+\frac{2m^{\beta}}{m^{\alpha}+m^{\beta}}u_{\parallel})} d u_{\parallel} \\
		&\hspace{0.4cm}\leq C_{\varepsilon} \int_{\mathbb{R}^3}\frac{(1+|v|+|u_{\parallel}|)^{\gamma-1}}{|u_{\parallel}|} \, e^{-c\{|u_{\parallel}|^2+|u_{\parallel}|\cdot|v_{\parallel}|\}} d u_{\parallel}.
	\end{align*}
	Since $|u_{\parallel}|\cdot|v_{\parallel}|=|v\cdot u_{\parallel}|$, and $u_{\parallel}\in \mathbb{R}^3$, we decompose $u_{\parallel}$ by
	\begin{equation*}
		u_{\parallel}=u^{0}_{\parallel}+u^{1}_{\parallel},\qquad  u^{0}_{\parallel} \parallel v, \qquad u^{1}_{\parallel} \perp v,
	\end{equation*}
	where $u^{0}_{\parallel}\in\mathbb{R}^1$, $u^{1}_{\parallel}\in\mathbb{R}^2$. So, it yields that
	\begin{align}
		&C_{\varepsilon} \int_{\mathbb{R}^1}\frac{(1+|v|+|u_{\parallel}|)^{\gamma-1}}{|u_{\parallel}|} \, e^{-c\{|u_{\parallel}|^2+|u_{\parallel}|\cdot|v_{\parallel}|\}} d u_{\parallel} \notag\\
		&\hspace{0.4cm}\leq C_{\varepsilon}(1+|v|)^{\gamma-1} \int_{\mathbb{R}^1} \, e^{-c|v\cdot u_{\parallel}|} d u^0_{\parallel}\int_{\mathbb{R}^2} \frac{1}{|u^1_{\parallel}|}e^{-c|u^1_{\parallel}|^2}d u^1_{\parallel}\notag\\
		&\hspace{0.4cm}\leq C_{\varepsilon}(1+|v|)^{\gamma-1} \int_{\mathbb{R}^1} \, e^{-c\{|v| u^0_{\parallel}\}} d u^0_{\parallel}\notag\\
		&\hspace{0.4cm}\leq C_{\varepsilon}(1+|v|)^{\gamma-2} \int_{\mathbb{R}^1} \, e^{-c\{|v| u^0_{\parallel}\}} d \{|v|u^0_{\parallel}\}\notag\\
		&\hspace{0.4cm} \leq C_{\varepsilon}(1+|v|)^{\gamma-2}.\label{lstintg270}
	\end{align}
	This completes the proof of remark \ref{RK2222} .
	\end{proof}
	The integral kernels $k_{\alpha \alpha}^{(2)}$ and $k_{\alpha \beta}^{(2)}$ likewise satisfy \eqref{222272}, \eqref{222273}, since $k_{\alpha \alpha}^{(2)}$ is symmetric with respect to $k_{\alpha \alpha}^{(1)}$, and $k_{\alpha \beta}^{(2)}$ possesses exponential decay in $v$, clearly satisfying these decay estimates. Consequently, the vector-valued integral kernel $\bm{k}(v,v_*)$ satisfies estimates \eqref{222272} and \eqref{222273}.

	\begin{lema}\label{LeMcoevlem} 
		There exists a $\delta>0$, such that
		\begin{equation}\label{LeMcoev}
			\left\langle \mathbf{L}\mathbf{f}, \mathbf{f}  \right\rangle_{L_v^2}\geq \delta |(\mathbf{I}-\mathbf{P}_0)\mathbf{f}|^2_{\nu}
		\end{equation}
		\begin{proof}
			Assume the conclusion does not hold, then there exists a sequence of vector functions $\mathbf{f}_n$ satisfying 
			\begin{equation}\label{lem23hgnj0}
				\left\langle \mathbf{f}_n, \bm{\chi}_j   \right\rangle_{L_v^2}=0, \quad n=1,2,\cdots, \quad  j=1,2,\cdots 6.
			\end{equation}
			such that $\left\langle \mathbf{L}\mathbf{f}_n, \mathbf{f}_n  \right\rangle_{L_v^2}\leq \frac{1}{n}$, and 
			\begin{equation*}
				\left\langle (\bm{\nu}:\mathbf{f}_n), \mathbf{f}_n  \right\rangle_{L_v^2}=|\mathbf{f}_n|^2_{\nu} \equiv 1. 
			\end{equation*}
			Then, there exists a subsequence of $\mathbf{f}_n$
			that converges weakly to $\mathbf{f}_0$ with respect to the inner product $\left\langle \cdot, \cdot \right\rangle_{\nu}$. The weak limit $\mathbf{f}_0$ satisfies
			\begin{equation*}
				|\mathbf{f}_0|^2_{\nu} \leq 1.
			\end{equation*}
			Note that
			\begin{equation*}
				\left\langle \mathbf{L}\mathbf{f}_n, \mathbf{f}_n  \right\rangle_{L_v^2}=\left\langle (\bm{\nu}:\mathbf{f}_n), \mathbf{f}_n  \right\rangle_{L_v^2}-\left\langle \mathbf{K}\mathbf{f}_n, \mathbf{f}_n  \right\rangle_{L_v^2}=1-\left\langle \mathbf{K}\mathbf{f}_n, \mathbf{f}_n  \right\rangle_{L_v^2}. 
			\end{equation*}
			By choosing $\varepsilon$ small, we split $\mathbf{K}=\mathbf{K}_s+\mathbf{K}_c$, where $\mathbf{K}_s$ satisfies 
			\begin{equation*}
				|\left\langle \mathbf{K}_s\mathbf{f}_n, \mathbf{f}_n  \right\rangle_{L_v^2}|\leq \varepsilon.
			\end{equation*}
			The compactness of $\mathbf{K}_c$ for $\gamma\leq 0$ implies
			\begin{equation*}
				\lim_{n \to \infty}|\mathbf{K}_c\mathbf{f}_n-\mathbf{K}_c\mathbf{f}_0|_{L_v^2}=0,
			\end{equation*}
			which gives 
			\begin{equation*}
				\lim_{n \to \infty}\left\langle \mathbf{K}_c\mathbf{f}_n, \mathbf{f}_n  \right\rangle_{L_v^2}=\left\langle \mathbf{K}_c\mathbf{f}_0, \mathbf{f}_0  \right\rangle_{L_v^2}.
			\end{equation*}
			Since $\varepsilon$ can be choosen arbitrarily, one gets 
			\begin{equation*}
				\lim_{n \to \infty}\left\langle \mathbf{K}\mathbf{f}_n, \mathbf{f}_n  \right\rangle_{L_v^2}=\left\langle \mathbf{K}\mathbf{f}_0, \mathbf{f}_0  \right\rangle_{L_v^2}.
			\end{equation*}
			We take limit $n\rightarrow \infty$ to obtain
			\begin{equation*}
				\left\langle \mathbf{L}\mathbf{f}_0, \mathbf{f}_0  \right\rangle_{L_v^2}=\left\langle (\bm{\nu}:\mathbf{f}_0), \mathbf{f}_0  \right\rangle_{L_v^2}-\left\langle \mathbf{K}\mathbf{f}_0, \mathbf{f}_0  \right\rangle_{L_v^2}\leq1-\lim_{n \to \infty}\left\langle \mathbf{K}\mathbf{f}_n, \mathbf{f}_n  \right\rangle_{L_v^2}=\lim_{n \to \infty}\left\langle \mathbf{L}\mathbf{f}_n, \mathbf{f}_n  \right\rangle_{L_v^2}=0. 
			\end{equation*}
			But $\mathbf{L}$ is a non-negative operator.
			Therefore, $\mathbf{f}_0 = \mathbf{P}_0 \mathbf{f}_0$, together with \eqref{lem23hgnj0}, we get $\mathbf{f}_0=\mathbf{0}$. 
			
			It contradicts $|\mathbf{f}_0|^2_{\nu} = 1$. 
		\end{proof}
	\end{lema}

	The following lemma is used to treat the high order derivatives.

	\begin{lema}\label{222maxdel}
		Let $\tilde{\beta}$ be  a multi-index and $l\geq0$. For any small positive $\eta$, there exists a  constant $C_{\eta}>0$ such that
		\begin{align}
			\big\langle w^{2l}  \partial_{\tilde{\beta}}[(\bm{\nu}:\mathbf{f})], \partial_{\tilde{\beta}} \mathbf{f} \big\rangle_{L_v^2} 
			&\geq  \big|w^{l}\partial_{\tilde{\beta}}\mathbf{f}\big|^2_{\nu}-\eta \sum_{|\tilde{\beta_1}|\leq|\tilde{\beta}|}\big|w^{l}\partial_{\tilde{\beta}_1}\mathbf{f}\big|^2_{\nu}-C_{\eta}\big|w^{l}\mathbf{f}\big|^2_{\nu};\label{lem24252tep}\\
			\big| \big\langle w^{2l} \partial_{\tilde{\beta}}[\mathbf{K} \mathbf{f}_1], \partial_{\tilde{\beta}} \mathbf{f}_2 \big\rangle\big| 
			&\leq \Bigg(\eta\sum_{|\tilde{\beta_1}|\leq|\tilde{\beta}|}|w^{l}\partial_{\tilde{\beta}_1}\mathbf{f}_1|_{\nu}+C_{\eta}|w^{l}\mathbf{f}_1|_{\nu}\Bigg)\big|w^{l}\partial_{\tilde{\beta}}\mathbf{f}_2\big|_{\nu}.\label{lem24253tep}
		\end{align}
	\end{lema}
	\begin{proof}
		By the chain rule of differentiation, 
		\begin{equation}\label{zzzdbbl}
			\big\langle w^{2l}  \partial_{\tilde{\beta}}[(\bm{\nu}:\mathbf{f})], \partial_{\tilde{\beta}} \mathbf{f} \big\rangle_{L_v^2}=\big\langle w^{2l}  [(\bm{\nu}:\partial_{\tilde{\beta}}\mathbf{f})], \partial_{\tilde{\beta}} \mathbf{f} \big\rangle_{L_v^2}+C_{\tilde{\beta}}^{\tilde{\beta_1}}\big\langle w^{2l}  [(\partial_{\tilde{\beta_1}}\bm{\nu}:\partial_{\tilde{\beta}-\tilde{\beta_1}}\mathbf{f})], \partial_{\tilde{\beta}} \mathbf{f} \big\rangle_{L_v^2}
		\end{equation}
		where $\tilde{\beta}\geq\tilde{\beta_1}$. Moreover, a direct calculation shows that
		\begin{equation}\label{nus24derivty}
			|\partial_{\tilde{\beta_1}}\bm{\nu}| \leq C (1+|v|)^{\gamma-1}.
		\end{equation}
		Then we obtain
		\begin{align*}
			\big\langle w^{2l}  [(\partial_{\tilde{\beta_1}}\bm{\nu}:\partial_{\tilde{\beta}-\tilde{\beta_1}}\mathbf{f})], \partial_{\tilde{\beta}} \mathbf{f} \big\rangle_{L_v^2} =& \big\langle w^{2l}  [(\partial_{\tilde{\beta_1}}\bm{\nu}:\partial_{\tilde{\beta}-\tilde{\beta_1}}\mathbf{f})], \mathbf{1}_{\{|v|>m\}}\partial_{\tilde{\beta}} \mathbf{f} \big\rangle_{L_v^2}\\
			&\hspace{0.4cm}+\big\langle w^{2l}  [(\partial_{\tilde{\beta_1}}\bm{\nu}:\partial_{\tilde{\beta}-\tilde{\beta_1}}\mathbf{f})], \mathbf{1}_{\{|v|\leq m\}}\partial_{\tilde{\beta}} \mathbf{f} \big\rangle_{L_v^2}.
		\end{align*}
		By \eqref{nus24derivty}, the part of $|v|\geq m$ is bounded by
		\begin{align*}
			\big\langle w^{2l}  [(\partial_{\tilde{\beta_1}}\bm{\nu}:\partial_{\tilde{\beta}-\tilde{\beta_1}}\mathbf{f})], \mathbf{1}_{\{|v|>m\}}\partial_{\tilde{\beta}} \mathbf{f} \big\rangle_{L_v^2} &\leq \frac{C}{m} \big|w^{l}\partial_{\tilde{\beta}-\tilde{\beta_1}}\mathbf{f}\big|_{\nu} \big|w^{l}\partial_{\tilde{\beta}}\mathbf{f}\big|_{\nu} \\
			& \leq  \frac{\eta}{2} \big|w^{l}\partial_{\tilde{\beta}-\tilde{\beta_1}}\mathbf{f}\big|_{\nu} \big|w^{l}\partial_{\tilde{\beta}}\mathbf{f}\big|_{\nu}.
		\end{align*}
		For the $m$ chosen above, by Cauchy-Schwarz inequality as well as the compact interpolation in the Sobolev space, the part of $|v|\leq m$ is bounded by
		\begin{align*}
			\big\langle w^{2l}  [(\partial_{\tilde{\beta_1}}\bm{\nu}:\partial_{\tilde{\beta}-\tilde{\beta_1}}\mathbf{f})], \mathbf{1}_{\{|v|\leq m\}}\partial_{\tilde{\beta}} \mathbf{f} \big\rangle_{L_v^2} \leq \frac{\eta}{2} \sum_{|\tilde{\beta_1}|=|\tilde{\beta}|} \big|w^{l}\partial_{\tilde{\beta_1}}\mathbf{f}\big|_{\nu}^2 + C_{\eta} \big|w^{l}\mathbf{f}\big|_{\nu}^2.
		\end{align*}
		Therefore, this concludes \eqref{lem24252tep}.
		
		Next we turn to the estimate for \eqref{lem24253tep}. From the decomposition $\mathbf{K}=\mathbf{K}_1+\mathbf{K}_2$, first we consider $\big\langle w^{2l} \partial_{\tilde{\beta}}\mathbf{K}_1 \mathbf{f}_1, \partial_{\tilde{\beta}}\mathbf{f}_2 \big\rangle_{L^{2}_{v}}$, which is bounded by
			\begin{align}
				 & \sum_{\alpha=A,B}\sum_{\beta=A,B} C_{\tilde{\beta}}^{\tilde{\beta_1}} \iiint_{\mathbb{S}^2\times(\mathbb{R}^3)^2} w^{2l}(v)|v_*|^\gamma  \partial_{\tilde{\beta_1}}\big(\sqrt{\mu^\alpha(v)} \sqrt{\mu^\beta(v_*+v)}\big)\notag\\
				 &\hspace{6cm}\times\partial_{\tilde{\beta}-\tilde{\beta_1}}f^\beta_1(v_*+v)\partial_{\tilde{\beta}}f^\alpha_2(v)b^{\alpha \beta}(\theta)\,d\omega \, dv_*\,dv \notag\\
				&\hspace{0.4cm}\leq\sum_{\alpha=A,B}\sum_{\beta=A,B} C_{\tilde{\beta}}^{\tilde{\beta_1}} \Bigg(\iiint_{\mathbb{S}^2\times(\mathbb{R}^3)^2} w^{2l}(v)|v_*|^\gamma  \partial_{\tilde{\beta_1}}\big(\sqrt{\mu^\alpha(v)} \sqrt{\mu^\beta(v_*+v)}\big)\label{lem24k11}\\
				&\hspace{5.5cm}\times\chi(|v_{*}|)\partial_{\tilde{\beta}-\tilde{\beta_1}}f^\beta_1(v_*+v)\partial_{\tilde{\beta}}f^\alpha_2(v)b^{\alpha \beta}(\theta)\,d\omega \, dv_*\,dv \notag\\
				&\hspace{1cm}+ \iiint_{\mathbb{S}^2\times(\mathbb{R}^3)^2} w^{2l}(v)|v_*|^\gamma  \partial_{\tilde{\beta_1}}\big(\sqrt{\mu^\alpha(v)} \sqrt{\mu^\beta(v_*+v)}\big)\label{lem24k12}\\
				&\hspace{5.2cm}\times\big[1-\chi(|v_{*}|)\big]\partial_{\tilde{\beta}-\tilde{\beta_1}}f^\beta_1(v_*+v)\partial_{\tilde{\beta}}f^\alpha_2(v)b^{\alpha \beta}(\theta)\,d\omega \, dv_*\,dv \Bigg)\notag,
			\end{align}
		where we have changed variable $v_*-v\rightarrow v_*$. Note that $\cos\theta=\frac{\omega\cdot(v-v_*)}{|v-v_*|}$ and $|b^{\alpha \beta}(\theta)|\leq C$, it yields that
		\begin{equation*}
			\Big|\partial_{\tilde{\beta_1}}\big(\sqrt{\mu^\alpha(v)} \sqrt{\mu^\beta(v_*+v)}\big)\Big|\leq C e^{-\frac{m_s}{8}(|v+v_*|^2+|v|^2)},
		\end{equation*}
		where we denote $m_s=\min\{m^A,m^B\}$.
		Applying a further change of variable and integrating $\omega$ over $\mathbb{S}^2$. \eqref{lem24k12} is bounded by
		\begin{align*}
			&C\sum_{\alpha=A,B}\sum_{\beta=A,B}\sum_{\tilde{\beta_1}}\iint_{(\mathbb{R}^3)^2} e^{-\frac{m_s}{8}(|v_*|^2+|v|^2)}|v_*-v|^\gamma w^{2l}(v) \\
			&\hspace{6cm}\times\big[1-\chi(|v-v_{*}|)\big]\big|\partial_{\tilde{\beta}-\tilde{\beta_1}}f^\beta_1(v_*)\partial_{\tilde{\beta}}f^\alpha_2(v)\big| \, dv_*\,dv \\
			&\hspace{0.5cm}\leq C\sum_{\alpha=A,B}\sum_{\beta=A,B}\sum_{\tilde{\beta_1}}\Bigg(\int_{\mathbb{R}^3} \int_{\mathbb{R}^3}e^{-\frac{m_s|v|^2}{8}}|v_*-v|^\gamma w^{2l}(v)\big[1-\chi(|v-v_{*}|)\big] dv \\
			&\hspace{8.2cm}\times e^{-\frac{m_s|v_*|^2}{8}}|\partial_{\tilde{\beta}-\tilde{\beta_1}}f^\beta_1(v_*)|^2 \, dv_*\,\Bigg)^{1/2}\\
			&\hspace{1cm} \times \Bigg(\int_{\mathbb{R}^3} \int_{\mathbb{R}^3}e^{-\frac{m_s|v_*|^2}{8}}|v_*-v|^\gamma \big[1-\chi(|v-v_{*}|)\big] dv_* \,e^{-\frac{m_s|v|^2}{8}}w^{2l}(v)|\partial_{\tilde{\beta}}f^\alpha_2(v)|^2 dv\Bigg)^{1/2}\\
			&\hspace{0.5cm}\leq  C \varepsilon^{3+\gamma}\sum_{\tilde{\beta_1}}\big|w^{l}\partial_{\tilde{\beta}-\tilde{\beta}_1}\mathbf{f}_1\big|_{\nu}\,\big|w^{l}\partial_{\tilde{\beta}}\mathbf{f}_2\big|_{\nu}\\ &\hspace{0.5cm}\leq \frac{\eta}{2}\Big(\sum_{\tilde{\beta}_1\leq\tilde{\beta}}\big|w^{l}\partial_{\tilde{\beta}_1}\mathbf{f}_1\big|_{\nu}\Big)\big|w^{l}\partial_{\tilde{\beta}}\mathbf{f}_2\big|_{\nu}
		\end{align*}
		For the small $\varepsilon>0$ chosen above, since $\chi_{\varepsilon}$ vanishes near the origin, an integration over $v_*$ in \eqref{lem24k11} yields
		\begin{align*}
			& \sum_{\alpha=A,B}\sum_{\beta=A,B} C_{\tilde{\beta}}^{\tilde{\beta_1}} \iiint_{\mathbb{S}^2\times(\mathbb{R}^3)^2} \partial_{\tilde{\beta}-\tilde{\beta_1}}^{v_*}\Big[|v_*|^\gamma \chi(|v_{*}|) \partial_{\tilde{\beta_1}}\big(\sqrt{\mu^\alpha(v)} \sqrt{\mu^\beta(v_*+v)}\big)\Big]\\
			&\hspace{6.5cm}\times w^{2l}(v)f^\beta_1(v_*+v)\partial_{\tilde{\beta}}f^\alpha_2(v)b^{\alpha \beta}(\theta)\,d\omega \, dv_*\,dv \\
			&\hspace{0.5cm}\leq \sum_{\alpha=A,B}\sum_{\beta=A,B}C_{\varepsilon} \iint_{(\mathbb{R}^3)^2} e^{-\frac{m_s}{8}(|v+v_*|^2+|v|^2)}|f^\beta_1(v_*+v)\partial_{\tilde{\beta}}f^\alpha_2(v)| dv_* \, dv \\
			&\hspace{0.5cm}\leq  \, C_{\varepsilon}\, \big|w^{l}\mathbf{f}_1\big|_{\nu} \big|w^{l}\partial_{\tilde{\beta}}\mathbf{f}_2\big|_{\nu}
		\end{align*}
		where we have used notation $\partial_{\tilde{\beta}-\tilde{\beta_1}}^{v_*}$ to denote the derivative with respect to $v_*$, in order to distinguish it from notation $\partial_{\tilde{\beta}-\tilde{\beta_1}}$ which represents the derivative with respect to $v$. Therefore, the estimate for the part of $\mathbf{K}_1$ can be derived.
		
		Now we consider the part for $\mathbf{K}_2$. $\mathbf{K}_2\mathbf{f}_1$ takes the form
		\begin{align}
			\sum_{\beta=A,B}&\Bigg(\iint_{\mathbb{S}^2\times\mathbb{R}^3} |u|^{\gamma} [\mu^{\beta}(u+v)]^{\frac{1}{2}}\big[1-\chi(|u|)\big]\big[\mu^{\beta}(v + u_{\perp}+\frac{m^{\beta}-m^{\alpha}}{m^{\alpha}+m^{\beta}}u_{\parallel})\big]^{\frac{1}{2}}\label{K2f1essigf11}\\
			&\hspace{5.5cm}\times f^\alpha_1(v + \frac{2m^{\beta}}{m^{\alpha}+m^{\beta}}u_{\parallel})b^{ \alpha \beta}(\theta) \, du \, d\omega \notag\\
			&\hspace{0.3cm}+ \iint_{\mathbb{S}^2\times\mathbb{R}^3} |u|^{\gamma} [\mu^{\beta}(u+v)]^{\frac{1}{2}}\big[1-\chi(|u|)\big]\big[\mu^{\alpha}(v + \frac{2m^{\beta}}{m^{\alpha}+m^{\beta}}u_{\parallel})\big]^{\frac{1}{2}}\label{K2f1essigf12}\\
			&\hspace{5cm}\times f^\beta_1(v + u_{\perp}+\frac{m^{\beta}-m^{\alpha}}{m^{\alpha}+m^{\beta}}u_{\parallel})b^{ \alpha \beta}(\theta) \, du \, d\omega \notag\\
			&\hspace{0.3cm}+\iint_{\mathbb{S}^2\times\mathbb{R}^3} |u|^{\gamma} [\mu^{\beta}(u+v)]^{\frac{1}{2}}\chi(|u|)\big[\mu^{\beta}(v + u_{\perp}+\frac{m^{\beta}-m^{\alpha}}{m^{\alpha}+m^{\beta}}u_{\parallel})\big]^{\frac{1}{2}}\label{K2f1etssigf13}\\
			&\hspace{5.5cm}\times f^\alpha_1(v + \frac{2m^{\beta}}{m^{\alpha}+m^{\beta}}u_{\parallel})b^{ \alpha \beta}(\theta) \, du \, d\omega \notag\\
			&\hspace{0.3cm}+ \iint_{\mathbb{S}^2\times\mathbb{R}^3} |u|^{\gamma} [\mu^{\beta}(u+v)]^{\frac{1}{2}}\chi(|u|)\big[\mu^{\alpha}(v + \frac{2m^{\beta}}{m^{\alpha}+m^{\beta}}u_{\parallel})\big]^{\frac{1}{2}}\label{K2f1ressigf14}\\
			&\hspace{5.3cm}\times f^\beta_1(v + u_{\perp}+\frac{m^{\beta}-m^{\alpha}}{m^{\alpha}+m^{\beta}}u_{\parallel})b^{ \alpha \beta}(\theta) \, du \, d\omega \Bigg). \notag
		\end{align}
		To compute $\partial_{\tilde{\beta}} \mathbf{K}_2\mathbf{f}_1$, differentiate the second term above directly with respect to $\beta$ using the product rule. By a further change of variables, this approach bypasses the need to compute derivatives of the singular kernel $|u - v|^\gamma$ in \eqref{K2f1essigf11} and \eqref{K2f1essigf12}. 
		Consequently, the term $\langle w^{2l} \partial_{\tilde{\beta}} \mathbf{K}_2 \mathbf{f}_1, \partial_{\tilde{\beta}} \mathbf{f}_2 \rangle$ can be decomposed as follows:
			\begin{align}
			&\sum_{\alpha=A,B}\sum_{\beta=A,B}\sum_{\tilde{\beta}_{i}}C_{\tilde{\beta}}\Bigg(\iiint_{\mathbb{S}^2\times(\mathbb{R}^3)^2} w^{2l}(v) |v-v_*|^{\gamma} \partial_{\tilde{\beta}_1}[\mu^{\beta}(v_*)]^{\frac{1}{2}}\big[1-\chi(|v-v_*|)\big]\partial_{\tilde{\beta}_2}[\mu^{\beta}(v_*')]^{\frac{1}{2}}\notag\\
			&\hspace{8.2cm}\times \partial_{\tilde{\beta}_3}f^\alpha_1(v')\partial_{\tilde{\beta}}f^\alpha_2(v)b^{ \alpha \beta}(\theta) \, du \,dv\, d\omega \label{sma248}\\
			&\hspace{0.3cm}+ \iiint_{\mathbb{S}^2\times(\mathbb{R}^3)^2} w^{2l}(v) |v-v_*|^{\gamma} \partial_{\tilde{\beta}_1}[\mu^{\beta}(v_*)]^{\frac{1}{2}}\big[1-\chi(|v-v_*|)\big]\partial_{\tilde{\beta}_2}[\mu^{\alpha}(v')]^{\frac{1}{2}}\notag\\
			&\hspace{9cm}\times \partial_{\tilde{\beta}_3}f^\beta(v_*')\partial_{\tilde{\beta}}f^\alpha_2(v)b^{ \alpha \beta}(\theta) \, dv_* \, dv\, d\omega \Bigg)\label{sma249}\\
			&\hspace{0.3cm}+\sum_{\alpha=A,B}\sum_{\beta=A,B}\Bigg(\int_{\mathbb{R}^3} \partial_{\tilde{\beta}}\Big[\iint_{\mathbb{S}^2\times\mathbb{R}^3}|v-v_*|^{\gamma} [\mu^{\beta}(v_*)]^{\frac{1}{2}}\chi(|v-v_*|)[\mu^{\beta}(v_*')]^{\frac{1}{2}} f^\alpha(v')b^{ \alpha \beta}(\theta) \, du \, d\omega \Big]\notag\\
				&\hspace{10.5cm}\times w^{2l}(v) \partial_{\tilde{\beta}}f^\alpha_2(v) \, dv\label{bga250}\\
			&\hspace{0.3cm}+ \int_{\mathbb{R}^3} \partial_{\tilde{\beta}}\Big( \iint_{\mathbb{S}^2\times\mathbb{R}^3} |v-v_*|^{\gamma} [\mu^{\beta}(v_*)]^{\frac{1}{2}}\chi(|v-v_*|)[\mu^{\alpha}(v')]^{\frac{1}{2}}
			 f^\beta(v_*')b^{ \alpha \beta}(\theta) \, du \, d\omega \Big)
			 w^{2l}(v) \partial_{\tilde{\beta}}f^\alpha_2(v) \, dv\Bigg), \label{bga251}
		\end{align}
		where $\tilde{\beta}=\tilde{\beta}_1+\tilde{\beta}_2+\tilde{\beta}_3$. Since 
		\begin{equation*}
			\big|\partial_{\tilde{\beta}}[\mu^{\beta}(\cdot)]^{\frac{1}{2}}\big| \leq C e^{-\frac{m_s|\cdot|^2}{8}}.
		\end{equation*}
		Following a similar procedure as in \eqref{lem22k2ssig} to \eqref{lem22cclssig}, the terms \eqref{sma248} and \eqref{sma249} are bounded by
		\begin{equation*}
			\varepsilon^{3+\gamma}\sum_{\tilde{\beta_1}}|w^{l}\partial_{\tilde{\beta}-\tilde{\beta}_1}\mathbf{f}_1|_{\nu}\,|w^{l}\partial_{\tilde{\beta}}\mathbf{f}_2|_{\nu} \leq \frac{\eta}{2}\Big(\sum_{\tilde{\beta}_1\leq\tilde{\beta}}|w^{l}\partial_{\tilde{\beta}_1}\mathbf{f}_1|_{\nu}\Big)|w^{l}\partial_{\tilde{\beta}}\mathbf{f}_2|_{\nu}.
		\end{equation*}
		For the small $\varepsilon>0$ above, \eqref{bga250}, \eqref{bga251}can be written as
		\begin{align}
		\sum_{\alpha,\beta}\sum_{\tilde{\beta}_{i}}&C_{\tilde{\beta}}	\Bigg( \int_{\mathbb{R}^3} \frac{1}{|u_{\parallel}|} \partial_{\tilde{\beta}_1}\big(e^{-\frac{m^\beta}{2}|\zeta_{\parallel}|^2}\big) e^{-\frac{m^\beta}{2}(\frac{m^\alpha}{m^\alpha+m^\beta})^2|u_{\parallel}|^2}\partial_{\tilde{\beta}_2}f^\alpha(v + \frac{2m^{\beta}}{m^{\alpha}+m^{\beta}}u_{\parallel}) \label{CE2BDSefsD1}\\
			&\hspace{0.2cm}\times \int_{\mathbb{R}^2} \partial_{\tilde{\beta}_1}\big(e^{-\frac{m^\beta}{2}|u_{\perp} + \zeta_{\perp}|^2}\big) \left[|u_{\parallel}|^2 + |u_{\perp}|^2\right]^{\frac{\gamma - 1}{2}} \chi\big(\sqrt{|u_{\parallel}|^2 + |u_{\perp}|^2}\big) \frac{b^{ \alpha \beta}(\theta)}{|\cos \theta|} \, du_{\perp} du_{\parallel}\notag\\
			 +&\int_{\mathbb{R}^3} \frac{1}{|u_{\parallel}|} e^{-\frac{1}{4}[|\xi_{\parallel}|^2+|\eta_{\parallel}|^2]}\int_{\mathbb{R}^2} \frac{b^{ \alpha \beta}(\theta)}{|\cos \theta|}\chi\big(\sqrt{|u_{\parallel}|^2 + |u_{\perp}|^2}\big)  \label{CE2BDfSHHD2}\\
			&\hspace{1.3cm} \times  e^{-\frac{1}{4}(|\xi_{\perp}|^2+|\eta_{\perp}|^2) } \left[|u_{\parallel}|^2 + |u_{\perp}|^2\right]^{\frac{\gamma - 1}{2}}f^\beta(v + u_{\perp}+\frac{m^{\beta}-m^{\alpha}}{m^{\alpha}+m^{\beta}}u_{\parallel})  \, du_{\perp} du_{\parallel} \Bigg)\notag.
		\end{align}
		By the expression of  $\eta_{\parallel},\xi_{\perp},\eta_{\perp},\xi_{\parallel}$ and $\zeta_{\perp}, \zeta_{\parallel}$, we deduce that
		\begin{equation*}
			|\nabla_v \eta_{\parallel}|+|\nabla_v \eta_{\perp}|+|\nabla_v \xi_{\parallel}|+|\nabla_v \xi_{\perp}|+|\nabla_v \zeta_{\parallel}|+|\nabla_v \zeta_{\perp}| \leq C,
		\end{equation*}
		which further gives
		\begin{align*}
			\big|\partial_{\tilde{\beta}_i}\big(e^{-\frac{m^\beta}{2}|\zeta_{\parallel}|^2}\big)\big|&\leq C e^{-\frac{m^\beta}{4}|\zeta_{\parallel}|^2},\\
			\big|\partial_{\tilde{\beta}_i}\big(e^{-\frac{m^\beta}{2}|u_{\perp} + \zeta_{\perp}|^2}\big)\big|&\leq C e^{-\frac{m^\beta}{4}|u_{\perp} + \zeta_{\perp}|^2},\\
			\big|\partial_{\tilde{\beta}_1}\big(e^{-\frac{1}{4}(|\xi_{\parallel}|^2+|\eta_{\parallel}|^2)}\big)\big|&\leq C e^{-\frac{1}{8}(|\xi_{\parallel}|^2+|\eta_{\parallel}|^2)},\\
			\big|\partial_{\tilde{\beta}_1}\big(e^{-\frac{1}{4}(|\xi_{\perp}|^2+|\eta_{\perp}|^2) }\big)\big|&\leq C e^{-\frac{1}{8}(|\xi_{\perp}|^2+|\eta_{\perp}|^2) }.
		\end{align*}
		Consequently, apart from the factors in the exponents, equation \eqref{CE2BDSefsD1},\eqref{CE2BDfSHHD2} is almost identical to \eqref{CE2easyPT},\eqref{CE2hadPT}. By choosing $m>0$ large enough, the integration domain is further partitioned into $(|v|+|v_*| \leq m)$ and $(|v|+|v_*| \geq m)$. 
		
		Applying analogous estimation techniques for  $\mathbf{K}_{2,r}^{\chi}\mathbf{f}_1$ as demonstrated in  \eqref{caseA1wgams}--\eqref{lem2230cgvb} and \eqref{sec2CaseB1}--\eqref{sec2CaseB4},   the contribution from the region $(|v|+|v_*| \geq m)$ can be bounded by
		\begin{equation*}
			\frac{C_{\varepsilon}}{m}\Bigg(\sum_{\tilde{\beta_1}}\big|w^{l}\partial_{\tilde{\beta}-\tilde{\beta}_1}\mathbf{f}_1\big|_{\nu}\Bigg)\,\big|w^{l}\partial_{\tilde{\beta}}\mathbf{f}_2\big|_{\nu} \leq \frac{\eta}{2}\Bigg(\sum_{\tilde{\beta}_1\leq\tilde{\beta}}\big|w^{l}\partial_{\tilde{\beta}_1}\mathbf{f}_1\big|_{\nu}\Bigg)\big|w^{l}\partial_{\tilde{\beta}}\mathbf{f}_2\big|_{\nu}.
		\end{equation*}
		Finally, we consider the remaining part, the operator  $\mathbf{K}_{2,z}^{\chi}$
			\begin{align*}
			&\sum_{\alpha,\beta}\sum_{\tilde{\beta}_{i}}C_{\tilde{\beta}}	\Bigg( \int_{\mathbb{R}^3} \frac{1}{|u_{\parallel}|} \partial_{\tilde{\beta}_1}\big(e^{-\frac{m^\beta}{2}|\zeta_{\parallel}|^2}\big) e^{-\frac{m^\beta}{2}(\frac{m^\alpha}{m^\alpha+m^\beta})^2|u_{\parallel}|^2}\partial_{\tilde{\beta}_2}f^\alpha(v + \frac{2m^{\beta}}{m^{\alpha}+m^{\beta}}u_{\parallel}) \\
			&\hspace{0.5cm}\times \int_{\mathbb{R}^2} \partial_{\tilde{\beta}_1}\big(e^{-\frac{m^\beta}{2}|u_{\perp} + \zeta_{\perp}|^2}\big) \left[|u_{\parallel}|^2 + |u_{\perp}|^2\right]^{\frac{\gamma - 1}{2}} \chi\big(\sqrt{|u_{\parallel}|^2 + |u_{\perp}|^2}\big) \frac{b^{ \alpha \beta}(\theta)}{|\cos \theta|}\mathbf{1}_{|v|+|v_*|\leq m} \, du_{\perp} du_{\parallel}\\
			&\hspace{0.2cm}+\int_{\mathbb{R}^3} \frac{1}{|u_{\parallel}|} e^{-\frac{1}{4}(|\xi_{\parallel}|^2+|\eta_{\parallel}|^2)}\int_{\mathbb{R}^2} \frac{b^{ \alpha \beta}(\theta)}{|\cos \theta|}\chi\big(\sqrt{|u_{\parallel}|^2 + |u_{\perp}|^2}\big)  \\
			& \hspace{0.5cm}\times  e^{-\frac{1}{4}(|\xi_{\perp}|^2+|\eta_{\perp}|^2) } \left[|u_{\parallel}|^2 + |u_{\perp}|^2\right]^{\frac{\gamma - 1}{2}}f^\beta(v + u_{\perp}+\frac{m^{\beta}-m^{\alpha}}{m^{\alpha}+m^{\beta}}u_{\parallel})\mathbf{1}_{|v|+|v_*|\leq m}  \, du_{\perp} du_{\parallel} \Bigg)
		\end{align*}
		is compact since $\frac{1}{|u_{\parallel}|} \in L^2_{loc}(\mathbb{R}^3)$. Hence, there exists a $C_{\eta}>0$ such that $\left \langle  w^{2l}\mathbf{K}_{2,r}^{\chi}\mathbf{f}_1,\mathbf{f}_2 \right \rangle_{\nu}$ is bounded by
		\begin{align*}
			 \Bigg(\frac{\eta}{4} \sum_{|\tilde{\beta_1}|\leq|\tilde{\beta}|} \big|w^{l}\partial_{\tilde{\beta_1}}\mathbf{f}\big|_{\nu}^2 + C_{\eta} \big|w^{l}\mathbf{f}\big|_{\nu}^2\Bigg)\big|w^{l}\partial_{\tilde{\beta}}\mathbf{f}_2\big|_{\nu}.
		\end{align*}
		Based on the above analysis, we have completed the proof of this lemma.
	\end{proof}

	\section{The nonlinear collision operator}
	This section is primarily devoted to discussing derivative estimates in both space and velocity for the nonlinear collision operator $\bf{\Gamma}\big[f_1,f_2\big]$. 
	\begin{align*}
	&\partial_{\tilde{\beta}}^{\tilde{\alpha}} \Gamma^{\alpha \beta} \big[f^{\alpha}_1, f^{\beta}_2\big]\\
	 &\hspace{1.2cm}= 
		\partial_{\tilde{\beta}}^{\tilde{\alpha}} \Bigg( 
	\int_{\mathbb{R}^3} \int_{S^2} |u|^\gamma e^{-\frac{m^{\beta}|u+v|^2}{4}} 
	f^{\alpha}_1(v + \frac{2m^{\beta}}{m^{\alpha}+m^{\beta}}u_{\parallel}) \\
	&\hspace{4.8cm} \times f^{\beta}_2(v + u_\perp+\frac{m^{\beta}-m^{\alpha}}{m^{\alpha}+m^{\beta}}u_{\parallel}) b^{\alpha\beta}(\theta) \, du \, d\omega 
	\Bigg) \\
	&\hspace{1.9cm}- 	\partial_{\tilde{\beta}}^{\tilde{\alpha}} \Bigg( 
	\int_{\mathbb{R}^3} \int_{S^2} |u|^\gamma e^{-\frac{m^{\beta}|u+v|^2}{4}} 
	f^{\alpha}_1(v) f^{\beta}_2(v+u) b^{\alpha\beta}(\theta) \, du \, d\omega 
	\Bigg) \\
	&\hspace{1.2cm}= \sum_{\tilde{\beta}=\tilde{\beta}_0+\tilde{\beta}_1+\tilde{\beta}_2} 
	\sum_{\tilde{\alpha}=\tilde{\alpha}_1+\tilde{\alpha}_2} C_{\tilde{\beta}}^{\tilde{\beta}_0 \tilde{\beta}_1 \tilde{\beta}_2} C_{\tilde{\alpha}}^{\tilde{\alpha}_1 \tilde{\alpha}_2} 
	\Gamma^{\alpha \beta}_0 \big[\partial_{\tilde{\beta}_1}^{\tilde{\alpha}_1}f^{\alpha}_1, \partial_{\tilde{\beta}_2}^{\tilde{\alpha}_2}f^{\beta}_2\big]
	\end{align*}
	 where we have  changed a variable $v_*-v \rightarrow u$. Subsequently, by the product rule and an inverse substitution of variables,
	\begin{align}
	&\Gamma^{\alpha \beta}_0 \big[\partial_{\tilde{\beta}_1}^{\tilde{\alpha}_1}f^{\alpha}_1, \partial_{\tilde{\beta}_2}^{\tilde{\alpha}_2}f^{\beta}_2\big]\notag\\
		&\hspace{1.2cm}= \int_{\mathbb{R}^3} \int_{S^2} |v_* - v|^\gamma 
		\partial_{\tilde{\beta}_0}  \big(e^{-\frac{m^{\beta}|v_*|^2}{4}}\big) 
		\partial_{\tilde{\beta}_1}^{\tilde{\alpha}_1}f^{\alpha}_1 (v') \partial_{\tilde{\beta}_2}^{\tilde{\alpha}_2}f^{\beta}_2 (v_*') 
		b^{\alpha\beta}(\theta) \, dv_* \, d\omega\notag\\
		&\hspace{1.6cm}- \partial_{\tilde{\beta}_1}^{\tilde{\alpha}_1}f^{\alpha}_1(v) 
		\int_{\mathbb{R}^3} \int_{S^2} |u - v|^\gamma 
		\partial_{\tilde{\beta}_0}  \big(e^{-\frac{m^{\beta}|v_*|^2}{4}}\big) 
		\partial_{\tilde{\beta}_2}^{\tilde{\alpha}_2}f^{\beta}_2 (v_*)
		b^{\alpha\beta}(\theta) \, dv_* \, d\omega \notag\\
		&\hspace{1.2cm}=: \Gamma_{0,\text{gain}}^{\alpha \beta} - \Gamma_{0,\text{loss}}^{\alpha \beta}. \label{Dsocald57}
	\end{align}

	\begin{lemma}\label{estimatesofnonlinears}
		Let $-3 < \gamma < 0$, $N \geq 8$. The multiple indicators $\tilde{\beta}_0, \tilde{\beta}_1, \tilde{\beta}_2$ and $\tilde{\alpha}_1, \tilde{\alpha}_2$ satisfy $\tilde{\beta}_0 + \tilde{\beta}_1 + \tilde{\beta}_2 = \beta$ and $\tilde{\alpha}_1 + \tilde{\alpha}_2 = \tilde{\alpha}$. Suppose  that $|\tilde{\beta}|\leq l$.  If $|\tilde{\alpha}_1|+|\tilde{\beta}_1|\geq \frac{N}{2}$, then  there exists some constant $C$ such that 
		\begin{align}
		&	\Big| \left\langle w^{2l}\Gamma^{\alpha \beta}_0 [\partial_{\tilde{\beta}_1}^{\tilde{\alpha}_1}f^{\alpha}_1, \partial_{\tilde{\beta}_2}^{\tilde{\alpha}_2}f^{\beta}_2],  \partial_{\tilde{\beta}}^{\tilde{\alpha}}f^{\alpha}_3 \right\rangle_{L_{x,v}^2} \Big|\notag	\\
			& \hspace{0.4cm} \leq C \Bigg(\sum_{|\tilde{\alpha}_i|+|\tilde{\beta}_i|\leq N}\big\Vert w^{|\tilde{\beta}_i|}\partial_{\tilde{\beta}_i}^{\tilde{\alpha}_i}\mathbf{f}_2\big\Vert_{L^{2}_{x,v}} \Bigg)\big\Vert w^{l}\partial_{\tilde{\beta}_1}^{\tilde{\alpha}_1}\mathbf{f}_1\big\Vert_{\nu} \big\Vert w^{l}\partial_{\tilde{\beta}}^{\tilde{\alpha}}\mathbf{f}_3\big\Vert_{\nu}\label{lem31nonl3f} \\
			& \hspace{0.7cm} + C \Bigg(\sum_{|\tilde{\alpha}_i|+|\tilde{\beta}_i|\leq N}\big\Vert w^{|\tilde{\beta}_2|}\partial_{\tilde{\beta}_2}^{\tilde{\alpha}_i}\mathbf{f}_2\big\Vert_{L^{2}_{x,v}} \Bigg)\big\Vert w^{l-|\tilde{\beta}_2|}\partial_{\tilde{\beta}_1}^{\tilde{\alpha}_1}\mathbf{f}_1\big\Vert_{\nu} \big\Vert w^{l}\partial_{\tilde{\beta}}^{\tilde{\alpha}}\mathbf{f}_3\big\Vert_{\nu}.\notag
		\end{align}
		Moreover, if $|\tilde{\alpha}_1|+|\tilde{\beta}_1|\leq \frac{N}{2}$, then
		\begin{align}
			&	\Big| \left\langle w^{2l}\Gamma^{\alpha \beta}_0 [\partial_{\tilde{\beta}_1}^{\tilde{\alpha}_1}f^{\alpha}_1, \partial_{\tilde{\beta}_2}^{\tilde{\alpha}_2}f^{\beta}_2],  \partial_{\tilde{\beta}}^{\tilde{\alpha}}f^{\alpha}_3 \right\rangle_{L_{x,v}^2} \Big|	\notag\\
			& \hspace{0.4cm} \leq C \Bigg(\sum_{|\tilde{\alpha}_i|+|\tilde{\beta}_i|\leq N}\big\Vert w^{|\tilde{\beta}_i|}\partial_{\tilde{\beta}_i}^{\tilde{\alpha}_i}\mathbf{f}_1\big\Vert_{L^{2}_{x,v}} \Bigg)\big\Vert w^{l}\partial_{\tilde{\beta}_2}^{\tilde{\alpha}_2}\mathbf{f}_2\big\Vert_{\nu} \big\Vert w^{l}\partial_{\tilde{\beta}}^{\tilde{\alpha}}\mathbf{f}_3\big\Vert_{\nu}\label{lem31nonl3s} \\
			& \hspace{0.7cm} + C \Bigg(\sum_{|\tilde{\alpha}_i|+|\tilde{\beta}_i|\leq N}\big\Vert w^{|\tilde{\beta}_1|}\partial_{\tilde{\beta}_1}^{\tilde{\alpha}_i}\mathbf{f}_1\big\Vert_{L^{2}_{x,v}} \Bigg)\big\Vert w^{l-|\tilde{\beta}_1|}\partial_{\tilde{\beta}_2}^{\tilde{\alpha}_2}\mathbf{f}_2\big\Vert_{\nu} \big\Vert w^{l}\partial_{\tilde{\beta}}^{\tilde{\alpha}}\mathbf{f}_3\big\Vert_{\nu}.\notag
		\end{align}
		\end{lemma}

		\begin{proof}
		First we consider the loss term $\Gamma_{0,\text{loss}}^{\alpha \beta}$. Since  $\partial_{\tilde{\beta}_0} \big(e^{-\frac{m^{\beta}|v_*|^2}{4}}\big)  \leq C  e^{-\frac{m^{\beta}|v_*|^2}{8}} $, then we have
		\begin{align*}
			&\int_{\mathbb{R}^3}  |v_* - v|^\gamma 
			\big| \partial_{\tilde{\beta}_0}  \big(e^{-\frac{m^{\beta}|v_*|^2}{4}}\big) 
			 \partial_{\tilde{\beta}_2}^{\tilde{\alpha}_2}f^{\beta}_2 (x,v_*)\big| 
			 \, dv_* \\
			&\hspace{0.5cm} \leq C \Big( \int_{\mathbb{R}^3} |v_* - v|^{\gamma} e^{-\frac{m^{\beta}|v_*|^2}{8}} \big| \partial_{\tilde{\beta}_2}^{\tilde{\alpha}_2}f^{\beta}_2 (x,v_*) \big|^2 dv_* \Big)^{1/2} 
			 \Big( \int_{\mathbb{R}^3} |v_* - v|^{\gamma} e^{-\frac{m^{\beta}|v_*|^2}{8}} dv_* \Big)^{1/2} \\
			&\hspace{0.5cm} \leq C \sup_{x,v_*} \Big| e^{-\frac{m^{\beta}|v_*|^2}{8}} \partial_{\tilde{\beta}_2}^{\tilde{\alpha}_2}f^{\beta}_2 (x,v_*) \Big| \Big( \int_{\mathbb{R}^3} |v_* - v|^{\gamma} e^{-\frac{m^{\beta}|v_*|^2}{8}} dv_* \Big)^{1/2} \\
			&\hspace{0.5cm} \leq C \sup_{x,v_*} \Big| e^{-\frac{m^{\beta}|v_*|^2}{8}} \partial_{\tilde{\beta}_2}^{\tilde{\alpha}_2}f^{\beta}_2 (x,v_*) \Big| \times \big(1 + |v|\big)^\gamma
		\end{align*}
		Notice that $|\tilde{\alpha}_1|+|\tilde{\alpha}_2|\leq \frac{N}{2}$, $N\geq 8$, the Sobolev embedding theorem implies $H^4(T^3 \times \mathbb{R}^3) \subset L^\infty
		$, then
		\begin{equation*}
			 \sup_{x,v_*} \Big| e^{-\frac{m^{\beta}|v_*|^2}{8}} \partial_{\tilde{\beta}_2}^{\tilde{\alpha}_2}f^{\beta}_2 (x,v_*) \Big| \leq C \sum_{|\tilde{\alpha}_i|+|\tilde{\beta}_i|\leq N}\big\Vert w^{|\tilde{\beta}_i|}\partial_{\tilde{\beta}_i}^{\tilde{\alpha}_i}\mathbf{f}_2\big\Vert_{L^{2}_{x,v}}
		\end{equation*}
		Since $b^{\alpha\beta}(\theta)\leq C$, it follows that
		\begin{align*}
			&	\Big| \left\langle w^{2l} \Gamma_{0,\text{loss}}^{\alpha \beta} [\partial_{\tilde{\beta}_1}^{\tilde{\alpha}_1}f^{\alpha}_1, \partial_{\tilde{\beta}_2}^{\tilde{\alpha}_2}f^{\beta}_2],  \partial_{\tilde{\beta}}^{\tilde{\alpha}}f^{\alpha}_3 \right\rangle_{L_{x,v}^2} \Big|	\\
			& \hspace{0.5cm} \leq C \iint_{\mathbb{R}^3\times\mathbb{T}^3} (1 + |v|)^\gamma w^{2l}(v) \big|\partial_{\tilde{\beta}_1}^{\tilde{\alpha}_1}f^{\alpha}_1 (x,v) \partial_{\tilde{\beta}}^{\tilde{\alpha}}f^{\alpha}_3 (x,v)\big|dv\, dx\\
			& \hspace{6.3cm}\times \Bigg(\sum_{|\tilde{\alpha}_i|+|\tilde{\beta}_i|\leq N}\big\Vert w^{|\tilde{\beta}_i|}\partial_{\tilde{\beta}_i}^{\tilde{\alpha}_i}\mathbf{f}_2\big\Vert_{L^{2}_{x,v}} \Bigg)\\
			& \hspace{0.5cm} \leq C \Bigg(\sum_{|\tilde{\alpha}_i|+|\tilde{\beta}_i|\leq N}\big\Vert w^{|\tilde{\beta}_i|}\partial_{\tilde{\beta}_i}^{\tilde{\alpha}_i}\mathbf{f}_2\big\Vert_{L^{2}_{x,v}} \Bigg)\big\Vert w^{l}\partial_{\tilde{\beta}_1}^{\tilde{\alpha}_1}\mathbf{f}_1\big\Vert_{\nu} \big\Vert w^{l}\partial_{\tilde{\beta}}^{\tilde{\alpha}}\mathbf{f}_3\big\Vert_{\nu}
		\end{align*}
		The case of $|\tilde{\alpha}_1|+|\tilde{\alpha}_1|\leq \frac{N}{2}$ is thereby proven for $\Gamma_{0,\text{loss}}^{\alpha \beta}$.
		
		Next we turn to the case of $|\tilde{\alpha}_2|+|\tilde{\beta}_2|\leq \frac{N}{2}$, and split the integration domain into three parts 
		\begin{align*}
			&D_1=\big(|v_*-v|\leq \frac{|v|}{2}\big),\hspace{1.2cm} D_2=\Big[\big(|v_*-v|\geq \frac{|v|}{2}\big)\cap \big(|v|\geq 1\big)\Big],\\
			& D_3=\Big[\big(|v_*-v|\geq \frac{|v|}{2}\big)\cap \big(|v|\leq 1\big)\Big].
		\end{align*}
		For the first part $(v_*,v)\in D_1$, $|v_*|\geq |v|-|v_*-v|\geq \frac{|v|}{2}$, which gives 
		\begin{equation*}
			e^{-\frac{m^{\beta}|v_*|^2}{8}} \leq e^{-\frac{m^{\beta}|v_*|^2}{16}}e^{-\frac{m^{\beta}|v|^2}{64}}.
		\end{equation*}
		Then, $\left\langle w^{2l}\Gamma^{\alpha \beta}_0 \big[\partial_{\tilde{\beta}_1}^{\tilde{\alpha}_1}f^{\alpha}_1, \partial_{\tilde{\beta}_2}^{\tilde{\alpha}_2}f^{\beta}_2\big],  \partial_{\tilde{\beta}}^{\tilde{\alpha}}f^{\alpha}_3 \right\rangle_{L_{x,v}^2}$ in $D_1$ can be bounded by
		\begin{align*}
			&C \iiint_{(\mathbb{R}^3)^2\times\mathbb{T}^3} \left| v_* - v \right|^\gamma e^{-\frac{m^{\beta}|v_*|^2}{16}}e^{-\frac{m^{\beta}|v|^2}{64}} w^{2l}(v)\\ 
			& \hspace{5cm}\times\big|\partial_{\tilde{\beta}_1}^{\tilde{\alpha}_1}f^{\alpha}_1 (x,v) \partial_{\tilde{\beta}_2}^{\tilde{\alpha}_2}f^{\beta}_2(x,v_*) \partial_{\tilde{\beta}}^{\tilde{\alpha}}f^{\alpha}_3 (x,v)\big| \,dv_*\,dv\,dx \\
			&\hspace{0.6cm}\leq C \Bigg( 
			\iiint_{(\mathbb{R}^3)^2\times\mathbb{T}^3} \left| v_* - v \right|^\gamma e^{-\frac{m^{\beta}|v_*|^2}{16}}e^{-\frac{m^{\beta}|v|^2}{64}} w^{2l}(v) 
			\big| \partial_{\tilde{\beta}_2}^{\tilde{\alpha}_2}f^{\beta}_2(x,v_*) \big|^2 \,dv_*\,dv\,dx 
			\Bigg)^{1/2} \\
			&\hspace{2cm} \times \Bigg( 
			\iiint_{(\mathbb{R}^3)^2\times\mathbb{T}^3} \left| v_* - v \right|^\gamma e^{-\frac{m^{\beta}|v_*|^2}{16}}e^{-\frac{m^{\beta}|v|^2}{64}} w^{2l}(v) \\
			&\hspace{6.3cm}\times\big| \partial_{\tilde{\beta}_1}^{\tilde{\alpha}_1}f^{\alpha}_1 (x,v) \big|^2 
			\big| \partial_{\tilde{\beta}}^{\tilde{\alpha}}f^{\alpha}_3 (x,v) \big|^2 \,dv_*\,dv\,dx 
			\Bigg)^{1/2}.
		\end{align*}
		‌Integration over $v \in \mathbb{R}^3$
		 in the first factor establishes an upper bound
		 \begin{align*}
		 &C \Bigg( \iint_{\mathbb{R}^3\times\mathbb{T}^3} \Big[ \int_{\mathbb{R}^3} |v_* - v|^\gamma e^{-\frac{m^{\beta}|v|^2}{32}} dv_* \Big] e^{-\frac{m^{\beta}|v|^2}{64}} w^{2l} \big| \partial_{\tilde{\beta}_1}^{\tilde{\alpha}_1}f^{\alpha}_1 (x,v) \big|^2 
		 \big| \partial_{\tilde{\beta}}^{\tilde{\alpha}}f^{\alpha}_3 (x,v) \big|^2 dv dx \Bigg)^{1/2}\\
		 &\hspace{10cm}\times\big\| w^l \partial_{\tilde{\beta}_2}^{\tilde{\alpha}_2}f^{\beta}_2 \big\|_{\nu}\\
		 &\hspace{0.5cm}\leq C \Bigg( \iint_{\mathbb{R}^3\times\mathbb{T}^3} e^{-\frac{m^{\beta}|v|^2}{64}} w^{2l} (1 + |v|)^\gamma \big| \partial_{\tilde{\beta}_1}^{\tilde{\alpha}_1}f^{\alpha}_1 (x,v) \big|^2 
		 \big| \partial_{\tilde{\beta}}^{\tilde{\alpha}}f^{\alpha}_3 (x,v) \big|^2 dv dx \Bigg)^{1/2}\\
		 &\hspace{10cm}\times\big\| w^l \partial_{\tilde{\beta}_2}^{\tilde{\alpha}_2}f^{\beta}_2 \big\|_{\nu}
		 \end{align*}
		 Once again applying the Sobolev embedding theorem,  
		 \begin{equation*}
		 	\sup_{x,v} \Big| e^{-\frac{m^{\beta}|v_*|^2}{64}} \partial_{\tilde{\beta}_1}^{\tilde{\alpha}_1}f^{\alpha}_1 (x,v) \Big| \leq C \sum_{|\tilde{\alpha}_i|+|\tilde{\beta}_i|\leq N}\big\Vert w^{|\tilde{\beta}_i|}\partial_{\tilde{\beta}_i}^{\tilde{\alpha}_i}\mathbf{f}_1\big\Vert_{L^{2}_{x,v}}
		 \end{equation*}
		 Thus we conclude the  case in $D_1=\big(|v_*-v|\leq \frac{|v|}{2}\big)$.
		 
		 Next consider $\Gamma^0_{\text{loss}}$ in $D_2=\Big[\big(|v_*-v|\geq \frac{|v|}{2}\big)\cap \big(|v|\geq 1\big)\Big]$. Since $\gamma < 0$, this implies $|u - v|^\gamma \le C|v|^\gamma$, and
		 \begin{align*}
		 &\int_{D_2} |v_* - v|^\gamma | \partial_{\tilde{\beta}_0} \big( e^{-\frac{m^{\beta}|v_*|^2}{4}}\big)  \partial_{\tilde{\beta}_1}^{\tilde{\alpha}_1}f^{\alpha}_1 (x,v)w^{2l} \partial_{\tilde{\beta}_2}^{\tilde{\alpha}_2}f^{\beta}_2(x,v_*) \partial_{\tilde{\beta}}^{\tilde{\alpha}}f^{\alpha}_3 (x,v)| dv_* dv dx \\
		 &\hspace{0.3cm}\leq C \int_{\mathbb{T}^3} \left(\int_{\mathbb{R}^3} e^{-\frac{m^{\beta}|v_*|^2}{8}} \big| \partial_{\tilde{\beta}_2}^{\tilde{\alpha}_2}f^{\beta}_2(x,v_*) \big| du \right) \left( \int_{|v| \ge 1} |v|^\gamma w^{2l} \big|\partial_{\tilde{\beta}_1}^{\tilde{\alpha}_1}f^{\alpha}_1 (x,v)  \partial_{\tilde{\beta}}^{\tilde{\alpha}}f^{\alpha}_3 (x,v)\big| dv \right) dx\\
		 &\hspace{0.3cm}\leq C \int_{\mathbb{T}^3} \left( \int_{\mathbb{R}^3} (1 + |v|)^\gamma w^{2l} | \partial_{\tilde{\beta}_1}^{\tilde{\alpha}_1}f^{\alpha}_1 |^2 dv \right)^{1/2}
		  \left( \int_{|v| \ge 1} |v|^\gamma w^{2l} | \partial_{\tilde{\beta}}^{\tilde{\alpha}}f^{\alpha}_3 |^2 dv \right)^{1/2}\big| w^l \partial_{\tilde{\beta}_2}^{\tilde{\alpha}_2}f^{\beta}_2 \big|_{\nu} dx  .
		\end{align*}
		 Since $|\alpha_2| + |\beta_2| \le N/2$, $N \ge 8$, and $W^{4,1}(\mathbb{T}^3) \subseteq L^\infty$,
		 $$
		 \sup_x \int_{\mathbb{R}^3} (1 + |v|)^\gamma w^{2l} | \partial_{\tilde{\beta}_1}^{\tilde{\alpha}_1}f^{\alpha}_1(x, v) |^2 dv \le C \sum_{|\alpha_i| + |\beta_1| \le N} \big\| w^l \partial_{\tilde{\beta}_1}^{\tilde{\alpha}_i}f^{\alpha}_1 \big\|_{L^{2}_{x,v}}^2. 
		 $$
		 Therefore, the part for $D_2=\Big[\big(|v_*-v|\geq \frac{|v|}{2}\big)\cap \big(|v|\geq 1\big)\Big]$ is bounded by
		 $$
		  C \Bigg( \sum_{|\tilde{\alpha}_i| + |\tilde{\beta}_1| \le N} \big\| w^l \partial_{\tilde{\beta}_1}^{\tilde{\alpha}_i}\mathbf{f}_1 \big\|_{L^{2}_{x,v}} \Bigg) \big\| w^l \partial_{\tilde{\beta}_2}^{\tilde{\alpha}_2}\mathbf{f}_2 \big\|_{\nu} \big\| w^l \partial_{\tilde{\beta}}^{\tilde{\alpha}}\mathbf{f}_3 \big\|_{\nu}.
		 $$
		 
		 Now we consider the part $D_3=\Big[\big(|v_*-v|\geq \frac{|v|}{2}\big)\cap \big(|v|\leq 1\big)\Big]$, it holds that $|u - v|^\gamma \le C|v|^\gamma$ and
		 \begin{align*}
		 	&\iiint_{D_3} |v_* - v|^\gamma | \partial_{\tilde{\beta}_0} \big( e^{-\frac{m^{\beta}|v_*|^2}{4}}\big) w^{2l} \big|\partial_{\tilde{\beta}_1}^{\tilde{\alpha}_1}f^{\alpha}_1 (x,v) \partial_{\tilde{\beta}_2}^{\tilde{\alpha}_2}f^{\beta}_2(x,v_*) \partial_{\tilde{\beta}}^{\tilde{\alpha}}f^{\alpha}_3 (x,v)\big| dv_* dv dx \\
		 	&\hspace{0.5cm}\leq C \int_{\mathbb{T}^3} \left( \int_{\mathbb{R}^3} |v_* - v|^\gamma e^{-\frac{m^{\beta}|v_*|^2}{8}} | \partial_{\tilde{\beta}_2}^{\tilde{\alpha}_2}f^{\beta}_2(x,v_*) | du \right) \\
		 	&\hspace{4.8cm}\quad \times \left( \int_{|v| \le 1} |v|^\gamma w^{2l} |\partial_{\tilde{\beta}_1}^{\tilde{\alpha}_1}f^{\alpha}_1 (x,v)  \partial_{\tilde{\beta}}^{\tilde{\alpha}}f^{\alpha}_3 (x,v) | dv \right) dx \\
		 	&\hspace{0.5cm}\leq C \int_{\mathbb{T}^3} \left( \int_{\mathbb{R}^3} |v_* - v|^\gamma e^{-\frac{m^{\beta}|v_*|^2}{8}} du \right)^{1/2} \left( \int_{\mathbb{R}^3} e^{-\frac{m^{\beta}|v_*|^2}{8}} | \partial_{\tilde{\beta}_2}^{\tilde{\alpha}_2}f^{\beta}_2(x,v_*) |^2 du \right)^{1/2} \\
		 	&\hspace{3.4cm} \times \left( \int_{|v| \le 1} |v|^\gamma w^{2l} | \partial_{\tilde{\beta}_1}^{\tilde{\alpha}_1}f^{\alpha}_1 |^2 dv \right)^{1/2} \left( \int_{|v| \le 1} w^{2l} | \partial_{\tilde{\beta}}^{\tilde{\alpha}}f^{\alpha}_3 |^2 dv \right)^{1/2} dx \\
		 	&\hspace{0.5cm}\leq C \int_{\mathbb{T}^3} \big| w^\theta \partial_{\tilde{\beta}_2}^{\tilde{\alpha}_2}f^{\beta}_2 \big|_{\nu} \left( \int_{|v| \le 1} |v|^\gamma w^{2l} | \partial_{\tilde{\beta}_1}^{\tilde{\alpha}_1}f^{\alpha}_1 |^2 dv \right)^{1/2} \left( \int_{|v| \le 1} w^{2l} | \partial_{\tilde{\beta}}^{\tilde{\alpha}}f^{\alpha}_3 |^2 dv \right)^{1/2} dx.
		 \end{align*}
		 Since $|\tilde{\alpha}_1| + |\tilde{\beta}_1| \leq \frac{N}{2}$ and $H^4(\mathbb{T}^3 \times \mathbb{R}^3) \subset L^\infty$,
		 $$
		 \int_{|v| \le 1} |v|^\gamma w^{2l} | \partial_{\tilde{\beta}_1}^{\tilde{\alpha}_1}f^{\alpha}_1 |^2 dv \le C \sup_{|v| \le 1, x \in \mathbb{T}^3} | \partial_{\tilde{\beta}_1}^{\tilde{\alpha}_1}f^{\alpha}_1 |^2 \le C \sum_{|\tilde{\alpha}_i| + |\tilde{\beta}_i| \le N} \big\| w^{|\tilde{\beta}_i|}  \partial_{\tilde{\beta}_i}^{\tilde{\alpha}_i}f^{\alpha}_1  \big\|_{L^{2}_{x,v}}^2.
		 $$
		 From the condition $\gamma > -3$, we can establish the following bound for the integral in $D_3$
		 \begin{equation*}
		 \iiint_{D_3} \leq C \Bigg( \sum_{|\tilde{\alpha}_i| + |\tilde{\beta}_i| \le N} \big\| w^{|\tilde{\beta}_i|}  \partial_{\tilde{\beta}_i}^{\tilde{\alpha}_i}\mathbf{f}_1  \big\|_{L^{2}_{x,v}} \Bigg) \big\| w^l \partial_{\tilde{\beta}_2}^{\tilde{\alpha}_2}\mathbf{f}_2 \big\|_{\nu} \big\| w^l \partial_{\tilde{\beta}}^{\tilde{\alpha}}\mathbf{f}_3 \big\|_{\nu}.
		\end{equation*}
		 This completes the estimate for  $\Gamma_{0,\text{loss}}^{\alpha \beta}$.
		 
		 Next, we turn to estimate the gain term $\Gamma_{0,\text{gain}}^{\alpha \beta}$, for which the integration domain $(v_*, v)$ is split into two parts:
		 $$
		 \big(|v_*| \ge \frac{|v|}{2}\big) \cup \big(|v_*| \le \frac{|v|}{2}\big).
		 $$
		 For the first part $\big(|v_*| \ge \frac{|v|}{2}\big)$,
		 \begin{equation*}
		 	e^{-\frac{m^{\beta}|v_*|^2}{8}} \leq e^{-\frac{m^{\beta}|v_*|^2}{16}}e^{-\frac{m^{\beta}|v|^2}{64}}.
		 \end{equation*}
		 Then,  $\left\langle w^{2l}\Gamma_{0,\text{gain}}^{\alpha \beta} [\partial_{\tilde{\beta}_1}^{\tilde{\alpha}_1}f^{\alpha}_1, \partial_{\tilde{\beta}_2}^{\tilde{\alpha}_2}f^{\beta}_2],  \partial_{\tilde{\beta}}^{\tilde{\alpha}}f^{\alpha}_3\right\rangle_{L_{x,v}^2}$ in this region is bounded by
		 $$
		 \begin{aligned}
		 	&\iiiint_{|v_*| \geq \frac{|v|}{2}} b^{\alpha\beta}(\theta)|v_* - v|^\gamma e^{-\frac{m^{\beta}|v_*|^2}{16}}e^{-\frac{m^{\beta}|v|^2}{64}} w^{2l} \big|\partial_{\tilde{\beta}_1}^{\tilde{\alpha}_1}f^{\alpha}_1 (v') \partial_{\tilde{\beta}_2}^{\tilde{\alpha}_2}f^{\beta}_2(v_*') \partial_{\tilde{\beta}}^{\tilde{\alpha}}f^{\alpha}_3 (v)\big| d\omega dv_*dvdx \\
		 	&\hspace{0.4cm}\leq C \left( \iiint_{(\mathbb{R}^3)^2\times\mathbb{T}^3} |v_* - v|^\gamma e^{-\frac{m^{\beta}|v_*|^2}{16}}e^{-\frac{m^{\beta}|v|^2}{64}} \big|\partial_{\tilde{\beta}_1}^{\tilde{\alpha}_1}f^{\alpha}_1 (v')\big|^2 \big|\partial_{\tilde{\beta}_2}^{\tilde{\alpha}_2}f^{\beta}_2(v_*')\big|^2 dv_*dvdx\right)^{1/2} \\
		 	&\hspace{2.5cm}\quad \times \left( \iiint_{(\mathbb{R}^3)^2\times\mathbb{T}^3} |v_* - v|^\gamma e^{-\frac{m^{\beta}|v_*|^2}{16}}e^{-\frac{m^{\beta}|v|^2}{64}} w^{2l}(v) \big|\partial_{\tilde{\beta}}^{\tilde{\alpha}}f^{\alpha}_3 (v)\big|^2 dv_*dvdx\right)^{1/2} \\
		 	&\hspace{0.4cm}\leq C \left( \iiint_{(\mathbb{R}^3)^2\times\mathbb{T}^3} |v_*' - v'|^\gamma e^{-\frac{m^{\beta}|v_*'|^2}{64}}e^{-\frac{m^{\beta}|v'|^2}{64}} \big|\partial_{\tilde{\beta}_1}^{\tilde{\alpha}_1}f^{\alpha}_1 (v')\big|^2 \big|\partial_{\tilde{\beta}_2}^{\tilde{\alpha}_2}f^{\beta}_2(v_*')\big|^2 dv_*'dv'dx \right)^{1/2} \\
		 	 &\hspace{11.1cm}\times\left\| w^l \partial_{\tilde{\beta}}^{\tilde{\alpha}}f^{\alpha}_3 (v)\right\|_{\nu}.
		 \end{aligned}
		 $$
	Since $l\geq0$ and $b^{\alpha\beta}\leq C$, we first  integrate with respect to $\omega$  in the first line above. Subsequently, we change variables  $(v_*,v)\rightarrow(v_*,v)$ in the second line.
	Based on observation, $f^{\alpha}_1$ and $f^{\beta}_2$ are symmetric in the second line.  Without loss of generality, we can assume that $|\tilde{\alpha}_1| + |\tilde{\beta}_1| \geq \frac{N}{2}$. Then it yields that
	 $$
	\sup_{x,v_*'} \Big\{ e^{-\frac{m^{\beta}|v_*'|^2}{128}} |\partial_{\tilde{\beta}_2}^{\tilde{\alpha}_2}f^{\beta}_2(v_*')|^2 \Big\} \leq C \sum_{|\tilde{\alpha}_i| + |\tilde{\beta}_i| \le N} \big\| w^{|\tilde{\beta}_i|} \partial_{\tilde{\beta}_i}^{\tilde{\alpha}_i} f^{\beta}_2 \big\|_{L^{2}_{x,v}}^2. 
	$$
	As a result, first integrating over $v_*$ in the first factor leads to
	\begin{align*}
		&\int_{\mathbb{R}^3} |v_*' - v'|^\gamma e^{-\frac{m^{\beta}|v_*'|^2}{64}} \big|\partial_{\tilde{\beta}_2}^{\tilde{\alpha}_2}f^{\beta}_2(v_*')\big|^2 dv_*' \\
		&\hspace{0.4cm} \leq C \Bigg(\sum_{|\tilde{\alpha}_i| + |\tilde{\beta}_i| \le N} \big\| w^{|\tilde{\beta}_i|} \partial_{\tilde{\beta}_i}^{\tilde{\alpha}_i} f^{\beta}_2 \big\|_{L^{2}_{x,v}}^2\Bigg) \int_{\mathbb{R}^3} |v_*' - v'|^\gamma e^{-\frac{m^{\beta}|v_*'|^2}{128}} dv_*' \\
		&\hspace{0.4cm}\leq C \sum_{|\tilde{\alpha}_i| + |\tilde{\beta}_i| \le N} \big\| w^{|\tilde{\beta}_i|} \partial_{\tilde{\beta}_i}^{\tilde{\alpha}_i} f^{\beta}_2 \big\|_{L^{2}_{x,v}}^2.
	\end{align*}
	and the first part in $\{|v_*| \geq \frac{|v|}{2}\}$  is bounded by
	 \begin{equation*}
	\int_{|v_*| \ge \frac{|v|}{2}} \leq C \Bigg( \sum_{|\tilde{\alpha}_i| + |\tilde{\beta}_i| \le N} \big\| w^{|\tilde{\beta}_i|}  \partial_{\tilde{\beta}_i}^{\tilde{\alpha}_i}\mathbf{f}_2  \big\|_{L^{2}_{x,v}} \Bigg) \big\| w^l \partial_{\tilde{\beta}_1}^{\tilde{\alpha}_1}\mathbf{f}_1 \big\|_{\nu} \big\| w^l \partial_{\tilde{\beta}}^{\tilde{\alpha}}\mathbf{f}_3 \big\|_{\nu}
	\end{equation*}

	Now we consider the second part $\big(|u| \leq \frac{|v|}{2}\big)$. If $|v - v_*| < \frac{|v|}{2}$, then we will get $|v_*| \ge |v| - |v - v_*| > \frac{|v|}{2}$. Therefore, 
	$$
	|v - v_*| \geq \frac{|v|}{2}.
	$$
	If further 
	$$
	|v| \ge 1,
	$$
	then  $\left\langle w^{2l} \Gamma_{0,\text{gain}}^{\alpha \beta} [\partial_{\tilde{\beta}_1}^{\tilde{\alpha}_1}f^{\alpha}_1, \partial_{\tilde{\beta}_2}^{\tilde{\alpha}_2}f^{\beta}_2]  \partial_{\tilde{\beta}}^{\tilde{\alpha}}f^{\alpha}_3\right\rangle_{L_{x,v}^2}$  is bounded by 

	\begin{align*}
		&\iiiint_{D_2\cap \{|u|\leq |v|\}} w^{2l} |v_* - v|^\gamma|\partial_{\tilde{\beta}_0}  \big(e^{-\frac{m^{\beta}|v_*|^2}{4}}\big) b^{\alpha\beta}(\theta)\\
		&\hspace{6cm}\times\big|\partial_{\tilde{\beta}_1}^{\tilde{\alpha}_1}f^{\alpha}_1 (v') \partial_{\tilde{\beta}_2}^{\tilde{\alpha}_2}f^{\beta}_2 (v_*') \partial_{\tilde{\beta}}^{\tilde{\alpha}}f^{\alpha}_3 (v)\big|
		 d \omega dv_* dv dx \\
		&\hspace{0.5cm}\le C \left( \iiint_{|v| \geq 1} |v|^\gamma e^{-\frac{m^{\beta}|v_*|^2}{4}} w^{2l} \big| \partial_{\tilde{\beta}_1}^{\tilde{\alpha}_1}f^{\alpha}_1 (v') \partial_{\tilde{\beta}_2}^{\tilde{\alpha}_2}f^{\beta}_2 (v_*') \big|^2 dv_* dv dx\right)^{1/2} \\
		&\hspace{5.5cm}\quad \times \left( \iiint_{|v| \ge 1} |v|^\gamma w^{2l} | \partial_{\tilde{\beta}}^{\tilde{\alpha}}f^{\alpha}_3 (v) |^2 dv_* dv dx \right)^{1/2} \\
		&\hspace{0.5cm}\leq C \left( \iiint_{|v| \ge 1} |v|^\gamma w^{2l}(v) \big| \partial_{\tilde{\beta}_1}^{\tilde{\alpha}_1}f^{\alpha}_1 (v') \partial_{\tilde{\beta}_2}^{\tilde{\alpha}_2}f^{\beta}_2 (v_*') \big|^2 dv_*'dv'dx \right)^{1/2} \big\| w^l \partial_{\tilde{\beta}}^{\tilde{\alpha}}f^{\alpha}_3 \big\|_{\nu},
	\end{align*}
	where we have used the fact that $|v_* - v|^\gamma \le 2^{-\gamma}|v|^\gamma$ and $dv_*'dv'=dv_*dv$. To estimate the first factor,  since $|v_*| \leq \frac{|v|}{2}$, from \eqref{F1vuvusacl}, we obtain
	\begin{equation*}
	|v_*'| + |v'| \le C(|v_*| + |v|) \le C|v|.
	\end{equation*}
	The condition $-3<\gamma < 0$, implies
	\begin{equation}\label{socaled6262d}
		|v|^\gamma \le C|v_*'|^\gamma,\ |v|^\gamma \le C|v'|^\gamma,
	\end{equation}
	and $w^{2l}(v) \le C \min\big\{w^{2l}(v_*'), w^{2l}(v')\big\}$, $\theta \ge 0$. Therefore, it holds that 
	\begin{equation*}
	\begin{aligned}
		&\iiint_{|v| \ge 1} |v|^\gamma w^{2l}(v) | \partial_{\tilde{\beta}_1}^{\tilde{\alpha}_1}f^{\alpha}_1 (v') \partial_{\tilde{\beta}_2}^{\tilde{\alpha}_2}f^{\beta}_2 (v_*') |^2 dv_*'dv'dx \\
		&\hspace{0.4cm}\le C \iiint_{|v| \ge 1} \min\big\{|v'|^\gamma, |v_*'|^\gamma, 1\big\}\times \min\big\{w^{2l}(v'), w^{2l}(v_*')\big\} \\
	 &\hspace{6.5cm}\times  \big| \partial_{\tilde{\beta}_1}^{\tilde{\alpha}_1}f^{\alpha}_1 (v')\big|^2 \big|\partial_{\tilde{\beta}_2}^{\tilde{\alpha}_2}f^{\beta}_2 (v_*') \big|^2 dv_*'dv'dx.
	\end{aligned}
	\end{equation*}
	On the other hand, we can also  assume $|\tilde{\alpha}_1| + |\tilde{\beta}_1| \leq \frac{N}{2}$, which gives
	\begin{align*}
		&\iiint_{D_2\cap \{|v_*|\leq |v|\}} w^{2l} |v_* - v|^\gamma|\partial_{\tilde{\beta}_0} \big(e^{-\frac{m^{\beta}|v_*|^2}{4}}\big) 
		\big|\partial_{\tilde{\beta}_1}^{\tilde{\alpha}_1}f^{\alpha}_1 (v')\big|^2 \big|\partial_{\tilde{\beta}_2}^{\tilde{\alpha}_2}f^{\beta}_2 (v_*') |^2
		dv dv_* dx \\
		&\hspace{0.4cm}\leq \iiint_{(\mathbb{R}^3)^2\times\mathbb{T}^3} (1 + |v_*'|)^\gamma w^{2l - 2|\tilde{\beta}_1|}(v_*') w^{2|\tilde{\beta}_1|}(v') 
		 \big| \partial_{\tilde{\beta}_1}^{\tilde{\alpha}_1}f^{\alpha}_1 (v') \big|^2 \big| \partial_{\tilde{\beta}_2}^{\tilde{\alpha}_2}f^{\beta}_2 (v_*') \big|^2 dv_*' dv' dx \\
		&\hspace{0.4cm}\leq \int_{\mathbb{T}^3} \left( \int_{\mathbb{R}^3} w^{2|\tilde{\beta}_1|} \big| \partial_{\tilde{\beta}_1}^{\tilde{\alpha}_1}f^{\alpha}_1  \big|^2 dv' \right)  \left( \int_{\mathbb{R}^3} (1 + |v_*'|)^\gamma w^{2l - 2|\tilde{\beta}_1|}(v_*') \big| \partial_{\tilde{\beta}_2}^{\tilde{\alpha}_2}f^{\beta}_2 \big|^2 dv_*' \right) dx \\
	&\hspace{0.4cm}\leq \sup_x \left( \int_{\mathbb{R}^3} w^{2|\tilde{\beta}_1|} \big| \partial_{\tilde{\beta}_1}^{\tilde{\alpha}_1}f^{\alpha}_1 \big|^2 dv' \right) \big\| w^{l - |\tilde{\beta}_1|} \partial_{\tilde{\beta}_2}^{\tilde{\alpha}_2}f^{\beta}_2 \big\|_{\nu}^2 \\
	&\hspace{0.4cm}\leq C \Bigg(\sum_{|\tilde{\alpha}_i| + |\tilde{\beta}_1| \le N} \big\| w^{|\tilde{\beta}_1|} \partial_{\tilde{\beta}_1}^{\tilde{\alpha}_i}f^{\alpha}_1 \big\|_{L^{2}_{x,v}}^2\Bigg) \big\| w^{l - |\beta_2|} \partial_{\tilde{\beta}_2}^{\tilde{\alpha}_2}f^{\beta}_2 \big\|_{\nu}^2,
	\end{align*}
	from $W^{4,1}(\mathbf{T}^3) \subset L^\infty$. Thus the proof for the case of $\big(|v_*| \leq \frac{|v|}{2}\big)\cap D_2$ is completed.
	
	Now the remaining case is $\big(|v_*| \leq \frac{|v|}{2}\big)\cap D_3$. In this region $|v_*| \leq \frac{1}{2}$ is valid. It holds that
	\begin{align*}
		&\iiint_{|v|\leq1,|v_*|\leq \frac{1}{2}} |v_* - v|^\gamma e^{-\frac{m^{\beta}|v_*|^2}{8}} w^{2l} \Big| \partial_{\tilde{\beta}_1}^{\tilde{\alpha}_1}f^{\alpha}_1 (v') \partial_{\tilde{\beta}_2}^{\tilde{\alpha}_2}f^{\beta}_2 (v_*') \partial_{\tilde{\beta}}^{\tilde{\alpha}}f^{\alpha}_3 (v) \Big| dv_* dv dx\\
		&\hspace{0.3cm}\leq C \iint_{|v|\leq1} \left( \int_{|v_*|\leq\frac{1}{2}} |v_* - v|^\gamma e^{-\frac{m^{\beta}|v_*|^2}{8}} \big| \partial_{\tilde{\beta}_1}^{\tilde{\alpha}_1}f^{\alpha}_1 (v') \partial_{\tilde{\beta}_2}^{\tilde{\alpha}_2}f^{\beta}_2 (v_*') \big| d\omega du \right) \big| \partial_{\tilde{\beta}}^{\tilde{\alpha}}f^{\alpha}_3 (v) \big| dvdx \\
		&\hspace{0.3cm}\leq C \iint_{|v|\leq1} \left( \int_{|v_*|\leq\frac{1}{2}} |v_* - v|^\gamma e^{-\frac{m^{\beta}|v_*|^2}{8}} dv_* \right)^{1/2} \\
		&\hspace{3.7cm}\quad \times \Bigg( \int |v|^\gamma \big| \partial_{\tilde{\beta}_1}^{\tilde{\alpha}_1}f^{\alpha}_1 (v') \big|^2 \big| \partial_{\tilde{\beta}_2}^{\tilde{\alpha}_2}f^{\beta}_2 (v_*') \big|^2 dv_* \Bigg)^{1/2} \big| \partial_{\tilde{\beta}}^{\tilde{\alpha}}f^{\alpha}_3 (v) \big| dvdx \\
		&\hspace{0.3cm}\leq C \int_{\mathbb{T}^3} \left( \int_{|v|\leq1,|v_*|\leq \frac{1}{2}} |v|^\gamma \big| \partial_{\tilde{\beta}_1}^{\tilde{\alpha}_1}f^{\alpha}_1 (v') \big|^2 \big| \partial_{\tilde{\beta}_2}^{\tilde{\alpha}_2}f^{\beta}_2 (v_*') \big|^2  dv_* \right)^{1/2}  \left( \int_{|v|\leq1} \big| \partial_{\tilde{\beta}}^{\tilde{\alpha}}f^{\alpha}_3 (v) \big|^2 dv \right)^{1/2} dx. \tag{63}
	\end{align*}
	But from \eqref{socaled6262d},
	\begin{align*}
		&\int_{|v|\leq1,|v_*|\leq \frac{1}{2}} |v|^\gamma \big| \partial_{\tilde{\beta}_1}^{\tilde{\alpha}_1}f^{\alpha}_1 (v') \big|^2 \big| \partial_{\tilde{\beta}_2}^{\tilde{\alpha}_2}f^{\beta}_2 (v_*') \big|^2  dv_* dv \\
		&\hspace{0.4cm}\leq C \int_{|v'|\leq C,|v_*'|\leq C} \min\big\{ |v'|^\gamma, |v_*'|^\gamma \big\} \times \big| \partial_{\tilde{\beta}_1}^{\tilde{\alpha}_1}f^{\alpha}_1 (v') \big|^2 \big| \partial_{\tilde{\beta}_2}^{\tilde{\alpha}_2}f^{\beta}_2 (v_*') \big|^2 dv_*'dv'.
	\end{align*}
	Assuming  $|\tilde{\alpha}_1| + |\tilde{\beta}_1| \geq \frac{N}{2}$,  the above is bounded by 
	\begin{align*}
		&C \Big( \int_{|v_*'| \le C} |v_*'|^\gamma | \partial_{\tilde{\beta}_2}^{\tilde{\alpha}_2}f^{\beta}_2 (v_*') |^2 dv_*' \Big) \int_{|v'| \leq C} \big|\partial_{\tilde{\beta}_1}^{\tilde{\alpha}_1}f^{\alpha}_1 (v')\big|^2 dv' \\
		&\hspace{0.4cm}\leq C \Big(\sup_{|v_*'| \le C} \big| \partial_{\tilde{\beta}_2}^{\tilde{\alpha}_2}f^{\beta}_2 (v_*') \big|^2\Big) \big|w^{l}\partial_{\tilde{\beta}_1}^{\tilde{\alpha}_1}f^{\alpha}_1 (v')\big|^2_{\nu} \\
		&\hspace{0.4cm}\leq C \Bigg(\sum_{|\tilde{\alpha}_i| + |\tilde{\beta}_i| \le N} \left\| w^{|\tilde{\beta}_i|} \partial_{\tilde{\beta}_2}^{\tilde{\alpha}_2}\mathbf{f}_2 \right\|_{L^{2}_{x,v}}^2 \Bigg) \big|w^{l}\partial_{\tilde{\beta}_1}^{\tilde{\alpha}_1}\mathbf{f}_1 \big|^2_{\nu}
	\end{align*}
	Through a further integration over $x$, then this lemma can be concluded.
	\end{proof}
		
	\begin{lemma}
		Let $|\tilde{\alpha}_1|+|\tilde{\alpha}_2|=N$, $N \geq 8$. \\
	$(1)$
		Let $\chi_{1}(x, v)$ be a smooth function with compact support, then 
		\begin{align}
			&\left\| \Gamma^{\alpha \beta} \big[\partial^{\tilde{\alpha}_1}f^{\alpha}_1, \partial^{\tilde{\alpha}_2}f^{\beta}_2\big] \chi_1 \right\|_{L^{2}_{x,v}}\notag\\
			&\hspace{1.8cm} \leq C \left( \sum_{| \tilde{\alpha}_i |+| \tilde{\beta}_i | \le N} \left\| w^{| \tilde{\beta}_i |} \partial_{\tilde{\beta}_i}^{\tilde{\alpha}_i}f^{\beta}_2 \right\|_{L^{2}_{x,v}} \right) \left\| \partial^{\tilde{\alpha}_1}f^{\alpha}_1 \right\|_{\nu},\hspace{0.4cm} \text{if}\ \left| \tilde{\alpha}_1 \right| \geq \frac{N}{2};\label{soacl6464}\\
			&\left\| \Gamma^{\alpha \beta} \big[\partial^{\tilde{\alpha}_1}f^{\alpha}_1, \partial^{\tilde{\alpha}_2}f^{\beta}_2\big] \chi_1 \right\|_{L^{2}_{x,v}}\notag\\
			&\hspace{1.8cm} \leq C \left( \sum_{| \tilde{\alpha}_i |+| \tilde{\beta}_i | \leq N} \left\| w^{| \tilde{\beta}_i |} \partial_{\tilde{\beta}_i}^{\tilde{\alpha}_i} f^{\alpha}_1\right\|_{L^{2}_{x,v}} \right) \left\| \partial^{\tilde{\alpha}_2}f^{\beta}_2 \right\|_{\nu},\hspace{0.4cm} \text{if}\ \left| \tilde{\alpha}_1 \right| \leq \frac{N}{2}.\label{soacl64642}
		\end{align}
		$(2)$
		Let $\bar{\chi}_1(v)$ be a smooth function such that
		$$
		\sup\left\{ \left| \bar{\chi}_1 \right|+\left| \nabla \bar{\chi}_1 \right|+\left| \nabla^2 \bar{\chi}_1 \right| \right\} \leq C \min \{\mu^A ,\mu^B \},
		$$
		then it holds that
		\begin{align}
			&\left| \int_{\mathbb{R}^3} \Gamma^{\alpha \beta} [\partial^{\tilde{\alpha}_1}f^{\alpha}_1, \partial^{\tilde{\alpha}_2}f^{\beta}_2] \bar{\chi}_1 dv \right|_{L^{2}_{x}} \notag\\
			&\hspace{1.8cm}\leq C \left( \sum_{| \tilde{\alpha}_i |+| \tilde{\beta}_i | \le N} \left\| w^{| \tilde{\beta}_i |} \partial_{\tilde{\beta}_i}^{\tilde{\alpha}_i}f^{\beta}_2 \right\|_{L^{2}_{x,v}} \right) \left\| \partial^{\tilde{\alpha}_1}f^{\alpha}_1 \right\|_{\nu},\hspace{0.4cm} \text{if}\ \left| \tilde{\alpha}_1 \right| \geq \frac{N}{2};\label{soacl6565}\\
			&\left| \int_{\mathbb{R}^3} \Gamma^{\alpha \beta} [\partial^{\tilde{\alpha}_1}f^{\alpha}_1, \partial^{\tilde{\alpha}_2}f^{\beta}_2]\bar{\chi}_1 dv \right|_{L^{2}_{x}}\notag\\
			&\hspace{1.8cm} \leq C \left( \sum_{| \tilde{\alpha}_i |+| \tilde{\beta}_i | \le N} \left\| w^{| \tilde{\beta}_i |} \partial_{\tilde{\beta}_i}^{\tilde{\alpha}_i} f^{\alpha}_1\right\|_{L^{2}_{x,v}} \right) \left\| \partial^{\tilde{\alpha}_2}f^{\beta}_2 \right\|_{\nu},\hspace{0.4cm} \text{if}\ \left| \tilde{\alpha}_1 \right| \leq \frac{N}{2}.\label{soacl65652}
		\end{align}
	\end{lemma}	
	
	\begin{proof}
		We note that
		\begin{align*}
			\Gamma^{\alpha \beta} [\partial^{\tilde{\alpha}_1}f^{\alpha}_1, \partial^{\tilde{\alpha}_2}f^{\beta}_2]	&\equiv \int_{\mathbb{R}^3} \int_{S^2} |v_* - v|^\gamma 
			 e^{-\frac{m^{\beta}|v_*|^2}{4}}
			\partial^{\tilde{\alpha}_1}f^{\alpha}_1 (v') \partial^{\tilde{\alpha}_2}f^{\beta}_2 (v_*') 
			b^{\alpha\beta}(\theta) \, dv_* \, d\omega\\
			&\hspace{1cm} - \partial^{\tilde{\alpha}_1}f^{\alpha}_1(v) 
			\int_{\mathbb{R}^3} \int_{S^2} |u - v|^\gamma 
			e^{-\frac{m^{\beta}|v_*|^2}{4}}
			\partial^{\tilde{\alpha}_2}f^{\beta}_2 (v_*)
			b^{\alpha\beta}(\theta) \, dv_* \, d\omega \\
			&=: \Gamma_{\text{gain}}^{\alpha \beta} - \Gamma_{\text{loss}}^{\alpha \beta}.
		\end{align*}
		In order to obtain \eqref{soacl6464} and \eqref{soacl64642}, we first estimate the second loss term $ \Gamma_{\text{loss}}^{\alpha \beta} \big[\partial^{\tilde{\alpha}_1}f^{\alpha}_1, \partial^{\tilde{\alpha}_2}f^{\beta}_2\big] $.
		
		If $ |\tilde{\alpha}_1| \geq \frac{N}{2} $, by $ H^4(\mathbf{T}^3 \times \mathbf{R}^3) \subset L^\infty $ and $ N \geq 8 $,
		\begin{align*}
		&\int_{\mathbb{R}^3} |v_* - v|^\gamma e^{-\frac{m^{\beta}|v_*|^2}{4}} \big|\partial^{\tilde{\alpha}_2}f^{\beta}_2 (x,v_*)\big| dv_*\\
		&\hspace{0.5cm}\leq C \sup_{x, v_*} \left| e^{-\frac{m^{\beta}|v_*|^2}{8}} \partial^{\tilde{\alpha}_2}f^{\beta}_2 (x,v_*) \right| \big(1 + |v|\big)^\gamma\\
		&\hspace{0.5cm}\leq C  \left(\sum_{| \tilde{\alpha}_i |+| \tilde{\beta}_i | \le N} \left\| w^{| \tilde{\beta}_i |} \partial_{\tilde{\beta}_i}^{\tilde{\alpha}_i} f^{\beta}_2\right\|_{L^{2}_{x,v}}\right) \big(1 + |v|\big)^\gamma.
		\end{align*}
		Therefore, $\left\| \Gamma_{\text{loss}}^{\alpha \beta} [\partial^{\tilde{\alpha}_1}f^{\alpha}_1, \partial^{\tilde{\alpha}_2}f^{\beta}_2] \chi_1 \right\|_{L^{2}_{x,v}}$ is bounded by
		\begin{align*}
		&\left(\iint_{\mathbb{R}^3\times\mathbb{T}^3} \left| \Gamma_{\text{loss}}^{\alpha \beta} [\partial^{\tilde{\alpha}_1}f^{\alpha}_1, \partial^{\tilde{\alpha}_2}f^{\beta}_2] \chi_1 \right|^2 dv dx\right)^{1/2}\\
		&\hspace{0.5cm}\leq C \left(\sum_{| \tilde{\alpha}_i |+| \tilde{\beta}_i | \le N} \left\| w^{| \tilde{\beta}_i |} \partial_{\tilde{\beta}_i}^{\tilde{\alpha}_i} f^{\beta}_2\right\|_{L^{2}_{x,v}}\right) \iint_{\mathbb{R}^3\times\mathbb{T}^3} (1 + |v|)^{\gamma} \chi_1^2 |\partial^{\tilde{\alpha}_1}f^{\alpha}_1|^2 dv dx\\
		&\hspace{0.5cm}\leq C \left(\sum_{| \tilde{\alpha}_i |+| \tilde{\beta}_i | \le N} \left\| w^{| \tilde{\beta}_i |} \partial_{\tilde{\beta}_i}^{\tilde{\alpha}_i} f^{\beta}_2\right\|_{L^{2}_{x,v}}\right) \left\| \partial^{\tilde{\alpha}_1}f^{\alpha}_1 \right\|_\nu.
		\end{align*}
		
		This concludes the case of $ |\tilde{\alpha}_1| \geq \frac{N}{2} $. If $ |\tilde{\alpha}_1| \leq \frac{N}{2} $, then the Cauchy-Schwarz inequality leads to
		\begin{align*}
		&\int_{\mathbb{R}^3} |v_* - v|^\gamma e^{-\frac{m^{\beta}|v_*|^2}{4}} |\partial^{\tilde{\alpha}_2}f^{\beta}_2 (x,v_*)| dv_*\\
		&\hspace{0.5cm}\leq \left( \int_{\mathbb{R}^3} |v_* - v|^\gamma e^{-\frac{m^{\beta}|v_*|^2}{4}} |\partial^{\tilde{\alpha}_2}f^{\beta}_2 (x,v_*)|^2 dv_* \right)^{1/2} (1 + |v|)^{\frac{\gamma}{2}}.
		\end{align*}
		Since \( \gamma > -3 \), and \( \chi_1 \) has compact support,
		\begin{equation*}
			\int_{\mathbb{R}^3}  |v_* - v|^\gamma (1 + |v|)^\gamma \chi_1 dv < \infty.
		\end{equation*}	
		Then we get 
		\begin{align*}
		&\iint_{\mathbb{R}^3\times\mathbb{T}^3}  \left| \Gamma_{\text{loss}}^{\alpha \beta} [\partial^{\tilde{\alpha}_1}f^{\alpha}_1, \partial^{\tilde{\alpha}_2}f^{\beta}_2] \chi_1 \right|^2  dv dx \\
		&\hspace{0.5cm}\leq C \iiint_{(\mathbb{R}^3)^2\times\mathbb{T}^3} |v_* - v|^\gamma e^{-\frac{m^{\beta}|v_*|^2}{4}} |\partial^{\tilde{\alpha}_1}f^{\alpha}_1(x,v)|^2 |\partial^{\tilde{\alpha}_2}f^{\beta}_2 (x,v_*)|^2  (1 + |v|)^\gamma \chi_1^2 dv_* dv dx,
		\end{align*}
		and it can be further bounded by
		\begin{align*}
		&\iint_{\mathbb{R}^3\times\mathbb{T}^3} e^{-\frac{m^{\beta}|v_*|^2}{4}} |\partial^{\tilde{\alpha}_2}f^{\beta}_2 (x,v_*)|^2\\
		&\hspace{0.8cm}\times \sup_{x,v} \Big\{ |\partial^{\tilde{\alpha}_1}f^{\alpha}_1(x,v)|^2 \chi_1(x, v)\Big\} \left( \int_{\mathbb{R}^3} |v_* - v|^\gamma (1 + |v|)^\gamma \chi_1 dv \right) dv_* dx\\
		&\hspace{0.5cm}\leq C  \left( \sum_{| \tilde{\alpha}_i |+| \tilde{\beta}_i | \leq N} \left\| w^{| \tilde{\beta}_i |} \partial_{\tilde{\beta}_i}^{\tilde{\alpha}_i} f^{\alpha}_1\right\|_{L^{2}_{x,v}} \right) \|\partial^{\tilde{\alpha}_2}f^{\beta}_2\|_{\nu}^2.
		\end{align*}
		The next step is to estimate the gain term  $\Gamma_{\text{gain}}^{\alpha \beta} [\partial^{\tilde{\alpha}_1}f^{\alpha}_1, \partial^{\tilde{\alpha}_2}f^{\beta}_2]$ in \eqref{soacl6464} and \eqref{soacl64642}. We notice the condition that $\chi_1(x,v)$ has a compact support, so one gets $\chi_1(x,v)\leq C\min \{\mu^A ,\mu^B \}$, and then

		\begin{align}
			&\iint_{\mathbb{R}^3\times\mathbb{T}^3}  \left| \Gamma_{\text{gain}}^{\alpha \beta} [\partial^{\tilde{\alpha}_1}f^{\alpha}_1, \partial^{\tilde{\alpha}_2}f^{\beta}_2] \chi_1 \right|^2  dv dx\notag\\
			&\hspace{0.5cm}\leq C\iiint_{(\mathbb{R}^3)^2\times\mathbb{T}^3} |u-v|^\gamma  e^{-\frac{m^{\beta}|v_*|^2}{4}}|\partial^{\tilde{\alpha}_1}f^{\alpha}_1(x,v')|^2 |\partial^{\tilde{\alpha}_2}f^{\beta}_2 (x,v_*')|^2 \chi_1^2(v) dv_*dv dx\notag\\
			&\hspace{8cm}\times\left(\int_{\mathbb{R}^3}|v_*-v|^\gamma e^{-\frac{m^{\beta}|v_*|^2}{4}} dv_*\right)\notag\\
			&\hspace{0.5cm}\leq C\iiint_{(\mathbb{R}^3)^2\times\mathbb{T}^3}|v_*'-v'|^\gamma e^{-\frac{m^{\beta}|v'|^2}{4}}e^{-\frac{m^{\beta}|v_*'|^2}{4}}|\partial^{\tilde{\alpha}_1}f^{\alpha}_1(x,v')|^2 |\partial^{\tilde{\alpha}_2}f^{\beta}_2 (x,v_*')|^2dv_*'dv'dx.\label{socal6666}
		\end{align}
		Based on the symmetry between $\partial^{\tilde{\alpha}_1}f^{\alpha}_1$ and $\partial^{\tilde{\alpha}_2}f^{\beta}_2$, without loss of generality, we assume $\left|\tilde{\alpha}_1 \right| \geq \frac{N}{2}$, so it yields that
		\begin{equation*}
		\sup_{x,v_*'}\Big\{e^{-\frac{m^{\beta}|v_*'|^2}{4}}|\partial^{\tilde{\alpha}_2}f^{\beta}_2 (x,v_*')|^2\Big\}\leq C \sum_{| \tilde{\alpha}_i |+| \tilde{\beta}_i | \le N} \left\| w^{| \tilde{\beta}_i |} \partial_{\tilde{\beta}_i}^{\tilde{\alpha}_i}f^{\beta}_2 \right\|_{L^{2}_{x,v}}^2.
		\end{equation*}
		Together with the fact $\int_{\mathbb{R}^3}|v_*'-v'|^\gamma e^{-\frac{m^{\beta}|v_*|^2}{4}} dv_*'<\infty$, $\gamma>-3$, we can bound  \eqref{socal6666} by
		
		\begin{equation}
		\begin{aligned}
			&\iint_{(\mathbb{R}^3)^2}\left(\int|v_*'-v'|^\gamma e^{-\frac{m^{\beta}|v_*'|^2}{4}}|\partial^{\tilde{\alpha}_2}f^{\beta}_2 (x,v_*')|^2dv_*'\right)e^{-\frac{m^{\beta}|v'|^2}{4}}|\partial^{\tilde{\alpha}_1}f^{\alpha}_1(x,v')|^2dv'dx\\
			&\hspace{0.5cm}\leq C\left( \sum_{| \tilde{\alpha}_i |+| \tilde{\beta}_i | \le N} \left\| w^{| \tilde{\beta}_i |} \partial_{\tilde{\beta}_i}^{\tilde{\alpha}_i}f^{\beta}_2 \right\|_{L^{2}_{x,v}} \right)^2\big\|\partial^{\tilde{\alpha}_1}f^{\alpha}_1\big\|^2_{\nu}.\label{socal666777}
		\end{aligned}
		\end{equation}
		This completes the proof of \eqref{soacl6464} and \eqref{soacl64642}.
		
		In order to verify \eqref{soacl6565} and \eqref{soacl65652}, since \( \bar{\chi}_1(v) \) exhibits exponential decay, \( \Gamma_{\text{loss}}^{\alpha \beta} [\partial^{\tilde{\alpha}_1}f^{\alpha}_1, \partial^{\tilde{\alpha}_2}f^{\beta}_2] \) can be estimated by
		\begin{align*}
		&\int_{\mathbb{R}^3} \Gamma_{\text{loss}}^{\alpha \beta} \big[\partial^{\tilde{\alpha}_1}f^{\alpha}_1, \partial^{\tilde{\alpha}_2}f^{\beta}_2\big] \bar{\chi}_1(v) dv\\
			&\hspace{0.5cm}\le C \int_{\mathbb{R}^3} \big|\partial^{\tilde{\alpha}_1}f^{\alpha}_1\bar{\chi}_1(v)\big| \int_{\mathbb{R}^3} |v_* - v|^\gamma e^{-\frac{m^{\beta}|v_*|^2}{4}} \big|\partial^{\tilde{\alpha}_2}f^{\beta}_2(v_*)\big|dv_*dv\\
			&\hspace{0.5cm}\le C \int_{\mathbb{R}^3} \big|\partial^{\tilde{\alpha}_1}f^{\alpha}_1\bar{\chi}_1(v)\big| \left( \int_{\mathbb{R}^3} |v_* - v|^\gamma e^{-\frac{m^{\beta}|v_*|^2}{4}} \big|\partial^{\tilde{\alpha}_2}f^{\beta}_2(v_*)\big|^2 dv_* \right)^{1/2} \\
			&\hspace{5.6cm}\times \left( \int_{\mathbb{R}^3} |v_* - v|^\gamma e^{-\frac{m^{\beta}|v_*|^2}{4}} dv_* \right)^{1/2}dv\\
			&\hspace{0.5cm}\le C \left( \iint_{(\mathbb{R}^3)^2} |v_* - v|^\gamma e^{-\frac{m^{\beta}|v_*|^2}{4}} \big|\partial^{\tilde{\alpha}_1}f^{\alpha}_1(v)\big|^2 \big|\bar{\chi}_1(v)\big|\big|\partial^{\tilde{\alpha}_2}f^{\beta}_2(v_*)\big|^2 dv_*dv \right)^{1/2},
		\end{align*}
		where we have used the Cauchy-Schwarz inequality for the \( v \) integration. We notice the symmetry between $\partial^{\tilde{\alpha}_1}f^{\alpha}_1$ and $\partial^{\tilde{\alpha}_2}f^{\beta}_2$. Without loss of generality, we assume $\left|\tilde{\alpha}_1 \right| \geq \frac{N}{2}$ , which gives
		\[
		\sup_{x,v_*} \left\{ e^{-\frac{m^{\beta}|v_*|^2}{8}} |\partial^{\tilde{\alpha}_2}f^{\beta}_2(v_*)|^2 \right\} \le C \sum_{| \tilde{\alpha}_i |+| \tilde{\beta}_i | \le N} \big\| w^{| \tilde{\beta}_i |} \partial_{\tilde{\beta}_i}^{\tilde{\alpha}_i}f^{\beta}_2 \big\|^2_{L^{2}_{x,v}}
		\]
		Since \( 0>\gamma > -3 \), we first integrate over \( u \), and then take the \( L^2 \) norm in the spatial variable \( x \) above concludes \eqref{soacl6565} and \eqref{soacl65652} for the loss term.
		
		To prove \eqref{soacl6565} and \eqref{soacl65652} for the gain term, since 
		$$
		\sup_{v}\left\{ \left| \bar{\chi}_1 \right|+\left| \nabla \bar{\chi}_1 \right|+\left| \nabla^2 \bar{\chi}_1 \right| \right\} \leq C \min \big\{\mu^A ,\mu^B \big\},
		$$
		we obtain
		\begin{align*}
		&\int_{\mathbb{R}^3} \Gamma_{\text{gain}}^{\alpha \beta} [\partial^{\tilde{\alpha}_1}f^{\alpha}_1, \partial^{\tilde{\alpha}_2}f^{\beta}_2] \bar{\chi}_1(v) dv\\
		&\hspace{0.5cm}\leq \iint_{(\mathbb{R}^3)^2} |v_* - v|^\gamma e^{-\frac{m^{\beta}|v_*|^2}{32}}|
		\partial^{\tilde{\alpha}_1}f^{\alpha}_1 (v') \partial^{\tilde{\alpha}_2}f^{\beta}_2 (v_*') \bar{\chi}_1(v)|\,dv_* dv  \\
		&\hspace{0.5cm}\le \int_{\mathbb{R}^3}\left(\int_{\mathbb{R}^3} |v_* - v|^\gamma e^{-\frac{m^{\beta}|v_*|^2}{32}}|\partial^{\tilde{\alpha}_1}f^{\alpha}_1 (v')|^2 |\partial^{\tilde{\alpha}_2}f^{\beta}_2 (v_*')|^2\,dv_*\right)^{1/2}\\
		&\hspace{5.3cm}\times\left(\int_{\mathbb{R}^3} |v_* - v|^\gamma e^{-\frac{m^{\beta}|v_*|^2}{32}}\,dv_*\right)^{1/2}|\bar{\chi}_1(v)|\,d v \\
		&\hspace{0.5cm}\le C\left(\iint_{(\mathbb{R}^3)^2} |v_* - v|^\gamma e^{-\frac{m^{\beta}|v_*|^2}{32}}|\partial^{\tilde{\alpha}_1}f^{\alpha}_1 (v')|^2 |\partial^{\tilde{\alpha}_2}f^{\beta}_2 (v_*')|^2|\chi_1|\,dv_* dv \right)^{1/2}\\
		&\hspace{0.5cm}\le C\left(\iint_{(\mathbb{R}^3)^2} |v_*' - v'|^\gamma e^{-\frac{m^{\beta}|v'|^2}{32}}e^{-\frac{m^{\beta}|v_*'|^2}{32}}|\partial^{\tilde{\alpha}_1}f^{\alpha}_1 (v')|^2 |\partial^{\tilde{\alpha}_2}f^{\beta}_2 (v_*')|^2\,dv_*' dv' \right)^{1/2},
		\end{align*}
		where we have applied the Cauchy-Schwarz inequality for \( v_* \). By a similar argument as  \eqref{socal666777}, it is bounded by
		\[
		C \left( \sum_{| \tilde{\alpha}_i |+| \tilde{\beta}_i | \le N} \big\| w^{| \tilde{\beta}_i |} \partial_{\tilde{\beta}_i}^{\tilde{\alpha}_i}f^{\beta}_2 \big\|_{L^{2}_{x,v}} \right) \big| \partial^{\tilde{\alpha}_1}f^{\alpha}_1 \big|_{\nu}.
		\]
		By a further integration in the spatial variable \( x \), we can conclude \eqref{soacl6565}. We notice the symmetry between $\partial^{\tilde{\alpha}_1}f^{\alpha}_1$ and $\partial^{\tilde{\alpha}_2}f^{\beta}_2$, the proof of case $\left|\tilde{\alpha}_1 \right| \leq \frac{N}{2}$ will be the same. Therefore, we complete the proof of this lemma.
	\end{proof}

		\begin{lemma}\label{lem333scd6}
			Let $|\tilde{\alpha}_1| + |\tilde{\alpha}_2| = N$, $N\geq 8$, and \(\chi_1(x, v)\) be a smooth function with  compact support. For any \(\eta > 0\), then it holds \\
			$(1)$
			If $\left|\tilde{\alpha}_1 \right| \geq \frac{N}{2}$, there exists \(f^{\alpha}_{1,\tilde{\alpha}_1}(x,v_*)\) such that \(\|f^{\alpha}_{1,\tilde{\alpha}_1}\| \leq C_\eta\|\partial^{\tilde{\alpha}_1}f^{\alpha}_1\|_{\nu}\), and
			\begin{align}
			&\iint_{\mathbb{R}^3\times\mathbb{T}^3} \Gamma^{\alpha \beta} [\partial^{\tilde{\alpha}_1}f^{\alpha}_1, \partial^{\tilde{\alpha}_2}f^{\beta}_2] \chi_1 dv dx\label{lem33F1}\\
			&\hspace{0.5cm}= \iint_{\mathbb{R}^3\times\mathbb{T}^3} f^{\alpha}_{1,\tilde{\alpha}_1} \partial^{\tilde{\alpha}_2}f^{\beta}_2\, dv_*dx + O(\eta^{3+\gamma}) \left( \sum_{| \tilde{\alpha}_i |+| \tilde{\beta}_i | \le N} \big\| w^{| \tilde{\beta}_i |} \partial_{\tilde{\beta}_i}^{\tilde{\alpha}_i}f^{\beta}_2 \big\|_{L^{2}_{x,v}} \right) \left\| \partial^{\tilde{\alpha}_1}f^{\alpha}_1 \right\|_{\nu}.\notag
			\end{align}
			$(2)$
			If \(\left|\tilde{\alpha}_1 \right| \leq \frac{N}{2}\), there exists \(f^{\beta}_{2,\tilde{\alpha}_2}(x,v)\) such that \(\|f^{\beta}_{2,\tilde{\alpha}_2}\| \leq C_\eta\|\partial^{\tilde{\alpha}_2}f^{\beta}_2\|_{\nu}\), and
			\begin{align}
			&\iint_{\mathbb{R}^3\times\mathbb{T}^3} \Gamma^{\alpha \beta} [\partial^{\tilde{\alpha}_1}f^{\alpha}_1, \partial^{\tilde{\alpha}_2}f^{\beta}_2] \chi_1 dv dx \label{lem33S2}\\
			&\hspace{0.5cm}= \iint_{\mathbb{R}^3\times\mathbb{T}^3} f^{\beta}_{2,\tilde{\alpha}_2} \partial^{\tilde{\alpha}_1}f^{\alpha}_1 \, dvdx + O(\eta^{3+\gamma}) \left( \sum_{| \tilde{\alpha}_i |+| \tilde{\beta}_i | \leq N} \big\| w^{| \tilde{\beta}_i |} \partial_{\tilde{\beta}_i}^{\tilde{\alpha}_i} f^{\alpha}_1\big\|_{L^{2}_{x,v}} \right) \big\| \partial^{\tilde{\alpha}_2}f^{\beta}_2 \big\|_{\nu}.\notag
			\end{align}
		\end{lemma}	
		\begin{proof}
			To verify \eqref{lem33F1}, we apply changes of variables $(v_*, v) \to (v_*', v') \to (v_*, v)$ for the gain term to obtain
			\begin{align}
			&\iint_{\mathbb{R}^3\times\mathbb{T}^3} \Gamma^{\alpha \beta} [\partial^{\tilde{\alpha}_1}f^{\alpha}_1, \partial^{\tilde{\alpha}_2}f^{\beta}_2] \chi_1(x, v) \, dvdx\notag\\
			&\hspace{0.3cm}= \iiiint_{\mathbb{S}^2\times(\mathbb{R}^3)^2\times\mathbb{T}^3} b^{\alpha \beta}(\theta)|v_*' - v'|^\gamma e^{-\frac{m^{\beta}|v_*|^2}{4}} \chi_1(x, v) \partial^{\tilde{\alpha}_1}f^{\alpha}_1(v') \partial^{\tilde{\alpha}_2}f^{\beta}_2(v_*') d \omega dv_* dv dx\notag\\
			&\hspace{1cm}- \iiiint_{\mathbb{S}^2\times(\mathbb{R}^3)^2\times\mathbb{T}^3} b^{\alpha \beta}(\theta)|v_* - v|^\gamma e^{-\frac{m^{\beta}|v_*|^2}{4}} \chi_1(x, v) \partial^{\tilde{\alpha}_1}f^{\alpha}_1(v) \partial^{\tilde{\alpha}_2}f^{\beta}_2(v_*)d \omega dv_* dv dx\notag\\
			&\hspace{0.3cm}= \iiiint_{\mathbb{S}^2\times(\mathbb{R}^3)^2\times\mathbb{T}^3} b^{\alpha \beta}(\theta)|v_* - v|^\gamma \left[ e^{-\frac{m^{\beta}|v_*'|^2}{4}} \chi_1(x, v') - e^{-\frac{m^{\beta}|v_*|^2}{4}} \chi_1(x, v) \right]\label{Socal312ORI}\\
			& \hspace{8cm}\times\partial^{\tilde{\alpha}_1}f^{\alpha}_1(v) \partial^{\tilde{\alpha}_2}f^{\beta}_2(v_*)d \omega dv_* dv dx. \notag
			\end{align}
			If $\left|\tilde{\alpha}_1 \right| \geq \frac{N}{2}$, by the cutoff function $\chi$ in \eqref{cutoffuctiondef} with $\varepsilon = \eta$, we define
			\begin{align}
			f^{\alpha}_{1,\tilde{\alpha}_1}(x,v_*) =& \iint_{\mathbb{S}^2\times\mathbb{R}^3} b^{\alpha \beta}(\theta)|v_* - v|^\gamma \chi(|v_* - v|) \partial^{\tilde{\alpha}_1}f^{\alpha}_1(v)\label{NefAalph1}\\
			&\hspace{0.4cm}\times \left[ e^{-\frac{m^{\beta}|v_*'|^2}{4}} \chi_1(x, v') - e^{-\frac{m^{\beta}|v_*|^2}{4}} \chi_1(x, v) \right] d\omega dv.\notag
			\end{align}
			Using \eqref{F1vuvusacl},
			\begin{equation*}
				\left| e^{-\frac{m^{\beta}|v_*'|^2}{4}} \chi_1(x, v') \right| \le C e^{-\frac{m^{\beta}|v_*'|^2}{4}}e^{-\frac{m^{\beta}|v'|^2}{4}} = C e^{-\frac{m^{\beta}(|v_*|^2+|v|^2)}{4}}.
			\end{equation*}
			Since $\chi$ vanishes near the origin,
			\begin{equation*}
				|f^{\alpha}_{1,\tilde{\alpha}_1}(x,v_*)| \le C_\eta e^{-\frac{m^{\beta}|v_*|^2}{8}} |\partial^{\tilde{\alpha}_1}f^{\alpha}_1|_{\nu} .
			\end{equation*}
			So, it yields 
			\begin{equation}\label{socal313behNE}
				\|f^{\alpha}_{1,\tilde{\alpha}_1}\| \le C_\eta \| \partial^{\tilde{\alpha}_1}f^{\alpha}_1\|_{\nu}.
			\end{equation}
			Notice by \eqref{Socal312ORI}, \eqref{NefAalph1} and \eqref{socal313behNE},
			\begin{equation*}
				\Big|\iint_{\mathbb{R}^3\times\mathbb{T}^3} \Gamma^{\alpha \beta} [\partial^{\tilde{\alpha}_1}f^{\alpha}_1, \partial^{\tilde{\alpha}_2}f^{\beta}_2] \chi_1 dv dx -\iint_{\mathbb{R}^3\times\mathbb{T}^3} f^{\alpha}_{1,\tilde{\alpha}_1} \partial^{\tilde{\alpha}_2}f^{\beta}_2\, dv_*dx\Big|
			\end{equation*}
			is bounded by
			\begin{align*}
			&C \iiint_{(\mathbb{R}^3)^2\times\mathbb{T}^3} |v_* - v|^\gamma \big[1 - \chi(|v_* - v|)\big] \big|\partial^{\tilde{\alpha}_1}f^{\alpha}_1(v) \partial^{\tilde{\alpha}_2}f^{\beta}_2(v_*)\big|e^{-\frac{m^{\beta}|v_*|^2}{4}}e^{-\frac{m^{\beta}|v|^2}{4}} dv_* dv dx\\
			&\hspace{0.4cm}\leq \sup_{x,v_*}\big\{e^{-\frac{m^{\beta}|v_*|^2}{16}}|\partial^{\tilde{\alpha}_2}f^{\beta}_2(v_*)|\big\} \int_{|v_* - v| \leq 2\eta} |v_*- v|^\gamma e^{-\frac{m^{\beta}|v_*|^2}{16}} dv_*\\
			&\hspace{7.3cm}\times \iint_{\mathbb{R}^3\times\mathbb{T}^3} e^{-\frac{m^{\beta}|v|^2}{8}}|\partial^{\tilde{\alpha}_1}f^{\alpha}_1(v)| dv dx\\
			&\hspace{0.4cm}\leq C\eta^{3 + \gamma} \left( \sum_{| \tilde{\alpha}_i |+| \tilde{\beta}_i | \le N} \big\| w^{| \tilde{\beta}_i |} \partial_{\tilde{\beta}_i}^{\tilde{\alpha}_i}f^{\beta}_2 \big\|_{L^{2}_{x,v}} \right) \big\| \partial^{\tilde{\alpha}_1}f^{\alpha}_1 \big\|_{\nu}.
			\end{align*}
			Thus we have proved \eqref{lem33F1}.
			
			Next, we turn to prove \eqref{lem33S2}, with $\left|\tilde{\alpha}_1 \right| \leq \frac{N}{2}$. By the same changes of variables, we define
			\begin{align}
				f^{\beta}_{2,\tilde{\alpha}_2}(x,v) =& \iint_{\mathbb{S}^2\times\mathbb{R}^3} b^{\alpha \beta}(\theta)|v_* - v|^\gamma \chi(|v_* - v|) \partial^{\tilde{\alpha}_2}f^{\beta}_2(v_*)\label{simalrTis313}\\
				&\hspace{0.4cm}\times \left( e^{-\frac{m^{\beta}|v_*'|^2}{4}} \chi_1(x, v') - e^{-\frac{m^{\beta}|v_*|^2}{4}} \chi_1(x, v) \right) d\omega dv_*.\notag
			\end{align}
			It satisfies
			\begin{equation}\label{estFsl315}
				|f^{\beta}_{2,\tilde{\alpha}_2}(x,v)| \le C_\eta e^{-\frac{m^{\beta}|v_*|^2}{8}} |\partial^{\tilde{\alpha}_2}f^{\beta}_2|_{\nu}.
			\end{equation}
			By \eqref{simalrTis313} and \eqref{estFsl315},
			\begin{equation*}
				\Big|\iint_{\mathbb{R}^3\times\mathbb{T}^3} \Gamma^{\alpha \beta} [\partial^{\tilde{\alpha}_1}f^{\alpha}_1, \partial^{\tilde{\alpha}_2}f^{\beta}_2] \chi_1 dv dx - \iint_{\mathbb{R}^3\times\mathbb{T}^3} f^{\beta}_{2,\tilde{\alpha}_2} \partial^{\tilde{\alpha}_1}f^{\alpha}_1\, dvdx\Big|
			\end{equation*}
			is bounded by
			\begin{align*}
				&C \iiint_{(\mathbb{R}^3)^2\times\mathbb{T}^3} |v_* - v|^\gamma \big[1 - \chi(|v_* - v|)\big]\Big|\partial^{\tilde{\alpha}_1}f^{\alpha}_1(v) \partial^{\tilde{\alpha}_2}f^{\beta}_2(v_*)\Big|\times \Big|e^{-\frac{m^{\beta}|v_*'|^2}{4}} \chi_1(x, v')\\
				&\hspace{8.5cm} - e^{-\frac{m^{\beta}|v_*|^2}{4}} \chi_1(x, v)\Big|dv_* dv dx\\
				&\hspace{0.4cm}\leq \sup_{\substack{x, v}} \big\{e^{-\frac{m^{\beta}|v|^2}{16}} |\partial^{\tilde{\alpha}_1}f^{\alpha}_1(v)|\big\} \int_{|v_* - v| \geq 2\eta} |v_* - v|^\gamma e^{-\frac{m^{\beta}|v|^2}{16}} dv\\
				&\hspace{7.3cm} \times \iint_{\mathbb{R}^3\times\mathbb{T}^3} e^{-\frac{m^{\beta}|v_*|^2}{8}} |\partial^{\tilde{\alpha}_2}f^{\beta}_2(v_*)| dv_* dx\\
				&\hspace{0.4cm}\le C \eta^{3 + \gamma} \Bigg( \sum_{| \tilde{\alpha}_i |+| \tilde{\beta}_i | \leq N} \big\| w^{| \tilde{\beta}_i |} \partial_{\tilde{\beta}_i}^{\tilde{\alpha}_i} f^{\alpha}_1\big\|_{L^{2}_{x,v}} \Bigg) \big\| \partial^{\tilde{\alpha}_2}f^{\beta}_2 \big\|_{\nu}.
			\end{align*}
			This concludes \eqref{lem33S2}. Therefore, we have proved this lemma.
		\end{proof}
	
	\section{Energy estimates}
	This section establishes a time-local solution for the Boltzmann equation. Our approach, therefore, relies on a uniform energy estimate for the following sequence of iterating approximate solutions.  
	
	\begin{align}
		&\left(\partial_t + v \cdot \nabla_x\right) F_{n+1}^{\alpha} + \sum_{\beta=A,B}F^{\alpha}_{n+1} \iint_{\mathbb{S}^2\times\mathbb{R}^3}b^{\alpha \beta}(\theta) |v-v_*|^{\gamma} F^{\beta}_n(v_*)  d\omega dv_* \notag \\
		&\quad = \sum_{\beta=A,B}\iint_{\mathbb{S}^2\times\mathbb{R}^3}b^{\alpha \beta}(\theta) |v-v_*|^{\gamma} F^{\alpha}_n(v') F^{\beta}_n(v'_*)d\omega dv_*, \label{socald70F}
	\end{align}
	with initial data 
	\begin{equation*}
	\mathbf{F}_{n+1}(0, x, v) = \mathbf{F}_{0}(x, v).
	\end{equation*}
	 Since \(  \mathbf{F}_{n+1} = \bm{\mu} + (\sqrt{\bm{\mu}}:  \mathbf{f}_{n+1}) \), and \( \mathbf{f}_{n+1}=(f_{n+1}^A,f_{n+1}^B)^T \) satisfies
	\begin{align}
		&\left(\partial_t + v \cdot \nabla_x + \nu^{\alpha}\right) f_{n+1}^{\alpha} + K^{\alpha} \mathbf{f}_{n} = \sum_{\beta=A,B} \Big(\Gamma^{\alpha\beta}_{\text{gain}}[f_n^{\alpha}, f_n^{\beta}] - \Gamma^{\alpha\beta}_{\text{loss}}[f^{\alpha}_n, f^{\beta}_{n+1}]\Big), \label{socal71f} \\
		&\hspace{3.4cm} f_{n+1}^{\alpha}(0, x, v) = f_{0}^{\alpha}(x, v). \notag
	\end{align}
	The most important thing is to get a uniform-in-$n$ estimate for the energy \( \mathcal{E}[\mathbf{f}_{n+1}(t)] \).

	\begin{lemma}\label{socedlem7}
		The sequence of approximate solutions $\{\mathbf{F}_{n}(t,x,v)\}$ is non-negative. For  sufficiently small ${\mathcal{E}}(\mathbf{f}_0)$, there exists $T^*({\mathcal{E}}(\mathbf{f}_0)) > 0$, such that 
		\begin{equation}\label{socald72}
			\sup_{\substack{0 \leq t \leq T^*,\ 1 \leq k \leq \infty}} {\mathcal{E}}[\mathbf{f}_k(t)] \leq C{\mathcal{E}}(\mathbf{f}_0).
		\end{equation}
	\end{lemma}
	
	\begin{proof}
		We prove this by induction over k. ‌It is straightforward to observe that \eqref{socald72} satisfies the case of $k=0$. We assume \eqref{socald72} holds for $k=n$. The linear equation \eqref{socald70F} admits a solution $\mathbf{F}_{n+1}$ given any $\mathbf{F}_{n} \ge 0$. Furthermore, the non-negativity of $F_0(x, v)$ ensures $\mathbf{F}_{n+1}(t, x, v) \ge 0$.
		
		To prove \eqref{socald72} for $k = n+1$, we need to estimate the mixed $x$ and $v$ derivatives of $f^{n+1}$. Taking $\partial_{\tilde{\beta}}^{\tilde{\alpha}}$ ($\tilde{\beta} \ne 0$) of \eqref{socal71f}, we obtain
		\[
		\begin{aligned}
			&[\partial_t + v \cdot \nabla_x]\partial_{\tilde{\beta}}^{\tilde{\alpha}} f_{n+1}^{\alpha} + \partial_{\tilde{\beta}}[\nu^{\alpha}\partial^{\tilde{\alpha}} f_{n+1}^{\alpha}] + \partial_{\tilde{\beta}}[K^{\alpha}\partial^{\tilde{\alpha}} \mathbf{f}_n] + C^{\tilde{\beta}}_{\tilde{\beta}_1}\partial_{\tilde{\beta}_1}v_j\partial^j \partial^{\tilde{\alpha}}_{\tilde{\beta}-\tilde{\beta}_1} f_{n+1}^{\alpha}\\
			&\hspace{0.7cm}=\sum_{\beta=A,B}\Big( \partial_{\tilde{\beta}}^{\tilde{\alpha}}\Gamma^{\alpha\beta}_{\text{gain}}[f_n^{\alpha}, f_n^{\beta}] - \partial_{\tilde{\beta}}^{\tilde{\alpha}}\Gamma^{\alpha\beta}_{\text{loss}}[f_{n+1}^{\alpha}, f_{n}^{\beta}]\Big),
		\end{aligned}
		\]
		where \(|\tilde{\beta}| = 1\). We multiply the above equation with \(w^{2|\tilde{\beta}|}\partial_{\tilde{\beta}}^{\tilde{\alpha}} f^{\alpha}_{n+1}\), integrating it over \(\mathbb{T}^3 \times \mathbb{R}^3\)  and then sum over $\alpha=A,B$ to obtain
		\begin{align}
			&\frac{1}{2}\frac{d}{dt}\big\|w^{|\tilde{\beta}|}\partial_{\tilde{\beta}}^{\tilde{\alpha}} \mathbf{f}_{n+1}\big\|^2_{L^{2}_{x,v}} + \sum_{\alpha=A,B}\left\langle w^{2|\tilde{\beta}|}\partial_{\tilde{\beta}} (\nu^{\alpha}\partial^{\tilde{\alpha}} f_{n+1}^{\alpha}), \partial_{\tilde{\beta}}^{\tilde{\alpha}} f_{n+1}^{\alpha} \right\rangle_{L^{2}_{x,v}} \notag\\
			&\hspace{1cm}+ \left\langle w^{2|\tilde{\beta}|}\partial_{\tilde{\beta}} (\mathbf{K}\partial^{\tilde{\alpha}} \mathbf{f}_n), \partial_{\tilde{\beta}}^{\tilde{\alpha}} \mathbf{f}_{n+1}\right\rangle_{L^{2}_{x,v}} + \left\langle w^{2|\tilde{\beta}|}C^{\tilde{\beta}_1}(\partial_{\tilde{\beta}_1}v_j)\partial^j \partial^{\tilde{\alpha}}_{\tilde{\beta}-\tilde{\beta}_1} \mathbf{f}_{n+1}, \partial_{\tilde{\beta}}^{\tilde{\alpha}} \mathbf{f}_{n+1}\right\rangle_{L^{2}_{x,v}} \notag\\
			&\hspace{0.5cm}= \sum_{\alpha=A,B}\sum_{\beta=A,B}\left\langle w^{2|\tilde{\beta}|}\Big( \partial_{\tilde{\beta}}^{\tilde{\alpha}}\Gamma^{\alpha\beta}_{\text{gain}}[f_n^{\alpha}, f_n^{\beta}] - \partial_{\tilde{\beta}}^{\tilde{\alpha}}\Gamma^{\alpha\beta}_{\text{loss}}[f_{n+1}^{\alpha}, f_{n}^{\beta}]\Big), \partial_{\tilde{\beta}}^{\tilde{\alpha}} f^{\alpha}_{n+1}\right\rangle_{L^{2}_{x,v}}.\label{socald73}
		\end{align}

		‌Next, we estimate the terms in \eqref{socald73} term by term. For any \(\eta > 0\), using Lemma \ref{222maxdel}, then integrating over \(\mathbb{T}^3\) yields
		\begin{align*}
			&\sum_{\alpha=A,B}\left\langle w^{2|\tilde{\beta}|}\partial_{\tilde{\beta}} (\nu^{\alpha}\partial^{\tilde{\alpha}} f_{n+1}^{\alpha}), \partial_{\tilde{\beta}}^{\tilde{\alpha}} f_{n+1}^{\alpha} \right\rangle_{L^{2}_{x,v}} \\
			&\hspace{3cm}\geq \big\|w^{|\tilde{\beta}|}\partial_{\tilde{\beta}}^{\tilde{\alpha}} \mathbf{f}_{n+1}\big\|^2_{\nu} - \eta \sum_{|\tilde{\beta}_1| \leq |\tilde{\beta}|} \big\|w^{|\tilde{\beta}_1|}\partial_{\tilde{\beta}_1}^{\tilde{\alpha}} \mathbf{f}_{n+1}\big\|^2_{\nu} - C_\eta \big\|\partial^{\tilde{\alpha}} \mathbf{f}_{n+1}\big\|^2_{\nu},
		\end{align*}
		and
		\begin{align*}
		&\left\langle w^{2|\tilde{\beta}|}\partial_{\tilde{\beta}} (\mathbf{K}\partial^{\tilde{\alpha}} \mathbf{f}_n), \partial_{\tilde{\beta}}^{\tilde{\alpha}} \mathbf{f}_{n+1}\right\rangle_{L^{2}_{x,v}} \le \left( \eta\sum_{|\tilde{\beta}_1| \le |\tilde{\beta}|} \big\|w^{|\tilde{\beta}_1|}\partial_{\tilde{\beta}_1}^{\tilde{\alpha}} \mathbf{f}_{n}\big\|_{\nu} + C_\eta \big\|\partial^\alpha \mathbf{f}_{n}\big\|_{\nu} \right)\\ 
		&\hspace{9.6cm}\times \big\|w^{|\tilde{\beta}|}\partial_{\tilde{\beta}}^{\tilde{\alpha}} \mathbf{f}_{n+1}\big\|_{\nu}.
		\end{align*}
		Since \(|\tilde{\beta}_1| = 1\), the fourth term in \eqref{socald73} is bounded by
		\begin{align*}
		&C\left\langle w^{2|\tilde{\beta}|}(\partial_{\tilde{\beta}_1}v_j)\partial^j \partial^{\tilde{\alpha}}_{\tilde{\beta}-\tilde{\beta}_1} \mathbf{f}_{n+1}, \partial_{\tilde{\beta}}^{\tilde{\alpha}} \mathbf{f}_{n+1}\right\rangle_{L^{2}_{x,v}}\\
		&\hspace{0.7cm}\leq \iint_{(\mathbb{R}^3)^2} w^{2|\tilde{\beta}|} \left|(\partial^j \partial^{\tilde{\alpha}}_{\tilde{\beta}-\tilde{\beta}_1} \mathbf{f}_{n+1}: \partial_{\tilde{\beta}}^{\tilde{\alpha}} \mathbf{f}_{n+1})\right| dv dx\\
		&\hspace{0.7cm}\leq \big\|w^{\frac{1}{2} + |\tilde{\beta}|}\partial_{\tilde{\beta}}^{\tilde{\alpha}} \mathbf{f}_{n+1}\big\|_{L^{2}_{x,v}} \times \big\|w^{\frac{1}{2} + (|\tilde{\beta}| - 1)}\partial^j \partial^{\tilde{\alpha}}_{\tilde{\beta}-\tilde{\beta}_1} \mathbf{f}_{n+1}\big\|_{L^{2}_{x,v}}\\
		&\hspace{0.7cm}\leq \eta \big\|w^{ |\tilde{\beta}|}\partial_{\tilde{\beta}}^{\tilde{\alpha}} \mathbf{f}_{n+1}\big\|_{\nu}^2 + C_\eta \big\|w^{(|\tilde{\beta}| - 1)}\partial^j \partial^{\tilde{\alpha}}_{\tilde{\beta}-\tilde{\beta}_1} \mathbf{f}_{n+1}\big\|^2_{\nu},
		\end{align*}
		where \(\|\cdot\|_\nu\) is equivalent to \(\|w^{\frac{1}{2}}\cdot\|_{L^{2}_{x,v}}\) because of \(w = (1 + |v|)^\gamma\).
		
		According to expression \eqref{Dsocald57} for \(\partial^\alpha_\beta \Gamma_{\text{loss}}[f^n, f^{n+1}]\) and \(\partial^\alpha_\beta \Gamma_{\text{gain}}[f^n, f^{n+1}]\), we apply Lemma \ref{estimatesofnonlinears} with \(l = |\beta|\), then it holds that
		\begin{align*}
			&\Big|\sum_{\alpha=A,B}\sum_{\beta=A,B}\left\langle w^{2|\tilde{\beta}|}\Big( \partial_{\tilde{\beta}}^{\tilde{\alpha}}\Gamma^{\alpha\beta}_{\text{gain}}[f_n^{\alpha}, f_n^{\beta}] - \partial_{\tilde{\beta}}^{\tilde{\alpha}}\Gamma^{\alpha\beta}_{\text{loss}}[f_{n+1}^{\alpha}, f_{n}^{\beta}]\Big), \partial_{\tilde{\beta}}^{\tilde{\alpha}} f^{\alpha}_{n+1}\right\rangle_{L^{2}_{x,v}}\Big| \\
			&\hspace{0.4cm}\le C\left[\Big(\sum_{| \tilde{\alpha}_i |+| \tilde{\beta}_i | \leq N} \left\| w^{| \tilde{\beta}_i |} \partial_{\tilde{\beta}_i}^{\tilde{\alpha}_i} \mathbf{f}_{n}\right\|_{L^{2}_{x,v}}\Big)\Big(\sum_{| \tilde{\beta}_1 | \leq | \tilde{\beta} |} \left\| w^{| \tilde{\beta}_1 |} \partial_{\tilde{\beta}_1}^{\tilde{\alpha}_1} \mathbf{f}_{n+1}\right\|_{\nu}\Big)\right. \\
			&\hspace{1.5cm} + \Big(\sum_{| \tilde{\alpha}_i |+| \tilde{\beta}_i | \leq N} \left\| w^{| \tilde{\beta}_i |} \partial_{\tilde{\beta}_i}^{\tilde{\alpha}_i} \mathbf{f}_{n}\right\|_{L^{2}_{x,v}}\Big)\Big(\sum_{| \tilde{\beta}_1 | \leq | \tilde{\beta} |} \left\| w^{| \tilde{\beta}_1 |} \partial_{\tilde{\beta}_1}^{\tilde{\alpha}_1} \mathbf{f}_{n+1}\right\|_{\nu}\Big) \\
			&\hspace{1.5cm} + \left.\Big(\sum_{| \tilde{\alpha}_i |+| \tilde{\beta}_i | \leq N} \left\| w^{| \tilde{\beta}_i |} \partial_{\tilde{\beta}_i}^{\tilde{\alpha}_i} \mathbf{f}_{n}\right\|_{L^{2}_{x,v}}\Big)\Big(\sum_{| \tilde{\beta}_1 | \leq | \tilde{\beta} |} \left\| w^{| \tilde{\beta}_1 |} \partial_{\tilde{\beta}_1}^{\tilde{\alpha}_1} \mathbf{f}_{n}\right\|_{\nu}\Big) \right] \times \left\|w^{ |\tilde{\beta}|}\partial_{\tilde{\beta}}^{\tilde{\alpha}} \mathbf{f}_{n+1}\right\|_{\nu}
		\end{align*}
		where we have chosen \(l - |\tilde{\beta}_1| \le |\tilde{\beta}_2|\) in Lemma \ref{estimatesofnonlinears}, with  \(\tilde{\alpha}_1 \le \tilde{\alpha}\), \(\tilde{\beta}_1 \le \tilde{\beta}\), and \(|\tilde{\alpha}_i| + |\tilde{\beta}_i| \le N\).
		
		 By \eqref{DFenergynorm}, we notice that
		\[
		\sup_{0 \le s \le t}\big\|w^{|\tilde{\beta}|}\partial_{\tilde{\beta}}^{\tilde{\alpha}} \mathbf{f}(s)\big\| \le C \sup_{0 \le s \le t}\mathcal{E}^{\frac{1}{2}}[\mathbf{f}(s)]
		\]
		then integrating \eqref{socald73} over [0, t], we derive
			\begin{align}
				&\frac{1}{2}\big\|w^{|\tilde{\beta}|}\partial_{\tilde{\beta}}^{\tilde{\alpha}} \mathbf{f}_{n+1}(t)\big\|^2_{L^{2}_{x,v}}+\int_{0}^{t}\big\|w^{|\tilde{\beta}|}\partial_{\tilde{\beta}}^{\tilde{\alpha}} \mathbf{f}_{n+1}(s)\big\|^2_{\nu} ds -\frac{1}{2}\big\|w^{|\tilde{\beta}|}\partial_{\tilde{\beta}}^{\tilde{\alpha}} \mathbf{f}_{n+1}(0)\big\|^2_{L^{2}_{x,v}} \notag\\
				&\hspace{0.3cm}\leq \eta \int_{0}^{t} \sum_{\big|\tilde{\beta}_{1}\big|\leq|\tilde{\beta}|}\left(\big\|w^{|\tilde{\beta}_1|}\partial_{\tilde{\beta}_1}^{\tilde{\alpha}} \mathbf{f}_{n+1}(s)\big\|^2_{\nu} + \big\|w^{|\tilde{\beta}_1|}\partial_{\tilde{\beta}_1}^{\tilde{\alpha}} \mathbf{f}_{n}(s)\big\|^2_{\nu} \right)ds \notag\\
				&\hspace{0.6cm}+C_{\eta} \sum_{\left|\tilde{\alpha}_{j}\right|\leq|\tilde{\alpha}|}\sum_{\left|\tilde{\beta}_{j}\right|<|\tilde{\beta}|} \int_{0}^{t} \big\|w^{|\tilde{\beta}_j|}\partial_{\tilde{\beta}_j}^{\tilde{\alpha}_j} \mathbf{f}_{n+1}(s)\big\|^2_{\nu}ds+C_{\eta} \int_{0}^{t}\left\|\partial^{\tilde{\alpha}} \mathbf{f}_{n}(s)\right\|_{\nu}^{2} ds \notag\\
				&\hspace{0.6cm}+C \sup _{0 \leq s \leq t} \mathcal{E}^{\frac{1}{2}}\left(\mathbf{f}_{n}(s)\right) \sum_{\left|\tilde{\alpha}_{1}\right|\leq|\tilde{\alpha}|}\sum_{\left|\tilde{\beta}_{1}\right|\leq|\tilde{\beta}|}\int_{0}^{t} \left(\big\|w^{|\tilde{\beta}_1|}\partial_{\tilde{\beta}_1}^{\tilde{\alpha}_1} \mathbf{f}_{n+1}(s)\big\|_{\nu}+\big\|w^{|\tilde{\beta}_1|}\partial_{\tilde{\beta}_1}^{\tilde{\alpha}_1} \mathbf{f}_{n}(s)\big\|_{\nu}\right) \notag\\
				&\hspace{9.5cm}\times\big\|w^{|\tilde{\beta}|}\partial_{\tilde{\beta}}^{\tilde{\alpha}} \mathbf{f}_{n+1}(s)\big\|_{\nu} ds \notag\\
				&\hspace{0.6cm}+C \sup _{0 \leq s \leq t} \mathcal{E}^{\frac{1}{2}}\left(\mathbf{f}_{n+1}(s)\right) \sum_{\left|\tilde{\alpha}_{1}\right|\leq|\tilde{\alpha}|}\sum_{\left|\tilde{\beta}_{1}\right|\leq|\tilde{\beta}|}\int_{0}^{t} \big\|w^{|\tilde{\beta}_1|}\partial_{\tilde{\beta}_1}^{\tilde{\alpha}_1} \mathbf{f}_{n}(s)\big\|_{\nu}\big\|w^{|\tilde{\beta}|}\partial_{\tilde{\beta}}^{\tilde{\alpha}} \mathbf{f}_{n+1}(s)\big\|_{\nu}ds, \label{socaed74}
			\end{align}
			 where $| \tilde{\alpha}_j |+| \tilde{\beta}_j | \leq N$, $\left|\tilde{\alpha}_{1}\right|\leq|\tilde{\alpha}|$. Since $|\tilde{\beta}_j| < |\tilde{\beta}|$, $\partial_{\tilde{\beta}}^{\tilde{\alpha}} \mathbf{f}_{n+1}$ can be estimated only in term of the $x$ derivatives $\partial^{\tilde{\alpha}} \mathbf{f}_{n+1}$ via an induction starting from $|\tilde{\beta}| = 1, 2,  \dots$ in \eqref{socaed74}. Through this induction process, one gets 
			\begin{align*}
				&\frac{1}{2}\big\|w^{|\tilde{\beta}|}\partial_{\tilde{\beta}}^{\tilde{\alpha}} \mathbf{f}_{n+1}(t)\big\|^2_{L^{2}_{x,v}}+\int_{0}^{t}\big\|w^{|\tilde{\beta}|}\partial_{\tilde{\beta}}^{\tilde{\alpha}} \mathbf{f}_{n+1}(s)\big\|^2_{\nu} ds \\
				&\hspace{0.5cm}\leq C \big\|w^{|\tilde{\beta}|}\partial_{\tilde{\beta}}^{\tilde{\alpha}} \mathbf{f}_{n+1}(0)\big\|^2_{L^{2}_{x,v}} \\
				&\hspace{0.8cm}+ C\eta \int_0^t \sum_{\left|\tilde{\beta}_{1}\right|\leq|\tilde{\beta}|}\left(\big\|w^{|\tilde{\beta}_1|}\partial_{\tilde{\beta}_1}^{\tilde{\alpha}} \mathbf{f}_{n+1}(s)\big\|^2_{\nu} + \big\|w^{|\tilde{\beta}_1|}\partial_{\tilde{\beta}_1}^{\tilde{\alpha}} \mathbf{f}_{n}(s)\big\|^2_{\nu} \right)ds \\
				&\hspace{0.8cm}+ C_\eta \sum_{|\tilde{\alpha}_j| \leq N} \int_0^t  \big\| \partial^{\tilde{\alpha}_j} \mathbf{f}_{n+1}(s) \big\|_{\nu}^2 ds + C_\eta \int_0^t \big\| \partial^{\tilde{\alpha}} \mathbf{f}_{n} (s)\big\|_{\nu}^2 ds\\
				&\hspace{0.8cm}+ C \sup_{0 \leq s \leq t} \mathcal{E}(\mathbf{f}_{n+1}(s)) \sup_{0 \leq s \leq t} \mathcal{E}^{\frac{1}{2}}(\mathbf{f}_{n}(s)) 
			+ C \sup_{0 \leq s \leq t} \mathcal{E}(\mathbf{f}_{n}(s)) \sup_{0 \leq s \leq t} \mathcal{E}^{\frac{1}{2}}(\mathbf{f}_{n+1}(s))
			\end{align*}
			where $|\tilde{\alpha}_j| \leq N$, $\tilde{\beta}_1 \leq \tilde{\beta}$, $\tilde{\alpha}_1 \leq \tilde{\alpha}$, and the inequality 
			\[
			\int_0^t \big\| w^{|\tilde{\beta}|}\partial_{\tilde{\beta}}^{\tilde{\alpha}} \mathbf{f}(s) \big\|^2 ds \leq C \sup_{0 \leq s \leq t} \mathcal{E}(\mathbf{f}(s))
			\]
			is used to estimate the last two terms in \eqref{socaed74}. Summing over $|\tilde{\alpha}| + |\tilde{\beta}| \leq N$,  $(\tilde{\beta} \neq 0)$, then it yields that
				\begin{align}
					&\sum_{|\tilde{\alpha}| + |\tilde{\beta}| \leq N}\left(\frac{1}{2}\big\|w^{|\tilde{\beta}|}\partial_{\tilde{\beta}}^{\tilde{\alpha}} \mathbf{f}_{n+1}(t)\big\|^2_{L^{2}_{x,v}}+\int_{0}^{t}\big\|w^{|\tilde{\beta}|}\partial_{\tilde{\beta}}^{\tilde{\alpha}} \mathbf{f}_{n+1}(s)\big\|^2_{\nu} ds\right)\notag\\
					&\hspace{0.3cm}\leq \sum_{|\tilde{\alpha}| + |\tilde{\beta}| \leq N}\left(C\big\|w^{|\tilde{\beta}|}\partial_{\tilde{\beta}}^{\tilde{\alpha}} \mathbf{f}_{n+1}(0)\big\|^2_{L^{2}_{x,v}} +C\eta \int_{0}^{t} \big\|w^{|\tilde{\beta}|}\partial_{\tilde{\beta}}^{\tilde{\alpha}} \mathbf{f}_{n+1}(s)\big\|^2_{\nu}ds  \right)\notag\\
					&\hspace{0.3cm}\quad + C_\eta \sum_{|\tilde{\alpha}| \leq N} \int_0^t \left( \big\| \partial^{\tilde{\alpha}} \mathbf{f}_{n+1}(s) \big\|_{\nu}^2 + \big\| \partial^{\tilde{\alpha}} \mathbf{f}_{n}(s) \big\|_{\nu}^2\right)ds \notag \\
					&\hspace{0.3cm}\quad + C \sup_{0 \le s \le t} \mathcal{E}(\mathbf{f}_{n+1}(s)) \sup_{0 \le s \le t} \mathcal{E}^{\frac{1}{2}}(\mathbf{f}_{n}(s)) 
					+ C\sup_{0 \le s \le t} \mathcal{E}(\mathbf{f}_{n}(s)) \sup_{0 \le s \le t} \mathcal{E}^{\frac{1}{2}}(\mathbf{f}_{n+1}(s)). \label{sced75}
				\end{align}
		Thus we complete the estimate for the mixed $x,v$- derivatives.
		
		The second step is to deal with  $x$-derivatives of $\mathbf{f}_{n+1}$. Taking $\partial^{\tilde{\alpha}}$ of \eqref{socal71f} gives
		\begin{align}
		&(\partial_t + v \cdot \nabla_x) \partial^{\tilde{\alpha}} f_{n+1}^\alpha + \nu^{\alpha} (\partial^{\tilde{\alpha}} f_{n+1}^\alpha) + K^\alpha(\partial^{\tilde{\alpha}} \mathbf{f}_{n}) \notag\\
		&\hspace{2cm}= \sum_{\beta=A,B}\Big(\partial^{\tilde{\alpha}} \Gamma^{\alpha\beta}_{\text{gain}}[f^{\alpha}_n, f^{\beta}_n] - \partial^{\tilde{\alpha}} \Gamma^{\alpha\beta}_{\text{loss}}[f^{\alpha}_{n+1}, f^{\beta}_n]\Big).
		\end{align}
		By Lemma\ref{estimatesofnonlinears} with $l = 0$, the nonlinear collision terms are bounded by
		\begin{align}
		&\frac{1}{2}\frac{d}{dt}\|\partial^{\tilde{\alpha}} \mathbf{f}_{n+1}\|^2_{L^{2}_{x,v}} + \sum_{\alpha=A,B}\left\langle \nu^{\alpha}\partial^{\tilde{\alpha}} f_{n+1}^{\alpha}, \partial^{\tilde{\alpha}} f_{n+1}^{\alpha} \right\rangle_{L^{2}_{x,v}} + \left\langle  (\mathbf{K}\partial^{\tilde{\alpha}} \mathbf{f}_n), \partial^{\tilde{\alpha}} \mathbf{f}_{n+1}\right\rangle_{L^{2}_{x,v}}  \notag\\
		&\hspace{0.5cm}= \sum_{\alpha=A,B}\sum_{\beta=A,B}\left\langle \big( \partial^{\tilde{\alpha}}\Gamma^{\alpha\beta}_{\text{gain}}[f_n^{\alpha}, f_n^{\beta}] - \partial^{\tilde{\alpha}}\Gamma^{\alpha\beta}_{\text{loss}}[f_{n+1}^{\alpha}, f_{n}^{\beta}]\big), \partial^{\tilde{\alpha}} f^{\alpha}_{n+1}\right\rangle_{L^{2}_{x,v}} \notag\\
		&\hspace{0.6cm}\le C\left[\Big(\sum_{| \tilde{\alpha}_i |+| \tilde{\beta}_i | \leq N} \big\| w^{| \tilde{\beta}_i |} \partial_{\tilde{\beta}_i}^{\tilde{\alpha}_i} \mathbf{f}_{n}\big\|_{L^{2}_{x,v}}\Big)\Big(\sum_{| \tilde{\alpha}_1 | \leq | \tilde{\alpha} |} \left\|  \partial^{\tilde{\alpha}_1} \mathbf{f}_{n}\right\|_{\nu}\Big)\right. \notag\\
		&\hspace{1.7cm} + \Big(\sum_{| \tilde{\alpha}_i |+| \tilde{\beta}_i | \leq N} \big\| w^{| \tilde{\beta}_i |} \partial_{\tilde{\beta}_i}^{\tilde{\alpha}_i} \mathbf{f}_{n}\big\|_{L^{2}_{x,v}}\Big)\Big(\sum_{| \tilde{\alpha}_1 | \leq | \tilde{\alpha} |} \left\|  \partial^{\tilde{\alpha}_1} \mathbf{f}_{n+1}\right\|_{\nu}\Big) \notag\\
		&\hspace{1.7cm} + \left.\Big(\sum_{| \tilde{\alpha}_i |+| \tilde{\beta}_i | \leq N} \big\| w^{| \tilde{\beta}_i |} \partial_{\tilde{\beta}_i}^{\tilde{\alpha}_i} \mathbf{f}_{n+1}\big\|_{L^{2}_{x,v}}\Big)\Big(\sum_{| \tilde{\alpha}_1 | \leq | \tilde{\alpha} |} \big\| \partial^{\tilde{\alpha}_1} \mathbf{f}_{n}\big\|_{\nu}\Big) \right] \times \big\|\partial^{\tilde{\alpha}} \mathbf{f}_{n+1}\big\|_{\nu}.\notag
		\end{align}
		Integrating the above equation over $[0, t]$ leads to
		\begin{align}
			&\frac{1}{2}\|\partial^{\tilde{\alpha}} \mathbf{f}_{n+1}(t)\|^2_{L^{2}_{x,v}} - \frac{1}{2}\|\partial^{\tilde{\alpha}} \mathbf{f}_{n+1}(0)\|^2_{L^{2}_{x,v}} + \int_0^t \|\partial^{\tilde{\alpha}} \mathbf{f}_{n+1}(s)\|^2_{\nu} ds \notag\\
			&\hspace{0.4cm}\leq C \sup_{0 \leq s \leq t} \mathcal{E}^{\frac{1}{2}}(\mathbf{f}_{n}(s)) \int_0^t \Big( \sum_{| \tilde{\alpha}_1 | \leq | \tilde{\alpha} |} \left\|  \partial^{\tilde{\alpha}_1} \mathbf{f}_{n}(s)\right\|_{\nu} + \left\|  \partial^{\tilde{\alpha}_1} \mathbf{f}_{n+1}(s)\right\|_{\nu} \Big) \|\partial^{\tilde{\alpha}} \mathbf{f}_{n+1}(s)\|_{\nu} ds \notag\\
			&\hspace{0.4cm}\quad + C \sup_{0 \leq s \leq t} \mathcal{E}^{\frac{1}{2}}(\mathbf{f}_{n+1}(s)) \int_0^t \Big( \sum_{| \tilde{\alpha}_1 | \leq | \tilde{\alpha} |} \left\|  \partial^{\tilde{\alpha}_1} \mathbf{f}_{n}(s)\right\|_{\nu}\Big) \|\partial^{\tilde{\alpha}} \mathbf{f}_{n+1}(s)\|_{\nu} ds \notag\\
			&\hspace{0.4cm}\quad - \int_0^t \left\langle  (\mathbf{K}\partial^{\tilde{\alpha}} \mathbf{f}_n)(s) , \partial^{\tilde{\alpha}} \mathbf{f}_{n+1}(s) \right\rangle_{L^{2}_{x,v}}ds,\label{scad7777}
		\end{align}
		We notice  from \eqref{socaed1616} that
		\[
		\begin{aligned}
			\Big|\int_0^t \left\langle  (\mathbf{K}\partial^{\tilde{\alpha}} \mathbf{f}_n)(s) , \partial^{\tilde{\alpha}} \mathbf{f}_{n+1}(s) \right\rangle_{L^{2}_{x,v}}ds\Big| &\leq C \int_0^t  \left\|  \partial^{\tilde{\alpha}} \mathbf{f}_{n}(s)\right\|^2_{\nu} + \left\|  \partial^{\tilde{\alpha}} \mathbf{f}_{n+1}(s)\right\|^2_{\nu} ds \\
			&\leq C t \sup_{0 \leq s \leq t} \left( \left\|  \partial^{\tilde{\alpha}} \mathbf{f}_{n}(s)\right\|^2_{L^{2}_{x,v}} + \left\|  \partial^{\tilde{\alpha}} \mathbf{f}_{n+1}(s)\right\|^2_{L^{2}_{x,v}} \right).
		\end{aligned}
		\]
		Together with \eqref{sced75}, this implies that for sufficiently small $\eta$
		\begin{align*}
			\mathcal{E}(\mathbf{f}_{n+1}(t)) &\leq C_0\mathcal{E}(\mathbf{f}_0) + Ct \left( \sup_{0 \leq s \leq t} \mathcal{E}(\mathbf{f}_{n+1})(s) + \sup_{0 \leq s \leq t} \mathcal{E}(\mathbf{f}_n)(s) \right)  \\
			&\quad + C \sup_{0 \leq s \leq t} \mathcal{E}(\mathbf{f}_{n+1})(s) \times\sup_{0 \leq s \leq t} \mathcal{E}^{\frac{1}{2}}(\mathbf{f}_n)(s)  \\
			&\quad + C \sup_{0 \leq s \leq t} \mathcal{E}(\mathbf{f}_n)(s) \times \sup_{0 \leq s \leq t} \mathcal{E}^{\frac{1}{2}}(\mathbf{f}_{n+1})(s)
		\end{align*}
		So, the last term is bounded by
		\begin{equation*}
			\frac{1}{4} \sup_{0 \leq s \leq t} \mathcal{E}(\mathbf{f}_{n+1})(s) + C \sup_{0 \leq s \leq t} \mathcal{E}^{\frac{3}{2}}(\mathbf{f}_n)(s)
		\end{equation*}
		Let $\sup_{0 \leq s \leq t} \mathcal{E}(\mathbf{f}_n(s)) \leq 2C_0\mathcal{E}(\mathbf{f}_0)$. Solving $\mathcal{E}(f^{n+1}(t))$ from above yields
		\begin{equation*}
			\left( \frac{3}{4} - CT^* - C\mathcal{E}^{\frac{1}{2}}(\mathbf{f}_0) \right)\sup_{0 \leq t \leq T^*} \mathcal{E}(\mathbf{f}_{n+1})(t) \leq \left( C_0 + C\mathcal{E}^{\frac{1}{2}}(\mathbf{f}_0) \right)\mathcal{E}^{\frac{1}{2}}(\mathbf{f}_0)
		\end{equation*}
		where $T^*$ and $\mathcal{E}(\mathbf{f}_0)$ are chosen sufficiently small. Therefore, this lemma follows.
		
	\end{proof}

	\begin{theorem}\label{soceadTm4}
		For any small $M > 0$, there exists $T^*(M) > 0$ and $M_1 > 0$, such that if the initial data satisfies 
		\[
		\mathcal{E}(\mathbf{f}^{\rm{in}}) = \sum_{|\tilde{\alpha}| + |\tilde{\beta}| \leq N} \big\| w^{|\tilde{\beta}|} \partial^{\tilde{\alpha}}_{\tilde{\beta}} \mathbf{f}^{\rm{in}} \big\|_2^2 \leq M_1,
		\]
		then there is a unique solution $\mathbf{f}(t, x, v)$ to equations \eqref{socal555EQSO}, such that
		\[
		\sup_{0 \leq t \leq T^*} \mathcal{E}\big[\mathbf{f}(t)\big] \leq M,
		\]
		where $\mathcal{E}\big[\mathbf{f}(t)\big]$ is continuous over $[0, T^*(M))$. If $\mathbf{F}^{\rm{in}}(x, v) = \bm{\mu} + (\bm{\sqrt{\mu}} : \mathbf{f}^{\rm{in}}) \geq 0$, then
		\[
		\mathbf{F}(t, x, v) = \bm{\mu} + (\bm{\sqrt{\mu}} : \mathbf{f}) \geq 0.
		\]
		Furthermore, for $\beta \neq 0$, the following energy estimate holds:
		\begin{align}
			&\sum_{|\tilde{\alpha}| + |\tilde{\beta}| \leq N} \left(\frac{1}{2} \big\|w^{|\tilde{\beta}|}\partial_{\tilde{\beta}}^{\tilde{\alpha}} \mathbf{f}(t)\big\|^2_{L^{2}_{x,v}} + \int_0^t  \big\|w^{|\tilde{\beta}|}\partial_{\tilde{\beta}}^{\tilde{\alpha}} \mathbf{f}(s)\big\|^2_{\nu} ds \right) \notag\\
			&\hspace{0.5cm}\leq C \sum_{|\tilde{\alpha}| + |\tilde{\beta}| \leq N} \left(\big\|w^{|\tilde{\beta}|}\partial_{\tilde{\beta}}^{\tilde{\alpha}} \mathbf{f}^{\rm{in}}\big\|^2_{L^{2}_{x,v}} +  \int_0^t \big\| \partial^{\tilde{\alpha}} \mathbf{f}(s) \big\|^2_{\nu} ds \right)
			 + C \sup_{0 \leq s \leq t} \mathcal{E}\big[\mathbf{f}(s)\big] \sup_{0 \leq s \leq t} \mathcal{E}^{\frac{1}{2}}\big[\mathbf{f}(s)\big].\label{socad7878}
		\end{align}
	\end{theorem}
	\begin{proof}
		Let $\mathbf{F}_{0}(t, x, v)\equiv \mathbf{F}^{\rm{in}}(x, v)$. By taking $n \to \infty$ in Lemma \ref{socedlem7}, we can obtain a classical solution $\mathbf{f}$, such that $\mathbf{F}(t, x, v) = \bm{\mu} + (\bm{\sqrt{\mu}} : \mathbf{f}) \geq 0$.
		
		To discuss the uniqueness of the solution, we suppose there is ‌another solution $\mathbf{g}$, satisfying $\sup_{0 \le t \le T^{*}} \mathcal{E}(\mathbf{g}) \le M$. Under the notation \eqref{Ldeopvec}, \eqref{Blackgamdef}, taking the difference of $\mathbf{f}$ and $\mathbf{g}$ leads to
		\begin{equation}\label{socaled7979}
		\left( \partial_{t} + v \cdot \nabla_{x} \right)[\mathbf{f} - \mathbf{g}] + \bm{L}[\mathbf{f} - \mathbf{g}] = \bm{\Gamma}[\mathbf{f} - \mathbf{g}, \mathbf{f}] + \bm{\Gamma}[\mathbf{g}, \mathbf{f} - \mathbf{g}],
		\end{equation}
		with 
		\begin{equation*}
			\mathbf{f}^{\rm{in}}(x, v) = \mathbf{g}^{\rm{in}}(x, v)
		\end{equation*}
		Notice that there are no derivatives on the right-hand side in \eqref{socaled7979}, we apply \eqref{lem31nonl3f} and \eqref{lem31nonl3s} in Lemma \ref{estimatesofnonlinears} with $\theta = 0$ to obtain
		\begin{align*}
		&\big\langle \bm{\Gamma}[\mathbf{f} - \mathbf{g}, \mathbf{f}] + \bm{\Gamma}[\mathbf{g}, \mathbf{f} - \mathbf{g}], \mathbf{f} - \mathbf{g} \big\rangle_{L^{2}_{x,v}} \\
		&\hspace{0.6cm}\leq C \left( \sum_{|\tilde{\alpha}| + |\tilde{\beta}| \leq 4} \big\| w^{| \tilde{\beta} |} \partial_{\tilde{\beta}}^{\tilde{\alpha}} \mathbf{f}\big\|_{L^{2}_{x,v}} + \big\| w^{| \tilde{\beta} |} \partial_{\tilde{\beta}}^{\tilde{\alpha}} \mathbf{g}\big\|_{L^{2}_{x,v}} \right) \big\|\mathbf{f} - \mathbf{g}\big\|_{\nu}^{2}.
		\end{align*}
		By \eqref{socaed1616} with $l = 0$, it holds that
		\[
		\left\langle \mathbf{L}[\mathbf{f} - \mathbf{g}], \mathbf{f} - \mathbf{g} \right\rangle_{L^{2}_{x,v}} \ge -C \|\mathbf{f} - \mathbf{g}\|_{\nu}^{2},
		\]
		Taking the $L^2$ inner product of both sides of equations \eqref{socaled7979} with the vector $(\mathbf{f} - \mathbf{g})$ leads to
		\begin{align*}
		&\big\|\mathbf{f}(t) - \mathbf{g}(t)\big\|_{L^{2}_{x,v}}^{2} + \int_{0}^{t} \big\|\mathbf{f}(s) - \mathbf{g}(s)\big\|_{\nu}^{2} ds \\
		&\hspace{0.5cm}\leq (C\sqrt{M} + C) \int_{0}^{t} \big\|\mathbf{f}(s) - \mathbf{g}(s)\big\|_{\nu}^{2} ds \\
		&\hspace{0.5cm}\leq C \int_{0}^{t} \big\|\mathbf{f}(s) - \mathbf{g}(s)\big\|_{L^{2}_{x,v}}^{2} ds.
		\end{align*}
	By definition \eqref{DFenergynorm}, $\mathbf{f} \equiv \mathbf{g}$ follows directly by Gronwall's inequality.
		
		To prove the continuity of 
		 $\mathcal{E}\big[\mathbf{f}(t)\big]$ with respect to $t$, take the limit  $n \to \infty$ in both \eqref{socaed74} and \eqref{scad7777} to get
		\[
		\Big|\mathcal{E}\big[\mathbf{f}(t)\big] - \mathcal{E}\big[\mathbf{f}(s)\big]\Big| \le C\left(1 + \sup_{s \le \tau \le t} \sqrt{\mathcal{E}\big[\mathbf{f}(\tau)\big]}\right) \sum_{|\tilde{\alpha}| + |\tilde{\beta}| \leq 4}\int_{s}^{t} \big\| w^{|\tilde{\beta}|} \partial_{\tilde{\beta}}^{\tilde{\alpha}} \mathbf{f}(\tau) \big\|_{\nu}^2 d\tau \to 0,
		\]
		as $t \to s$, since $\big\| w^{|\tilde{\beta}|} \partial_{\tilde{\beta}}^{\tilde{\alpha}} \mathbf{f}(\tau) \big\|^2_{\nu}$ is integrable in $\tau$. 
		
		Finally, taking the limit as $n \to \infty$ in \eqref{sced75} yields \eqref{socad7878}.
		\end{proof}
		The energy estimate derived in \eqref{socad7878} is very important because it shows the necessity of controlling the term $\int_{0}^{t} \left\| \partial^{\tilde{\alpha}} \mathbf{f}(s) \right\|_{\nu}^{2} ds$, while also revealing the need for more refined estimates to spatial derivatives.

		\begin{lemma}\label{socadLemm888}
			Assume $\mathbf{f}(t, x, v)$ satisfies equations \eqref{socal555EQSO} for $0 \le t \le T$ with
			\begin{equation}
			\sup_{0 \le s \le T} \sum_{|\tilde{\alpha}| + |\tilde{\beta}| \leq N} \big\| w^{| \tilde{\beta} |} \partial_{\tilde{\beta}}^{\tilde{\alpha}} \mathbf{f}(s)\big\|_{L^{2}_{x,v}} \leq M,\label{socad8080}
			\end{equation}
			for $0<M \leq 1$ small. Then there exists some $C > 0$ such that for $0 \le s \le t \le T$,
			\begin{equation}
			\sum_{|\tilde{\alpha}| \le N} \big\|\partial^{\tilde{\alpha}} \mathbf{f}(t)\big\|_{L^{2}_{x,v}}^2 + \sum_{|\tilde{\alpha}| \le N} \int_s^t \big\|\partial^{\tilde{\alpha}} \mathbf{f}(\tau)\big\|_{\nu}^2 d\tau \le e^{C(t - s)} \sum_{|\tilde{\alpha}| \le N} \big\|\partial^{\tilde{\alpha}} \mathbf{f}(s)\big\|_{L^{2}_{x,v}}^2.\label{socad8181}
			\end{equation}
			Moreover,
			\begin{equation}
			\sum_{|\tilde{\alpha}| \le N} \big\|w^{\frac{1}{2}} \partial^{\tilde{\alpha}} \mathbf{f}(t)\big\|_{L^{2}_{x,v}}^2 \le e^{C(t - s)} \sum_{|\tilde{\alpha}| \le N} \big\|w^{\frac{1}{2}} \partial^{\tilde{\alpha}} \mathbf{f}(s)\big\|^2_{L^{2}_{x,v}},\label{socad8282}
			\end{equation}
			\begin{equation}
			\sum_{|\tilde{\alpha}| \le N} \int_s^t \big\|\partial^{\tilde{\alpha}} \mathbf{f}(\tau)\big\|_{\nu}^2 d\tau \ge C\left[1 - e^{-C(t - s)}\right] \sum_{|\tilde{\alpha}| \le N} \big\|w^{\frac{1}{2}} \partial^{\tilde{\alpha}} \mathbf{f}(s)\big\|^2_{L^{2}_{x,v}}.\label{socad8383}
			\end{equation}
		\end{lemma}
		\begin{proof}
			First, we prove \eqref{socad8181}. Notice that
			\begin{equation}
			\frac{1}{2}\frac{d}{dt}\big\|\partial^{\tilde{\alpha}} \mathbf{f}(t)\big\|^2_{L^{2}_{x,v}} +   \left\langle \mathbf{L}[\partial^{\tilde{\alpha}} \mathbf{f}], \partial^{\tilde{\alpha}} \mathbf{f} \right\rangle_{L^{2}_{x,v}} = \left\langle \partial^{\tilde{\alpha}} \bm{\Gamma}[\mathbf{f}, \mathbf{f}], \partial^{\tilde{\alpha}} \mathbf{f} \right\rangle_{L^{2}_{x,v}}. 
			\end{equation}
			Since $L$ is a non-negative operator satisfying \eqref{socad8080}, we apply \eqref{lem31nonl3f} and \eqref{lem31nonl3s} with $l = 0$ and $\mathbf{f}_1 = \mathbf{f}_2 = \mathbf{f}_3 = \mathbf{f}$ to get
			\begin{align}
				\frac{1}{2}\frac{d}{dt}\big\|\partial^{\tilde{\alpha}} \mathbf{f}(t)\big\|_{L^{2}_{x,v}}^2 &\leq C\Bigg( \sum_{|\tilde{\alpha}_1|+|\tilde{\beta}_1|\leq N}\big\Vert w^{|\tilde{\beta}_1|}\partial_{\tilde{\beta}_1}^{\tilde{\alpha}_1}\mathbf{f}\big\Vert_{L^{2}_{x,v}} \Bigg)\\
				&\hspace{3.6cm}\times\Bigg( \sum_{|\tilde{\alpha}_2| \le \tilde{\alpha}} \big\|\partial^{\tilde{\alpha}_2} \mathbf{f} \big\|_{\nu} \Bigg)
				\big\|\partial^{\tilde{\alpha}} \mathbf{f}\big\|_{\nu} \notag\\
				&\le C\sqrt{M}\Bigg(\sum_{|\tilde{\alpha}_2| \le \tilde{\alpha}} \big\|\partial^{\tilde{\alpha}_2} \mathbf{f}\big\|_{\nu} \Bigg)\big\|\partial^{\tilde{\alpha}} \mathbf{f}\big\|_{\nu}. 
			\end{align}
			Summing over $|{\tilde{\alpha}}| \leq N$ yields
			\begin{equation*}
			\sum_{|\tilde{\alpha}| \le N}\frac{d}{dt}\big\|\partial^{\tilde{\alpha}} \mathbf{f} \big\|_{L^{2}_{x,v}}^2 \le C\sum_{|\tilde{\alpha}| \le N}\big\|\partial^{\tilde{\alpha}} \mathbf{f} \big\|_{L^{2}_{x,v}}^2.
			\end{equation*}
			Then, by the Gronwall's inequality,
			\begin{equation*}
				\sum_{|\tilde{\alpha}| \le N}\big\|\partial^{\tilde{\alpha}} \mathbf{f}(t)\big\|_{L^{2}_{x,v}}^2 \le e^{C(t-s)}\sum_{|\tilde{\alpha}| \le N}\big\|\partial^{\tilde{\alpha}} \mathbf{f}(s)\big\|_{L^{2}_{x,v}}^2.
			\end{equation*}
			Substituting this estimate into the right-hand side of \eqref{socad8181} completes the proof.
			
			To establish \eqref{socad8282}, we observe that
			\begin{align}
				&\frac{1}{2}\frac{d}{dt}\big\|w^{\frac{1}{2}} \partial^{\tilde{\alpha}} \mathbf{f}\big\|_{L^{2}_{x,v}}^2 + \big\langle w\mathbf{L}[\partial^{\tilde{\alpha}} \mathbf{f}], \partial^{\tilde{\alpha}} \mathbf{f} \big\rangle_{L^{2}_{x,v}}\notag \\
				&\hspace{1.5cm}= \big\langle w\partial^{\tilde{\alpha}} \Gamma\big[\mathbf{f}, \mathbf{f}\big], \partial^{\tilde{\alpha}} \mathbf{f} \big\rangle_{L^{2}_{x,v}} \notag\\
				&\hspace{1.5cm}\le \Bigg( \sum_{| \tilde{\alpha}_i |+| \tilde{\beta}_i | \leq N} \big\| w^{| \tilde{\beta}_i |} \partial_{\tilde{\beta}_i}^{\tilde{\alpha}_i} \mathbf{f}\big\|_{L^{2}_{x,v}} \Bigg)\Bigg( \sum_{|\tilde{\alpha}_2| \le \tilde{\alpha}} \big\|w^{\frac{1}{2}}\partial^{\tilde{\alpha}_2} \mathbf{f}\big\|_{\nu} \Bigg)\big\|w^{\frac{1}{2}}\partial^{\tilde{\alpha}} \mathbf{f}\big\|_{\nu} \notag\\
				&\hspace{1.5cm}\le C\sqrt{M}\Bigg( \sum_{|\tilde{\alpha}_2| \leq \tilde{\alpha}} \big\|w^{\frac{1}{2}}\partial^{\tilde{\alpha}_2} \mathbf{f}\big\|_{\nu} \Bigg)\big\|w^{\frac{1}{2}}\partial^{\tilde{\alpha}} \mathbf{f}\big\|_{\nu}, \label{socad8686}
			\end{align}
			by \eqref{lem31nonl3f} and \eqref{lem31nonl3s} with $l = \frac{1}{2}$.  Notice from \eqref{socaed1616}, we have
			\begin{equation*}
				\left| \left\langle w\mathbf{L}[\partial^{\tilde{\alpha}} \mathbf{f}], \partial^{\tilde{\alpha}} \mathbf{f} \right\rangle_{L^{2}_{x,v}} \right| \leq C \big\|w^{\frac{1}{2}}\partial^{\tilde{\alpha}} \mathbf{f}\big\|^2_{\nu} 
			\end{equation*}
			Then we obtain
			\begin{equation*}
				\sum_{|\tilde{\alpha}| \le N}\frac{d}{dt} \big\|w^{\frac{1}{2}} \partial^{\tilde{\alpha}} \mathbf{f}\big\|_{L^{2}_{x,v}}^2 \le C \sum_{|\tilde{\alpha}| \le N} \big\|w^{\frac{1}{2}} \partial^{\tilde{\alpha}} \mathbf{f}\big\|_{\nu}^2 \leq C \sum_{|\tilde{\alpha}| \leq N} \big\|w^{\frac{1}{2}} \partial^{\tilde{\alpha}} \mathbf{f}\big\|_{L^{2}_{x,v}}^2.
			\end{equation*}
			Consequently, \eqref{socad8282} is a direct result of Gronwall's lemma.
			
			It remains to prove \eqref{socad8383}. We integrate \eqref{socad8686} over $[s,t]$ to obtain
			\begin{gather*}
			\begin{aligned}
			\|w^{\frac{1}{2}} \partial^{\tilde{\alpha}} \mathbf{f}(t)\|_{L^{2}_{x,v}}^2
			 &\geq \|w^{\frac{1}{2}} \partial^{\tilde{\alpha}} \mathbf{f}(s)\|_{L^{2}_{x,v}}^2 - \int_s^t \left\langle w\mathbf{L}[\partial^{\tilde{\alpha}} \mathbf{f}], \partial^{\tilde{\alpha}} \mathbf{f} \right\rangle_{L^{2}_{x,v}} d \tau \\
			&\hspace{0.4cm}- \int_s^t \Big( \sum_{| \tilde{\alpha}_i |+| \tilde{\beta}_i | \leq N} \big\| w^{| \tilde{\beta}_i |} \partial_{\tilde{\beta}_i}^{\tilde{\alpha}_i} \mathbf{f}\big\|_{L^{2}_{x,v}} \Big)\\
			&\hspace{3cm}\times \Big( \sum_{|\tilde{\alpha}_2| \le \tilde{\alpha}} \big\|w^{\frac{1}{2}}\partial^{\tilde{\alpha}_2} \mathbf{f}\big\|_{\nu} \Big) \big\|w^{\frac{1}{2}}\partial^{\tilde{\alpha}} \mathbf{f}\big\|_{\nu} d \tau\\
			&\ge \big\|w^{\frac{1}{2}} \partial^{\tilde{\alpha}} \mathbf{f}(s)\big\|_{L^{2}_{x,v}}^2 - [C + \sqrt{M}] \sum_{|\tilde{\alpha}| \leq N} \int_s^t \big\|w^{\frac{1}{2}}\partial^{\tilde{\alpha}} \mathbf{f}(\tau)\big\|^2_{\nu} d\tau\\
			&\ge \big\|w^{\frac{1}{2}} \partial^{\tilde{\alpha}} \mathbf{f}(s)\big\|_{L^{2}_{x,v}}^2 - C \sum_{|\tilde{\alpha}| \leq N} \int_s^t  \big\|\partial^{\tilde{\alpha}} \mathbf{f}(\tau)\big\|_{\nu}^2 d\tau
			\end{aligned}
			\end{gather*}
			Finally, letting
			\begin{equation*}
				\mathbf{G}(t) \equiv \sum_{|\tilde{\alpha}| \le N} \int_s^t \big\|\partial^{\tilde{\alpha}} \mathbf{f}(\tau)\big\|_{\nu}^2 d\tau
			\end{equation*}
			it follows that
			\begin{equation*}
			\frac{d}{dt}	\mathbf{G}(t) = \sum_{|\tilde{\alpha}| \le N} \big\|\partial^{\tilde{\alpha}} \mathbf{f}(t)\big\|_{\nu}^2 \ge C \big\|w^{\frac{1}{2}}\partial^{\tilde{\alpha}} \mathbf{f}(t)\big\|_{L^{2}_{x,v}}^2 \ge \big\|w^{\frac{1}{2}}\partial^{\tilde{\alpha}} \mathbf{f}(s)\big\|_{L^{2}_{x,v}}^2 - C \mathbf{G}(t)
			\end{equation*}
			because $\|\partial^{\tilde{\alpha}} \mathbf{f}(t)\|_{\nu}^2$ is equivalent to $\|w^{\frac{1}{2}}\partial^{\tilde{\alpha}} \mathbf{f}(t)\|^2_{L^{2}_{x,v}}$. The proof of \eqref{socad8383} is completed by integrating the above first-order differential inequality across $[s,t]$. 
		\end{proof}
		
		\section{Positivity of $\int_{0}^{1} \big\langle \mathbf{L}[\partial^{\tilde{\alpha}} \mathbf{f} (s)], \partial^{\tilde{\alpha}} \mathbf{f} (s) \big\rangle_{L^{2}_{x,v}} \, ds$}
		In this section, we are going to establish the positivity of $\int_{0}^{1} \big\langle \mathbf{L}[\partial^{\tilde{\alpha}} \mathbf{f} (s)], \partial^{\tilde{\alpha}} \mathbf{f} (s) \big\rangle_{L^{2}_{x,v}} \, ds$ for the solution $\mathbf{f}(t,x,v)$ to the Boltzmann equations \eqref{socal555EQSO}. Additionally, it is important to emphasize that conservation laws prove crucial in the proof process.

		\begin{lemma}\label{socedlem9}
		Assume $ \mathbf{f}(t, x, v) $ satisfies \eqref{socal555EQSO} for $ 0 \leq t \leq T $,$ (T \geq 1) $. Suppose that the initial data $ \mathbf{f}^{\rm{in}}(x, v) $ satisfies the conservation laws \eqref{conseerlaw}, and the small amplitude assumption \eqref{socad8080} is also valid for sufficiently small $ 0 <M<1 $. Then there exists a constant $ 0 < \delta_M < 1 $, such that
		\begin{equation}\label{socedlem9ct}
		\sum_{|\tilde{\alpha}| \le N} \int_{0}^{1} \big\langle \mathbf{L}[\partial^{\tilde{\alpha}} \mathbf{f} (s)], \partial^{\tilde{\alpha}} \mathbf{f} (s) \big\rangle_{L^{2}_{x,v}} \, ds \geq \delta_M \sum_{|\alpha| \le N} \int_0^1 \|\partial^{\tilde{\alpha}} \mathbf{f}(s)\|_{\nu}^2 \,ds.
		\end{equation}
		\end{lemma}
		
		We prove by contradiction. Assume that this lemma's conclusion does not hold, then there exists a sequence of solutions $ \mathbf{f}_n(t, x, v) $ to \eqref{socal555EQSO}, such that for $ 0 \le t \le T $, 
		\begin{equation*}
			\sup_{0 \le s \le T} \sum_{|\tilde{\alpha}| + |\tilde{\beta}| \leq N} \big\| w^{| \tilde{\beta} |} \partial_{\tilde{\beta}}^{\tilde{\alpha}} \mathbf{f}_n(s)\big\|_{L^{2}_{x,v}} \leq M,
		\end{equation*}
		and
		$$
		0 \leq   \sum_{|\tilde{\alpha}| \leq N} \int_{0}^{1} \big\langle \mathbf{L}[\partial^{\tilde{\alpha}} \mathbf{f}_n (s)], \partial^{\tilde{\alpha}} \mathbf{f}_n (s) \big\rangle_{L^{2}_{x,v}} \, ds\leq \frac{1}{n} \sum_{|\tilde{\alpha}| \le N} \int_0^1 \| \partial^{\tilde{\alpha}} \mathbf{f}_n \|_{\nu}^2\, ds.
		$$
		Now, we define the normalized sequence of $\mathbf{Z}_n=(Z^A_n,Z^B_n)^T$, where
		\begin{equation}\label{normalized8711}
		Z^{\alpha}_n(t, x, v) \equiv \frac{f^{\alpha}_n(t, x, v)}{\sqrt{\sum_{|\tilde{\alpha}| \leq N} \int_0^1 \| \partial^{\tilde{\alpha}} \mathbf{f}_n(s) \|_{\nu}^2 \,ds}}. 
		\end{equation}
		So we have 
		\begin{equation}\label{normalized8722}
			\sum_{|\tilde{\alpha}| \leq N} \int_0^1 \| \partial^{\tilde{\alpha}} \mathbf{Z}_n(s) \|_{\nu}^2 \,ds = 1, 
		\end{equation}
		which leads to
		\begin{equation}\label{normalized888}
		0 \le \int_0^1 \sum_{|\tilde{\alpha}| \leq N} \big\langle \mathbf{L}[\partial^{\tilde{\alpha}} \mathbf{Z}_n (s)], \partial^{\tilde{\alpha}} \mathbf{Z}_n (s) \big\rangle_{L^{2}_{x,v}} \, ds \leq \frac{1}{n}. 
		\end{equation}
		By Lemma \ref{socadLemm888}, it yields that
		$$
		\sum_{|\tilde{\alpha}| \le N} \big\| w^{\frac{1}{2}} \partial^{\tilde{\alpha}} \mathbf{f}_n(t) \big\|_{L^{2}_{x,v}}^2 \le C \sum_{|\tilde{\alpha}| \le N} \big\| w^{\frac{1}{2}} \partial^{\tilde{\alpha}} \mathbf{f}_n(0) \big\|_{L^{2}_{x,v}}^2,
		$$
		and
		$$
		\sum_{|\tilde{\alpha}| \le N} \int_0^1 \big\| \partial^{\tilde{\alpha}} \mathbf{f}_n(s) \big\|_{\nu}^2 ds \ge C \sum_{|\tilde{\alpha}| \leq N} \big\| w^{\frac{1}{2}} \partial^{\tilde{\alpha}} \mathbf{f}_n(0) \big\|_{L^{2}_{x,v}}^2.
		$$
		We know from \eqref{normalized8711} and \eqref{normalized8722} that
		\begin{equation}
		\sup_{0 \le t \le 1} \sum_{|\tilde{\alpha}| \le N} \big\| w^{\frac{1}{2}} \partial^{\tilde{\alpha}} \mathbf{Z}_n(t) \big\|_{L^{2}_{x,v}}^2 \le C, \label{socad8989}
		\end{equation}
		which is uniform in $ n $. Since $ \partial^{\tilde{\alpha}} \mathbf{f}_n $ satisfies
		$$
		(\partial_t + v \cdot \nabla_x) 
		 \partial^{\tilde{\alpha}} \mathbf{f}_n + \mathbf{L}\big[\partial^{\tilde{\alpha}} \mathbf{f}_n\big] \equiv \partial^{\tilde{\alpha}} \bm{\Gamma}\big[\mathbf{f}_n, \mathbf{f}_n\big].
		$$
		We divide the above equations  by $ \sqrt{\sum_{|\tilde{\alpha}| \le N} \int_0^1 \| \partial^{\tilde{\alpha}} \mathbf{f}_n(s) \|_{\nu}^2 \,ds}$\,  to obtain
		\begin{align}
			&(\partial_t + v \cdot \nabla_x) \partial^\alpha \mathbf{Z}_n + L[\partial^\alpha \mathbf{Z}_n]\notag \\
			&\hspace{1.5cm}= \sum_{|\alpha_1| \le N/2} C_{\tilde{\alpha}}^{\tilde{\alpha}_1} \Gamma[\partial^{\alpha_1} \mathbf{f}_n, \partial^{\alpha - \alpha_1} \mathbf{Z}_n] + \sum_{|\alpha_1| > N/2} C_{\tilde{\alpha}}^{\tilde{\alpha}_1} \Gamma[\partial^{\alpha_1} \mathbf{Z}_n, \partial^{\alpha - \alpha_1} \mathbf{f}_n]. \label{socaed9090}
		\end{align}
		Moreover, $\mathbf{Z}_n$ also satisfies the conservation laws:
		\begin{equation}\label{coserforZZZ}
			\begin{aligned}
				\int_{\mathbb{T}^3} \big\langle \mathbf{Z}_n, (\mathbf{e}_i:\sqrt{\bm{\mu}}) \big\rangle_{L_v^2}  dx  &= 0, \quad i=1,2,\\
				\int_{\mathbb{T}^3} \big\langle \mathbf{Z}_n, (v_j \mathbf{m}:\sqrt{\bm{\mu}}) \big\rangle_{L_v^2} dx &= 0, \quad j=1,2,3,\\
				\int_{\mathbb{T}^3} \big\langle \mathbf{Z}_n,( \frac{|v|^2}{2}\mathbf{m}:\sqrt{\bm{\mu}})   \big\rangle_{L_v^2}dx &= 0.
			\end{aligned}
		\end{equation}
		where $\bm{\mu}=(\mu^A,\mu^B)^T$, $\mathbf{m} = (m^A,m^B )^T$, and $ \mathbf{e}_j$ is the $j^{th}$ unit vector in $\mathbb{R}^2$.
		
		From \eqref{socad8080} and \eqref{socad8989}, we know that both $\{\mathbf{Z}_n(t,x,v)\}$ and $\{\mathbf{f}_n(t,x,v)\}$ have weakly convergent subsequences, denoted as 
		$\mathbf{Z}(t,x,v)$ and $\mathbf{f}(t,x,v)$ respectively. It is easy to verify that their weak limits satisfy
		\begin{equation*}
			\partial^\alpha \mathbf{Z}_n(t,x,v) \rightharpoonup \partial^\alpha \mathbf{Z}(t,x,v)\hspace{0.5cm}  \text{as}\hspace{0.5cm}   n \to \infty ,
		\end{equation*}
		  weakly with respect to the inner product $ \int_0^1 \langle \mathbf{f}(s), \mathbf{f}(s) \rangle_{\nu} ds $, and
		\begin{equation*}
		 w^{| \tilde{\beta} |} \partial_{\tilde{\beta}}^{\tilde{\alpha}} \mathbf{f}_n(t,x,v) \rightharpoonup  w^{| \tilde{\beta} |} \partial_{\tilde{\beta}}^{\tilde{\alpha}} \mathbf{f}(t,x,v) \hspace{0.5cm}  \text{as}\hspace{0.5cm}   n \to \infty ,
		\end{equation*}
		weakly with respect to the function space $ \big[L^2([0,1] \times \mathbb{T}^3 \times \mathbb{R}^3)\big]^2$.
		
		The subsequent proof is divided into two main steps.
		The first step is to prove 
		\begin{equation*}
		 \mathbf{Z}(t,x,v) = (\mathbf{a}(t,x):\sqrt{\bm{\mu}}) + (\mathbf{b}(t,x) \cdot v) (\mathbf{m}:\sqrt{\bm{\mu}}) + c(t,x)|v|^2(\mathbf{m}:\sqrt{\bm{\mu}})
		\end{equation*}
		where $\mathbf{a}=(a^A,a^B)^T$, $\mathbf{b}=(b_1,b_2,b_3)^T$ are two vectors. We also split $ \mathbf{Z}_n(t,x,v) $ by
		$$
		\mathbf{Z}_n(t,x,v) = P_0 \mathbf{Z}_n + \{I - P_0\}\mathbf{Z}_n = \sum_{j=1}^6 \big\langle \mathbf{Z}_n(t,x,\cdot), \tilde{\bm{\chi}}_j \big\rangle_{L^{2}_{v}} \tilde{\bm{\chi}}_j(v) + \{I - P_0\}\mathbf{Z}_n,
		$$
		where $\tilde{\bm{\chi}}_j(v)$ is a orthonormal basis of $\mathbb{R}^3(v)$ defined in \eqref{deforthonormalB}. So it also holds that
		\begin{align}
			\partial^{\tilde{\alpha}} \mathbf{Z}_n(t,x,v) &= P_0 \partial^{\tilde{\alpha}} \mathbf{Z}_n + \{I - P_0\}\partial^{\tilde{\alpha}} \mathbf{Z}_n \notag \\
			&= \sum_{j=1}^6 \big\langle \partial^{\tilde{\alpha}} \mathbf{Z}_n(t,x,v), \tilde{\bm{\chi}}_j \big\rangle_{L^{2}_{v}} \tilde{\bm{\chi}}_j(v) + \{I - P_0\}\partial^{\tilde{\alpha}} \mathbf{Z}_n. \label{socade9393}
		\end{align}
		By \eqref{normalized888} and Lemma \ref{lemma2.2}, as $n\rightarrow \infty$, we deduce that
		\begin{align}
		\sum_{|{\tilde{\alpha}}| \leq N} \int_{0}^{1} \left\| \{I - P_{0}\} \partial^{\tilde{\alpha}} \mathbf{Z}_{n} \right\|_{\nu}^{2} ds \leq \frac{1}{\delta} \sum_{|{\tilde{\alpha}}| \leq N} \int_{0}^{1} \big\langle \mathbf{L}\big[\partial^{\tilde{\alpha}} \mathbf{Z}_{n}\big], \partial^{\tilde{\alpha}} \mathbf{Z}_{n} \big\rangle_{L^{2}_{x,v}} ds \to 0. \label{socade9494}
		\end{align}
		To show that $ \mathbf{Z}(t, x, v) $  is not identically zero, it suffices to show that
		\begin{align}
		\sum_{1 \leq j \leq 6} \int_{0}^{1} \left\| \left\langle \partial^{\tilde{\alpha}} \mathbf{Z}_{n}, \tilde{\bm{\chi}}_j \right\rangle_{L^{2}_{v}} \tilde{\bm{\chi}}_j - \left\langle \partial^{\tilde{\alpha}} \mathbf{Z}, \tilde{\bm{\chi}}_j \right\rangle_{L^{2}_{v}} \tilde{\bm{\chi}}_j \right\|_{\nu}^{2} ds \to 0. \label{socade9595}
		\end{align}
		This reduces to proving that
		\begin{align}
		\int_{0}^{1} \int_{\mathbb{T}^{3}} \left| \left\langle \partial^{\tilde{\alpha}} \mathbf{Z}_{n}, \tilde{\bm{\chi}}_j \right\rangle_{L^{2}_{v}} - \left\langle \partial^{\tilde{\alpha}} \mathbf{Z}, \tilde{\bm{\chi}}_j \right\rangle_{L^{2}_{v}} \right|^{2} dxds \to 0. \label{socade9696}
		\end{align}
		\begin{proof}[Proof of \eqref{socade9696}]
		
		For any $ 1>\eta > 0 $, we define a smooth cutoff function 
		in $ [0, 1] \times \mathbb{T}^{3} \times \mathbb{R}^{3} $ 
		 $$ \chi_{1}(t, x,v) \equiv 1,\hspace{0.5cm}  \text{if} \hspace{0.2cm}(t, x,v)\in  [\eta, 1 - \eta] \times \mathbb{T}^{3} \times ( |v| \leq \frac{1}{\eta}). $$
		Then we split $\big\langle \partial^{\tilde{\alpha}} \mathbf{Z}_{n}(t, x, v), \tilde{\bm{\chi}}_j \big\rangle_{L^{2}_{v}}$ into
		\begin{equation}
		\left\langle \partial^{\tilde{\alpha}} \mathbf{Z}_{n}(t, x, v), \tilde{\bm{\chi}}_j \right\rangle_{L^{2}_{v}} = \big\langle (1 - \chi_{1})\partial^{\tilde{\alpha}} \mathbf{Z}_{n}(t, x, v), \tilde{\bm{\chi}}_j \big\rangle_{L^{2}_{v}} + \big\langle \chi_{1}\partial^{\tilde{\alpha}} \mathbf{Z}_{n}(t, x, v), \tilde{\bm{\chi}}_j \big\rangle_{L^{2}_{v}}. \label{socad9797}
		\end{equation}
		For the first term above, we have
		\begin{align}
			&\int_{0}^{1} \int_{\mathbb{T}^{3}} \left| \left\langle (1 - \chi_{1})\partial^{\tilde{\alpha}} \mathbf{Z}_{n}, \tilde{\bm{\chi}}_j \right\rangle_{L^{2}_{v}} \right|^{2} + \left| \left\langle (1 - \chi_{1})\partial^{\tilde{\alpha}} \mathbf{Z}, \tilde{\bm{\chi}}_j \right\rangle_{L^{2}_{v}} \right|^{2} dx\,ds \notag\\
			&\hspace{0.4cm}\leq C \int_{0}^{1} \iint_{\mathbb{T}^{3} \times \mathbb{R}^{3}} (1 - \chi_{1}) |\partial^{\tilde{\alpha}} \mathbf{Z}_{n}|^{2} |\tilde{\bm{\chi}}_j|\,dv\, dx\, dt \notag\\
			&\hspace{1cm}+ C \int_{0}^{1} \iint_{\mathbb{T}^{3} \times \mathbb{R}^{3}} (1 - \chi_{1}) |\partial^{\tilde{\alpha}} \mathbf{Z}|^{2} |\tilde{\bm{\chi}}_j|\,dv\, dx\, dt \notag\\
			&\hspace{0.4cm}\leq C \Bigg\{\iiint_{0 \le s \le \eta} +  \iiint_{1 - \eta \le s \le 1} +  \iiint_{|v| \geq 1/\eta}\Bigg\}\Big[\big(|\partial^{\tilde{\alpha}} \mathbf{Z}|^{2}+|\partial^{\tilde{\alpha}} \mathbf{Z}_{n}|^{2}\big)(1 - \chi_{1})|\tilde{\bm{\chi}}_j|\Big] dv \,dx\, dt.\label{socd97.597.5}
		\end{align}
		Since $\sup_{s} \big\{\| w^{1/2} \mathbf{Z}_{n} \|_{L^{2}_{x,v}} + \| w^{1/2} \mathbf{Z}\|_{L^{2}_{x,v}}\big\} \le C$, and \eqref{socad8989}, it holds that
		$$
		|\tilde{\bm{\chi}}_j(v)| \leq o(\eta)(1 + |v|)^{\gamma}, \quad \text{for } |v| \ge 1/\eta,
		$$
		so \eqref{socd97.597.5} can be further bounded by
		\begin{equation}
		C\eta \times \sup_{s} \iint_{\mathbb{T}^{3} \times \mathbb{R}^{3}} (1 + |v|)^{\gamma} \big(|\partial^{\tilde{\alpha}} \mathbf{Z}_{n}(s, x, v)|^{2} + |\partial^{\tilde{\alpha}} \mathbf{Z}(s, x, v)|^{2}\big) dv\, dx \leq C\eta. \label{socad9898}
		\end{equation}
		
		Next, we need to prove that the second term $\left\langle \chi_{1}\partial^{\tilde{\alpha}} \mathbf{Z}_{n}(t, x, v), \tilde{\bm{\chi}}_j \right\rangle_{L^{2}_{v}}$ in \eqref{socad9797} is uniformly bounded in $\big[H^{\frac{1}{4}}([0, 1] \times \mathbb{T}^{3})\big]^2$. Notice by \eqref{socaed9090}, $\chi_{1}\partial^{\tilde{\alpha}} \mathbf{Z}_{n}$ satisfies
		\begin{align}
			&(\partial_{t} + v \cdot \nabla_{x})[\chi_{1}\partial^{\tilde{\alpha}} \mathbf{Z}_{n}] \notag\\
			&\hspace{1.5cm}= -\chi_{1} \mathbf{L}\big[\partial^{\tilde{\alpha}} \mathbf{Z}_{n}\big] + \partial^{\tilde{\alpha}} \mathbf{Z}_{n} (\partial_{t} + v \cdot \nabla_{x})\chi_{1} + \sum_{\tilde{\alpha}'} C_{\tilde{\alpha}}^{\tilde{\alpha}'} \bm{\Gamma}\big[\partial^{\tilde{\alpha}'} \mathbf{f}_{n}, \partial^{\tilde{\alpha} - \tilde{\alpha}'} \mathbf{Z}_{n}\big]\chi_{1} \notag\\
			&\hspace{1.9cm} + \sum_{\tilde{\alpha}'} C_{\tilde{\alpha}}^{\tilde{\alpha}'} \Gamma[\partial^{\tilde{\alpha}'} \mathbf{Z}_{n}, \partial^{\tilde{\alpha} - \tilde{\alpha}'} \mathbf{f}_{n}]\chi_{1}. \label{socad9999}
		\end{align}
		
		To show that each term on the right-hand side of \eqref{socad9999} is uniformly bounded in $ \big[L^2([0,1] \times \mathbb{T}^3 \times \mathbb{R}^3)\big]^2 $. Firstly, it can be  observed that
		
		$$
		\partial^{\tilde{\alpha}} \mathbf{Z}_n (\partial_t + v \cdot \nabla_x) \chi_1 \in \big[L^2([0,1] \times \mathbb{T}^3 \times \mathbb{R}^3)\big]^2.
		$$
		 Together with \eqref{socad8080} and \eqref{normalized8722}, and using \eqref{soacl6464} - \eqref{soacl65652}, we obtain that
		
		$$
		\left( \sum_{|\tilde{\alpha}'| \leq N/2} \Gamma[\partial^{\tilde{\alpha}'} \mathbf{f}_n, \partial^{\tilde{\alpha} - \tilde{\alpha}'} \mathbf{Z}_n] + \sum_{|\tilde{\alpha}'| > N/2} \Gamma[\partial^{\tilde{\alpha}'} \mathbf{Z}_n, \partial^{\tilde{\alpha} - \tilde{\alpha}'} \mathbf{f}_n] \right) \chi_1
		$$
		is uniformly bounded in $ \big[L^2([0,1] \times \mathbb{T}^3 \times \mathbb{R}^3)\big]^2 $. Moreover, by \eqref{nuKsuanzidef},

		\begin{align*}
			&\int_{\mathbb{R}^{3}} \big|\chi_1 L^{\alpha}\big[\partial^{\tilde{\alpha}} \mathbf{Z}_n\big]\big|^2 dv \\
			&\hspace{0.5cm}\leq \sum_{\beta=A,B} \Bigg(\int_{\mathbb{R}^{3}} \big|\chi_1 \nu^{\alpha} \partial^{\tilde{\alpha}} Z^{\alpha}_n\big|^2 dv + C \int_{\mathbb{R}^{3}} \Big| \int_{\mathbb{R}^{3}} |v_* - v|^\gamma \sqrt{\mu^\alpha(v)} \sqrt{\mu^\beta(v_*)}\partial^{\tilde{\alpha}} Z_n^\beta(v_*)  dv_* \Big|^2 \chi_1^2\, dv \\
			&\hspace{2.2cm} + C \int_{\mathbb{R}^{3}} \Big| \int_{\mathbb{R}^{3}} |v_* - v|^\gamma \sqrt{\mu^\beta(v_*)}\\
			&\hspace{5.3cm}\times \Big[ \sqrt{\mu^\alpha(v')} \partial^{\tilde{\alpha}} Z^{\beta}_n(v_*') + \sqrt{\mu^\beta(v_*')} \partial^{\tilde{\alpha}} Z_n^{\alpha}(v') \Big] d\omega dv_*' \Big|^2 \chi_1^2(v) dv \Bigg) \\
		&\hspace{0.5cm}\leq \sum_{\beta=A,B} \Bigg( \big|\partial^{\tilde{\alpha}} Z^{\alpha}_{n}\big|_{\nu}^{2} + C \int_{\mathbb{R}^{3}} |v_* - v|^{\gamma} \sqrt{\mu^\alpha(v)} \sqrt{\mu^\beta(v_*)}\big|\partial^{\tilde{\alpha}} Z_n^\beta(v_*)\big|^{2} \chi_{1}^{2} \,dv_*dv \\
		&\hspace{1.6cm}+ C \int_{\mathbb{R}^{3}} |v_*' - v'|^{\gamma} \sqrt{\mu^\beta(v_*)} \Big[ \sqrt{\mu^\alpha(v')}\big|\partial^{\tilde{\alpha}} Z^{\beta}_n(v_*')\big|^{2} + \sqrt{\mu^\beta(v_*')}\big|\partial^{\tilde{\alpha}} Z_n^{\alpha}(v')\big|^{2}\Big] \chi_{1}^{2}(v) dv_*'dv' \Bigg)\\
		&\hspace{0.5cm}\leq C  \big|\partial^{\tilde{\alpha}} \mathbf{Z}_{n}\big|_{\nu}^{2},
		\end{align*}
		where $\chi_{1}$ has compact support. 
		
		So, the right-hand side in \eqref{socad9999} is uniformly bounded in $\big[L^{2}([0, 1] \times \mathbb{T}^{3} \times \mathbb{R}^{3})\big]^2$. Applying the averaging lemma in \cite{avrlem}, it yields that
		$$
		\big\langle \chi_{1}\partial^{\tilde{\alpha}} \mathbf{Z}_{n}(t, x, v), \tilde{\bm{\chi}}_j \big\rangle_{L^{2}_{v}} = \int_{\mathbb{R}^{3}} \chi_{1}(t, x, v)\big(\partial^{\tilde{\alpha}} \mathbf{Z}_{n}(t, x, v):\tilde{\bm{\chi}}_j(v)\big)\, dv \,\in \big[H^{\frac{1}{4}}([0, 1] \times \mathbb{T}^{3})\big]^2
		$$
		which is uniform in $n$. Notice by \eqref{socad8080} and $\sum_{|\tilde{\alpha}| \leq N} \int_{0}^{1} \|\partial^{\tilde{\alpha}} \mathbf{Z}_{n}\|_{\nu}^{2} ds = 1$,  we deduce that
		$$
		\big\langle \chi_{1}\partial^{\tilde{\alpha}} \mathbf{Z}_{n}(t, x, v), \tilde{\bm{\chi}}_j \big\rangle_{L^{2}_{v}} \to \big\langle \chi_{1}\partial^{\tilde{\alpha}} \mathbf{Z}(t, x, v), \tilde{\bm{\chi}}_j \big\rangle_{L^{2}_{v}} \hspace{0.5cm} \text{ in }\hspace{0.2cm} \big[L^{2}([0, 1] \times \mathbb{T}^{3})\big]^{2}.
		$$
		The proofs of \eqref{socade9595} and \eqref{socade9696} follow by combining this with \eqref{socad8080} and \eqref{socad9898}. 
		\end{proof} 
		Consequently, it holds that
		$$
		\int_{0}^{1} \|\partial^{\tilde{\alpha}} \mathbf{Z}_{n} - \partial^{\tilde{\alpha}} \mathbf{Z}\|_{\nu}^{2}\, ds \to 0.
		$$
		Then we have
		\begin{equation}
		\int_{0}^{1} \|\partial^{\tilde{\alpha}} \mathbf{Z}\|_{\nu}^{2}\, ds = 1. \label{socaed100100}
		\end{equation}
		Letting $ n \to \infty $ in \eqref{socade9393}, we obtain
		$$
		\begin{aligned}
			\mathbf{Z}(t,x,v) &= P_0 \mathbf{Z} \\
			&= \sum_{j=1}^6 \big\langle \partial^{\tilde{\alpha}} \mathbf{Z}(t,x,v), \tilde{\bm{\chi}}_j \big\rangle_{L^{2}_{v}} \tilde{\bm{\chi}}_j(v) \\
			&= (\mathbf{a}(t,x):\sqrt{\bm{\mu}}) + (\mathbf{b}(t,x) \cdot v) (\mathbf{m}:\sqrt{\bm{\mu}}) + c(t,x)|v|^2(\mathbf{m}:\sqrt{\bm{\mu}}).
		\end{aligned}
		$$
		Notice that
		\begin{equation*}
			(\mathbf{e}_1:\sqrt{\bm{\mu}}),\hspace{0.3cm} (\mathbf{e}_2:\sqrt{\bm{\mu}}),\hspace{0.3cm} v_i(\mathbf{m}:\sqrt{\bm{\mu}})\,(i=1,2,3), \hspace{0.3cm} |v|^2(\mathbf{m}:\sqrt{\bm{\mu}}),
		\end{equation*}
		 these six vectors above are linearly independent.  Therefore, $ \mathbf{a}(t,x),\,\mathbf{b}(t,x),\,c(t,x) $ can be expressed as a linear combinations of
		$$
		\big\langle  \mathbf{Z}(t,x,v), (\mathbf{e}_1:\sqrt{\bm{\mu}}) \big\rangle_{L^{2}_{v}},\,
		\hspace{0.5cm}\big\langle  \mathbf{Z}(t,x,v), v_i(\mathbf{m}:\sqrt{\bm{\mu}}) \big\rangle_{L^{2}_{v}}\,i=1,2,3,
		$$
		and
		$$
		\big\langle  \mathbf{Z}(t,x,v), (\mathbf{e}_2:\sqrt{\bm{\mu}}) \big\rangle_{L^{2}_{v}},\,\hspace{0.5cm}\big\langle  \mathbf{Z}(t,x,v), |v|^2(\mathbf{m}:\sqrt{\bm{\mu}}) \big\rangle_{L^{2}_{v}}.
		$$
		By \eqref{socad8989}, we deduce that
		\begin{equation}
		\sup_{0 \leq t \leq 1} \sum_{|\tilde{\alpha}| \leq N} \left[ \big|\partial^{\tilde{\alpha}} \mathbf{a}(t)\big|_{L^{2}_{x}}^2 + \big|\partial^{\tilde{\alpha}} \mathbf{b}(t)\big|_{L^{2}_{x}}^2 + \big|\partial^{\tilde{\alpha}} c(t)\big|_{L^{2}_{x}}^2 \right] \leq C. \label{socad101101}
		\end{equation}
		To derive the equations for the limiting function $ \mathbf{Z}(t,x,v) $, it is essential to prove 
		\begin{align}
			&\sum_{|\tilde{\alpha}_1| \leq N/2} \bm{\Gamma}\big[\partial^{\tilde{\alpha}_1} \mathbf{f}_n, \partial^{\tilde{\alpha} - \tilde{\alpha}_1} \mathbf{Z}_n\big] + \sum_{|\tilde{\alpha}_1| > N/2} \bm{\Gamma}\big[\partial^{\tilde{\alpha}_1} \mathbf{Z}_n, \partial^{\tilde{\alpha} - \tilde{\alpha}_1} \mathbf{f}_n\big] \notag\\
			&\hspace{0.4cm}\to \sum_{|\tilde{\alpha}_1| \leq N/2} \bm{\Gamma}\big[\partial^{\tilde{\alpha}_1} \mathbf{f}, \partial^{\tilde{\alpha} - \tilde{\alpha}_1} \mathbf{Z}\big] + \sum_{|\tilde{\alpha}_1| > N/2} \bm{\Gamma}\big[\partial^{\tilde{\alpha}_1} \mathbf{Z}, \partial^{\tilde{\alpha} - \tilde{\alpha}_1} \mathbf{f}\big] \quad \text{in } \mathcal{D}'. \label{secd102102}
		\end{align}

		\begin{proof}[Proof of \eqref{secd102102}]
		We split \eqref{secd102102} into two parts:
		\begin{align*}
		\sum_{|\tilde{\alpha}_1| \leq N/2} \bm{\Gamma}\big[\partial^{\tilde{\alpha}_1} \mathbf{f}_n, \partial^{\tilde{\alpha} - \tilde{\alpha}_1} \mathbf{Z}_n - \partial^{\tilde{\alpha} - \tilde{\alpha}_1} \mathbf{Z}\big] + \sum_{|\tilde{\alpha}_1| > N/2} \bm{\Gamma}\big[\partial^{\tilde{\alpha}_1} \mathbf{Z}_n - \partial^{\tilde{\alpha}_1} \mathbf{Z}, \partial^{\tilde{\alpha} - \tilde{\alpha}_1} \mathbf{f}_n\big]\\
		+ \sum_{|\tilde{\alpha}_1| \leq N/2} \bm{\Gamma}\big[\partial^{\tilde{\alpha}_1} \mathbf{f}_n - \partial^{\tilde{\alpha}_1} \mathbf{f}, \partial^{\tilde{\alpha} - \tilde{\alpha}_1} \mathbf{Z}\big] + \sum_{|\tilde{\alpha}_1| > N/2} \bm{\Gamma}\big[\partial^{\tilde{\alpha}_1} \mathbf{Z}_n, \partial^{\tilde{\alpha} - \tilde{\alpha}_1} \mathbf{f}_n - \partial^{\tilde{\alpha} - \tilde{\alpha}_1} \mathbf{f}\big].
		\end{align*}
		By \eqref{socaed100100} and \eqref{soacl6464}, the first two terms converge to zero in $\mathcal{D}'$. 
		
		Since $\partial^{\tilde{\alpha}_1} \mathbf{f}_n - \partial^{\tilde{\alpha}_1} \mathbf{f}$ and $\partial^{\tilde{\alpha} - \tilde{\alpha}_1} \mathbf{f}_n - \partial^{\tilde{\alpha} - \tilde{\alpha}_1} \mathbf{f}$ converge weakly in $\big[L^2([0,1] \times \mathbb{T}^3 \times \mathbb{R}^3)\big]^2$, the last two terms also converge to zero in $\mathcal{D}'$. We first choose an $\eta$ ‌sufficiently small, and then let $n \to \infty$ in Lemma \ref{lem333scd6}. Consequently, \eqref{secd102102} follows.
		\end{proof}
		Letting $n \to \infty$ in \eqref{socaed9090} and separating the linear parts from the nonlinear parts yields
		\begin{align}
			&(\partial_t + v \cdot \nabla_x) \partial^{\tilde{\alpha}} \mathbf{Z} \notag\\
			&\hspace{1.3cm}= \sum_{|\tilde{\alpha}_1| \leq N/2} C_{\tilde{\alpha}}^{\tilde{\alpha}_1} \Gamma[\partial^{\tilde{\alpha}_1} \mathbf{f}, \partial^{\tilde{\alpha} - \tilde{\alpha}_1} \mathbf{Z}] + \sum_{|\tilde{\alpha}_1| > N/2} C_\alpha^{\tilde{\alpha}_1} \Gamma[\partial^{\tilde{\alpha}_1} \mathbf{Z}, \partial^{\tilde{\alpha} - \tilde{\alpha}_1} \mathbf{f}]\notag \\
			&\hspace{1.3cm}=: \mathbf{h}_{\tilde{\alpha}}, \label{socad103103}
		\end{align}
		Moreover, the conservation laws also hold as:
		\begin{equation}\label{conseerlawFZZ}
			\begin{aligned}
				\int_{\mathbb{T}^3} \big\langle \mathbf{Z}, (\mathbf{e}_i:\sqrt{\bm{\mu}}) \big\rangle_{L_v^2}  dx  &= 0, \quad i=1,2,\\
				\int_{\mathbb{T}^3} \big\langle \mathbf{Z}, (v_j \mathbf{m}:\sqrt{\bm{\mu}}) \big\rangle_{L_v^2} dx &= 0, \quad j=1,2,3,\\
				\int_{\mathbb{T}^3} \big\langle \mathbf{Z}, (\frac{|v|^2}{2}\mathbf{m} :\sqrt{\bm{\mu}})  \big\rangle_{L_v^2}dx &= 0.
			\end{aligned}
		\end{equation}
		
		The second step is to prove 
		\begin{equation*}
			\int_{0}^{1} \sum_{|\tilde{\alpha}| \leq N} \|\partial^{\tilde{\alpha}} \mathbf{Z}(s)\|^{2}_{\nu} \,ds \leq C M.
		\end{equation*}
		If we choose $M$ sufficiently small, this will lead to a contradiction with \eqref{socaed100100}. After substituting
		\begin{equation*}
		\mathbf{Z}(t,x,v)	= (\mathbf{a}(t,x):\sqrt{\bm{\mu}}) + (\mathbf{b}(t,x) \cdot v) (\mathbf{m}:\sqrt{\bm{\mu}}) + c(t,x)|v|^2(\mathbf{m}:\sqrt{\bm{\mu}})
		\end{equation*}
		 into equations \eqref{socad103103}, we obtain
		\begin{equation}
		\begin{aligned}
			&m^{\alpha}\nabla_{x}\partial^{\tilde{\alpha}} c \cdot v|v|^{2}\sqrt{\mu^{\alpha}} +  \{m^{\alpha}\partial^{\tilde{\alpha}} \dot{c}|v|^{2} +m^{\alpha}v \cdot \nabla_{x}(\partial^{\tilde{\alpha}}\mathbf{b} \cdot v)\}\sqrt{\mu^{\alpha}} \\
			&\hspace{1cm}+ \{m^{\alpha}\partial^{\tilde{\alpha}}\mathbf{b} + \nabla_{x}\partial^{\tilde{\alpha}} a^{\alpha}\} \cdot v\sqrt{\mu^{\alpha}} + \partial^{\tilde{\alpha}}\dot{a^{\alpha}}\sqrt{\mu^{\alpha}} = h^{\alpha}_{\tilde{\alpha}}(\mathbf{f}, \mathbf{Z}), \label{socad105105}
		\end{aligned}
		\end{equation}
		where $\alpha \in \{A,B\},\,$$\mathbf{a}(t,x)=(a^A,a^B)^T(t,x)$ and $|\alpha| \leq N$. The dot $`\cdot'$ 
		represents Newton's notation for the derivative with respect to time. 
		Since $\sqrt{\mu^{\alpha}},\ v_i\sqrt{\mu^{\alpha}},\ v_i^2\sqrt{\mu^{\alpha}},\ v_i v_j\sqrt{\mu^{\alpha}},\ v_i|v|^2\sqrt{\mu^{\alpha}}$ are linearly independent, 
		this set of vectors can be transformed into an orthonormal basis $\hat{\bm{\chi}}_k^{\alpha}$ via Gram-Schmidt orthogonalization, which satisfies
		$$
		[\sqrt{\mu^{\alpha}}, v_{i}\sqrt{\mu^{\alpha}}, v_{i}^{2}\sqrt{\mu^{\alpha}}, v_{i}v_{j}\sqrt{\mu^{\alpha}}, |v|^{2}v_{i}\sqrt{\mu^{\alpha}}]A^{\alpha}_{13 \times 13} = [\hat{\bm{\chi}}_k^{\alpha}],
		$$
		‌where $1 \leq i \neq j \leq 3,\, 1 \leq k\leq 13$ and $A^{\alpha}_{13 \times 13}$ is the matrix representing the linear transformation with $\det A^{\alpha}_{13 \times 13} \neq 0$. 
		Then we compare the coefficients of $\sqrt{\mu^{\alpha}}$, $v_{i}\sqrt{\mu^{\alpha}}$, $v_{i}^{2}\sqrt{\mu^{\alpha}}$, $v_{i}v_{j}\sqrt{\mu^{\alpha}}$ and $|v|^{2}v_{i}\sqrt{\mu^{\alpha}}$,  on both sides of \eqref{socad105105} to obtain
		\begin{align}
			m^{\alpha}\partial^i \partial^{\tilde{\alpha}} c =& h_{ci,\tilde{\alpha}}^\alpha,\label{socad10606}\\
			m^{\alpha}\{\partial^{\tilde{\alpha}} \dot{c} + \partial^i \partial^{\tilde{\alpha}} b_i\} =& h_{i,\tilde{\alpha}}^\alpha,\label{socad10707}\\
				m^{\alpha}\{\partial^i \partial^{\tilde{\alpha}} b_j + \partial^j \partial^{\tilde{\alpha}} b_i\} =& h_{ij,\tilde{\alpha}}^\alpha,\quad i \neq j,\label{socad10808}\\
			m^{\alpha}\partial^{\tilde{\alpha}} \dot{b}_i + \partial^i \partial^{\tilde{\alpha}} a^{\alpha} =& h_{bi,\tilde{\alpha}}^\alpha,\label{socad10909}\\
			\partial^{\tilde{\alpha}} \dot{a^{\alpha}} =& h_{a,\tilde{\alpha}}^\alpha,\label{socad110100}
		\end{align}
		where $h_{a,\tilde{\alpha}}^\alpha,\ h_{bi,\tilde{\alpha}}^\alpha,\ h_{i,\tilde{\alpha}}^\alpha,\ h_{ij,\tilde{\alpha}}^\alpha,\ h_{ci,\tilde{\alpha}}^\alpha$  can be expressed as follows:
		$$
		\left[ h_{a,\tilde{\alpha}}^\alpha,\ h_{bi,\tilde{\alpha}}^\alpha,\ h_{i,\tilde{\alpha}}^\alpha,\ h_{ij,\tilde{\alpha}}^\alpha,\ h_{ci,\tilde{\alpha}}^\alpha \right] = \left[ \int_{\mathbf{R}^3} h^\alpha_{\tilde{\alpha}}(t, x, v) \hat{\bm{\chi}}_k^{\alpha}(v) dv \right] A.
		$$
		Applying \eqref{soacl6565},\eqref{soacl65652} and \eqref{socad101101} to $\mathbf{h}_{\tilde{\alpha}}$ in \eqref{socad103103}, we get
		\begin{align}
		&\sum_{|\alpha| \le N}	\sup_{0 \leq s \le 1}  \Big\{ |h_{a,\tilde{\alpha}}^\alpha(s)|_{L_x^2} + |h_{bi,\tilde{\alpha}}^\alpha(s)|_{L_x^2} + |h_{ci,\tilde{\alpha}}^\alpha(s)|_{L_x^2} + |h_{i,\tilde{\alpha}}^\alpha(s)|_{L_x^2} + |h_{ij,\tilde{\alpha}}^\alpha(s)|_{L_x^2} \Big\}\notag\\
		&\hspace{0.4cm}\leq C \sqrt{M}.\label{socad111111}
		\end{align}
		
		Next, we use equations \eqref{socad10606}--\eqref{socad110100} to show that $\mathbf{a}(t,x)$, $\mathbf{b}(t,x)$ and $c(t,x)$ are of the order $O(\sqrt{M})$. By \eqref{socad10606}, $\nabla\partial^{\tilde{\alpha}} c(t, x)$ is bounded by
		\begin{equation}
		\sum_{|\tilde{\alpha}| \le N} |\nabla\partial^{\tilde{\alpha}} c|_{L_x^2} \leq C \sum_i |h_{ci,\tilde{\alpha}}^\alpha|_{L_x^2} \le C \sqrt{M}.\label{socaed112112} 
		\end{equation}
		Now, we turn to $b_i(t, x)$, for $1 \le i \le 3$. Taking $\partial^i$ of \eqref{socad10707} and $\partial^j$ of \eqref{socad10808} yields
		\begin{align*}
			m^{\alpha}\Delta \partial^{\tilde{\alpha}} b_i &= m^{\alpha}\sum_j \partial^{jj} \partial^{\tilde{\alpha}} b_i \\
			&= m^{\alpha}\sum_{j \neq i} \partial^{jj} \partial^{\tilde{\alpha}} b_i + m^{\alpha}\partial^{ii} \partial^{\tilde{\alpha}} b_i \\
			&= m^{\alpha}\sum_{j \neq i} \left[ -\partial^{ji} \partial^{\tilde{\alpha}} b_j + \partial^j h_{ij,\tilde{\alpha}}^\alpha \right] + \left[ \partial^i h_{i,\tilde{\alpha}}^\alpha - m^{\alpha}\partial^i \partial^{\tilde{\alpha}} \dot{c} \right] \\
			&= \sum_{j \neq i} \left[ m^{\alpha}\partial^i \partial^{\tilde{\alpha}} \dot{c} - \partial^i h_{j,\tilde{\alpha}}^\alpha \right] - m^{\alpha}\partial^i \partial^{\tilde{\alpha}} \dot{c} + \sum_{j \neq i} \partial^j h_{ij,\tilde{\alpha}}^\alpha + \partial^i h_{i,\tilde{\alpha}}^\alpha \\
			&= 2m^{\alpha}\partial^i \partial^{\tilde{\alpha}} \dot{c} - m^{\alpha}\partial^i \partial^{\tilde{\alpha}} \dot{c} - \sum_{j \neq i} \left[ \partial^i h_{j,\tilde{\alpha}}^\alpha - \partial^j h_{ij,\tilde{\alpha}}^\alpha \right] + \partial^i h_{i,\tilde{\alpha}}^\alpha \\
			&= m^{\alpha}\partial^i \partial^{\tilde{\alpha}} \dot{c} - \sum_{j \neq i} \left[ \partial^i h_{j,\tilde{\alpha}}^\alpha - \partial^j h_{ij,\tilde{\alpha}}^\alpha \right] + \partial^i h_{i,\tilde{\alpha}}^\alpha\\
			&= -m^{\alpha}\partial^{i i} \partial^{\tilde{\alpha}} b_i + \partial^i h_{i,\tilde{\alpha}}^\alpha - \sum_{j \neq i} \left[ \partial^i h_{j,\tilde{\alpha}}^\alpha - \partial^j h_{ij,\tilde{\alpha}}^\alpha \right] + \partial^i h_{i,\tilde{\alpha}}^\alpha \\
			&= -m^{\alpha}\partial^{i i} \partial^{\tilde{\alpha}} b_i + 2\partial^i h_{i,\tilde{\alpha}}^\alpha - \sum_{j \neq i} \left[ \partial^i h_{j,\tilde{\alpha}}^\alpha - \partial^j h_{ij,\tilde{\alpha}}^\alpha \right].
		\end{align*}
		Multiplying $\Delta \partial^{\tilde{\alpha}} b_i$ with $\partial^{\tilde{\alpha}} b_i$, then integrating over $\mathbb{T}^3$, we obtain
		\begin{equation}
		\sum_{|\tilde{\alpha}| \leq N} \left| \partial^i \partial^{\tilde{\alpha}} \mathbf{b} \right|_{L_x^2} \leq C \sum_{i,j} \left[ | h_{i,\tilde{\alpha}}^\alpha |_{L_x^2} + | h_{ij,\tilde{\alpha}}^\alpha |_{L_x^2} \right] \leq C \sqrt{M}. \label{soedc113113}
		\end{equation}
		An integration over a time interval is needed when estimating $\nabla\mathbf{a}(t,x)$. Without loss of generality, we assume $t \geq \frac{1}{2}$ (otherwise, consider the time interval $[t, 1]$). Integrating \eqref{socad10909} over $[0, t]$ yields,
		\begin{equation}
		\partial^{\tilde{\alpha}} b_i(t) - \partial^{\tilde{\alpha}} b_i(0) + \int_0^t \partial^i \partial^{\tilde{\alpha}} a^{\alpha}(s) ds = \int_0^t h_{bi,\tilde{\alpha}}^\alpha(s) ds. \label{socad114114}
		\end{equation}
		where $0 \leq |\tilde{\alpha}| \le N - 1$. Since $\partial^i \partial^{\tilde{\alpha}} \dot{a^{\alpha}} = \partial^i h_{a,\tilde{\alpha}}^\alpha \in L^2([0, 1] \times \mathbb{T}^3 \times \mathbb{R}^3)$, 
		$$
		\partial^i \partial^{\tilde{\alpha}} a^{\alpha}(s) = \partial^i \partial^{\tilde{\alpha}} a^{\alpha}(t) + \int_t^s \partial^i \partial^{\tilde{\alpha}}  \dot{a^{\alpha}}(\tau) d\tau.
		$$
		Plugging the above equation into \eqref{socad114114},and integrating over $[0, t]$, we get
		\begin{equation}
		\partial^i \partial^{\tilde{\alpha}} a^{\alpha}(t) = -\frac{1}{t} \left[ \partial^{\tilde{\alpha}} b_i(t) - \partial^{\tilde{\alpha}} b_i(0) \right] - \frac{1}{t} \int_0^t \int_t^s \partial^i \partial^{\tilde{\alpha}}  \dot{a^{\alpha}}(\tau) d\tau ds + \frac{1}{t} \int_0^t h_{bi,\tilde{\alpha}}^\alpha ds.\label{extrasec51}
		\end{equation}
		 Applying  $\partial^i$ again to both sides of equation \eqref{extrasec51}, we obtain
		\begin{align}
		\Delta\partial^{\tilde{\alpha}} a^{\alpha}(t) =& -\frac{1}{t} \sum_i \left[ \partial^i \partial^{\tilde{\alpha}} b_i(t) - \partial^i \partial^{\tilde{\alpha}} b_i(0) \right] \notag\\
		&- \frac{1}{t} \int_0^t \int_t^s \Delta h_{a,\tilde{\alpha}}^\alpha(\tau) d\tau ds + \frac{1}{t} \sum_i \int_0^t \partial^i h_{bi,\tilde{\alpha}}^\alpha ds. \label{extrel529}
		\end{align}
		Multiplying $\Delta\partial^{\tilde{\alpha}} a^{\alpha}(t)$ with $\partial^{\tilde{\alpha}} a^{\alpha}$, equation \eqref{extrel529} implies
		$$
		|\nabla \partial^{\tilde{\alpha}} a^\alpha|_{L_x^2} \le C \left[ |\partial^{\tilde{\alpha}} \mathbf{b}|_{L_x^2} + |\nabla_x h_{a,\tilde{\alpha}}^\alpha|_{L_x^2} + |h_{bi,\tilde{\alpha}}^\alpha|_{L_x^2} \right].
		$$
		Therefore, combining the above equality with \eqref{socaed112112}, \eqref{soedc113113} and \eqref{socad111111}, yields
		$$
		\sum_{\tilde{\alpha} \neq 0, |\tilde{\alpha}| \leq N} \int_0^1 \left[ \big|\partial^{\tilde{\alpha}} \mathbf{a}\big|_{L_x^2}^2 + \big|\partial^{\tilde{\alpha}} \mathbf{b}\big|_{L_x^2}^2 + \big|\partial^{\tilde{\alpha}} c\big|_{L_x^2}^2 \right] (s) ds \le C M.
		$$
		Finally, we estimate the $L^2_x$ norm of $\mathbf{a}(t,x), \mathbf{b}(t,x)$ and $c(t,x)$. Taking $\tilde{\alpha} = 0$ in  \eqref{socad10606}--\eqref{socad110100}, yields
		$$
		\int_0^1 \left[ |\dot{\mathbf{a}}|_{L_x^2}^2 + |\dot{\mathbf{b}}|_{L_x^2}^2 + |\dot{c}|_{L_x^2}^2 \right] (s) ds \le C M.
		$$
		By the Poincaré inequality in $[0,1] \times \mathbf{T}^3$, it holds that
		
		\begin{align*}
			&\int_0^1 \Big[ |\mathbf{a}|_{L_x^2}^2 + |\mathbf{b}|_{L_x^2}^2 + |c|_{L_x^2}^2 \Big] (s) ds \\
			&\hspace{2cm}\leq C \int_0^1 \Big[ |\nabla_{t,x} \mathbf{a}|_{L_x^2}^2 + |\nabla_{t,x} \mathbf{b}|_{L_x^2}^2 + |\nabla_{t,x} c|_{L_x^2}^2 \Big] (s) ds \\
			&\hspace{2.5cm} + C|\mathbb{T}^3| \left( \Big| \int_0^1 \int_{\mathbb{T}^3} \mathbf{a} \, dtdx \Big|^2 + \Big| \int_0^1 \int_{\mathbb{T}^3} \mathbf{b} \, dtdx \Big|^2 + \Big| \int_0^1 \int_{\mathbb{T}^3} c \, dtdx \Big|^2 \right) \\
			&\hspace{2cm}\leq CM + C|\mathbb{T}^3| \left( \Big| \int_0^1 \int_{\mathbb{T}^3} \mathbf{a} \, dtdx \Big|^2 + \Big| \int_0^1 \int_{\mathbb{T}^3} \mathbf{b} \, dtdx \Big|^2 + \Big| \int_0^1 \int_{\mathbb{T}^3} c \, dtdx \Big|^2 \right).
		\end{align*}
		But from the conservation laws in \eqref{conseerlawFZZ}, we have $\int_0^1 \int_{\mathbb{T}^3} \mathbf{b}\, dx\,dt = 0$, and
		$$
		\left| \int_0^1 \int_{\mathbb{T}^3}\mathbf{a}(t,x) \,dx\,dt\right| + \left| \int_0^1 \int_{\mathbb{T}^3} c (t,x) \,dx\,dt\right| = 0.
		$$
		This implies that $\int_0^1 \left[ |\mathbf{a}(s)|_{L_x^2}^2 + |\mathbf{b}(s)|_{L_x^2}^2 + |c(s)|_{L_x^2}^2 \right] (s) ds \le CM$, and therefore
		$$
		\sum_{|\tilde{\alpha}| \le N} \int_0^1 \left[ |\partial^{\tilde{\alpha}} \mathbf{a}|_{L_x^2}^2 + |\partial^{\tilde{\alpha}} \mathbf{b}|_{L_x^2}^2 + |\partial^{\tilde{\alpha}} c|_{L_x^2}^2 \right] (s) ds \le CM.
		$$
		This is a contradiction to $\sum_{|\tilde{\alpha}| \leq N} \int_0^1 \| \mathbf{Z} \|_{\nu}^2 \,ds = 1$, if $M$ is chosen to be small. This lemma thus follows. $\square$
		
		Proof of Theorem \ref{maintheorem22}. Assumption \eqref{socad8080} is valid for $0 \le t \le T_{*}$ because 
		$$
		\sup_{0 \le s \le T} \sum_{|\alpha| + |\beta| \le N} \left\| |w|^{|\beta|} \partial_\beta^\alpha f(s) \right\|_2^2 \le \sup_{0 \le s \le T} \mathcal{E}(f(s)) \le M.
		$$
		If $T_{*}\leq 1$, the conclusion follows naturally.
		If $T_{*}> 1$, there exists an integer $n \geq 1$ such that $n< T_{*}\leq n+1$. We apply Lemma \ref{socedlem9} to each interval $[0,  1],[1,  2],\ldots [n-1, n], [n, T_*]$. 
		\section*{6. Global existence}
		
		Proof of Theorem \ref{maintheorem11}. We choose $0 \le M \le 1$ and $ 8 \leq N$, such that both Lemma \ref{socedlem9} and Theorem \ref{soceadTm4} hold. Let the initial data satisfy
		\begin{equation*}
			\mathcal{E}[\mathbf{f}^{\rm{in}}] = \sum_{|\tilde{\alpha}|+|\tilde{\beta}|\leq N}  \frac{1}{2} \big\| w^{|\tilde{\beta}|} \partial_{\tilde{\beta}}^{\tilde{\alpha}} \mathbf{f}_0 \big\|^2_{L^{2}_{x,v}}\leq M_1. 
		\end{equation*}
		as in Theorem \ref{soceadTm4}. Let
		$$
		T = \sup_t \left\{ t : \sup_{0 \le s \le t} \mathcal{E}\big[\mathbf{f}(s)\big] \le M \right\} > 0.
		$$
		First, we estimate the $ x $ derivatives for $ \mathbf{f} $. By \eqref{scad7777} we get
		$$
		\frac{1}{2}\|\partial^{\tilde{\alpha}} \mathbf{f}\|^2_{L^{2}_{x,v}}(t) + \int_{t_1}^t \langle [\partial^{\tilde{\alpha}} \mathbf{f}], \partial^{\tilde{\alpha}} \mathbf{f} \rangle_{L^{2}_{x,v}} ds \leq \frac{1}{2}\|\partial^{\tilde{\alpha}} \mathbf{f}^{\rm{in}}\|^2_{L^{2}_{x,v}} + \int_{0}^t \langle \partial^{\tilde{\alpha}} \bm{\Gamma}[\mathbf{f}, \mathbf{f}], \partial^{\tilde{\alpha}} \mathbf{f} \rangle_{L^{2}_{x,v}} ds.
		$$
		For any $ t = t_1 + n $, where $ 0 \le t_1 \le 1 $ and $ n $ is an integer. By Theorem \ref{maintheorem22}, and the fact the linear operator $\mathbf{L}$ is nonnegative in $[0, t_1]$, we obtain
		\begin{align*}
			&\frac{1}{2}\sum_{|\tilde{\alpha}| \leq N}\|\partial^{\tilde{\alpha}} \mathbf{f}\|^2_{L^{2}_{x,v}}(t) + \delta_M \int_{t_1}^t \sum_{|\tilde{\alpha}| \leq N}\|\partial^{\tilde{\alpha}} \mathbf{f}(s)\|_{\nu}^2 ds + \delta_M \int_{0}^{t_1} \sum_{|\tilde{\alpha}| \leq N}\|\partial^{\tilde{\alpha}} \mathbf{f}(s)\|_{\nu}^2 ds \\
			&\leq C\mathcal{E}[\mathbf{f}^{\rm{in}}] + \int_{0}^t \sum_{|\tilde{\alpha}| \leq N}\langle \partial^{\tilde{\alpha}} \Gamma[\mathbf{f}, \mathbf{f}], \partial^{\tilde{\alpha}} \mathbf{f} \rangle_{L^{2}_{x,v}} ds + \delta_M \int_{0}^{t_1} \sum_{|\tilde{\alpha}| \leq N}\|\partial^{\tilde{\alpha}} \mathbf{f}(s)\|_{\nu}^2 ds \\
			&\leq C\mathcal{E}[\mathbf{f}^{\rm{in}}] + C \sup_{0 \leq s \leq t} \sqrt{\mathcal{E}[\mathbf{f}(s)]} \sum_{|\tilde{\alpha}| \leq N} \int_{0}^t \|\partial^{\tilde{\alpha}} \mathbf{f}(s)\|_{\nu}^2 ds \\
			&+ \delta_M \int_{0}^{t_1} \sum_{|\tilde{\alpha}| \leq N}\|\partial^{\tilde{\alpha}} \mathbf{f}(s)\|_{\nu}^2 ds,
		\end{align*}
		where $ 0 < \delta_M < 1 $. From Lemma \ref{socadLemm888} with $ t_1 \leq 1 $,
		$$
		\sum_{|\tilde{\alpha}| \leq N} \int_{0}^{t_1} \|\partial^{\tilde{\alpha}} \mathbf{f}(s)\|_{\nu}^2 ds \leq C\|\partial^{\tilde{\alpha}} \mathbf{f}^{\rm{in}}\|^2_{L^{2}_{x,v}}.
		$$
		Therefore, there exists $ C_M > 0 $ such that
		\begin{equation}\label{socde116}
		\sup_{0 \leq s \leq t} \sum_{|\tilde{\alpha}| \leq N} \left[ \frac{1}{2} \|\partial^{\tilde{\alpha}} \mathbf{f}(s)\|^2_{L^{2}_{x,v}} + \int_0^s \|\partial^{\tilde{\alpha}} \mathbf{f}(s)\|^2_{\nu} ds \right]
		\le C_M \mathcal{E}[\mathbf{f}^{\rm{in}}] + C_M \left[ \sup_{0 \leq s \le t} \mathcal{E}\big[\mathbf{f}(s)\big] \right]^{3/2}.
		\end{equation}
		Substitute  \eqref{socde116} into the right-hand side of \eqref{socad7878} to get, for some other constant $ C_M \ge 1 $,
		\begin{equation}\label{socde117}
		\sup_{0 \le s \leq t} \mathcal{E}\big[\mathbf{f}(s)\big] \le C_M \mathcal{E}[\mathbf{f}^{\rm{in}}]  + C_M \left[ \sup_{0 \leq s \le t} \mathcal{E}\big[\mathbf{f}(s)\big] \right]^{3/2}.
		\end{equation}
		To construct a global solution, for such chosen $ M > 0, C_M > 0 $, we further choose $ M_2 \le M_1 < M $ such that
		$$
		C_M M_2^{1/2} \leq \frac{1}{2}.
		$$
		Finally let
		$$
		\mathcal{E}[\mathbf{f}^{\rm{in}}] \le \varepsilon_0 \equiv \frac{M_2}{4C_M} \leq \frac{M_2}{2} \leq \frac{M_1}{2} < \frac{M}{2},
		$$
		and let $ C_0 = 2C_M $. By Theorem \ref{soceadTm4},
		$$
		T_2 = \sup_t \left\{ t : \sup_{0 \le s \le t} \mathcal{E}\big[\mathbf{f}(s)\big] \le M_2 \right\} > 0.
		$$
		 Clearly we have $ 0 \le t \le T_2 \le T $ and $ M_2^{1/2} C_M \le 1/2 $. It follows by \eqref{socde117} that
		\begin{align*}
		&\sup_{0 \le s \le T_2} \mathcal{E}\big[\mathbf{f}(s)\big] \leq C_M \mathcal{E}(f_0) + C_M \left[ \sup_{0 \leq s \leq T_2} \mathcal{E}\big[\mathbf{f}(s)\big] \right]^{3/2}\\
		&\hspace{0.4cm}\leq C_M \mathcal{E}[\mathbf{f}^{\rm{in}}] + \left[ \sup_{0 \le s \le T_2} \mathcal{E}\big[\mathbf{f}(s)\big] \right] \left[ C_M \left( \sup_{0 \leq s \leq T_2} \mathcal{E}\big[\mathbf{f}(s)\big] \right)^{1/2} \right]\\
		&\hspace{0.4cm}\leq C_M \mathcal{E}[\mathbf{f}^{\rm{in}}] + \frac{1}{2} \sup_{0 \leq s \le T_2} [\mathcal{E}\big[\mathbf{f}(s)\big]].
		\end{align*}
	Therefore,
		\begin{equation}\label{socdlaste}
		\sup_{0 \leq s \le T_2} \mathcal{E}\big[\mathbf{f}(s)\big] \le 2 C_M \mathcal{E}[\mathbf{f}^{\rm{in}}] = C_0 \mathcal{E}[\mathbf{f}^{\rm{in}}] \le 2 C_M \times \frac{M_2}{4 C_M} =  \frac{M_2}{2} < M_2.
		\end{equation}
		
		Comparing \eqref{socdlaste} with the definition of $ T_2 $, implies $ T_2 = \infty $ from the continuity of $ \mathcal{E}\big[\mathbf{f}(s)\big] $. Consequently, we have completed the proof of Theorem \ref{maintheorem11}.

	\subsection*{Acknowledgements}
	The authors are grateful for the valuable comments made by the referees. This research was
	supported by the National Natural Science Foundation of China (Grant 12271356).

\end{document}